\documentclass[reqno]{amsart}
\usepackage[utf8]{inputenc}
\usepackage{amssymb}
\usepackage{mdframed}
\usepackage{bm,bbm}
\usepackage{euscript}
\usepackage{xcolor}
\usepackage{tikz-cd}
\usepackage{comment}
\usepackage[all,cmtip]{xy}

\usepackage[left= 3.4 cm, right= 3.4 cm, top= 2.8 cm, bottom=2.7 cm, foot= 1.2 cm, marginparwidth=2.7 cm, marginparsep=0.5 cm]{geometry}

\numberwithin{equation}{section}
\usepackage{soul}

\usepackage{hyperref}
\hypersetup{colorlinks}
\definecolor{darkred}{rgb}{0.5,0,0}
\definecolor{darkgreen}{rgb}{0,0.5,0}
\definecolor{darkblue}{rgb}{0,0,0.5}
\hypersetup{colorlinks, linkcolor=darkblue, filecolor=darkgreen, urlcolor=darkred, citecolor=darkblue}
\makeatletter 
\@addtoreset{equation}{section}
\makeatother  

\numberwithin{equation}{section}

\newtheorem{thm}{Theorem}[section]
\newtheorem{cor}[thm]{Corollary}

\newtheorem{prop}[thm]{Proposition}

\newtheorem{lemma}[thm]{Lemma}
\newtheorem{def-lemma}[thm]{Definition-Lemma}
\theoremstyle{definition}
\newtheorem{defn}[thm]{Definition}
\theoremstyle{remark}
\newtheorem{rem}[thm]{Remark}
\newtheorem{hyp}[thm]{Hypothesis}

\newcommand{\beq}{\begin{equation}}
\newcommand{\eeq}{\end{equation}}
\newcommand{\beqn}{\begin{equation*}}
\newcommand{\eeqn}{\end{equation*}}
\newcommand{\ov}{\overline}
\newcommand{\mb}{\mathbb}
\newcommand{\mc}{\mathcal}

\newcommand{\RepOrbTop}{\operatorname{RepOrbTop}}
\newcommand{\pt}{\operatorname{pt}}

\newcommand{\scrC}{\EuScript C}

\newcommand{\scrE}{\EuScript E}
\newcommand{\scrF}{\EuScript F}

\newcommand{\scrH}{\EuScript H}

\newcommand{\scrL}{\EuScript L}
\newcommand{\scrM}{\EuScript M}

\newcommand{\scrT}{\EuScript T}

\newcommand{\bmu}{\bm{mu}}
\newcommand{\bMU}{\bm{MU}}

\newcommand{\ud}{\underline}

\newcommand{\ev}{{\rm ev}}

\makeatletter
\newcommand{\colim@}[2]{%
  \vtop{\m@th\ialign{##\cr
    \hfil$#1\operator@font colim$\hfil\cr
    \noalign{\nointerlineskip\kern1.5\ex@}#2\cr
    \noalign{\nointerlineskip\kern-\ex@}\cr}}%
}
\newcommand{\colim}{%
  \mathop{\mathpalette\colim@{\rightarrowfill@\textstyle}}\nmlimits@
}
\makeatother

\title[Bordism-valued GW]{Bordism and resolution of singularities \\ \Small  with applications to Gromov-Witten theory}

\author{Mohammed Abouzaid}
\address{Department of Mathematics, Stanford University, Stanford, CA, 94305, USA}
\email{abouzaid@stanford.edu}

\author{Shaoyun Bai}
\address{Department of Mathematics, MIT, Cambridge, MA 02139, USA}
\email{shaoyunb@mit.edu}

\begin{document}

\thanks{The first and second authors were respectively supported by NSF grants DMS-2103805 and DMS-2404843}

\begin{abstract}
  We adapt algorithms for resolving the singularities of complex algebraic varieties to prove that the natural map of homology theories from complex bordism to the bordism theory of complex derived orbifolds splits. In equivariant stable homotopy theory, our techniques yield a splitting of homology theories for the map from bordism to the equivariant bordism theory of a finite group $\Gamma$, given by assigning to a manifold its product with $\Gamma$. In symplectic topology, and  using recent work of Abouzaid-McLean-Smith and Hirschi-Swaminathan, we conclude that one can define complex cobordism-valued Gromov-Witten invariant for arbitrary (closed) symplectic manifolds. We apply our results to constrain the topology of the space of Hamiltonian fibrations over $S^2$. The methods we develop apply to normally complex orbifolds, and will hence lead to applications in symplectic topology that rely on moduli spaces of holomorphic curves with Lagrangian boundary conditions.
\end{abstract}

\maketitle

\setcounter{tocdepth}{1}
\tableofcontents

\section{Introduction}

\subsection{Applications: equivariant stable homotopy theory}
\label{sec:appl-equiv-stable}

While our initial motivation for this work arose from symplectic topology (see Section \ref{subsec:intro-application} below), we begin with a discussion of its consequences for equivariant stable homotopy, as they may be familiar to a wider set of potential readers. The reader who is interested largely in the study of (derived) orbifolds and in the methods we use to study them may skip instead to Section \ref{sec:results:-bord-orbisp}.

Associated to each finite group $\Gamma$, and each theory of bordism (oriented, complex, spin, ...) is an equivariant homology theory which assigns to a topological space $X$, equipped with a $\Gamma$-action, the graded abelian group of bordism classes of $\Gamma$-manifolds equipped with a $\Gamma$-equivariant map to $X$.  This theory is naturally a module over the bordism ring of manifolds, arising from taking products. In the case where $X$ is a point, the first extensive investigation of these groups goes back to Conner and Floyd   \cite{ConnerFloyd1964}. In this paper, we specialize our discussion to complex bordism, which we denote by $\Omega^{U,\Gamma}_*(X)$, whose elements are represented by bordism classes of $\Gamma$-manifolds equipped with  an almost complex structure on the direct sum of the tangent bundle with ${\mb R}^k$ (with trivial $\Gamma$-action). 
\begin{rem}
We formulate all our results as statements about homology theories because this will not take us too far afield from the other goals of the paper which are described in the remainder of this introduction. There are two ways to refine the results to statements about spectra: one can either pass to the dual cohomology theories, following the method described in Section \ref{sec:revi-geom-bord}, or one can directly construct representing bordism spectra by extending our constructions from manifolds with boundary to those with corners \cite{Quinn1995,AbouzaidBlumberg2024}.
\end{rem}

\begin{rem}
  In homotopy theoretic literature, there are analogous groups called \emph{geometric bordism groups}, denoted by $\bmu^{\Gamma}_*$, that can be described using Thom spaces on equivariant Grassmanians \cite{wasserman}\footnote{Recently, Schwede has been advocating the use of the notation $\bmu^{\Gamma}_*(X)$ for these groups \cite{schwede}.}. The geometric description of these groups is in terms of  bordism classes of $\Gamma$-manifolds equipped with an almost complex structure on the direct sum of the tangent bundle with ${\mb R}^k \oplus V$, where $V$ is a \emph{complex representation} of $\Gamma$, so that there is a natural morphism of homology theories $ \Omega^{U,\Gamma}_* \to \bmu^{\Gamma}_*$. It follows from Theorem \ref{thm:inclusion_geometric_homotopical_splits} below that this map of equivariant homology theories splits.
\end{rem}

The product of a stably complex manifold $M$ with $\Gamma$ is tautologically equipped with a $\Gamma$-invariant stable complex structure. Assuming that $X$ is a space equipped with the trivial $\Gamma$-action, this construction yields a map $\Omega^{U}_*(X) \to  \Omega^{U,\Gamma}_*(X) $,  which can be readily seen to be multiplicative, and functorial with respect to maps of topological spaces.  In particular, this map is naturally a map of modules over the (ordinary) bordism ring $\Omega^{U}_*$. 
\begin{thm} \label{thm:inclusion_ordinary_geometric_splits} 
  For each finite group $\Gamma$, the map of homology theories $\Omega^{U} \to \Omega^{U,\Gamma} $ splits as  a map of $\Omega^{U}_*$-modules.
\end{thm}
While the proof, based on the results of Sections \ref{sec:abelianize} and \ref{sec:bergh}, will be provided in Section \ref{sec:equivariant-bordisms-no-d}, the basic idea is the following: the map in the other direction assigns to any $\Gamma$-manifold a resolution of the singularities of its quotient orbifold.

\begin{rem}
There is an altogether different map $ \Omega^{U} \to \Omega^{U,\Gamma} $ obtained by equipping a manifold with the trivial $\Gamma$ action, and which has an obvious splitting given by forgetting the group action. These two maps are the special cases of a construction that assigns to each finite $\Gamma$ set $S$ the map $ \Omega^{U} \to \Omega^{U,\Gamma} $ which takes the product of a manifold with $S$. As we indicate in Remark \ref{rem:other_choices_of_splitting}, our results produce a splitting of the map from a direct sum of copies of $\Omega^U$ indexed by irreducible $\Gamma$-sets (i.e. conjugacy classes of subgroups) to $ \Omega^{U,\Gamma} $.

\end{rem}

\begin{rem}
The computation of the homology theories $ \Omega^{U,\Gamma}  $ has been largely open since their introduction by Conner and Floyd \cite{ConnerFloyd1964}. For the theory $ \bmu^{\Gamma}$, the bordism groups of a point, for cyclic groups, were computed in \cite{Kosniowski1983} for cyclic groups  (c.f. the survey \cite{May1996} for a discussion of more general abelian groups).
\end{rem}
\begin{rem}
  The naive generalizations of this result, to oriented bordism and to the Lie group setting, both fail. In the oriented bordism case, this is due to Rosenzweig \cite{Rosenzweig1972} who proved that every torsion element of the oriented bordism group lies in the kernel of the transfer map to the bordism group of oriented manifolds with orientation preserving involutions (the proof is completely explicit, relying on Anderson's presentation of the torsion subgroup of the oriented bordism groups in terms of projective bundles \cite{Anderson1966}).  On the other hand, for Lie groups, one immediately sees that the circle bounds equivariantly. We can, however, prove an analogue of the above result for equivariant bordism of \emph{almost-free} manifolds (i.e. those with the property that all stabilizer groups are finite).
\end{rem}
Most of the work on equivariant bordism theories has been concerned with the following more sophisticated notion \cite{tomdieck1970}: assign to a complex $\Gamma$-representation $V$ and a space $X$ the bordism group $ \Omega^{U,\Gamma}_{* + \dim V}(X \times V)$, generated by $\Gamma$-manifolds equipped with a proper equivariant map to $X \times V$, and define $d \Omega^{U,\Gamma}_*(X)$ to be the colimit of these groups with respect to the maps
\begin{equation}
    \Omega^{U,\Gamma}_{* + \dim V}(X \times V) \to   \Omega^{U,\Gamma}_{* + \dim W}(X \times W) ,
\end{equation}
associated to an inclusion $V \subset W$ of complex $\Gamma$-representations,  which are given by taking the product with a complementary subspace to $V$.  This theory is usually called \emph{homotopical bordism}, and following Schwede, the corresponding groups are denoted in some of the recent literature in equivariant homotopy theory by a variant of the symbol $ \bMU_*^{\Gamma}(X)$.
\begin{thm} \label{thm:inclusion_geometric_homotopical_splits} 
  For each finite group $\Gamma$, the map of equivariant homology theories $\Omega^{U,\Gamma} \to d \Omega^{U,\Gamma}  $ splits as a map of $\Omega^{U}_*$ modules.
\end{thm}
\begin{rem} \label{rem:external_not_internal_multiplicative}
  We suspect that the map we construct fails to be compatible with the module structure over the larger bordism ring $\Omega^{U, \Gamma}_*$. We could, however, prove compatiblity with external multiplication, i.e. the maps generalising the $\Omega^{U}_*$  module map
  \begin{equation}
    \Omega^{U,\Gamma_1} \times  \Omega^{U,\Gamma_2} \to \Omega^{U,\Gamma_1 \times \Gamma_2}. 
  \end{equation}
  The internal multiplication arises from the composition of the external multiplication with the restriction map along the diagonal of $\Gamma \times \Gamma$, and there is no reason for the splitting we construct to be compatible with restriction to subgroups.
\end{rem}
The proof of this result is given in Section \ref{sec:equivariant-bordisms}, building upon the results of Sections \ref{sec:ATW} and \ref{sec:derive-manifolds}. The basic idea is the following: the difference between geometric and homotopical bordism measures the failure of equivariant transversality. By considering higher jets, one may nonetheless formulate a notion of \emph{genericity} for equivariant maps which are modelled after holomorphic maps. Generic fibers of equivariant maps have singularities that are prescribed by representation theoretic data, and that are modelled by algebraic varieties. The splitting is provided by a resolution of singularities of such zero loci.
\begin{rem}
 In the case of abelian groups, the map of coefficients  $\bmu^{\Gamma}_* \to  \bMU_*^{\Gamma}$ is known to be injective \cite{Loffler1973,tomdieck1970}, but we are not aware of any result either at the level of homology theories, or for general groups.
\end{rem}

Our proofs of Theorems \ref{thm:inclusion_ordinary_geometric_splits} and \ref{thm:inclusion_geometric_homotopical_splits} are uniform for all groups $\Gamma$.  Instead of using the language of global homology theories (c.f. \cite{schwede}), we give explicit statements, formulating the compatibility of the second result with two maps on bordism groups associated to a group homomorphism $\Gamma \to \Gamma'$: the first assigns to a $\Gamma$-manifold $M$ the $\Gamma'$-manifold $M \times_{\Gamma} \Gamma'$, yielding, for each $\Gamma$-space $X$, a map from the $\Gamma$-equivariant bordism of $X$ to the $\Gamma'$-equivariant bordism of $X \times_\Gamma \Gamma'$, while the second pulls back the $\Gamma'$-action on a manifold to $\Gamma$,  yielding, for each $\Gamma'$-space $Y$, a map from $\Gamma'$-equivariant bordism to $\Gamma$-equivariant bordism:
\begin{thm}\label{thm:split-change-group-compatible}
There is a distinguished choice of splitting in Theorem \ref{thm:inclusion_geometric_homotopical_splits}, so that the following holds for a group  homomorphism $\Gamma \to \Gamma'$:
  \begin{enumerate}
  \item If the homomorphism is injective, and $(X,A)$ is a pair of $\Gamma$-spaces, the following diagram commutes:
  \begin{equation}
    \begin{tikzcd}
      d \Omega^{U,\Gamma} (X,A) \ar[r] \ar[d] &  d \Omega^{U,\Gamma'} (X \times_{\Gamma} \Gamma', A \times_{\Gamma} \Gamma') \ar[d] \\
       \Omega^{U,\Gamma} (X,A) \ar[r]  & \Omega^{U,\Gamma'} (X \times_{\Gamma} \Gamma', A \times_{\Gamma} \Gamma') .
    \end{tikzcd}
  \end{equation}
  \item If the homomorphism is surjective, and $(Y,B)$ is a $\Gamma'$-space, the following diagram commutes:
  \begin{equation}
    \begin{tikzcd}
      d \Omega^{U,\Gamma'} (Y,B) \ar[r] \ar[d] &  d \Omega^{U,\Gamma} (Y,B) \ar[d] \\
       \Omega^{U,\Gamma'} (Y,B) \ar[r]  & \Omega^{U,\Gamma} (Y,B) .
    \end{tikzcd}
  \end{equation}
  \end{enumerate}
\end{thm}

Returning to Theorem \ref{thm:inclusion_ordinary_geometric_splits}, we begin with the result that asserts the existence of a structure map which is not part of the usual axioms for equivariant homology theories:
\begin{thm}
Given an surjective group homomorphism $\Gamma \to \Gamma'$, there is a natural transformation of homology theories
  \begin{equation} \label{eq:natural_transformation-resolve-singularity}
         \Omega^{U,\Gamma}  \Rightarrow \Omega^{U,\Gamma'}
       \end{equation} 
on the category of pairs of $\Gamma'$-spaces.  This natural transformation is functorial with respect to group homomorphisms in the sense that, given a surjection $\Gamma' \to \Gamma''$, and a pair $(Y,B)$ of $\Gamma''$-spaces, the following diagram commutes:
    \begin{equation}
    \begin{tikzcd}
    \Omega^{U,\Gamma} (Y,B) \ar[dr] \ar[r] &   \Omega^{U,\Gamma'} (Y,B) \ar[d] \\
       & \Omega^{U,\Gamma''} (Y,B) .
    \end{tikzcd}
  \end{equation}       
\end{thm}
As one might guess from the discussion following Theorem \ref{thm:inclusion_ordinary_geometric_splits}, the map assigns to a $\Gamma$-manifold a resolution of the singularties of its quotient by the kernel of the homomorphism. The naturality of the resolution algorithm implies that the result carries a residual $\Gamma'$-action. Our methods do not seem well-adapted to extend this result to homotopical bordism.

The relationship of the above statements to Theorem \ref{thm:inclusion_ordinary_geometric_splits}, as well as to more familiar maps in equivariant homotopy theory, is most easily formulated if we assume that the surjection splits. In that case, given a pair $(Y,B)$ of $\Gamma'$-spaces, we consider the composition:
\begin{equation} \label{eq:composition_two_maps}
    \Omega^{U,\Gamma'} (Y,B) \to \Omega^{U,\Gamma} (Y \times_{\Gamma'} \Gamma, B\times_{\Gamma'} \Gamma) \to \Omega^{U,\Gamma} (Y , B), 
  \end{equation}
  where the first map is induced by the injection $\Gamma' \to \Gamma$, and the second by the surjection $\Gamma \to \Gamma'$. This defines a natural transformation of functors on the category of such pairs. 
\begin{thm} \label{thm:splitting_compatible_surjections}
The composition of the natural transformations of Equations \eqref{eq:composition_two_maps} and \eqref{eq:natural_transformation-resolve-singularity} is the identity on  the category of pairs of $\Gamma'$-spaces.
\end{thm}

\begin{rem}
  There has been some recent work showing that the coefficients of both geometric and homotopical bordism are non-trivial in odd degree \cite{Samperton2022,Kriz2022,AngelSampertonSegoviaUribe2023}. The examples that they study arise from free actions, which can be identified with $\Omega^{U}_*(B\Gamma)$; the composition of the splitting constructed in Theorem \ref{thm:inclusion_ordinary_geometric_splits} (when evaluated on a point) with the natural map from free bordism to equivariant bordism agrees with the map induced by the projection from $B\Gamma$ to the point. In this sense,  Theorem \ref{thm:inclusion_ordinary_geometric_splits} can be seen as largely orthogonal to the existing literature, as it concerns the relationship between the free and non-free part of equivariant bordism.

On the other hand, Theorem \ref{thm:inclusion_geometric_homotopical_splits} implies that a counterexample to the evenness conjecture for $ \Omega^{U,\Gamma}$ directly yields a counterexample for homotopical bordism.
  
\end{rem}

\subsection{Applications: enumerative geometry and symplectic topology}\label{subsec:intro-application}
We now turn attention to our initial motivation for this work. The reader who is interested in understanding the methods of this paper may entirely skip this section, and proceed to Section \ref{sec:results:-bord-orbisp}, where we discuss bordism of orbifolds and derived orbifolds.

Classically, Gromov--Witten invariants of complex algebraic varieties or symplectic manifolds are constructed from virtual fundamental classes of moduli spaces of stable maps. Such a class lives in the rational Chow group or the rational homology group of the corresponding moduli space. It has been speculated since the beginning of Gromov--Witten theory \cite[Section 1.5]{kontsevich-94} that it should be possible to formulate variants of these virtual fundamental classes taking value in complex cobordism. On the symplectic side, early work in the subject \cite{mcduff-87} has already showcased the use of cobordism classes of moduli spaces of pseudo-holomorphic maps to distinguish symplectic structures.

The results of this paper lead to complex bordism valued Gromov--Witten invariants for general symplectic manifolds which can be formulated as follows: given a closed symplectic manifold $(X, \omega)$, let $\ov{\scrM}_{g,k}(X,J,A)$ be the moduli space of stable genus $g$ maps with $k$ marked points which represent a homology class $A$ on $X$, and which are pseudo-holomorphic with respect to an $\omega$-tamed almost complex structure $J$. The space $\ov{\scrM}_{g,k}(X,J,A)$ comes with a continuous map
\begin{equation}
    \ov{\scrM}_{g,k}(X,J,A) \to \ov{\scrM}_{g,k} \times X^k ,
\end{equation}
which is the product of the stabilization map to the  (coarse space of)  Deligne--Mumford space $\ov{\scrM}_{g,k}$ and the evaluation map at the $k$ marked points.

Whenever all points in $\ov{\scrM}_{g,k}(X,J,A) $ are regular, and have trivial isotropy, this moduli space is a closed manifold, and while it does not appear to admit a canonical smooth structure in general, there is an associated well-defined complex cobordism class, arising from a stable isomorphism between the tangent space and the kernel of the linearised $\bar{\partial}$ operator, and the fact that this operator is complex linear up to terms of order $0$. In this (extremely) limited context, which already fails, for example, when considering genus-$0$ Gromov--Witten theory with target $X = \mathbb{C} \mathbb{P}^1$, one obtains a Gromov--Witten invariant valued in the complex cobordism group of $ \ov{\scrM}_{g,k} \times X^k$, which is defined as the image of this class under evaluation. The following result asserts that we may generalize this to all symplectic manifolds:

\begin{thm}\label{thm:intro}
For any closed symplectic manifold $X$, any pair $(g,k)$ of non-negative integers, and any integral second homology class $A$ on $X$, there is a complex cobordism class on $ \ov{\scrM}_{g,k} \times X^k$ which is associated to the moduli space of genus $g$ stable maps with $k$ marked points in homology class $A$. This class is independent of the choice of almost complex structure $J$, and agrees with the class defined by 
$\ov{\scrM}_{g,k}(X,J,A)$ whenever it is a transversely cut out manifold.
\end{thm}

Note that due to the failure of transversality and the nontriviality of automorphisms of stable maps, $\ov{\scrM}_{g,k}(X,J,A)$ is only a compact orbispace in general. The content of the above theorem is that we can construct a stably complex \emph{manifold} from a derived orbifold
/global Kuranishi chart presentation of $\ov{\scrM}_{g,k}(X,J,A)$ such that the resulting bordism class is independent of auxiliary choices. This is achieved by developing recent advances on \emph{normally complex polynomial perturbations} \cite{Bai_Xu_2022} and adapting various methods of \emph{resolution of singularities} in characteristic $0$ to our context. We refer the reader to Theorem \ref{thm:GW-bordism} for a more precise statement, and we will discuss our construction in later part of the introduction.

\begin{rem}
Coates--Givental \cite{Coates-Givental} proposed a definition of rational quantum cobordisms by studying the Chern--Dold character of the virtual tangent bundle of $\ov{\scrM}_{g,k}(X,J,A)$ and integrating it against the rational virtual fundamental class in the spirit of Hirzebruch--Riemann--Roch theorem. In contrast, our definition is integral, and directly constructs stably complex manifolds.
\end{rem}

  As a proof of principle, we use our construction to establish a family version of the homological splitting theorem from \cite{abouzaid2021complex}. For a closed symplectic manifold $(X, \omega)$, denote by $\Omega \mathrm{Ham}(M, \omega)$ the space of loops of Hamiltonian diffeomorphisms which are based at the identity. To each point in $\Omega \mathrm{Ham}(M, \omega)$, we have an associated sweep-out map $S^1 \times X \to X$, which we can formulate as a map of based spaces
\beq
\begin{aligned}
  \Omega \mathrm{Ham}(X, \omega) \wedge  S^1 & \to \mathrm{Aut}(X), 
\end{aligned}
\eeq
where $\mathrm{Aut}$ is the space of self-homotopy equivalences, with basepoint the identity.
Passing to spectra,  we obtain a map
\beq
 \Omega \mathrm{Ham}(X, \omega) \wedge  S^1 \wedge {\mb S}  \to F( X_+\wedge {\mb S},  X_+\wedge {\mb S}). 
 \eeq
 where the notation $F$ stands for the mapping spectrum, taking value in the component of self-homotopy equivalences.
  Our main result is formulated by further passing to complex cobordism:

\begin{thm}\label{thm:splitting-B} 
The map
\beq\label{eqn:B-sweepout}
 \Omega \mathrm{Ham}(X, \omega) \wedge  S^1 \to    F( X_+\wedge {\mb S},  X_+  \wedge MU) 
\eeq
is null-homotopic.
\end{thm}

Composing this result with the inclusion of any point in $ \Omega \mathrm{Ham}(X, \omega)$ recovers \cite[Theorem 1.3]{abouzaid2021complex} as a special case. However, unlike \emph{loc. cit.}, our proof does not depend on any chromatic homotopy theory. In particular, while it would not be too difficult to obtain such a null-homotopy \emph{after $p$-completion} for each prime $p$ from the techniques of \cite{abouzaid2021complex}, we do not know how to produce a compatible collection of null-homotopies from the chromatic construction.

\subsection{Results:  bordism for orbispaces}
\label{sec:results:-bord-orbisp}
The key notions that give rise to the results discussed in Sections \ref{sec:appl-equiv-stable} and \ref{subsec:intro-application} are best formulated in the language of bordism theories for orbispaces.

Following Pardon \cite{pardon20}, we consider the model $\RepOrbTop$ for the category of \emph{orbispaces}, modelled by orbi-CW complexes, whose objects are obtained by gluing products of discs with the orbifolds associated to $\pt/\Gamma$, for finite groups $\Gamma$, along maps which are representable, i.e. induce injections of stabilizer groups. The morphisms in  $\RepOrbTop$ are also assumed to be representable. The restriction to representable morphisms should be considered as a mild assumption, as discussed in the introduction to \cite{pardon20}.

The classical geometric approach to defining bordism as a homology theory (cf. \cite{Atiyah61}) can be imitated for derived orbifolds (cf.  \cite{joyce2007kuranishi, pardon20}), so that one obtains a bordism homology theory. We specifically consider the stably complex orbifold bordism group $\Omega^{U}_*(X) $ of an orbifold CW complex $X$ which is generated by representable maps $M \to X $ whose domain is a closed orbifold, equipped with a stable complex structure, modulo bordism over $X$. This is evidently functorial under representable maps, and can be considered as a homology theory on orbispaces (the precise formulation requires introducing a category of pairs, which we omit to simplify the discussion).

A more elaborate notion is that of a \emph{derived orbifold}, namely a triple $({\mc U}, {\mc E}, {\mc S})$ where ${\mc U}$ is an orbifold and  ${\mc S}: {\mc U} \to {\mc E}$ is a section of a vector bundle ${\mc E} \to {\mc U}$. A derived orbifold is called compact if ${\mc S}^{-1}(0)$ has compact coarse space, and a stable complex structure on $({\mc U}, {\mc E}, {\mc S})$ is a complex structure on the virtual vector bundle $T{\mc U} - {\mc E}$, which is exhibited by third vector bundle ${\mc E'} $ and complex structures on $ T{\mc U} \oplus {\mc E'}$ and $ T{\mc U} \oplus {\mc E'}$. The stably complex derived orbifold group $d \Omega^{U}(X)$ of an orbispace $X$ is defined to be group generated  by compact derived orbifolds with a choice of stable complex structure, whose underlying orbifold ${\mc U}$ is equipped with a representable map to $X$, modulo bordisms over $X$. This homology group is naturally graded by the virtual dimension $\dim {\mc U} - \dim {\mc E}$. The detailed definition is recalled in Section \ref{subsec:d-Omega-C}.

\begin{rem}
In our implementation the geometric applications discussed above, derived orbifolds are obtained as global quotients by compact Lie groups following \cite{abouzaid2021complex}: i.e.  ${\mc U}$ is presented as the quotient of a manifold by an action with finite isotropy. A non-trivial result of Pardon \cite{pardon19} implies that all derived orbifolds arise as global quotients. The earlier literature was built instead on the variants of the notion of a Kuranishi spaces \cite{Fukaya_Ono_integer,pardon-VFC,Abouzaid_Blumberg}, which have their own  associated bordism theories, which are not explicitly formulated in the literature. We expect all such theories to be equivalent, and to agree with derived orbifold bordism.
\end{rem}

Each closed stably complex orbifold $M$ may be considered as a derived orbifold with trivial underlying vector bundle ${\mc E}$. The analogue in the presence of boundary implies that there is a tautological map of homology theories from ordinary to derived bordism:
\beq
\Omega^{U} \to d\Omega^{U}.
\eeq

We can now state the orbispace versions of the equivariant results from Section \ref{sec:appl-equiv-stable}. We start with Theorem \ref{thm:inclusion_geometric_homotopical_splits}:
\begin{thm}\label{thm:main-derived-to-underived} 
There is a natural transformation
\beq\label{eqn:splitting-map-orbi}
{\mc Z}: d\Omega^{U} \to \Omega^{U} 
\eeq
of homology theories on orbispaces {respecting the $\Omega^{U}$-module structures}, which defines a splitting of the inclusion of orbifold bordism in its derived version.  
\end{thm}
\begin{rem}
  This result is apparently stronger than the one we formulated for equivariant bordism theories, as we are asserting linearity with respect to the orbifold bordism group, rather than the bordism group of manifolds as in Theorem \ref{thm:inclusion_geometric_homotopical_splits}. This is a consequence of the fact that the natural analogue of the product of orbifolds, in the equivariant setting, is the \emph{external} product, as discussed in Remark \ref{rem:external_not_internal_multiplicative}.
\end{rem}

In order to state the analogue of Theorem \ref{thm:inclusion_ordinary_geometric_splits}, we need to be able to define orbifold bordism as a homology theory on the category of topological spaces. This can be done by dropping the representability assumption, but, as noted in  \cite{pardon20}, it is more convenient to instead work with the functor $R : \mathrm{Top} \to  \RepOrbTop$ which assigns to each space its product with the universal orbispace, in such a way that the set of homotopy classes of maps $M \to X$ from an orbifold to $X$ is in bijective correspondence with homotopy classes of representable-maps $M \to R(X)$.
\begin{thm} \label{thm:split_manifold_to_orbifold}
  There is a natural transformation
  \begin{equation}
    \label{eq:splitting_orbifold_ordinary}
    \Omega^{U} \circ R \to \Omega^{U}
  \end{equation}
  of homology theories of topological spaces, splitting the inclusion of manifold bordism into its orbifold version, and which respects the module structure over $\Omega^U_* $.
\end{thm}

Combining Theorem \ref{thm:main-derived-to-underived} and Theorem \ref{thm:split_manifold_to_orbifold}, we conclude: 
\begin{thm}\label{thm:main}
There is a natural transformation
\beq\label{eqn:splitting-map}
{\mc Z}: d\Omega^{U} \circ R \to \Omega^{U}
\eeq
of homology theories of topological spaces, whose composite with the tautological map is the identity on manifold bordism.  \qed
\end{thm}

\begin{rem}
This result generalizes the natural transformation of homology theories
\beq
\mc{FOP}: d\Omega^{U} \to H(-;{\mb Z})
\eeq
constructed in \cite[Theorem 1.4]{Bai_Xu_2022}, where $d\Omega^{U}$ is written as $\ov{\Omega}^{\mathrm{der}, {\mb C}}$.
\end{rem}

\begin{rem}
  Since there are orbifolds which cannot be represented as global quotients by finite groups, we cannot expect to describe derived orbifold bordism in terms of equivariant bordism theories for finite groups. However, given Pardon's result \cite{pardon20} relating geometric bordism theory $d\Omega^{U}$ to Schwede's global spectrum $\mathbf{MU}$ \cite{schwede}, from which all equivariant bordism spectra can be recovered, we expect that the above splitting theorem can be interpreted in terms of global homotopy theory, extending the results described in Section \ref{sec:appl-equiv-stable}.
  
  Regarding multiplicative structures, we note that the outcome of \cite{Abouzaid_Blumberg} is the existence of an $A_\infty$ ring spectrum representing derived orbifold bordism; forthcoming work will refine this to an $E_\infty$ ring structure, for which the natural map from complex bordism respects the $E_\infty$ ring map. A more homotopical model of this spectrum should be closely related to Pardon's constructions, leading to the following natural question: does this map split as $E_\infty$ spectra? We expect that the construction of such a map would greatly simplify the putative construction of (Hamiltonian) Floer homotopy type over the complex cobordism, see the discussion in Section \ref{subsec:future}.
\end{rem}

\subsection{Methods of proof}
To connect Theorem \ref{thm:main} with Theorem \ref{thm:intro}, we use recent constructions of global Kuranishi charts for $\ov{\scrM}_{g,k}(X,J,A)$ \cite{abouzaid2021complex, hirschi2022global, AbouzaidMcLeanSmith2023} that produce a well-defined class in $d\Omega_*^{U}(\ov{\scrM}_{g,k} \times X^k)$ from the corresponding derived orbifold presentations of $\ov{\scrM}_{g,k}(X,J,A)$. The complex cobordism class which is the subject of  Theorem \ref{thm:intro} is simply obtained from applying the splitting map ${\mc Z}$ to the aforementioned class. Accordingly, we focus on explaining the construction of the splitting.

As alluded to earlier, our construction uses geometric perturbations of the defining section ${\mc S}$ of a compact stably complex derived orbifold $({\mc U}, {\mc E}, {\mc S})$, which can be seen as following in the footsteps of the classical symplectic approach to Gromov--Witten theory \cite{Li-Tian, fukaya-ono, ruan, siebert1996gromov, HWZ-GW}.  This is in contrast to the recent constructions of virtual fundamental classes in ordinary and generalized homology theories \cite{pardon-VFC,Abouzaid_Blumberg, abouzaid2021complex}, which use Poincar\'e duality and Thom isomorphisms. The perturbation approach we take is a modification of the Fukaya--Ono--Parker (FOP) perturbation scheme developed by the second author and Xu \cite{Bai_Xu_2022}, which draws inspiration from \cite{Fukaya_Ono_integer, BParker_integer}. The main difference with the previous work of the second author with Xu is that we do not use the canonical Whitney stratifications, which is not refined enough for us to resolve the singularities of the zero loci of the perturbed sections.

The output of this perturbation will not in general be a manifold, and in fact not even an orbifold. Indeed, the local model is given by a choice of a finite group $\Gamma$, a pair $V$ and $W$ of finite-dimensional complex representations of $\Gamma$ with $V$ being faithful, and we have ${\mc U} = [V / \Gamma]$, ${\mc E} = [(V \times W) / \Gamma]$. In this situation, the section ${\mc S}$ is determined by a $\Gamma$-equivariant polynomial map $f: V \to W$, and while a  generic choice ensures that the zero-locus $f^{-1}(0)$ is smooth away from the point in $V$ with non-trivial stabilizer, there is no general procedure for guaranteeing that this variety is everywhere smooth. The main insight of this paper is that we can compatibly resolve the resulting singularities. We shall do so in three steps:

We first apply Abramovich--Temkin--W{\l}odarczyk's \emph{functorial embedded resolution of singularities} algorithm \cite{abramovich2019functorial} to build a \emph{compact stably complex orbifold} from a compact stably complex derived orbifold $({\mc U}, {\mc E}, {\mc S})$. In the local model, this amounts to resolving the singularities of  $f^{-1}(0)$. The functoriality of this algorithm ensures that the resolution is in fact $\Gamma$-equivariant, and the corresponding quotient orbifold is the output of this step. In reality, we need to patch such local constructions together, and figure out a way to deal with the non-algebraic nature of a general $({\mc U}, {\mc E}, {\mc S})$. The functoriality of Abramovich--Temkin--W{\l}odarczyk's method is of essential importance in achieving this.

\begin{rem}
Given that Hironaka's original construction of resolution of singularities was not sufficiently functorial, it seems implausible that we would be able to alternatively use it. On the other hand, we expect to be able to use later proofs of resolution of singularities in the lineage of Abhyankar \cite{Abhyankar1982}, at the cost of additional complexity in our arguments. 
\end{rem}

The second step is concerned with making all the stabilizers of a given compact stably complex orbifold \emph{abelian}. Our method follows the work of Bergh--Rydh \cite{bergh2019functorial} and we apply iterative blow-ups to abelianize the orbifold. Roughly speaking, one defines an integer-valued function on the coarse space of an orbifold to measure `how far' the orbifold is from being abelian, and blows up along the locus over which such a function achieves its maximum. 

In the third step, we modify an abelian stably complex orbifold so that the coarse space of the resulting orbifold is a \emph{smooth manifold}. This step heavily relies on Bergh's functorial destackification result \cite{bergh} for algebraic abelian orbifolds, which is rather technical, but the basic idea is simply that the coarse space of the orbifold obtained as the quotient of $\mathbb{C}$ by a cyclic group action is naturally identified again with $\mathbb{C}$, and hence is smooth.

Of course, one needs to ensure that each step is compatible with the equivalence relations entering the definition of $d\Omega_*^{U}$, and that invariance holds despite the various auxiliary choice.

\subsection{Outlook and further context}\label{subsec:future}
Although we state our main result for the homology theory $d\Omega_*^{U}$, our construction actually applies to \emph{normally complex orbifolds}. Such a structure on an orbifold ${\mc U}$ consists of the data of stable complex structures on the normal bundles of all strata, which are compatible with inclusions. If $({\mc U}, {\mc E}, {\mc S})$ is a compact derived orbifold with ${\mc U}$ normally complex and ${\mc E}$ complex, we obtain a cobordism class of smooth manifolds from the algorithms in Section \ref{subsec:splitting-mani} and Section \ref{subsec:splitting}.

The applicability of our results to normally complex orbifolds opens up much wider applications in symplectic topology. In fact, all the moduli spaces of pseudo-holomorphic curves appearing in Gromov--Witten theory, Hamiltonian Floer theory, and Lagrangian Floer theory are expected to be presented as the zero loci of the section from certain normally complex derived orbifolds. The essential observation is that all the nontrivial automorphisms of stable maps with at least one marked point  are caused by closed irreducible component of the domain, and accordingly, the nontrivial (virtual) representations produced from the equivariant Fredholm indices of the linear deformation operator are complex.

On the Hamiltonian Floer side, the availability of global Kuranishi charts and the package of derived orbifold flow categories together with the results in this paper allows us to envision a Hamiltonian Floer homotopy type as an $MU$-module for general symplectic manifolds. However, due to the failure of multiplicativity under Cartesian products of all the resolution methods we use, our result can not be directly applied to produce such a homotopy type. We expect that a more refined certain variants of our method, adapted to the specific stratifications appearing in Hamiltonian Floer theory, will allow us to obtain such a refinement, and we shall return to this task in the future.

Finally, it is an open problem to define complex bordism/algebraic cobordism valued Gromov--Witten invariants using algebraic geometry building on \cite{li-tian-algebraic, bf97}. There is a fundamental step in our approach (Lemma \ref{lem:local-FOP-regular}) which uses a general position argument in the smooth category, so it is unclear if the methods of this paper can be reproduced purely in the algebraic category. Nevertheless, our invariants $(\ov{\mathbf{M}}_{g.k}(X,A), \mathrm{st} \times \mathrm{ev})$ are invariants of deformation classes of smooth projective algebraic varieties. Developing calculation tools and exploring structural aspects of these invariants provides another set of open problems.

\subsection{Outline of the paper}
This paper is organized as follows. In Section \ref{sec:prelim}, we recall basic notions about orbifolds, including normal complex structures, and FOP perturbations using the topological stack point of view. In the subsequent sections, we discuss different steps of the resolution of singularities algorithm, ordered in what we perceive to be increasing complexity.
\begin{itemize}
    \item Following \cite{bergh2019functorial}, we present how to functorially make the stabilizers of a normally complex orbifold abelian by a sequence of ordinary blow-ups in Section \ref{sec:abelianize}.
    \item In Section \ref{sec:bergh}, we adapt the approach in \cite{bergh} to our setting to make the coarse space of an abelian normally complex orbifold smooth by a combination of ordinary blow-ups and root stack constructions.
    \item Before proceeding to the constructions involving derived orbifolds, we describe, in Section \ref{sec:equivariant-bordisms-no-d}, an algorithm using the preceding results to extract a smooth, stably complex manifold, from a stably complex orbifold.
    \item After recalling the results on functorial embedded resolution of singularities from \cite{abramovich2019functorial} and verifying various implications on the resolutions of the universal zero loci in Section \ref{sec:ATW}, we show how to combine the FOP perturbation scheme with resolution of singularities to produce a normally complex orbifold from a normally complex derived orbifold in Section~\ref{sec:derive-manifolds}.
    \item  Section \ref{sec:split} assembles the discussion of the preceding two section to associate a stably complex orbifold to a stably complex derived orbifold. The results stated in Section \ref{sec:appl-equiv-stable} follow as corollaries.
    \item Finally, in Section \ref{sec:applications}, we provide proofs of the applications stated in Section \ref{subsec:intro-application}.
\end{itemize}

\subsection*{Acknowledgement} The second author would like to express his most sincere gratitude to Guangbo Xu for the collaboration \cite{Bai_Xu_2022, bai2022arnold}, which made it possible to start thinking about this work. We are grateful to John Pardon for crystallizing the direct approach to resolve the singularities of normally complex orbifolds, which replaces an early cumbersome attempt. We thank Dan Abramovich, David Rydh, and Jaroslaw W{\l}odarczyk for helpful email correspondences which helped us better understand resolution of singularities in algebraic geometry, and John Greenlees for comments about an early draft. We are also grateful to the Simons Laufer Mathematical Sciences Institute for its hospitality in Fall 2022, while part of this work was conducted.

\section{Normally complex orbifolds}\label{sec:prelim}

\subsection{Orbifolds and vector bundles}
We recall basic notions about orbifolds using the language of topological stacks following the expositions in \cite{pardon19,pardon20}. We choose not to formulate our construction in terms of Satake and Thurston's simpler point of view because we cannot avoid introducing ineffective suborbifolds in our subsequent constructions. However, nearly all the orbifolds appearing as ambient spaces are effective, except for setting up different versions of orbifold bordism theory.

Let $\mathbf{Top}$ be the category with topological spaces as objects and continuous maps as morphisms. Denote by $\mathbf{Grpd}$ the $2$-category of essentially small groupoids. A \emph{stack} is a functor $F: \mathbf{Top}^{\mathrm{op}} \to \mathbf{Grpd}$ satisfying descent, namely, for a topological space $U$ with an open cover $\{ U_\alpha \}$, the functor
\beq
F(U) \to \mathrm{Eq}\big( F(U_{\alpha}) \rightrightarrows F(U_{\alpha} \cap U_{\beta}) \mathrel{\substack{\textstyle\rightarrow\\[-0.6ex]
                      \textstyle\rightarrow \\[-0.6ex]
                      \textstyle\rightarrow}} F(U_{\alpha} \cap U_{\beta} \cap U_{\gamma}) \big)
\eeq
is an equivalence. Stacks form a $2$-category denoted by $\mathrm{Shv}(\mathbf{Top}, \mathbf{Grpd})$. The Yoneda embedding $\mathbf{Top} \to \mathrm{Shv}(\mathbf{Top}, \mathbf{Grpd})$ defined by taking $U \in \mathbf{Top}$ to $\mathrm{Hom}_{\mathbf{Top}}(-,U)$ is a fully faithful embedding, with (essential) image called \emph{representable} stacks.

A morphism of stacks ${\mc X} \to {\mc Y}$ is called \emph{representable} if for every topological space $U$, viewed as a stack via Yoneda embedding, and any morphism $U \to {\mc Y}$, the fiber product ${\mc X} \times_{\mc Y} U$ is representable. Given a property $P$ of maps of topological spaces which is preserved under pullback, a representable morphism ${\mc X} \to {\mc Y}$ is said to satisfy property $P$ if for every topological space $U$ and any morphism $U \to {\mc Y}$, the map from the fiber product ${\mc X} \times_{\mc Y} U \to U$ has property $P$. Properties preserved under pullbacks include being injective, surjective, open, a (closed) embedding, separated, proper, admitting local sections, and \'etale (finite covering space). As an example, we say that $\{ {\mc X}_{\alpha} \to {\mc Y} \}$ is an \emph{open cover} if each for any topological space $U$ and morphism $U \to {\mc Y}$, each of the map ${\mc X}_\alpha \times_{\mc Y} U \to U$ is an open embedding, and the collection of  open subsets $\{ {\mc X}_\alpha \times_{\mc Y} U \}$ forms an open cover of $U$.

The \emph{coarse space} $|{\mc X}|$ of a stack ${\mc X}$,  is the topological space characterized as the initial object in the category of maps from ${\mc X}$ to the Yoneda image of $\mathbf{Top}$, which is equivalently understood as the the image of ${\mc X}$ under the left adjoint of the Yoneda embedding. The coarse space $|{\mc X}|$ is also called the set of \emph{points} on ${\mc X}$, because it indeed consists of isomorphism classes of morphisms of the form $* \to {\mc X}$. The \emph{isotropy group} or \emph{stabilizer} of a point $x \in |{\mc X}|$ is the automorphism group of the morphism $* \to {\mc X}$ induced by $x$. A stack is called \emph{compact} if its coarse space $|{\mc X}|$ is a compact topological space.

\begin{defn}
\begin{enumerate}
    \item A stack ${\mc X}$ is called \emph{topological} if there exists a topological space $U$ and a representable surjective morphism $U \to {\mc X}$ admitting local sections. Under such setting, $U$ is called an \emph{atlas}.
    \item A \emph{(separated) orbispace} is a topological stack ${\mc X}$ admitting an \emph{\'etale} atlas $U \to {\mc X}$ such that the diagonal morphism ${\mc X} \to {\mc X} \times {\mc X}$ is proper.
    \item An orbispace ${\mc X}$ is called a \emph{topological orbifold} if it admits an \'etale atlas $U \to {\mc X}$ such that $U$ is locally homeomorphic to ${\mb R}^n$. Here $n$ is defined as the dimension of ${\mc X}$.
    \item A \emph{(smooth) orbifold} is a  topological orbifold ${\mc X}$ equipped with a \emph{smooth structure}: i.e., a choice of \'etale atlas $U \to {\mc X}$ defining the topological orbifold structure together with a choice of smooth structure on $U$ such that the two smooth structures on $U \times_{\mc X} U$ defined by pulling back the smooth structure on $U$ via the two projections agree.
    \item A map of smooth orbifolds $f: {\mc X} \to {\mc Y}$ is called \emph{smooth} if there exist smooth \'etale atlases $U_{\mc X} \to {\mc X}$ and $U_{\mc Y} \to {\mc Y}$ such that $f$ can be lifted to a smooth map $U_{\mc X} \to U_{\mc Y}$.
\end{enumerate}
\end{defn}

Topological stacks can be constructed from topological groupoids. Recall that a topological groupoid consists of a pair of topological spaces $M$ and $O$, with source and target maps $s,t: M \to O$, composition morphism $m: M \times_{O} M \to M$ (here the fiber product is defined using $s$ and $t$), as well an an identity morphism $i: O \to M$ (the identity map), satisfying the axioms of a groupoid (in particular, the existence of an inverse $\iota: M \to M$). A topological groupoid $M \rightrightarrows O$ gives rise to a stack $[M \rightrightarrows O]$ as follows: for a topological space $U'$, the groupoid associated to $[M \rightrightarrows O](U')$ has objects given by open covers $\{U'_{\alpha}\}$ of $U'$ endowed with maps $U'_\alpha \to O$ and $U'_{\alpha} \cap U'_{\beta} \to M$ satisfying suitable compatibility conditions. The (iso)morphisms of two such $\{U'_{\alpha}\}$ and $\{U''_{\alpha'}\}$ are specified by maps $U'_{\alpha} \cap U''_{\alpha'}$ satisfying certain compatibility conditions. It is straightforward to check that $[M \rightrightarrows O]$ satisfies descent. The natural map $O \to [M \rightrightarrows O]$ is defined by evaluating at $U'$ with the trivial cover $\{ U'\}$, and the fiber product $O \times_{[M \rightrightarrows O]} O$ is identified with $M$. One can show that $O \to [M \rightrightarrows O]$ is representable by applying the local sections of $O \to [M \rightrightarrows O]$ to the representable morphism $M \cong O \times_{[M \rightrightarrows O]} O \to O$.

On the other hand, for a topological stack ${\mc X}$ with an atlas $U \to {\mc X}$, we can look at the topological space $U \times_{\mc X} U$. Let $s,t: U \times_{\mc X} U \to U$ be the projection map to be first and second component respectively. Let $m: (U \times_{\mc X} U) \times_{U} (U \times_{\mc X} U) = U \times_{\mc X} U \times_{\mc X} U \to U \times_{\mc X} U$ be the map forgetting the middle factor, and let $i: U \to U \times_{\mc X} U$ be the diagonal map.  This data defines a topological groupoid, with inverse given by the map $\iota: U \times_{\mc X} U \to U \times_{\mc X} U$ exchanging the two factors, and the topological stack associated to this groupoid $[U \times_{\mc X} U \rightrightarrows U]$ is isomorphic to ${\mc X}$. Indeed, there exists a natural map $[U \times_{\mc X} U \rightrightarrows U] \to {\mc X}$ by the descent property of ${\mc X}$ which defines an equivalence.

As a classical example, given a topological space $X$ acted on continuously by a topological group $G$, the \emph{action groupoid} $G \times X \rightrightarrows X$ has source and target map given by $(g,x) \mapsto x$ and $(g,x) \mapsto g \cdot x$ respectively, composition map given by  $m: (g_1, g_2, x) \mapsto (g_1 \cdot g_2, x)$, and identity map $i: x \mapsto (e,x)$. The resulting stack is denoted by $X // G$ coming with a morphism $X \to X//G$. The coarse space $|X // G|$ is the quotient $X/G$.

Using the equivalence exhibited as above, we see that:

\begin{lemma}\label{lem:orb-groupoid}
A topological groupoid $(M, O, s,t,i,\iota,m)$ gives rise an $n$-dimensional smooth orbifold if the following holds:
\begin{itemize}
    \item both $M$ and $O$ are second countable Hausdorff topogical spaces locally homeomorphic to ${\mb R}^n$, equipped with smooth structures, so that the transition functions between charts are smooth;
    \item both the source and target map $s,t$ are smooth and define local diffeomorphisms (the \emph{\'etale} property);
    \item all the continuous maps $i,\iota,m$ are smooth with respect to the above smooth structure;
    \item the map $s \times t: M \to O \times O$ is proper (the \emph{separatedness} property). \qed
\end{itemize}
\end{lemma}

This result makes connection with the exposition in, e.g., \cite[Section 1.4]{ALR07}, in which this is referred to as a \emph{proper \'etale Lie groupoid}. Note that in \emph{loc. cit.}, an orbifold is defined to be a topological space $|{\mc X}|$ equipped with a \emph{(Morita) equivalence class} of orbifold structures, i.e., a proper \'etale Lie groupoid such that the coarse space of the associated topological stack is homeomorphic to $|{\mc X}|$. In the sequel, we will use the notation ${\mc U}$ to denote smooth orbifolds. We do not distinguish between isomorphic smooth structures on ${\mc U}$. In the special setting of the action groupoid from a smooth action of a finite group on a smooth manifold, we call the corresponding orbifold a \emph{finite quotient orbifold}.

\begin{prop}\label{prop:orbi-chart}
A stack ${\mc U}$ is an orbifold if and only if the coarse space $|{\mc U}|$ is Hausdorff and it admits an open cover by finite quotient orbifolds $\{ U_\alpha // \Gamma_{\alpha} \to {\mc U} \}$.

Accordingly, any smooth \'etale atlas $U \to {\mc U}$ admits a refinement by a covering of finite quotient orbifolds: there exists an open cover $\{ U_\alpha // \Gamma_{\alpha} \to {\mc U} \}$ such that the map $U_{\alpha} \to U_\alpha // \Gamma_{\alpha} \to {\mc U}$ factors through $U_\alpha \to U$.
\end{prop}

\begin{proof}
This is \cite[Proposition 3.3, Corollary 3.5]{pardon19} in the setting of orbifolds.
\end{proof}

Using such an open cover, the stabilizer $\Gamma_x$  of a point $x \in |{\mc U}|$ is isomorphic to the stabilizer of any preimage of $x$ in $U_{\alpha}$ of the $\Gamma_\alpha$-action if $x \in U_{\alpha} / \Gamma_{\alpha}$. An orbifold ${\mc U}$ is called \emph{effective} if it admits an open cover by finite quotient orbifolds $\{ U_\alpha // \Gamma_{\alpha} \to {\mc U} \}$ such that the $\Gamma_{\alpha}$-action on $U_{\alpha}$ is faithful. If ${\mc U}$ is effective, we can further refine any \'etale atlas such that each $U_{\alpha}$ can be identified with a $\Gamma_{\alpha}$-invariant open subset of the Euclidean space ${\mb R}^n$ which is a faithful $\Gamma_{\alpha}$-representation. When this is the case, the open embedding $U_{\alpha} // \Gamma_{\alpha} \to {\mc U}$ is called an \emph{orbifold chart}.

In our later discussion of orbifold bordism theories, we need to encounter \emph{orbifolds with boundary}, which are locally modeled on $({\mb R}_{\geq 0} \times {\mb R}^{n-1}) // \Gamma$ such that $\Gamma$ acts linearly and the action of $\Gamma$ on the ${\mb R}_{\geq 0}$-factor is trivial. To avoid repeating explanations, we will mostly state definitions and constructions for orbifolds without mentioning the counterparts for orbifolds with boundary unless it is necessary.

Next, we discuss objects living on stacks or orbifolds. A \emph{vector bundle} over a stack ${\mc X}$ is a stack ${\mc E}$ with a representable morphism ${\mc E} \to {\mc X}$, which is further endowed with a pair of representable morphisms ${\mc E} \times_{\mc X} {\mc E} \to {\mc E}$ and  ${\mb R} \times {\mc E} \to {\mc E}$ such that for any topological space $X$ with a morphism $X \to {\mc X}$, there exists an open cover $\{U_\alpha\}$ of $X$ and integers $n_{\alpha} \geq 0$, with the property that for any $\alpha$, the pullback of ${\mc E} \times_{\mc X} U_{\alpha}$ is isomorphic to ${\mb R}^{n_{\alpha}} \times U_{\alpha} \to U_{\alpha}$ and the two morphisms are identified with fiberwise additions and scalar multiplications. A \emph{complex} vector bundle ${\mc E} \to {\mc X}$ is be defined by replacing ${
\mb R}$ by ${\mb C}$ in the above definition. We just need to encounter the cases when all the $n_{\alpha}$ are equal. In the case of an orbifold ${\mc U}$, a \emph{smooth vector bundle} over ${\mc U}$ of rank $k$ is defined to be a vector bundle ${\mc E} \to {\mc U}$ where ${\mc E}$ is a smooth orbifold with smooth projection map, so that for any smooth map from a smooth manifold $X \to {\mc U}$, there exists an open cover $\{ U_{\alpha} \}$ of $X$, such that the pullback ${\mc E} \times_{\mc U} U_{\alpha} \to U_{\alpha}$ is smoothly isomorphic to ${\mb R}^k \times U_{\alpha} \to U_{\alpha}$. Complex smooth vector bundles can be defined similarly. One can easily formulate operations on vector bundles, including pullback, direct sum, dual, and tensor products, and notions of subbundles, quotient bundles, and bundle morphisms, by ``evaluating" the notions along topological spaces. Unless otherwise stated, a vector bundle over an orbifold is always assumed to be smooth in our subsequent discussion.

Continuing our example of finite quotient orbifold, given a smooth manifold $X$ acted on smoothly by a finite group $\Gamma$, together with a $\Gamma$-representation $W$, we obtain a smooth vector bundle $(X \times W) // \Gamma \to X // \Gamma$. These are the building blocks of vector bundles. Indeed, given a vector bundle ${\mc E} \to {\mc U}$, for any \'etale atlas $U \to {\mc U}$, the pullback ${\mc E} \times_{\mc U} U$ is a locally trivial vector bundle under some open over $\{ U_{\alpha} \}$ of $U$. By refining the \'etale atlas $\coprod_{\alpha} U_{\alpha}$, using Proposition \ref{prop:orbi-chart}, we can assume that the restriction of the \'etale chart over $U_{\alpha}$ factors through $U_{\alpha} \to U_{\alpha} // \Gamma_{\alpha}$ for some finite group $\Gamma_{\alpha}$. By the local triviality of ${\mc E} \times_{\mc U} U$, we see that ${\mc E} \times_{\mc U} U |_{U_{\alpha}}$ is the product bundle $U_{\alpha} \times W_{\alpha}$ for some $\Gamma_{\alpha}$-representation $W_{\alpha}$. We will call the corresponding bundle map 
\beq
\begin{tikzcd}
(U_{\alpha} \times W_{\alpha}) // \Gamma_{\alpha} \arrow[d] \arrow[r] & {\mc E} \arrow[d] \\
U_{\alpha} // \Gamma_{\alpha} \arrow[r]                               & {\mc U}
\end{tikzcd}
\eeq
a \emph{bundle chart} in line with the more traditional convention.

\begin{lemma}
Any orbifold ${\mc U}$ admits a tangent bundle $T{\mc U}$, which is an orbifold with a smooth map $T{\mc U} \to {\mc U}$ defining a vector bundle over ${\mc U}$.
\end{lemma}
\begin{proof}
Let $U \to {\mc U}$ be a smooth \'etale atlas, and denote by $[U \times_{\mc U} U \rightrightarrows U]$ the associated groupoid. Using the smooth structure on $U$, we can define its tangent bundle $TU$ and the projection map $TU \to U$. The fiber product $U \times_{\mc U} U$ acquires a unique smooth structure by the pullback along either factor, so we can define its tangent bundle $T (U \times_{\mc U} U)$. Then we claim that the following defines a groupoid structure on the pair $T (U \times_{\mc U} U), TU$. 
\begin{itemize}
    \item $s,t: T (U \times_{\mc U} U) \to TU$ defined by the differentials of the projection to the two factors respectively;
    \item $i,\iota$ are also defined by taking the differentials of the corresponding maps for the pair $U \times_{\mc U} U, U$.
    \item to define $m: T (U \times_{\mc U} U) \times_{TU} T (U \times_{\mc U} U) \to T (U \times_{\mc U} U)$, note that the domain can be identified with $T((U \times_{\mc U} U) \times_U (U \times_{\mc U} U))$, we can again take the differential.
\end{itemize}
It is easy to see that they satisfy the requirements of being a smooth orbifold provided that ${\mc U}$ is an orbifold. We denote the associated smooth orbifold by $T{\mc U}$.

The map of stacks $T{\mc U} \to {\mc U}$ is constructed by observing that the projection maps $T (U \times_{\mc U} U) \to U \times_{\mc U} U$ and $TU \to U$ is compatible with the groupoid structures. The fiberwise additions and scalar multiplications on the smooth manifolds in the above construction descend to representable maps of stacks. Finally, the smoothness of pullback can be checked \'etale locally.
\end{proof}

A smooth map of orbifolds $f: {\mc U}' \to {\mc U}$ induces a smooth bundle map $df: T{\mc U} \to T{\mc U}'$ covering $f$. This can be seen again by looking at the associated groupoids.

\begin{defn}
\begin{enumerate}
\item A smooth map of orbifolds is called an \emph{immersion} if its lift to some (equivalently any) \'etale atlas is an immersion.

\item Given an orbifold ${\mc U}$, a \emph{suborbifold} is an orbifold ${\mc U}'$ together with a smooth immersion ${\mc U}' \to {\mc U}$ which defines an isomorphism onto a closed substack of ${\mc U}$.

\item For a suborbifold ${\mc U}'$ in an orbifold ${\mc U}$, its \emph{normal bundle} $N_{\mc U} {\mc U}' \to {\mc U}'$ is defined to be the quotient vector bundle $T{\mc U}|_{{\mc U}'} / T{\mc U'}$.
\end{enumerate}
\end{defn}

We complete this section by introducing the main objects that we study:
\begin{defn}
  A \emph{derived orbifold (with boundary)}  is a triple $({\mc U}, {\mc E}, {\mc S})$ in which ${\mc U}$ is a orbifold (with boundary), ${\mc E} \to {\mc U}$ is a vector bundle, and ${\mc S}: {\mc U} \to {\mc E}$ is a smooth section. Such a derived orbifold is called \emph{compact} if ${\mc S}^{-1}(0)$ is compact. The \emph{dimension} of such $({\mc U}, {\mc E}, {\mc S})$ is $\dim {\mc U} - \dim {\mc E}$.
\end{defn}

Given a derived orbifold (with boundary) $({\mc U}, {\mc E}, {\mc S})$, we have the following two operations.
\begin{itemize}
    \item \emph{Restriction}: if ${\mc U}' \subset {\mc U}$ is an open suborbifold containing ${\mc S}^{-1}(0)$, we see that $({\mc U}', {\mc E}|_{{\mc U}'}, {\mc S}_{{\mc U}'})$ is also a derived orbifold (with boundary).
    \item \emph{Stabilization}: if $\pi_{{\mc E}'}: {\mc E}' \to {\mc U}$ is a vector bundle, denote by $\tau_{{\mc E}'}: {\mc E}' \to \pi_{{\mc E}'}^* {\mc E}'$ the tautological section. Then $({\mc E}', \pi_{{\mc E}'}^*{\mc E} \oplus \pi_{{\mc E}'}^* {\mc E}', \pi^*{\mc S} \oplus \tau_{{\mc E}'})$ is also a derived orbifold (with boundary).
\end{itemize}

Note that the zero locus is not changed by either restriction or stabilization. Therefore, restrictions and stabilizations of compact derived orbifold remain compact. For notational convenience, we will use ${\mc D} = ({\mc U}, {\mc E}, {\mc S})$ to denote a derived orbifold (with boundary).

\subsection{Differential geometry on orbifolds: metrics and normal complex structures} \label{sec:derived-orbifolds} 

We shall be using, within the setting of orbifolds, some basic notions of differential geometry. To begin, we introduce the following:

\begin{defn}
\begin{enumerate}
    \item A \emph{metric} on a vector bundle ${\mc E} \to {\mc X}$ is a map $g: {\mc E} \times_{\mc X} {\mc E} \to {\mb R}$ satisfying $g(w_1, w_2) = g(w_2, w_1)$, $g(w_1, r \cdot w_2) = r\cdot g(w_1, w_2)$ for $r \in {\mb R}$ and $g(w,w) \geq 0$ with equality if and only if $w=0$. 
    \item A \emph{hermitian metric} on a complex vector bundle ${\mc E} \to {\mc X}$ is a map $h: {\mc E} \times_{\mc X} {\mc E}$ satisfying $h(w_1, w_2) = \overline{h(w_2, w_1)}$, $h(w_1, c \cdot w_2) = c \cdot h(w_1, w_2)$ for $c \in {\mb C}$ and $h(w,w) \geq 0$ with equality if and only if $w=0$.
\end{enumerate}
\end{defn}

As usual, a metric on $T{\mc U}$ is called a \emph{Riemannian metric}. The exponential map of any smooth \'etale atlas descends to a unique smooth map $\exp: T{\mc U} \to {\mc U}$, also referred to as the \emph{exponential map}. Other differential-geometric notions, including connections, parallel transport, geodesics, etc, can also be defined for orbifolds in the natural way.

Next, we introduce the following structure which is particularly interesting from the point of view of symplectic topology, as it arises in the study of various (thickened) moduli spaces of $J$-holomorphic curves. From a technical point of view, this is the most natural class of orbifolds to which one can apply algebro-geometric methods to simplify the singularities of the coarse spaces, as we will see in later sections.

\begin{defn}\label{defn:normal-complex-manifold}
Let $M$ be a smooth manifold, and suppose that $\Gamma$ is a finite group acting smoothly on $M$. For any subgroup $\Gamma' \leq \Gamma$, denote by $M^{\Gamma'}$ the fixed point locus of the restricted $\Gamma'$-action. A \emph{normal complex structure} on the $\Gamma$-manifold $M$ is a collection of $\Gamma'$-equivariant complex structures on the normal bundle
\beq
I_{\Gamma'}: N M^{\Gamma'} \to N M^{\Gamma'}
\eeq
for all $\Gamma' \leq \Gamma$ satisfying the following compatibility condition: for any pair of subgroups $\Gamma_2 < \Gamma_1 \leq \Gamma$ the embedding of vector bundles
\begin{equation}
    N M^{\Gamma_2}|_{M^{\Gamma_1}} \to N M^{\Gamma_1} 
\end{equation}
along $M^{\Gamma_1}$ respects the complex structures, and the canonical splitting
\begin{equation}\label{eqn:normal-splitting}
    N M^{\Gamma_1}  \cong N M^{\Gamma_2}|_{M^{\Gamma_1}} \oplus (N M^{\Gamma_2}|_{M^{\Gamma_1}})^{\perp}
\end{equation}
as the fiberwise direct sum of $\Gamma_2$-representations is preserved by $I_{\Gamma_1}$.
\end{defn}

Using the splitting \eqref{eqn:normal-splitting}, the above definition asks that the two resulting complex structures from $I_{\Gamma_1}$ and $I_{\Gamma_2}$ on $N M^{\Gamma_2}|_{M^{\Gamma_1}}$ agree.

Given a group homomorphism $\Gamma_1 \to \Gamma_2$, and an equivariant smooth open embedding $f: M_1 \to M_2$ from a $\Gamma_1$-manifold $M_1$ to a $\Gamma_2$-manifold $M_2$, we may pull back a normal complex strucure on $M_2$ to a normal complex structure on $M_1$, assuming that $f$ defines an injection on the stabilizers:
\begin{defn}\label{defn:normal-complex}
Given a smooth orbifold ${\mc U}$, a \emph{normal complex structure} consists of:
\begin{enumerate}
    \item an \'etale atlas of the form $U = \coprod U_{\alpha}$ for which each $U_{\alpha}$ is a smooth manifold acted on smoothly by a finite group $\Gamma_{\alpha}$, such that $\{ U_{\alpha} // \Gamma_{\alpha} \}$ is a cover of ${\mc U}$ by open substacks;
    \item a normal complex structure on each $\Gamma_{\alpha}$-manifold $U_{\alpha}$ such that the two normal complex structures on $U \times_{\mc U} U$ obtained from the pullback of two projections maps are isomorphic.
\end{enumerate}
\end{defn}

\begin{rem}\label{rem:normal-complex}
  Normal complex structures and their appearance in symplectic topology were first discovered by Fukaya--Ono in \cite{Fukaya_Ono_integer}. The first systematic investigation appeared in \cite{Bai_Xu_2022}, and the definition from there is stated for effective orbifolds using the language of orbifold atlas, which we recall for the sake of concreteness. For an orbifold ${\mc U}$, viewed as a Hausdorff space $|{\mc U}|$ equiped with an equivalence class of orbifold atlases, a normal complex structure on ${\mc U}$ is a collection of $\Gamma'$-equivariant complex structures on the normal bundle of the $\Gamma'$-fixed point locus $U^{\Gamma'}$ respecting the fiberwise $\Gamma'$-action on each orbifold chart $\{(U, \Gamma, \psi)\}$ with $\Gamma' \leq \Gamma$
\beq
I_{\Gamma'}: NU^{\Gamma'} \to NU^{\Gamma'}
\eeq
subject to the following requirements:
\begin{enumerate}
    \item for each orbifold chart $(U, \Gamma)$ and $y \in U$ with stabilizer $\Gamma_y$, suppose we have a pair of subgroups $\Gamma_2 < \Gamma_1 \leq \Gamma_y$, then for the $I_{\Gamma_1}$-compatible decomposition
    \beq
    N_y U^{\Gamma_1} = N_y U^{\Gamma_2} \oplus (TU^{\Gamma_2} \cap N_y U^{\Gamma_1})
    \eeq
    of $\Gamma_1$-representations, the restriction of $I_{\Gamma_1}$ along the component $N_y U^{\Gamma_2}$ coincides with the complex structure $I_{\Gamma_2}$;
    \item if we have a chart embedding $\psi_{12}: (U_1, \Gamma_1) \to (U_2, \Gamma_2)$, and for $y \in U_1$, we identify the representatives of the stabilizer of $\psi_1(y) = \psi_2 \circ \psi_{12}(y)$ given by $\Gamma_y$ and $\Gamma_{\psi(y)}$ under the group injection $\Gamma_1 \to \Gamma_2$, then the differential $d \psi_{12}: NU_1^{\Gamma_y} \to NU_2^{\Gamma_{\psi(y)}}$ intertwines with the complex structures $I_{\Gamma_y}$ and $I_{\Gamma_{\psi(y)}}$.
\end{enumerate}
It is straightforward to see that this is equivalent to Definition \ref{defn:normal-complex}, as the compatibility condition (2) is encoded in the requirement (2) of Definition \ref{defn:normal-complex}.
\end{rem}

A normally complex orbifold is defined to be an orbifold equipped with an isomorphism class of normal complex structures. Note that \emph{any} manifold endowed with the trivial orbifold structure is automatically a normally complex orbifold. For another class of examples, recall that an orbifold ${\mc U}$ is called \emph{almost complex} if its tangent bundle $T{\mc U}$ is equipped with an isomorphism class of complex structures. Then any almost complex orbifold is also normally complex. Smooth Deligne--Mumford stacks over ${\mb C}$ produce examples of almost complex orbifolds.

A choice of normal complex structure provides a finer refinement of an orbifold than the standard stratification associated to the isotropy group, by including the choice of representation:
\begin{defn}\label{defn:iso-type}
  The set of \emph{isotropy types} is the set of equivalence classes of pairs
\beq
(\Gamma, V),
\eeq
where $\Gamma$ is a finite group and $V$ is a finite-dimensional complex $\Gamma$-representation which does not contain any trivial summand. Here $(\Gamma, V)$ is declared to be equivalent to $(\Gamma', V')$ if there is an isomorphism of groups $\Gamma \xrightarrow{\sim} \Gamma'$ under which $V$ and $V'$ are isomorphic representations.
\end{defn}
\begin{rem}
  The standard terminology is to use isotropy to refer only to the choice of group, but the more refine decomposition will be essential for out construction, and is consistent with the formulation in \cite{Bai_Xu_2022}.
\end{rem}
The set of isotropy types admits a natural partial order as follows. We say that $(\Gamma', V') \leq (\Gamma, V)$ if there exists an injective homomorphism  $\Gamma' \hookrightarrow \Gamma$ such that $V'$ is isomorphic to the quotient of $V$ by its $\Gamma'$-fixed subspace.

\begin{defn}\label{defn:normal-invariant}
Given a normally complex orbifold ${\mc U}$, the \emph{normal invariant} at $x \in |{\mc U}|$ is the isotropy type of the pair
\beq
(\Gamma_y, N_y U^{\Gamma_y})
\eeq
where $U // \Gamma \to {\mc U}$ is smoothly embedded as an open substack and $x$ is the image of $y$ under $U/\Gamma \to |{\mc U}|$.
\end{defn}

It is easy to see that different choices of \'etale charts covering $x$ and different choices of inverse images of $x$, define equivalent isotropy types. Therefore, the normal invariant of $x \in {\mc U}$ is well-defined.

In order to formulate the notion of normally complex derived orbifold, we need the notion of a normally complex vector bundle. We again start with the equivariant setting.

\begin{defn}\label{defn:bundle-equivariant-complex}
Let $M$ be a smooth $\Gamma$-manifold and suppose $E \to M$ is a $\Gamma$-equivariant vector bundle. Given $\Gamma' \leq \Gamma$, introduce the notation
\beq\label{eqn:bundle-E-split}
E|_{M^{\Gamma'}} = \mathring{E}^{\Gamma'} \oplus \check{E}^{\Gamma'}
\eeq
which decomposes each fiber of $E|_{M^{\Gamma'}}$ as a direct sum of trivial and nontrivial $\Gamma'$-representations. Then a \emph{normal complex structure} on $E$ is a collection of $\Gamma'$-equivariant complex structures
\beq
J_{\Gamma'}: \check{E}^{\Gamma'} \to \check{E}^{\Gamma'}
\eeq
for all $\Gamma' \leq \Gamma$ such that for any pair of subgroups $\Gamma_2 < \Gamma_1 \leq \Gamma$, the embedding of vector bundles 
\begin{equation}
    \check{E}^{\Gamma_1}|_{M^\Gamma_2}  \to  \check{E}^{\Gamma_2},
\end{equation}
along $ M^{\Gamma_2}  $ is a map of complex vector bundles, and the canonical decomposition $\check{E}^{\Gamma_2} = \check{E}^{\Gamma_1}|_{M^{\Gamma_2}} \oplus (\mathring{E}^{\Gamma_1} \cap \check{E}^{\Gamma_2})$ is preserved by the complex structure $J_{\Gamma_2}$.
\end{defn}

As before, given a group homomorphism $\Gamma_1 \to \Gamma_2$, the pullback of a normally complex vector bundle on under an equivariant smooth open embedding $f: M_1 \to M_2$ from a $\Gamma_1$-manifold $M_1$ to a $\Gamma_2$-manifold $M_2$, has a natural normally complex structure whenever $f$ induces an injection on the stabilizer groups.

\begin{rem}
    The fixed point locus $M^{\Gamma'}$ may have different connected components with different dimensions, and the objects $\mathring{E}^{\Gamma'}$ and $\check{E}^{\Gamma'}$ only define vector bundles over these connected components. To simplify the notations, we will stick to these slightly imprecise notations which should not cause confusions.
\end{rem}

\begin{defn}\label{defn:nc-bundle}
A \emph{normal complex structure} on a vector bundle ${\mc E} \to {\mc U}$ consists of:
\begin{enumerate}
    \item an \'etale atlas of ${\mc E}$ of the form $E = \coprod_{\alpha} E_{\alpha}$ covering an \'etale atlas of ${\mc U}$ of the form $\coprod_{\alpha} U_{\alpha}$ as a vector bundle such that:
    \begin{itemize}
    \item $U_{\alpha}$ is a smooth $\Gamma_{\alpha}$-manifold, $E_{\alpha} \to U_{\alpha}$ is a $\Gamma_{\alpha}$-equivariant vector bundle, $E_{\alpha} // \Gamma_{\alpha} \to {\mc E}$ is an open embedding which is further a bundle map covering an open embedding $U_{\alpha} // \Gamma_{\alpha} \to {\mc U}$;
        \item $E_{\alpha} \to U_{\alpha}$ is equipped with a normal complex structure;
    \end{itemize}
    \item for the vector bundle $E \times_{\mc E} E \to U \times_{\mc U} U$, the two normal complex structures induced from the pullback of two projection maps $U \times_{\mc U} U \to U$ agree.
\end{enumerate}
\end{defn}

\begin{rem}
  Continuing Remark \ref{rem:normal-complex}, we recall the arguably more concrete formulation of normal complex structures on vector bundles in terms of bundle charts as appeared in \cite{Bai_Xu_2022}.
For a vector bundle ${\mc E} \to {\mc U}$, a normal complex structure is a collection of $\Gamma'$-equivariant complex structures 
\beq
J_{\Gamma'}: \check{E}^{\Gamma'} \to \check{E}^{\Gamma'}
\eeq
for each bundle chart $(U, \Gamma, E)$ and $\Gamma' \leq \Gamma$ subject to the following conditions:
\begin{enumerate}
    \item for each bundle chart $(U, \Gamma, E)$ and $y \in U$ with stabilizer $\Gamma_y$, suppose we have a pair of subgroups $\Gamma_2 < \Gamma_1 \leq \Gamma_y$, then for the $J_{\Gamma_1}$-compatible decomposition
    \beq
    \check{E}_y^{\Gamma_1} = \check{E}_y^{\Gamma_2} \oplus (\mathring{E}_y^{\Gamma_2} \cap \check{E}_y^{\Gamma_1})
    \eeq
    of $\Gamma_1$-representations, the restriction of $J_{\Gamma_1}$ along the component $\check{E}_y^{\Gamma_2}$ coincides with the complex structure $J_{\Gamma_2}$;
    \item if we have a bundle chart embedding $\hat{\psi}_{12}: (U_1, \Gamma_1, E_1) \to (U_2, \Gamma_2, E_2)$, and for $y \in U_1$, we identify the representatives of the stabilizer of $\psi_1(y) = \psi_2 \circ \psi_{12}(y)$ given by $\Gamma_y$ and $\Gamma_{\psi(y)}$ under the group injection $\Gamma_1 \to \Gamma_2$, then the bundle map $\hat{\psi}_{12}: \check{E}_1^{\Gamma_y} \to \check{E}_2^{\Gamma_{\psi(y)}}$ intertwines with the complex structures $J_{\Gamma_y}$ and $J_{\Gamma_{\psi(y)}}$.
\end{enumerate}
\end{rem}

The constructions of this paper are formulated for the following class of derived orbifolds:

\begin{defn}
  A \emph{normal complex structure} on a derived orbifold ${\mc D} = ({\mc U}, {\mc E}, {\mc S})$ (with boundary) consists of normal complex structures on $\mc U$ and $\mc E$. 
\end{defn}
 Of course, restrictions of derived orbifolds preserves normal complex structure, so does the stabilization by complex vector bundles.

\subsection{Fukaya--Ono--Parker sections} 
In this subsection, we recall the definition of \emph{normally complex polynomial perturbations/sections} introduced by Fukaya--Ono \cite{Fukaya_Ono_integer} and Parker \cite{BParker_integer}. Following \cite{Bai_Xu_2022}, we call such sections the Fukaya--Ono--Parker (FOP) sections.

We need to discuss a distinguished class of metric and connection structures on orbifolds and vector bundles to get prepared for our discussion of the above distinguished class of sections.

Let $M$ be a smooth manifold, and suppose that $\Gamma$ is a finite group acting smoothly on $M$. For any subgroup $\Gamma' \leq \Gamma$, let $M^{\Gamma'}$ be the fixed point locus of the restricted $\Gamma'$-action. Choosing a $\Gamma$-equivariant Riemannian metric $g$ on $M$,  then the splitting
\beq
TM|_{M^{\Gamma'}} \cong TM^{\Gamma'} \oplus NM^{\Gamma'},
\eeq
equipps $\pi_{NM^{\Gamma'}}: NM^{\Gamma'} \to M^{\Gamma'}$ with a bundle metric and a connection from the Levi-Civita connection. Using the splitting $T(NM^{\Gamma'}) = (\pi_{NM^{\Gamma'}})^* TM^{\Gamma'} \oplus (\pi_{NM^{\Gamma'}})^* NM^{\Gamma'}$ induced by the connection, the bundle metric on $NM^{\Gamma'}$ and the induced Riemannian metric on $M^{\Gamma'}$ together equip $NM^{\Gamma'}$ with a Riemannian metric $g^{NM^{\Gamma'}}$. On the other hand, the normal exponential map
\beq
\exp: NM^{\Gamma'} \to M
\eeq
defines a diffeomorphism onto its image when restricted along a sufficiently small open neighborhood of the $0$-section $M^{\Gamma'} \subset M^{\Gamma'}$.

\begin{defn}\label{defn:manifold-straight-metric}
  A $\Gamma$-equivariant Riemannian metric $g$ on $M$ is called  (normally) \emph{straightened} if for any $\Gamma' \leq \Gamma$, the push-forward of the metric $g^{NM^{\Gamma'}}$ on $NM^{\Gamma'}$ under the normal exponential map agrees with the restriction of $g$ to an open neighborhood of $M^{\Gamma'}$ over which $\exp$ defines a diffeomorphism.
\end{defn}

A straightened metric has the following feature: given a group $\Gamma_1 \leq \Gamma $, consider the decomposition of vector bundles over $M^{\Gamma_1}$
\beq\label{eqn:decomposition-normal}
NM^{\Gamma_1} = (\mathring{NM^{\Gamma_1}})^{\Gamma_2} \oplus (\check{NM^{\Gamma_1}})^{\Gamma_2}
\eeq
associated to the restriction of the fiberwise $\Gamma_1$ action to a subgroup $\Gamma_2 < \Gamma_1$, where the first component is the $\Gamma_2$-trivial component. The normal exponential map defines
\beq
\exp: (\mathring{NM^{\Gamma_1}})^{\Gamma_2} \to M^{\Gamma_2}.
\eeq
Given $y \in M^{\Gamma_1}$ and a vector in $N_y M^{\Gamma_1}$ written as $(\mathring{v}_{\Gamma_2}, \check{v}_{\Gamma_2})$ under the decomposition of Equation \eqref{eqn:decomposition-normal}, one can check that
\beq
\exp_{x}(\mathring{v}_{\Gamma_2}, \check{v}_{\Gamma_2}) = \exp_{\exp_x(\mathring{v}_{\Gamma_2})}(\check{v}_{\Gamma_2}).
\eeq
In other words, if the metric $g$ is straightened, the induced metric along the normal directions behaves like an equivariant flat metric.

\begin{defn}\label{defn:straight-metric}
Given a smooth orbifold ${\mc U}$, a \emph{straightened Riemannian metric} on ${\mc U}$ is a Riemannian metric on ${\mc U}$ such that 
\begin{enumerate}
\item there exists an \'etale atlas of the form $U = \coprod U_{\alpha}$ for which each $U_{\alpha}$ is a smooth manifold acted on smoothly by a finite group $\Gamma_{\alpha}$, such that $\{ U_{\alpha} // \Gamma_{\alpha} \}$ is a cover of ${\mc U}$ by open substacks;
    \item the pullback of $g$ to each $U_{\alpha}$ defines a straightened $\Gamma_{\alpha}$-equivariant Riemannian metric.
\end{enumerate}
\end{defn}

The existence of straightened Riemannian metrics and a relative form for the inductive construction is established by \cite[Lemma 3.15]{Bai_Xu_2022}. The reader can refer to the cited reference for a full proof.

We need a similar structure for vector bundles. Let $M$ be a smooth manifold with a $\Gamma$-action, and let $E \to M$ be a $\Gamma$-equivariant vector bundle equipped with a $\Gamma$-equivariant connection $\nabla$. Suppose that $M$ is endowed with a $\Gamma$-equivariant straightened Riemannian metric $g$. For a subgroup $\Gamma' \leq \Gamma$, we have the projection $\pi_{NM^{\Gamma'}}: NM^{\Gamma'} \to M^{\Gamma'}$. Over a disc bundle $D_r(NM^{\Gamma'}) \subset NM^{\Gamma'}$, the parallel transport along normal geodesics induces a $\Gamma'$-equivariant bundle isomorphism
\beq\label{eqn:bundle-E-isomorphism}
\rho^{\Gamma'}: (\pi_{NM^{\Gamma'}})^* (E|_{M^{\Gamma'}}) \cong E|_{D_r(NM^{\Gamma'})}.
\eeq
The splitting of $E|_{M^{\Gamma'}}$ as a direct sum $E|_{M^{\Gamma'}} = \mathring{E}^{\Gamma'} \oplus \check{E}^{\Gamma'}$ from Equation\eqref{eqn:bundle-E-split},  extends to a splitting
\beq\label{eqn:bundle-restriction-split}
E|_{D_r(NM^{\Gamma'})} = \mathring{E}^{\Gamma'} \oplus \check{E}^{\Gamma'}.
\eeq
using the isomorphism of Equation  \eqref{eqn:bundle-E-isomorphism}. 

\begin{defn}\label{defn:straight-connection}
For a $\Gamma$-manifold $M$ endowed with a $\Gamma$-equivariant straightened Riemannian metric and a $\Gamma$-equivariant vector bundle $E \to M$ endowed with a $\Gamma$-equivariant connection $\nabla$, the connection is said to be \emph{straightened} if for any $\Gamma' \leq \Gamma$, after identifying a disc bundle $D_r(NM^{\Gamma'})$ with an open subset of $M^{\Gamma'}$ under the exponential map, we have 
\beq
(\rho^{\Gamma'})^* (\pi_{NM^{\Gamma'}})^* \nabla|_{E|_{M^{\Gamma'}}} = \nabla |_{E|_{D_r(NM^{\Gamma'})}}.
\eeq
\end{defn}

Again, straightened connections can be pulled back by equivariant smooth open maps which define an injection on the stabilizers. 

\begin{defn}\label{defn:vector bundle-straighting}
A \emph{straightening} on a vector bundle ${\mc E} \to {\mc U}$ consists of a straightend Riemannian metric $g$ on ${\mc U}$ and a connection $\nabla$ on ${\mc E}$ such that:
\begin{enumerate}
    \item we have an \'etale atlas of ${\mc E}$ of the form $E = \coprod_{\alpha} E_{\alpha}$ covering an \'etale atlas of ${\mc U}$ of the form $\coprod_{\alpha} U_{\alpha}$ as a vector bundle such that:
   \begin{itemize}
        \item $U_{\alpha}$ is a smooth $\Gamma_{\alpha}$-manifold, $E_{\alpha} \to U_{\alpha}$ is a $\Gamma_{\alpha}$-equivariant vector bundle, $E_{\alpha} // \Gamma_{\alpha} \to {\mc E}$ is an open embedding which is further a bundle map covering an open embedding $U_{\alpha} // \Gamma_{\alpha} \to {\mc U}$;
        \item the straightened Riemannian metric is defined with respect to the atlas $\coprod_{\alpha} U_{\alpha} \to {\mc U}$;
    \end{itemize}
    \item the pullback of $\nabla$ to each $E_{\alpha} \to U_{\alpha}$ defines a straightened $\Gamma_{\alpha}$-equivariant connection under the induced straightened $\Gamma_{\alpha}$-equivariant metric on $U_{\alpha}$.
\end{enumerate}
\end{defn}

Straightenings of ${\mc E} \to {\mc U}$ are used to obtain decompositions as in Equation \eqref{eqn:bundle-restriction-split}. Their existence and a relative form of the construction is the content of \cite[Lemma 3.20]{Bai_Xu_2022}.

If $\pi_{{\mc E}'}: {\mc E}' \to {\mc U}$ is another vector bundle, consider  the \emph{stabilization} $\pi_{{\mc E}'}^* {\mc E} \oplus \pi_{{\mc E}'}^* {\mc E}' \to {\mc E}'$. If $\pi_{{\mc E}'}: {\mc E}' \to {\mc U}$ comes with a straightening, \cite{Bai_Xu_2022} specifies a class of bundle metrics on ${\mc E}'$ which is also called straightened. Such a structure can be used to equip the stabilization $\pi_{{\mc E}'}^* {\mc E} \oplus \pi_{{\mc E}'}^* {\mc E}' \to {\mc E}'$ with a straightening if furthermore, ${\mc E} \to {\mc U}$ has a straightening, see \cite[Section 3.4]{Bai_Xu_2022}. This fact will be used in the study of derived orbifold bordisms.

With these preparations, we can proceed to define FOP sections following \cite[Section 3.6]{Bai_Xu_2022}. As usual, we discuss the corresponding notion in the equivariant setting first.

Let $\Gamma$ be a finite group and let $(V, W)$ be a pair of finite-dimensional complex $\Gamma$-representations. Given a positive integer $d$, denote by
\beq
\mathrm{Poly}_{d}^{\Gamma}(V,W)
\eeq
the space of $\Gamma$-equivariant polynomial maps from $V$ to $W$ of degree at most $d$. It admits an action of $GL(V)^{\Gamma} \times GL(W)^{\Gamma}$ by automorphisms, where $GL(V)^{\Gamma}$ resp. $GL(W)^{\Gamma}$ denotes elements in $GL(V)$ resp. $GL(W)$ invariant under the $\Gamma$-conjugation action.

Suppose that $M$ is a smooth $\Gamma$-manifold with a normal complex structure in the sense of Definition \ref{defn:normal-complex-manifold} and a straightened Riemannian metric in the sense of Definition \ref{defn:manifold-straight-metric}. Furthermore, suppose that we are given a $\Gamma$-equivariant vector bundle $E \to M$ equipped with a normal complex structure (cf. Definition \ref{defn:bundle-equivariant-complex}) and a straightened connection (cf. Definition \ref{defn:straight-connection}). Given $\Gamma' \leq \Gamma$ and a connected component of the fixed point locus $M^{\Gamma'}$, any fiber of $NM^{\Gamma'}$ defines the same complex $\Gamma'$-representation and we write it as $V'$. Similarly, using the decomposition of Equation \eqref{eqn:bundle-E-split}, we denote the fibers of $\check{E}^{\Gamma'}$ by $W'$ as a complex $\Gamma$-representation. Using the structure map of the complex vector bundles $NM^{\Gamma'}$ and $\check{E}^{\Gamma'}$, we can define a vector bundle
\beq
\mathrm{Poly}_d^{\Gamma'}(NM^{\Gamma'}, \check{E}^{\Gamma'})
\eeq
over $M^{\Gamma'}$ whose fibers are given by $\mathrm{Poly}_{d}^{\Gamma'}(V',W')$.

\begin{defn}\label{defn:FOP-lift-equivariant}
Let $E \to M$ be as above. Let $S: M \to E$ be a smooth $\Gamma$-equivariant section, which is decomposed as 
\beq
S = (\mathring{S}_{\Gamma'}, \check{S}_{\Gamma'}): D_r (NM^{\Gamma'}) \to \mathring{E}^{\Gamma'} \oplus \check{E}^{\Gamma'}
\eeq
over the disc bundle $D_r (NM^{\Gamma'})$. The section $S$ is said to have a \emph{normally complex lift} of degree at most $d$ near $M^{\Gamma'}$ if there exists a smooth bundle map from a disc bundle
\beq
\check{\mathfrak{s}}_{\Gamma'}: D_r(NM^{\Gamma'}) \to \mathrm{Poly}_d^{\Gamma'}(NM^{\Gamma'}, \check{E}^{\Gamma'})
\eeq
such that $\check{S}_{\Gamma'} = \mathrm{ev} \circ \check{\mathfrak{s}}_{\Gamma'}$, where $\mathrm{ev}$ is the bundle map
\beq
NM^{\Gamma'} \times_{M^{\Gamma'}} \mathrm{Poly}_d^{\Gamma'}(NM^{\Gamma'}, \check{E}^{\Gamma'}) \to \check{E}^{\Gamma'}
\eeq
defined by the fiberwise evaluation map. If the section $S$ admits such a lift near $M^{\Gamma'}$ for all $\Gamma' \leq \Gamma$, we say $S$ has a normally complex lift of degree at most $d$.
\end{defn}

In this definition, the straightenings are used to make various identifications, so the notion of normally complex lift depends on such structures. For a subgroup $\Gamma'' < \Gamma'$, it can be verified that a normally complex lift of degree at most $d$ near $M^{\Gamma'}$ defines a normally complex lift of degree at most $d$ near $M^{\Gamma'} \cap M^{\Gamma''}$, so lifts near different fixed point loci are compatible.

\begin{defn}\label{defn:FOP-section}
Let ${\mc E} \to {\mc U}$ be a vector bundle. Suppose that both ${\mc U}$ and ${\mc E}$ have a normal complex structure and that ${\mc E} \to {\mc U}$ is endowed with a straightening in the sense of Definition \ref{defn:vector bundle-straighting}. Assume that the normal complex structures and the straightening are both specified using \'etale atlases $E = \coprod_{\alpha} E_{\alpha}$ and $U = \coprod_{\alpha} U_{\alpha}$ as in Definition \ref{defn:vector bundle-straighting}. A section ${\mc S}: {\mc U} \to {\mc E}$ is said to have a \emph{normally complex lift} of degree at most $d$ if, for each chart $\alpha$, its pullback to the $\Gamma_{\alpha}$-equivariant vector bundle $E_{\alpha} \to U_{\alpha}$ admits a normally complex lift of degree at most $d$. If this is the case, we call ${\mc S}$ an \emph{FOP section}.
\end{defn}

\begin{rem}
Note that any two \'etale atlases of ${\mc U}$ depicted in Proposition \ref{prop:orbi-chart} can be refined to a common \'etale atlas of the same form, and a similar statement holds for \'etale atlases of ${\mc U}$ formed by bundle charts. From this, one notices that the notions of normal complex structures and straightenings do not really depend on the choices of \'etale atlas in the definition. A particular choice of \'etale atlas just makes the definition transparent. Similarly, the existence of normally complex lift of ${\mc S}: {\mc U} \to {\mc E}$ does not depend on the choice of \'etale atlases, though it is important to notice that it does depend on both the normal complex structures and the straightenings prescribed in the definition.
\end{rem}

We finally state some properties of FOP sections, whose proof can be found in \cite[Section 3.6]{Bai_Xu_2022}.

\begin{prop}\label{prop:FOP-prop}
Let ${\mc E} \to {\mc U}$ be a vector bundle such that both ${\mc U}$ and ${\mc E}$ are normally complex, and ${\mc E} \to {\mc U}$ is endowed with a straightening.
\begin{enumerate}
    \item  The space of FOP sections is a module over $C^{\infty}({\mc U})$.
    \item Suppose that we have an arbitrary continuous section ${\mc S}: {\mc U} \to {\mc E}$ such that $|{\mc S}^{-1}(0)|$ is compact. Let ${\mc U}' \subset {\mc U}$ be an open suborbifold containing ${\mc S}^{-1}(0)$ with $|{\mc U}'|$ being precompact. Then there exists $d \in {\mb Z}_{\geq 1}$ so that for any $\epsilon > 0$, there exists a smooth section ${\mc S}': {\mc U} \to {\mc E}$ which is an FOP section of degree at most $d$ near ${\mc U}'$ such that
    \beq
    \| {\mc S}' - {\mc S} \|_{C^0({\mc U}')} < \epsilon.
    \eeq
\end{enumerate} \qed
\end{prop} 

The importance of (1) is that we can use cut-off functions to define FOP sections in a local-to-global fashion. The item (2) allows us to approximate arbitrary sections by FOP sections, which are crucial for establishing equivariant transversality.

\section{Abelianizing normally complex orbifolds}\label{sec:abelianize}
In this section, we follow Bergh--Rydh \cite{bergh2019functorial} to perform blow-ups of normally complex orbifolds to obtain normally complex orbifolds so that the representation associated to each point is a direct sum of $1$-dimensional representations (in the terminology of \cite{bergh2019functorial}, this is formulated as abelianess of \emph{geometric stabilizers}). The situation in \cite{bergh2019functorial} is more complicated than ours, because the arguments in \emph{loc. cit.} are designed to handle families of possibly non-Deligne--Mumford stacks over a field of arbitrary characteristic.

The strategy of Bergh--Rydh is quite straightforward, and relies on measuring the `non-abelianness' of a normally complex orbifold as follows:  the stabilizer group $\Gamma_x$ of a point $x$ in an orbifold $ \mc U$  acts on the tangent space $T_x {\mc U}$. Associate to this point the dimension of the direct sum of all irreducible complex representations of complex dimension at least $2$.  Note that if the invariant takes value $0$ everywhere on ${\mc U}$, then if ${\mc U}$ is effective, and all the stabilizers inject into a product of $U(1)$, which means that all the stabilizers are abelian. More generally, assuming that ${\mc U}$ is connected with generic stabilizer $\Gamma_{\mc U}$, every stabilizer is an extension of an abelian group by $\Gamma_{\mc U} $.

\subsection{Blow-ups} 
Due to the absence of holomorphic structures, constructing blow-up in the almost complex category such that the blow-down map is pseudo-holomorphic is a delicate question, which indeed cannot be achieved in general. Our construction of blow-up simply specifies a contractible family of normally complex orbifolds, each of which admits a smooth map onto the normally complex orbifolds we started with. In particular, the cobordism class of our blow-up construction is canonically defined and this suffices for our purpose.

Given an orbifold ${\mc U}$ and a complex vector bundle ${\mc E} \to {\mc U}$, denote by ${\mb P}({\mc E}) \to {\mc U}$ the projectivization of ${\mc E}$. Then the blow-up of the total space ${\mc E}$ along the zero-section ${\mc U}$ can be described as the orbifold
\beq
\mathrm{Bl}_{\mc U} {\mc E} := \{ (v,l) \in {\mb P}({\mc E}) \times_{\mc U} {\mc E} | l \text{ is contained in the complex line defined by }v \},
\eeq
where the description can easily be understood by pulling back to \'etale charts. The projectivization ${\mb P}({\mc E})$ embeds into $\mathrm{Bl}_{\mc U} {\mc E}$ as the locus where $l=0$, known as the exceptional divisor. Moreover, there is a natural projection/blow-down map $\mathrm{Bl}_{\mc U} {\mc E} \to {\mc E}$ defined by $(v,l) \mapsto l$, which induces an identification $\mathrm{Bl}_{\mc U} {\mc E} \setminus {\mb P}({\mc E}) \xrightarrow{\sim} {\mc E} \setminus {\mc U}$. In general, if we would like to blow up an orbifold along a suborbifold, we need to require the normal bundle of the given suborbifold to have a complex structure. Additionally, in order to stay in the realm of smooth normally complex orbifolds, we need certain conditions to ensure that the blow-up admits a normal complex structure. This motivates the following definition.

\begin{defn}
Let $M$ be a smooth $\Gamma$-manifold and suppose that $g$ is a $\Gamma$-equivariant Riemannian metric on $M$. A normal complex structure (cf. Definition \ref{defn:normal-complex-manifold}) is called \emph{flattened} with respect to $g$ if there exists $\delta > 0$ such that for any subgroup $\Gamma' \leq \Gamma$, the (linearization of) normal exponential map over the $\delta$-disc bundle of the normal bundle of $M^{\Gamma'}$
\beq
\exp: D_{\delta} M^{\Gamma'} \to M
\eeq
intertwines with the normal complex structures under the group homomorphism $\Gamma' \to \Gamma$.

Similarly, given a smooth orbifold ${\mc U}$ and a Riemannian metric $g$ on ${\mc U}$, a normal complex structure on ${\mc U}$ (cf. Definition \ref{defn:normal-complex}) is called \emph{flattened} with respect to $g$ if its pullbacks to the defining \'etale charts are straightened with respect to the pullback equivariant Riemannian metric.
\end{defn}

\begin{lemma}\label{lem:flattening}
Suppose that $M$ is a smooth $\Gamma$-manifold equipped with a normal complex structure. Let $g$ be a $\Gamma$-invariant Riemannian metric on $M$ such that for any $\Gamma' \leq \Gamma$, the normal exponential map along $M^{\Gamma'}$ has positive injectivity radius. Then the given normal complex structure can be deformed to a flattened one with respect to $g$. Moreover, the space of such flattened Riemannian metrics are contractible.
\end{lemma}
\begin{proof}
The construction is based on induction. We define a partial order on the set of subgroups of $\Gamma'$ as follows: for $\Gamma_1, \Gamma_2 \leq \Gamma$, we declare $\Gamma_1 \preceq \Gamma_2$ if and only if $M^{\Gamma_1} \subset M^{\Gamma_2}$. 

Staring from any $\Gamma_1 \leq \Gamma$ such that $\Gamma_1$ is minimal under the relation $\preceq$, we see that $M^{\Gamma_1}$ is a closed submanifold and the normal bundle $NM^{\Gamma_1}$ has a $\Gamma_1$-equivariant complex structure. Choose $\delta > 0$ so that $\exp: D_{\delta} M^{\Gamma_1} \to M$ defines a $\Gamma_1$-equivariant diffeomorphism onto its image. To construct a straightened normal complex structure near $M^{\Gamma_1}$, we need to make sure that for any $\Gamma_1 \preceq \Gamma_2$, the linearization of the exponential map along the normal directions of the $\Gamma_2$ fixed point loci
\beq
\exp: N_{D_{\delta} M^{\Gamma_1}} (D_{\delta} M^{\Gamma_1})^{\Gamma_2} \to NM^{\Gamma_2}
\eeq
intertwines with the fiberwise $\Gamma_2$-equivariant complex structures. For such a pair $\Gamma_1 \preceq \Gamma_2$, let us denote the complex structure on $N_{D_{\delta} M^{\Gamma_1}} (D_{\delta} M^{\Gamma_1})^{\Gamma_2}$ by $J_0$ and the pullback of the complex structure on $NM^{\Gamma_2}$ under $\exp$ by $J$. Here we view both $J_0$ and $J$ as $\Gamma_2$-equivariant endomorphisms of the vector bundle $N_{D_{\delta} M^{\Gamma_1}} (D_{\delta} M^{\Gamma_1})^{\Gamma_2}$ whose square is $-\mathrm{id}$. Then we claim that there exists a $\Gamma_2$-equivariant bundle isomorphism
\beq
I_{J_0, J}: (N_{D_{\delta} M^{\Gamma_1}} (D_{\delta} M^{\Gamma_1})^{\Gamma_2}, J_0) \to (N_{D_{\delta} M^{\Gamma_1}} (D_{\delta} M^{\Gamma_1})^{\Gamma_2}, J)
\eeq
of the form $I_{J_0, J} = \mathrm{id} + \frac12 J_0 A_{J_0, J}$ which pulls back $J$ to $J_0$, where $A_{J_0, J}$ is a $\Gamma_2$-equivariant endomorphism which anti-commutes with $J_0$. Indeed, explicitly, we can write down the formula
\beq
A_{J_0, J} = -2 J_0 (\mathrm{id} + J_0 J)(\mathrm{id}-J_0 J)^{-1}.
\eeq
Now if we choose a smooth $\Gamma_2$-invariant function
\beq
\varphi: N_{M^{\Gamma_2}} M^{\Gamma_1} \to {\mb R}
\eeq
which is equal to $0$ near the $0$-section and is equal to $1$ outside a small disc bundle over the $0$-section, we can define a $\Gamma_2$-equivariant complex structure on $N_{D_{\delta} M^{\Gamma_1}} (D_{\delta} M^{\Gamma_1})^{\Gamma_2}$ by the formula
\beq
(\mathrm{id} + \frac12 \varphi J_0 A_{J_0, J}) \circ J_0 \circ (\mathrm{id} + \frac12 \varphi J_0 A_{J_0, J})^{-1}.
\eeq
If we push forward such a complex structure to $NM^{\Gamma_2}$ and extend it by the original normal complex structure, based on the construction, we see that it is straightened with respect to $g$ along $M^{\Gamma_2} \cap \exp(D_{\delta} M^{\Gamma_1})$. If we range over all $\Gamma_1 \preceq \Gamma_2$, and perform such a cut-off construction using induction based on the induced partial orders on the set of all such $\Gamma_2$, the compatibility of different cut-off constructions makes sure that the normal complex structure can be straightened near $M^{\Gamma_1}$.

Then we can prove the existence by induction using the partial order $\preceq$. The details are similar to the proof of \cite[Lemma 3.15]{Bai_Xu_2022}, to which we refer the reader. By construction, such a newly constructed normal complex structure can be deformed to the original one by varying the cut-off functions. Note that the only choice we make in the construction is a sequence of cut-off functions. For different choices, we can convexly interpolate between the resulting straightened normal complex structures upon restricting to smaller open subsets containing the fixed point loci. The contractibility of the space of our straightened normal complex structures follows directly.
\end{proof}

\begin{cor}
If ${\mc U}$ is a compact orbifold with a normal complex structure, then given any Riemannian metric $g$ on ${\mc U}$, there exists a contractible family of flattened normal complex structures on ${\mc U}$ with respect to $g$.
\end{cor}
\begin{proof}
The proof is similar to the proof of Lemma \ref{lem:flattening}. However, the induction is based on the following partial order on the set of isotropy types (cf. Definition \ref{defn:iso-type}): we declare $(\Gamma_1, V_1) \preceq (\Gamma_2, V_2)$ if and only if $\Gamma_2 \leq \Gamma_1$ and $V_2$ is the complement of $\Gamma_2$-invariant subspace of $V_1$. Such a partial order induces a partial order on the set of normal invariants of ${\mc U}$ (cf. Definition \ref{defn:normal-invariant}), which leads to a filtration of ${\mc U}$ by closed subspaces. Lemma \ref{lem:flattening} describes the construction over \'etale charts, which can be globalized for orbifolds.
\end{proof}

\begin{defn}\label{defn:suborbifold-normal-complex}
Let ${\mc U}$ be an orbifold with a normal complex structure. Suppose that ${\mc U}' \subset {\mc U}$ is a suborbifold such that $N_{\mc U} {\mc U}'$ admits a complex structure. The complex structure on $N_{\mc U} {\mc U}'$ is \emph{compatible} with the normal complex structure if for any $x \in |{\mc U}'|$ and an open embedding $U // \Gamma \to {\mc U}$ with $y \mapsto x$ under $U / \Gamma \to |{\mc U}|$, the normal space of the preimage of ${\mc U}'$ in $U$ at $y$ with the induced complex structure from $N_{\mc U} {\mc U}'$ is a complex subspace of $N_y U^{\Gamma_y}$, which is endowed with the complex structure from the normal complex structure on ${\mc U}$.
\end{defn}

In practice, we will construct complex structures on $N_{\mc U} {\mc U}'$ using the normal complex structure on ${\mc U}$ using concrete geometric information. Under the situation of Definition \ref{defn:suborbifold-normal-complex}, we will also say ${\mc U}'$ is a suborbifold \emph{with complex normal bundle}.

\begin{defn}
Let ${\mc U}$ be a compact orbifold with a normal complex structure, which is flattened with respect to a Riemannian metric $g$ on ${\mc U}$. Suppose that ${\mc U}' \subset {\mc U}$ is a suborbifold whose normal bundle $N_{\mc U} {\mc U}'$ admits a complex structure compatible with the normal complex structure on ${\mc U}$. The \emph{blow-up} of ${\mc U}$ along ${\mc U}'$, written as $\mathrm{Bl}_{{\mc U}'}{\mc U}$, is defined to be the pushout of the diagram
\beq
\begin{tikzcd}
                                                              &  & {\mc U}                                                        \\
{ D_{\delta} {\mc U}' \setminus {\mc U}'} \arrow[rru, "\exp"] \arrow[rrd, "\text{inclusion}"'] &  &                                                                \\
                                                              &  & \mathrm{Bl}_{\mc U'} (N_{\mc U} {\mc U}') |_{D_{\delta} {\mc U}'},
\end{tikzcd}
\eeq
which comes with a map $\mathrm{Bl}_{{\mc U}'}{\mc U} \to {\mc U}$ called the \emph{blow-down} map. Here $\delta > 0$ is chosen so that the normal exponential map over the disc bundle $D_{\delta} {\mc U}'$ defines an open embedding of normal complex orbifolds, and $\mathrm{Bl}_{\mc U'} (N_{\mc U} {\mc U}') |_{D_{\delta} {\mc U}'}$ denotes the preimage of $D_{\delta} {\mc U}' \subset N_{\mc U} {\mc U}'$ under the blow-down map $\mathrm{Bl}_{{\mc U}'} (N_{\mc U} {\mc U}') \to N_{\mc U} {\mc U}'$. The image of ${\mb P}(N_{\mc U} {\mc U}')$ is called the \emph{exceptional divisor}.
\end{defn}

According to the definition, $\mathrm{Bl}_{{\mc U}'}{\mc U}$ admits a normal complex structure, and the blow-down map is a smooth map. If we merely have a suborbifold ${\mc U}' \subset {\mc U}$ such that $N_{\mc U} {\mc U}'$ has a complex structure compatible with a given normal complex structure on ${\mc U}$, after choosing a Riemannian metric on ${\mc U}$, we can deform the normal complex structure on ${\mc U}$ to obtain a flattened normal complex structure so that we can perform the blow-up construction afterwards. Note that the space of auxiliary choices made in the construction is contractible. In particular, $\mathrm{Bl}_{{\mc U}'}{\mc U}$ is well-defined up to isomorphism of orbifolds, and the resulting normal complex structures are all isotopic to each other. From now on, we will refer to $\mathrm{Bl}_{{\mc U}'}{\mc U} \to {\mc U}$ as \emph{the} blow-up of ${\mc U}$ along ${\mc U}'$, though the reader should keep in mind that certain choices need to be made.

\begin{rem}
    Note that any open suborbifold ${\mc U}_1$ of a normally complex orbifold ${\mc U}_2$ acquires a natural normal complex structure. For a blow-up construction of ${\mc U}'$ whose center ${\mc U}_1'$ is the pullback of an suborbifold ${\mc U}_2'$ of ${\mc U}$ with complex normal bundle, we see that $\mathrm{Bl}_{{\mc U}_1'} {\mc U}_1$ is an open suborbifold of $\mathrm{Bl}_{{\mc U}_2'} {\mc U}_2$ provided that the flattened normal complex structures are compatible. When we state functoriality properties concerning resolutions of singularities under open embeddings, we always encounter this situation.
\end{rem}

\subsection{The abelianization algorithm}

For the statement of the next result, we say that an orbifold is \emph{geometrically abelian} if the representation of the stabilizer associated to each point is a direct sum of $1$-dimensional representations. This is equivalent to the requirement that this stabilizer group be an extension of an abelian group by the stabilizer at a generic point.

\begin{thm}[Abelianization]\label{thm:abel} 
Suppose that ${\mc U}$ is a compact normally complex orbifold. Then there exists a geometrically abelian compact normally complex orbifold ${\mc U}^{ab}$ with a smooth surjective map
\beq
\pi^{ab}: {\mc U}^{ab} \to {\mc U}
\eeq
which is a composition of a sequence of blow-down maps such that for each blow-up, the center is a suborbifold with complex normal bundle.

Given an open embedding of compact normally complex orbifolds ${\mc U}_1 \to {\mc U}_2$, we have an open embedding of compact normally complex orbifolds ${\mc U}_1^{ab} \to {\mc U}_2^{ab}$ such that the following diagram is commutative:
\beq
\begin{tikzcd}
{\mc U}_1^{ab} \arrow[r] \arrow[d, "\pi^{ab}_1"'] & {\mc U}_2^{ab} \arrow[d, "\pi^{ab}_2"] \\
{\mc U}_1 \arrow[r]                               & {\mc U}_2.                           
\end{tikzcd}
\eeq
\end{thm}

The rest of this subsection is devoted to a proof of this theorem. Throughout this section, ${\mc U}$ stands for a compact normally complex orbifold.

Given a representative of an isotropy type $(\Gamma, V)$ (cf. Definition \ref{defn:iso-type}), decompose $V$ into a direct sum of irreducible representations
\beq\label{eqn:decompose-V}
V = \chi_1 \oplus \cdots \oplus \chi_r \oplus \tau_1 \oplus \cdots \oplus \tau_s,
\eeq
where $\chi_1, \dots, \chi_r$ are $1$-dimensional complex $\Gamma$-representations and $\tau_1, \dots, \tau_s$ are complex $\Gamma$-representations of dimension $\geq 2$. The $1$-dimensional $\Gamma$-representations define characters
\beq
\chi_i : \Gamma \to {\mb C}^*
\eeq
and we can define subgroup of $\Gamma$
\beq\label{eqn:Gamma-bar}
\ov{\Gamma} := \bigcap_{i} \ker(\chi_i),
\eeq
and construct a new isotropy type
\beq\label{eqn:iso-type-bar}
(\ov{\Gamma}, \ov{V}) := (\ov{\Gamma}, \tau_1 \oplus \cdots \oplus \tau_s).
\eeq
It follows from the definition that the equivalence class of $(\ov{\Gamma}, \ov{V})$ is independent of the choice of the representative $(\Gamma, V)$.

\begin{defn}
Define the \emph{non-abelianness} of an isotropy type $(\Gamma, V)$ to be $\dim_{\mb C} \ov{V}$, and we denote it by $\mathrm{ord}^{na}(\Gamma, V)$.
\end{defn}

For a normally complex orbifold ${\mc U}$, the non-abelianness function defines a map
\beq
\mathrm{ord}^{na}: |{\mc U}| \to {\mb Z}_{\geq 0},
\eeq
which takes $x \in |{\mc U}|$ to the value of $\mathrm{ord}^{na}$ at the normal invariant of $x$. We see that all the points in ${\mc U}$ have stabilizer given by an extension of the generic stabilizer group by an abelian group if and only if $\mathrm{ord}^{na} \equiv 0$ all over $|{\mc U}|$. In a nutshell, the proof of Theorem \ref{thm:abel} is based on iteratively blowing up the loci along which $\mathrm{ord}^{na}$ achieves its maximum.

\begin{lemma}\label{lem:ord-na-semi-continuous}
$\mathrm{ord}^{na}: |{\mc U}| \to {\mb Z}_{\geq 0}$ is an upper semi-continuous function.
\end{lemma}
\begin{proof}
It suffices to show that for a sequence $\{x_{\nu}\} \subset |{\mc U}|$ with $x_{\nu} \to x$, we have
\beq
\mathrm{ord}^{na}(x) \geq \limsup_{\nu} \mathrm{ord}^{na}(x_{\nu}).
\eeq
Without loss of generality, we can assume that ${\mc U} = \tilde{V} // \Gamma$ such that:
\begin{itemize}
    \item $\Gamma$ is a finite group;
    \item $\tilde{V}$ is a faithful $\Gamma$-representation with a decomposition $V = {\mb R}^k \oplus V$, where $\Gamma$ acts trivially on ${\mb R}^k$ and $V$ is a complex $\Gamma$-representation without trivial summand;
    \item $x = 0 \in \tilde{V}$.
\end{itemize}
Then the stabilizer of any point $x_{\nu}$ can be identified with a subgroup $\Gamma_{x_{\nu}} \leq \Gamma$. By restricting along $\Gamma_{x_{\nu}}$, the $\Gamma$-representation $V$ can be decomposed as
\beq
V = \mathring{V}^{\Gamma_{x_{\nu}}} \oplus \check{V}^{\Gamma_{x_{\nu}}},
\eeq
in which $\mathring{V}^{\Gamma_{x_{\nu}}}$ is the trivial isotypic piece. Using the decomposition \eqref{eqn:decompose-V}, we know that the direct sum of all irreducible $\Gamma_{x_{\nu}}$-representations of dimension $\geq 2$ in $\check{V}^{\Gamma_{x_{\nu}}}$ is a subspace of
\beq\label{eqn:nonabelian-intersect}
(\tau_1 \oplus \cdots \oplus \tau_s) \cap \check{V}^{\Gamma_{x_{\nu}}}.
\eeq
Because the normal invariant of $x_{\nu}$ is represented by $(\Gamma_{x_{\nu}}, \check{V}^{\Gamma_{x_{\nu}}})$, we see that $\mathrm{ord}^{na}(x_{\nu}) \leq \mathrm{ord}^{na}(x)$ and the assertion follows.
\end{proof}

Denote by $|{\mc U}_{na}|$ the subspace of $|{\mc U}|$ consists of all the points at which $\mathrm{ord}^{na}$ achieves its maximum, then we immediately have

\begin{cor}
$|{\mc U}_{na}|$ is a closed subset of $|{\mc U}|$. \qed
\end{cor}

In fact, we can look at more refined structures.

\begin{lemma}\label{lem:abelianize-center}
The closed substack ${\mc U}_{na}$ which is defined to be the preimage of $|{\mc U}_{na}|$ of the coarse space map ${\mc U} \to |{\mc U}|$ is a suborbifold, which can be equipped with complex normal bundle.
\end{lemma}
\begin{proof}
We make a preliminary observation: given an isotropy type represented by $(\Gamma, V)$, then the subgroup  $\ov{\Gamma}$ of $\Gamma$ appearing in Equation \eqref{eqn:iso-type-bar}, is \emph{normal}.

Let $\coprod_{\alpha} U_{\alpha} \to {\mc U}$ be an \'etale atlas which equips ${\mc U}$ with a normal complex structure as in Definition \ref{defn:normal-complex}. Based on Proposition \ref{prop:orbi-chart}, passing to a refinement if necessary, we can assume that $U_{\alpha}$ is an open subset of an ${\mb R}$-linear faithful $\Gamma_{\alpha}$-representation $\tilde{V}_{\alpha}$ such that $\tilde{V}_{\alpha} = {\mb R}^{k_{\alpha}} \oplus V_{\alpha}$, where ${\mb R}^{k_{\alpha}}$ is the trivial isotypic piece and $V_{\alpha}$ is a complex $\Gamma_{\alpha}$-representation without any trivial summand.

Suppose that $x \in |{\mc U}_{na}|$ is contained in the image of $U_{\alpha} / \Gamma_{\alpha}$. Up to a further refinement, we can assume that all such \'etale charts are \emph{centered} at $x$, i.e. $0 \in U_{\alpha} \subset {\mb R}^{k_{\alpha}} \oplus V_{\alpha}$ and $\Gamma_{\alpha}$ is isomorphic to $\Gamma_x$. Using the subgroup $\ov{\Gamma}_{\alpha} \leq \Gamma_{\alpha}$ (cf. \eqref{eqn:Gamma-bar}), we obtain a decomposition
\beq\label{eqn:decompose-V-alpha}
V_{\alpha} = \mathring{V}_{\alpha}^{\ov{\Gamma}_{\alpha}} \oplus \check{V}_{\alpha}^{\ov{\Gamma}_{\alpha}},
\eeq
in which $\mathring{V}_{\alpha}^{\ov{\Gamma}_{\alpha}}$ is the trivial isotypic piece of the restricted $\ov{\Gamma}_{\alpha}$-representation. Then we know that $|{\mc U}_{na}|$ is covered by the union of images of
\beq\label{eqn:decompose-U-alpha}
\big( U_{\alpha} \cap ({\mb R}^{k_{\alpha}} \oplus \mathring{V}_{\alpha}^{\ov{\Gamma}_{\alpha}}) \big) / \Gamma_{\alpha},
\eeq
and any point in it is mapped into $|{\mc U}_{na}|$. Note that $\mathring{V}_{\alpha}^{\ov{\Gamma}_{\alpha}}$ is preserved by $\Gamma_{\alpha}$ because $\ov{\Gamma}_{\alpha} \leq \Gamma_{\alpha}$ is a normal subgroup. The disjoint union
\beq
\coprod_{\alpha}  U_{\alpha} \cap ({\mb R}^{k_{\alpha}} \oplus \mathring{V}_{\alpha}^{\ov{\Gamma}_{\alpha}})
\eeq
with $({\mc U}_{\alpha} / \Gamma_{\alpha}) \cap |{\mc U}_{na}| \neq \emptyset$  defines a smooth \'etale atlas of the closed substack ${\mc U}_{\alpha}$. Indeed, the compatibility is already encoded in the orbifold structure on ${\mc U}$ visualized through $\coprod_{\alpha} U_{\alpha} \to {\mc U}$. By construction, we see that ${\mc U}_{na} \to {\mc U}$ is a suborbifold.

Finally, using the decomposition in Equation \eqref{eqn:decompose-V-alpha}, the complex structure on $\check{V}_{\alpha}^{\ov{\Gamma}_{\alpha}}$, as part of the normal complex structure on ${\mc U}$, endows the normal bundle $N_{\mc U} {\mc U}_{na}$ with a complex structure.
\end{proof}

\begin{proof}[Proof of Theorem \ref{thm:abel}]
If the function $\mathrm{ord}^{na}$ identically vanishes on ${\mc U}$, the orbifold is already geometrically abelian, so we may assume that the maximally non-abelian locus  ${\mc U}_{na}$ as in Lemma \ref{lem:abelianize-center} is non-empty.

Blow up ${\mc U}$ along ${\mc U}_{na}$ to obtain $\mathrm{Bl}_{{\mc U}_{na}} {\mc U}$. Then we claim that
\beq
\max_{|\mathrm{Bl}_{{\mc U}_{na}} {\mc U}|} \mathrm{ord}^{na} < \max_{|{\mc U}|} \mathrm{ord}^{na}.
\eeq
Because the blow-down map $\mathrm{Bl}_{{\mc U}_{na}} {\mc U} \to {\mc U}$ is an isomorphism away from ${\mc U}_{na}$, by the maximality of $\mathrm{ord}^{na}$ along ${\mc U}_{na}$, we only need to verify the above inequality over the exceptional divisor ${\mb P}(N_{\mc U} {\mc U}_{na})$. It suffices to check things \'etale locally. Using the notations from \eqref{eqn:decompose-V-alpha} and \eqref{eqn:decompose-U-alpha},  we can locally identify ${\mc U}_{na} \subset {\mc U}$ as
\beq\label{eqn:subset-local}
({\mb R}^{k_{\alpha}} \oplus \mathring{V}_{\alpha}^{\ov{\Gamma}_{\alpha}}) // \Gamma_{\alpha} \subset V_{\alpha} // \Gamma_{\alpha}.
\eeq
Then the exceptional divisor ${\mb P}(N_{\mc U} {\mc U}_{na})$ can be locally identified with 
\beq
\big( ({\mb R}^{k_{\alpha}} \oplus \mathring{V}_{\alpha}^{\ov{\Gamma}_{\alpha}}) \times {\mb P}(\check{V}_{\alpha}^{\ov{\Gamma}_{\alpha}}) \big) // \Gamma_{\alpha},
\eeq
where $\Gamma_{\alpha}$ acts on ${\mb P}(\check{V}_{\alpha}^{\ov{\Gamma}_{\alpha}})$ as projective linear transformations. Then $\mathrm{Bl}_{{\mc U}_{na}} {\mc U}$ can be locally identified with the global finite quotient orbifold
\beq\label{eqn:blowup-local}
\big( ({\mb R}^{k_{\alpha}} \oplus \mathring{V}_{\alpha}^{\ov{\Gamma}_{\alpha}}) \times {\mc O}_{\mb P(\check{V}_{\alpha}^{\ov{\Gamma}_{\alpha}})}(-1) \big) // \Gamma_{\alpha},
\eeq
where ${\mc O}_{\mb P(\check{V}_{\alpha}^{\ov{\Gamma}_{\alpha}})}(-1) \to {\mb P}(\check{V}_{\alpha}^{\ov{\Gamma}_{\alpha}})$ is the tautological line bundle with the natural $\Gamma_{\alpha}$-action. Note that the space of ${\mc O}_{\mb P(\check{V}_{\alpha}^{\ov{\Gamma}_{\alpha}})}(-1) \setminus \mb P(\check{V}_{\alpha}^{\ov{\Gamma}_{\alpha}})$ can be identified with the complement of the origin of $\check{V}_{\alpha}^{\ov{\Gamma}_{\alpha}}$. On the other hand, the normal invariant of any point in the coarse space of the exceptional divisor is equal to the normal invariant of any nonzero point (in the coarse space) lying on the line specified by such a point in the projective space.
For any point in the coarse space of the left side of \eqref{eqn:subset-local}, we know the value of $\mathrm{ord}^{na}$ is equal to $\dim_{\mb C} \check{V}_{\alpha}^{\ov{\Gamma}_{\alpha}}$. On the other hand, {the isotropy group of any vector in $\check{V}_{\alpha}^{\ov{\Gamma}_{\alpha}} \setminus \{0\}$ has a nontrivial invariant subspace containing the span of this vector in $\check{V}_{\alpha}^{\ov{\Gamma}_{\alpha}}$}, which implies that the value of $\mathrm{ord}^{na}$ at any point in the exceptional divisor is strictly less than $\dim_{\mb C} \check{V}_{\alpha}^{\ov{\Gamma}_{\alpha}}$. As a result, $\max \mathrm{ord}^{na}$ gets decreased under the blow-up.

Due to the compactness of ${\mc U}$, after blowing up along ${\mc U}_{na}$ for finitely many times, we can guarantee that $\mathrm{ord}^{na} \equiv 0$. This final output is defined to be ${\mc U}^{ab}$, and the composition of blow-down maps defines the map $\pi^{ab}: {\mc U}^{ab} \to {\mc U}$ satisfying all the requirements. Moreover, it follows from the construction that abelianization is functorial under open embeddings.
\end{proof}

\subsection{Discussions}
Due to the elementary nature of the abelianization procedure, it deserves some further comments to make connections with results from the literature. The reader is invited to figure out methods to abelianize  normally complex orbifolds.

Our choice of invariant, which follows \cite{bergh2019functorial}, is referred to as the \emph{codimension of stackiness} in \emph{loc. cit.}. Alternatively, as pointed to the authors by David Rydh, one could use the order of $\ov{\Gamma}$ as an invariant.

If ${\mc U}$ is the finite global quotient orbifold constructed from a complex manifold acted on complex analytically by a finite group, there is an alternative construction due to Batyrev \cite{batyrev2000canonical}. There, the invariant measuring the `nonabelian-ness' is simply the maximal value of orders of non-abelian groups appearing as stabilizers of ${\mc U}$. There is a variant of ${\mc U}_{na}$ in this setting, which is constructed from the fixed point locus of $\ov{\Gamma}$ if $\Gamma$ achieves the maximum of the `nonabelian-ness' invariant here. However, because there may be multiple choices of $\Gamma$, the union of the corresponding fixed point loci of $\ov{\Gamma}$ may not be a smooth suborbifold. In the algebraic or complex analytic setting, this union is a subvariety with simple normal crossings. The crossings need to be removed using blow-ups before performing other blow-ups to reduce the `nonabelian-ness'. It should be possible to adapt a variant of Batyrev's algorithm in our setting, though the technicalities to be deployed may be more substantial.

As observed in \cite[Theorem 4.1]{RY00}, the functorial resolution in characteristic $0$ of Bierstone--Milman \cite{Bierstone-Milman} (or other approaches which make Hironaka's original method functorial) can be used to abelianize smooth Deligne--Mumford stacks over ${\mb C}$ (we learned this from \cite{bergh}). The idea is based on performing blow-ups such that any point with nontrivial stabilizer is contained in a (locally invariant) simple normal crossings divisor. If this is the case, the action of the stabilizer on the normal space of the divisor factors through ${\mb G}_m^k$. Such a method may also work in the almost complex category, but it has the disadvantage that the resolution procedure is rather implicit.

\section{Desingularizing the coarse spaces}\label{sec:bergh}

In this section, we follow Bergh's functorial destackification algorithm of Deligne--Mumford stacks with abelian stabilizers \cite{bergh} to perform blow-up and root stack constructions to obtain from any compact and geometrically abelian normally complex orbifold,  a compact and effective normally complex orbifold whose coarse space is a smooth manifold. As in the abelianization procedure, the fact that the algebro-geometric approach in \cite{bergh} can be applied to study normally complex orbifolds is eventually due to the fact that all the invariants which specify the center of blow-up or the divisor along which the root stack is taken are derived from the normal invariant Definition \ref{defn:normal-invariant}, and these centers are suborbifolds with complex normal bundle.

A minor difference with the setup in \cite{bergh} is that we do not assume effectiveness of the underlying orbifold. In our presentation, for a finite group $\Gamma$, we will use the same symbol to denote a $1$-dimensional complex $\Gamma$-representation and its associated character.

\subsection{Root stacks}
\label{sec:root-stacks}

We begin by defining the \emph{root stack} procedure to incorporate ideas of \emph{stacky blow-ups} from, e.g., \cite{bergh, bergh2019functorial}.

\begin{defn}
Suppose that ${\mc U}$ is a normally complex orbifold. A real codimension $2$ suborbifold ${\mc D} \subset {\mc U}$ is called a \emph{divisor} if $N_{\mc U} {\mc D}$ admits a complex structure which is compatible with the normal complex structure on ${\mc U}$ (cf. Definition \ref{defn:suborbifold-normal-complex}).
\end{defn}

A familiar construction from, e.g., \cite[Section 1.1]{GH78}, shows that given a divisor ${\mc D} \subset {\mc U}$, there exists a complex line bundle $L_{\mc D} \to {\mc U}$ characterized by:
\begin{enumerate}
    \item $L_{\mc D}$ is isomorphic to the product bundle over ${\mc U} \setminus {\mc D}$;
    \item identifying an open neighborhood of ${\mc D} \subset N_{\mc U} {\mc D}$ with an open subset in ${\mc U}$ containing ${\mc D}$, the restriction of $L_{\mc D}$ to such an open subset is isomorphic to the pullback of $N_{\mc U} {\mc D}$ under the projection $N_{\mc U} {\mc D} \to {\mc D}$.
\end{enumerate}
Moreover, there exists a section $\tau_{\mc D}: {\mc U} \to L_{\mc D}$ whose vanishing locus is exactly ${\mc D}$.

\begin{defn}
Suppose ${\mc D}$ is a divisor in a normally complex orbifold ${\mc U}$. Let $L_{\mc D} \to {\mc U}$ be the associated line bundle and write $L_{\mc D}^{\times}$ for the complement of the zero-section. Consider the bundle map
\beq
\begin{tikzcd}
L_{\mc D}^{\times} \times {\mb C} \arrow[rd] \arrow[rr] &         & L_{\mc D} \arrow[ld] \\
                                                        & {\mc U} &                     
\end{tikzcd}
\eeq
defined by mapping $(l, c) \in L_{\mc D}^{\times} \times {\mb C}$ to $c^d l$, where $d \in {\mb Z}_{\geq 1}$. Denote the preimage of the embedding $\tau_{D}: {\mc U} \to L_{\mc D}$ in $L_{\mc D}^{\times} \times {\mb C}$ by $V(\tau_{\mc D})$. Then the \emph{$d$-th root stack} of ${\mc U}$ along ${\mc D}$ is defined to be the quotient orbifold 
\beq
V(\tau_{\mc D}) / {\mb C}^*,
\eeq
where the group ${\mb C}^*$ acts on $V(\tau_{\mc D})$ by restricting the ${\mb C}^*$-action on $L_{\mc D}^{\times} \times {\mb C}$ given by 
\beq
(z, (l, c) ) \to (z^{-d} \cdot l, z \cdot c), \quad \quad \quad z \in {\mb C}^*.
\eeq
\end{defn}

\begin{rem}
Our construction of root stacks is simply an implementation of the construction from algebraic geometry, see, e.g., \cite[Appendix B]{GW-DM}.
\end{rem}

\subsection{Statements and definitions}
We now state the main theorem of this section. Then we introduce fundamental definitions and a collection of invariants which measure how far the coarse space of ${\mc U}$ is being smooth. 

\begin{defn}
A \emph{normally complex pair} $({\mc U}, \mathbf{D})$ consists of:
\begin{enumerate}
    \item a compact normally complex \emph{geometrically abelian} orbifold ${\mc U}$;
    \item an ordered collection $\mathbf{D} = \{ {\mc D}_1, \dots, {\mc D}_r \}$ of divisors of ${\mc U}$ which intersect transversely, such that for any $x \in |{\mc D}_i|$, the action of $\Gamma_x$ on the corresponding fiber of $N_{\mc U} {\mc D}_i$ is nontrivial.
\end{enumerate}
\end{defn}

This is the analogue of \emph{standard pair} from \cite[Definition 2.1]{bergh}. The requirement (2) is the counterpart of simple normal crossings divisors in our context, which we have strengthened by requiring the non-triviality of the action of stabilizer groups,  
{ since a divisor whose normal bundle carries the trivial representation can be neglected in our construction.}
Note that the components ${\mc D}_i$ are not assumed to be connected, and it is important to keep track of the orderings. For two normally complex pairs $({\mc U}', \mathbf{D}')$ and $({\mc U}, \mathbf{D})$, we use the notation $f: ({\mc U}', \mathbf{D}') \to ({\mc U}, \mathbf{D})$ to denote a smooth map $f: {\mc U}' \to {\mc U}$ which maps $\mathbf{D}'$ into $\mathbf{D}$ such that the orderings are respected.

\begin{defn}[c.f. {\cite[Definition 2.7]{bergh}}]\label{defn:pair-distinguish}
  A normally complex pair with \emph{distinguished structure} is a triple $({\mc U}, \mathbf{D}, \acute{\mathbf{D}})$ such that $({\mc U}, \mathbf{D})$ is a normally complex pair, $\acute{\mathbf{D}} \subset \mathbf{D}$ is a subcollection of divisors from $\mathbf{D}$ which are not preceded, under the ordering of $\mathbf{D}$, by any element of $\mathbf{D} \setminus \acute{\mathbf{D}}$.
\end{defn}

The notation $f: ({\mc U}', \mathbf{D}', \acute{\mathbf{D}}') \to ({\mc U}, \mathbf{D}, \acute{\mathbf{D}})$ should be self-explanatory based on the previous definition.

\begin{thm}\label{thm:deorbi}
If $({\mc U}, \mathbf{D})$ is a normally complex pair, then there exists a sequence of normally complex pairs 
\beq
\pi^M: (\tilde{\mc U}, \tilde{\mathbf{D}}) := ({\mc U}_n, \mathbf{D}_n) \xrightarrow{\pi_n} \cdots \xrightarrow{\pi_1} ({\mc U}_0, \mathbf{D}_0) = ({\mc U}, \mathbf{D})
\eeq
such that:
\begin{enumerate}
    \item each $\pi_i: {\mc U}_i \to {\mc U}_{i-1}$ is a smooth map induced from either the blow-up along a suborbifold with complex normal bundle in ${\mc U}_{i-1}$, or a root stack construction along a divisor in $\mathbf{D}_{i-1}$;
    \item the coarse space $|\tilde{\mc U}|$ is a smooth \emph{manifold}.
\end{enumerate}
If $f: ({\mc U}', \mathbf{D}') \to ({\mc U}, \mathbf{D})$ defines an open suborbifold ${\mc U}' \subset {\mc U}$ such that $\mathbf{D}' = f^* \mathbf{D}$ and the normal complex structures are compatible, then there exists a smooth map $\tilde{f}: (\tilde{\mc U}', \tilde{\mathbf{D}}') \to (\tilde{\mc U}, \tilde{\mathbf{D}})$ defining an open suborbifold $\tilde{f}: \tilde{\mc U}' \to \tilde{\mc U}$ such that the commutative diagram
\beq
\begin{tikzcd}
{(\tilde{\mc U}', \tilde{\mathbf{D}}')} \arrow[d] \arrow[r, "\tilde{f}"] & {(\tilde{\mc U}, \tilde{\mathbf{D}})} \arrow[d] \\
{({\mc U}', \mathbf{D}')} \arrow[r, "f"]                                 & {({\mc U}, \mathbf{D})}                        
\end{tikzcd}
\eeq
is Cartesian. Accordingly, by passing to coarse space, we obtain a smooth embedding 
\beq
|\tilde{f}|: |\tilde{\mc U}'| \to |\tilde{\mc U}|.
\eeq
\end{thm}

The above construction is referred to as a \emph{deorbification} process. If we start with $\mathbf{D} = \emptyset$, it provides a method to obtain an orbifold with smooth coarse space from an normally complex orbifold ${\mc U}$ with only abelian geometric stabilizers. Although the divisors could be eliminated in the final statement which asserts the smoothness of $|\tilde{\mc U}|$, they play a pivotal role in the proof. As can be inferred from the statement, such a deorbification is functorial in under open embeddings. This functoriality is important for the compatibility between local and global constructions. We will prove Theorem \ref{thm:deorbi} at the end of this section after finishing all the preparatory work.

Let $({\mc U}, \mathbf{D})$ be a normally complex pair. Given $x \in |{\mc U}|$, let $(\Gamma, V)$ be a representative of its normal invariant (cf. Definition \ref{defn:normal-invariant}). For $i \in \{ 1, \dots, r\}$, define
\beq
\chi_i =
\left\{
\begin{aligned}
\text{the character associated to $N_{{\mc D}_i} {\mc U}$} \subset V, \quad \quad \quad &\text{if } x \in |{\mc D}_i|,\\
 0 \quad \quad \quad & \text{if } x \notin |{\mc D}_i|.
\end{aligned}
\right.
\eeq
The characters $\{ \chi_i \}$ inherits an ordering from the ordering of the divisors. If $\dim_{\mb C} V = s$, we can decompose it into a direct sum of irreducible representations
\beq\label{eqn:decompose-char}
V \cong \chi_{i_1} \oplus \cdots \oplus \chi_{i_{r'}} \oplus \chi_{i_{r' +1}} \oplus \cdots \oplus \chi_s = \chi_1 \oplus \cdots \oplus \chi_s,
\eeq
in which the representations/characters $\chi_{i_1}, \dots \chi_{i_{r'}}$ come from the normal directions of the divisors and $\chi_{i_{r' +1}}, \dots, \chi_s$ are the `residual' directions.

All of our subsequent invariants are constructed out of the decomposition of the normal invariant \eqref{eqn:decompose-char}. Given a finite group $\Gamma$, we use the notation $\chi(\Gamma)$ to denote the group of characters of $\Gamma$, which is also known as the \emph{Cartier dual}. For a sequence of characters $\chi_1, \dots, \chi_l$, write $\langle \chi_1, \dots, \chi_l \rangle$ for the subgroup of $\chi(\Gamma)$ generated by these characters. Using this notation, define the subgroups of $\chi(\Gamma)$
\beq\label{eqn:chi-gamma-div}
\chi(x) := \langle \chi_1, \dots, \chi_s \rangle, \quad \quad \quad \chi_{div}(x) := \langle \chi_{i_1}, \dots, \chi_{i_{r'}} \rangle
\eeq
following \eqref{eqn:decompose-char}. Concretely, $\chi_{div}(x)$ is the group of $\Gamma$-characters generated by those associated to the normal representations of the divisors in $\mathbf{D}$ passing through $x$.

\begin{defn}[{\cite[Definition 7.1]{bergh}}]\label{defn:independence-ind}
  Let $({\mc U}, \mathbf{D})$ be a normally complex pair. Let $(\Gamma, V)$ be a representative of the normal invariant of $x \in |{\mc U}|$ with decomposition \eqref{eqn:decompose-char}. The \emph{independence index}
  of $x$ is defined to be the number of indices $i \in \{1, \dots, s\}$ such that
\beq\label{eqn:independent}
\langle \chi_i \rangle \cap \langle \chi_1, \dots, \hat{\chi_i}, \dots, \chi_s \rangle = \{0\}.
\eeq
\end{defn}
Note that if the independence index of $x$ is $0$, the coarse space $|{\mc U}|$ is naturally a smooth manifold near $x$ because locally the coarse space, which is a quotient space, is homeomorphic to ${\mb R}^k \times V$. A trivial but illustrative example is the following: for the cyclic group  $C_2$ of order $2$, let ${\mb C}(0)$ be the trivial representation and ${\mb C}(-1)$ be the sign representation. Then the coarse space $|({\mb C}(0) \oplus {\mb C}(-1)) // C_2|$ is smooth, but $|({\mb C}(-1) \oplus {\mb C}(-1)) // C_2|$ is not smooth. The independence index does not use any information from $\mathbf{D}$.

\begin{defn}[{\cite[Definition 7.4]{bergh}}]\label{defn:toroidal-ind}
Suppose $({\mc U}, \mathbf{D})$ is a normally complex pair. Let $(\Gamma, V)$ be a representative of the normal invariant of $x \in |{\mc U}|$ with decomposition \eqref{eqn:decompose-char}. The \emph{toroidal index} of $x$ is defined to be the number $s - r'$.
\end{defn}
If the toroidal index of $x$ is $0$, then the $\Gamma$-representation $V$ is decomposed into the direct sum of the normal representations from divisors in $\mathbf{D}$ passing through $x$. If this is the case, $x$ admits an orbifold chart modeled on \emph{toric orbifolds} along the normal directions, whose partial deorbification process will be recalled in Section \ref{subsec:toric-recall}. Reducing the local model to toric orbifolds allows us to use combinatorial methods to achieve the deorbification/destackification.

Abusing the notation, write $\mathbf{D}$ the totally ordered set consisting of the indices of the divisors in $\mathbf{D}$. For a normally complex pair $({\mc U}, \mathbf{D})$ and $x \in |{\mc U}|$, denote by $\chi(x)^{\mathbf D}$ the set of $|\mathbf{D}|$-tuples of characters in the group $\chi(x)$, where each component which comes from the normal representation of a divisor in $\mathbf{D}$ should be thought of as marked by the corresponding index.

\begin{defn}[{\cite[Definition 7.8]{bergh}}]\label{defn:divisor-type}
Given a point $x \in |{\mc U}|$ as above with decomposition \eqref{eqn:decompose-char}, define the tuple $\mathbf{v}(x) \in \chi(x)^{\mathbf D}$ by
\beq
\mathbf{v}(x)_i =
\left\{
\begin{aligned}
0 \quad \quad \quad & \text{if } x \notin |{\mc D}_i| \text{ or }x \in |{\mc D}_i| \text{ and } \langle \chi_i \rangle \cap \langle \chi_1, \dots, \hat{\chi_i}, \dots, \chi_s \rangle = \{0\},\\
\chi_i, \quad \quad \quad &\text{if } x \in |{\mc D}_i| \text{ and } \langle \chi_i \rangle \subset \langle \chi_1, \dots, \hat{\chi_i}, \dots, \chi_s \rangle,\\
\end{aligned}
\right.
\eeq
where $\chi_i$ is the character associated to $N_{{\mc D}_i} {\mc U}$. Denote by $A(x)$ the subgroup of $\chi(x)$ generated by characters in $\mathbf{v}(x)$. Then the \emph{divisorial type} of $x$ is defined to be the pair
\beq
(A(x), \mathbf{v}(x)).
\eeq
\end{defn}

Note that all the characters in $\mathbf{v}(x)$ are marked by a divisor from $\mathbf{D}$. We can introduce a well-ordering on the set of divisorial types of all points in $|{\mc U}|$. Let $j_1, \dots, j_m$ be all the indices in $\mathbf{D}$ such that $\mathbf{v}(x)_{j_1}, \dots, \mathbf{v}(x)_{j_m} \neq 0$, and we assume that $j_1, \dots, j_m$ is in the decreasing order. Then we have a short exact sequence of abelian groups
\beq\label{eqn:divisor-type-exact}
0 \to K(x) \to {\mb Z}^m \to A(x) \to 0
\eeq
where ${\mb Z}^m \to A(x)$ is defined by $\mathbf{v}(x)_i \mapsto \chi_i$ and $K(x)$ is the kernel. Because $A(x)$ is a finite group, $K(x)$ is a free ${\mb Z}$-module of rank $m$. Then we can write the map $K(x) \to {\mb Z}^m$ as a matrix with integer entries $H=(a_{ij})_{m \times m}$ using the Hermite normal form, which is characterized by:
\begin{enumerate}
    \item $H$ is upper triangular;
    \item $a_{ii} > 0$, and $a_{ii} > a_{ij} \geq 0$ for all $j > i$.
\end{enumerate}
Using the ordering on $\mathbf{D}$, the order on the set of divisorial types is induced from the lexicographic order on the set of corresponding $H$: here we first look at the lexicographic order of the sequence $(j_1, \dots, j_m)$, then the lexicographic order on entries of $H$, where the entries are first ordered by rows, with high row numbers being more significant, then ordered by columns, with lower column numbers being more significant. Such a well-order refines the following partial order on the set of divisorial types: $(A(x), \mathbf{v}(x)) \geq (A(x'), \mathbf{v}(x'))$ if there exists a surjective group homomorphism $A(x) \to A(x')$ taking $\mathbf{v}(x)$ to $\mathbf{v}(x')$. The reader should consult \cite[Section 7]{bergh} for more explanations and illustrations through examples.

\begin{defn}[{\cite[Definition 7.13]{bergh}}]\label{defn:divisor-index}
Given a normally complex pair $({\mc U}, \mathbf{D})$, let $(\Gamma, V)$ be a representative of the isotropy type of $x \in |{\mc U}|$. Using the decomposition \eqref{eqn:decompose-char}, define the \emph{divisorial index} of $x$ to be the number of characters from $\{ \chi_{i_{r' +1}}, \dots, \chi_s \}$ specified in \eqref{eqn:decompose-char} which does not lie in $\chi_{div}(x)$ from \eqref{eqn:chi-gamma-div}. $({\mc U}, \mathbf{D})$ is said to be \emph{divisorial} at $x$ if the divisorial index of $x$ is $0$.
\end{defn}

\begin{rem}
If $({\mc U}, \mathbf{D})$ is divisorial at all the points of $|{\mc U}|$, it has the following feature. For any ${\mc D}_i \in \mathbf{D}$, let $L_{{\mc D}_i}$ be the dual line bundle and denote by $P^{{\mb C}^*}_{{\mc D}_i}$ the associated principal ${\mb C}^*$-bundle. Then the fiber product
\beq
P^{{\mb C}^*}_{{\mc D}_1} \times_{\mc U} \cdots \times_{\mc U} P^{{\mb C}^*}_{{\mc D}_r}
\eeq
has pure isotropy, so that the coarse space is actually a \emph{smooth manifold}. It suffices to check this near each $x \in |{\mc U}|$, where the statement reduces to the claim that, in the effective case, the isotropy is trivial. Indeed, because $\chi_{div}(x) = \chi(x)$ at any $x \in |{\mc U}|$, the action of $\Gamma_x = \Gamma$ on the direct sum of the normal representations of the divisors in $\mathbf{D}$ passing through $x$, which is $\chi_{i_1} \oplus \cdots \oplus \chi_{i_{r'}}$ from \eqref{eqn:decompose-char}, is \emph{faithful} due to the effectiveness assumption of ${\mc U}$. As a result, the induced group homomorphism $\Gamma_x \to ({\mb C}^*)^r$, which is defined to be the product of $\Gamma_x$-representation which is the fiber of $L_{{\mc D}_i}$ over $x$, is injective.  In other words, the direct sum $L_{{\mc D}_1} \oplus \cdots \oplus L_{{\mc D}_r}$ is a \emph{faithful vector bundle} is the sense of \cite[Remark 1.6]{pardon19}. 
\end{rem}

\begin{defn}[{\cite[Definition 7.16]{bergh}}]\label{defn:D-div-index}
Let $({\mc U}, \mathbf{D})$ be   a normally complex pair, and let $x$ be a point at which it is divisorial, with decomposition of its normal invariant as in Equation \eqref{eqn:decompose-char}. 

If $x$ lies on a divisor ${\mc D} \in \mathbf{D}$, define the subgroup $\chi_{div}(x)^{\mc D}$ of $\chi_{div}(x)$ (cf. Equation \eqref{eqn:chi-gamma-div})  to be generated by all the characters from $\{ \chi_{i_1}, \dots, \chi_{i_{r'}} \}$ except the one corresponding to ${\mc D}$. Write
\beq\label{eqn:chi-D-x}
\chi_{\mc D}(x) := \chi_{div}(x) / \chi_{div}(x)^{\mc D},
\eeq
which is a cyclic group. Because all the characters from  Equation \eqref{eqn:decompose-char} are contained in $\chi_{div}(x)$, for any character $\chi_i$ from $\{\chi_{i_{r' + 1} }, \dots, \chi_s \}$, there exists a nonnegative integer $c_i$ such that the image of $\chi_i$ under the projection
\beq\label{eqn:char-proj-D}
\chi(x) = \chi_{div}(x) \to \chi_{\mc D}(x)
\eeq
is $c_i$ times the generator. Then the \emph{${\mc D}$-divisorial index} at $x$ is defined to be
\beq\label{eqn:d-div-index}
\sum_{i=i_{r'} + 1}^{s} c_i.
\eeq

If $x \notin |{\mc D}|$, its ${\mc D}$-divisorial index is defined to be $0$.
\end{defn}

Continuing using the notations as above, if for any $\chi_i \in \{ \chi_{i_1}, \dots, \chi_{i_{r'}} \}$ we have
\beq
\langle \chi_i \rangle \cap \langle \chi_{i_1}, \dots, \hat{\chi_i}, \dots \chi_{i_{r'}} \rangle = \{0\},
\eeq
one can see that if the ${\mc D}$-divisorial index vanishes at $x$ for all ${\mc D} \in \mathbf{D}$, the so does the independence index from Definition \ref{defn:independence-ind}. Accordingly, in this situation, the coarse space $|{\mc U}|$ is smooth near $x$.

We will make use of the invariants described in this subsection to achieve the functorial deorbification Theorem \ref{thm:deorbi} step by step.

\subsection{Toric local models}\label{subsec:toric-recall}
Suppose that $(\Gamma, V)$ is a representative of an isotropy type such that $V$ is decomposed as
\beq\label{eqn:decompse-V-another}
V = \chi_1 \oplus \cdots \oplus \chi_s
\eeq
in which each $\chi_i$ is a $1$-dimensional complex $\Gamma$-representation, and we use the same notation to denote the associated $\Gamma$-character. The goal of this subsection is to extend \cite[Algorithm A]{bergh}, which defines a ``partial" resolution in the case $\Gamma$ is abelian to our setting. We begin by recalling the basic notions of toric orbifolds:
\begin{defn}
An \emph{orbifold fan} is a triple $\mathbf{\Sigma} = (N, \Sigma, \beta)$ where $N$ is a finite rank free abelian group, $\Sigma$ is a simplicial fan in $N_{\mb R} = N \otimes_{\mb Z} {\mb R}$, and $\beta: {\mb Z}^{\Sigma(1)} \to N$ is a group homomorphism from the free abelian group generated by rays in $\Sigma$ which maps each generator $\rho \in \Sigma(1)$ to a non-zero lattice point on the ray $\rho$.
\end{defn}

Any orbifold fan gives rise to a toric orbifold $X_{\mathbf{\Sigma}}$: $\beta$ induces  an exact sequence
\begin{equation} \label{eq:exact_sequence_groups_toric}
1 \to  \Delta(\Sigma) \to    {\mb Z}^{\Sigma(1)} \otimes_{\mb Z} \mb C^* \to N \otimes_{\mb Z} \mb C^* \to 1.
\end{equation}
One way to express $X_{\mathbf{\Sigma}}$ is to start with affine space $\mb C^{ \Sigma(1)}$ , on which $ {\mb Z}^{\Sigma(1)} \otimes_{\mb Z} \mb C^*$ acts, remove the unstable locus $Z(\Sigma)$ with respect to the action of $\Delta(\Sigma) $, and take the corresponding quotient for this group.

In the special case of a $\Gamma$-representation $V$ equipped with a decomposition as in Equation \eqref{eqn:decompse-V-another}, if we write $\chi(V) := \langle \chi_1, \dots, \chi_s \rangle \subset \chi(\Gamma)$, we have a group homomorphism
\beq
{\mb Z}^s \to \chi(V)
\eeq
whose kernel $K$ is free of rank $s$. By applying the dual $\hom(-,{\mb Z})$ to the short exact sequence $0 \to K \to {\mb Z}^s \to \chi(V) \to 0$, we obtain a homomorphism
\beq\label{eqn:hom-dual}
\beta :  \hom({\mb Z}^s,{\mb Z}) \to \hom(K,{\mb Z}).
\eeq
Now take $N = \hom(K,{\mb Z})$ and let $\Sigma \subset N_{\mb R}$ be the fan associated to the cone corresponding to the `first quadrant'. In this case, we have that the group $\Delta(\Sigma)  $ in Equation \eqref{eq:exact_sequence_groups_toric} is exactly the image of $\Gamma$ in the torus. The Cox construction of such an orbifold fan recovers the orbifold quotient of $V$ by the image of $\Gamma$.

Toric orbifolds with an ordering on toric divisors are examples of normally complex pairs. Given an ordering on the rays of an orbifold fan, one can define a distinguished structure as in Definition \ref{defn:pair-distinguish} for a collection of toric divisors.

If $\mathbf{\Sigma} = (N, \Sigma, \beta)$ is an orbifold fan, for a cone $\sigma \subset \Sigma$, define $v_{\sigma} = \sum_{\rho \in \sigma(1)} \beta(\rho) \in N_{\mb R}$ and let $\rho_{v_{\sigma}}$ be the ray generated by $v_{\sigma}$. Then the \emph{star subdivision} of $\mathbf{\Sigma}$ along $\sigma$ is the orbifold fan $\mathbf{\Sigma}^*(\sigma) = (N, \Sigma^*(\sigma), \beta^*)$ where $\Sigma^*(\sigma)$ is obtained from $\Sigma$ by subdivision using $v_{\sigma}$, and $\beta^*$ extends $\beta$ by additionally mapping the generator associated to $\rho_{v_{\sigma}}$ to $v_{\sigma}$. This is the combinatorial counterpart of blow-up, as familiar from toric blow-ups of toric varieties.

There is also a reformulation of taking root stacks using orbifold fans. Namely, given an orbifold fan $\mathbf{\Sigma} = (N, \Sigma, \beta)$ and a map $\mathbf{d}: \Sigma(1) \to {\mb Z}_{\geq 1}$, we can define a homomorphism $\beta': {\mb Z}^{\Sigma(1)} \to N$ with $\beta'(\rho) = \mathbf{d}(\rho)\beta(\rho)$. Then we obtain a new orbifold fan $\Sigma_{\mathbf{d}^{-1} \rho} := (N, \Sigma, \beta')$. It is straightforward to see that the corresponding toric orbifold is constructed by iteratively taking $\mathbf{d}(\rho)$-th root stack along the toric divisor specified by $\rho \in \Sigma(1)$.

\begin{prop}[{\cite[Algorithm A]{bergh}}]\label{prop:algo-A}
Given an orbifold fan $\mathbf{\Sigma}_0$ with an ordered set of rays having a distinguished structure, there exists a sequence
\beq
\mathbf{\Sigma}_n \to \cdots \to \mathbf{\Sigma}_0
\eeq
of toric fans such that for any $0 \leq i \leq n-1$, the orbifold fan $\mathbf{\Sigma}_{i+1}$ is obtained from $\mathbf{\Sigma}_{i}$ by a star subdivision or a root stack construction. Moreover, the following properties hold.
\begin{enumerate}
    \item Each orbifold fan $\mathbf{\Sigma}_{i}$ has an ordered set of rays with a distinguished structure.
    \item for $\mathbf{\Sigma}_n$, let $(X_{\mathbf{\Sigma}_n}, \mathbf{D}_n, \acute{\mathbf{D}}_n)$ be the associated normally complex pair with a distinguished structure. Then for any $x \in |X_{\mathbf{\Sigma}_n}|$, if the normal invariant of $x$ is represented by $(\Gamma, V)$ with decomposition \eqref{eqn:decompose-char}, then \eqref{eqn:independent} holds for any $\chi_i$ from $N_{X_{\mathbf{\Sigma}_n}} {\mc D}_i$ with ${\mc D}_i \in \acute{\mathbf{D}}_n$.
\end{enumerate}
\end{prop}

The content of the above statement can be summarized as ``making the distinguished divisors independent". We refer the reader to the original reference for full proof. We remark that our definition of normally complex pairs requires the ambient orbifold to be compact, which is not necessarily satisfied by $X_{\mathbf{\Sigma}_n}$, but we hope that the meaning of the above statement is transparent.

We shall use the above result in Section \ref{sec:proof-theor-refthm:d} to lower the independence the independence index of a possibly ineffective orbifold whose underlying effective orbifold is abelian, as all the steps in the construction of Proposition \ref{prop:algo-A} lift to the ineffective case.

\begin{rem}[{\cite[Remark 7.11]{bergh}}]\label{rem:orbi-fan}
    In the case of Definition \ref{defn:divisor-type}, for each divisorial type $(A(x), \mathbf{v}(x))$, using the short exact sequence \eqref{eqn:divisor-type-exact}, we can define an orbifold fan $\mathbf{\Sigma} = (N, \Sigma, \beta)$ with $N = \hom(K(x), {\mb Z})$ and $\beta$ being the dual of $K(x) \to {\mb Z}^m$. This orbifold fan has a unique maximal cone, and the rays in $\Sigma(1)$ are labeled by the corresponding indices in $\mathbf{D}$.
\end{rem}

\subsection{Divisorialification}
We establish the following analogue of \cite[Algorithm C]{bergh} for normally complex pairs.

\begin{prop}\label{prop:divisorial}
If $({\mc U}, \mathbf{D})$ is a normally complex pair then there exists a sequence of normally complex pairs
\beq
({\mc U}^{div}, \mathbf{D}^{div}) := ({\mc U}_n, \mathbf{D}_n) \xrightarrow{\pi_n} \cdots \xrightarrow{\pi_1} ({\mc U}_0, \mathbf{D}_0) = ({\mc U}, \mathbf{D})
\eeq
such that:
\begin{enumerate}
    \item each $\pi_i: {\mc U}_i \to {\mc U}_{i-1}$ is a smooth map induced from the blow-up along a suborbifold with complex normal bundle in ${\mc U}_{i-1}$;
    \item for any point $x \in |{\mc U}_n|$, the divisorial index (cf. Definition \ref{defn:divisor-index}) at $x$ defined using the pair $({\mc U}_n, \mathbf{D}_n)$ is $0$.
\end{enumerate}
If $f: ({\mc U}', \mathbf{D}') \to ({\mc U}, \mathbf{D})$ is an open embedding of a suborbifold ${\mc U}' \subset {\mc U}$ such that $\mathbf{D}' = f^* \mathbf{D}$ and the normal complex structures are compatible, then there exists a smooth map $\tilde{f}: (({\mc U}')^{div}, (\mathbf{D}')^{div}) \to ({\mc U}^{div}, \mathbf{D}^{div})$ which defines an open embedding of orbifolds $\tilde{f}: \tilde{\mc U}' \to \tilde{\mc U}$ with compatible normal complex structures such that the commutative diagram
\beq
\begin{tikzcd}
{(({\mc U}')^{div}, (\mathbf{D}')^{div})} \arrow[d] \arrow[r, "\tilde{f}"] & {({\mc U}^{div}, \mathbf{D}^{div})} \arrow[d] \\
{({\mc U}', \mathbf{D}')} \arrow[r, "f"]                                 & {({\mc U}, \mathbf{D})}                        
\end{tikzcd}
\eeq
is Cartesian.
\end{prop}

Note that we only need to blow up along suborbifolds with complex normal bundle to achieve divisorialification. The construction is based on blowing up along the locus at which the divisorial index from Definition \ref{defn:divisor-index} achieves its maximum. The following is a necessary preliminary result.

\begin{lemma}\label{lem:divisorial-center}
Let $({\mc U}, \mathbf{D})$ be a normally complex pair. Let $|{\mc U}_{div}| \subset |{\mc U}|$ be the subset which consists of points at which the divisorial index achieves its maximum. Then $|{\mc U}_{div}|$ is a closed subset of $|{\mc U}|$, and the closed substack ${\mc U}_{div} \subset {\mc U}$ which is defined to be the preimage of $|{\mc U}_{div}|$ under the coarse space map ${\mc U} \to |{\mc U}|$ is a suborbifold which is equipped with complex normal bundle.
\end{lemma}
\begin{proof}
If all the points in $|{\mc U}|$ have divisorial index $0$, we have ${\mc U}_{div} = {\mc U}$ and there is nothing to prove. So let us assume that there is at least one point in $|{\mc U}|$ at which the divisorial index function is nonzero.

To show that $|{\mc U}_{div}|$ is closed, it suffices to show that the divisorial index function is a upper semi-continuous function on $|{\mc U}|$. To this end, let $\{ x_{\nu} \} \subset |{\mc U}|$ be a sequence such that $x_{\nu} \to x$. Suppose the normal invariant of $x$ is represented by $(\Gamma, V)$ which comes with the decomposition \eqref{eqn:decompose-char}. Without loss of generality, we can assume that $x$ is covered by an orbifold chart ${\mb R}^k \times V$ such that $\Gamma$ acts on ${\mb R}^k$ trivially. Note that it suffices to figure out the argument for $k=0$ which we assume from now on.  Let the divisorial index of $x$ be $d$ and suppose that it is realized by characters $\chi_{s-d+1}, \dots, \chi_{s}$. By assumption, over the orbifold chart, we can identify the preimages of the divisors ${\mc D}_{i_1}, \dots, {\mc D}_{i_{r'}}$ with the complex codimension $1$ subspaces $V / \chi_{i_1}, \dots, V / \chi_{i_{r'}}$. Here the quotient should be understood as taking the complement. For $x_{\nu} = x'$, we can assume that $\Gamma_{x'} = \Gamma'$ is a subgroup of $\Gamma$. The interesting case is when $\Gamma'$ is a proper subgroup of $\Gamma$, otherwise the divisorial index of $x_{\nu}$ agrees with that of $x$, so let us assume $\Gamma' < \Gamma$. Denote the induced decomposition \eqref{eqn:decompose-char} of the restriction of $V$ to $\Gamma'$ by
\beq
\chi_{i_1}' \oplus \cdots \oplus \chi_{i_{r'}}' \oplus \chi_{i_{r' +1}}' \oplus \cdots \oplus \chi_s'.
\eeq
Then observe that
\beq
\chi_{div}(x') = \langle \chi_{i_1}', \dots, \chi_{i_{r'}}' \rangle.
\eeq
Consequently, the characters which are not contained in $\chi_{div}(x')$ form a subset of $\{ \chi_{s-d+1}', \dots, \chi_{s}' \}$, so we see that the divisorial index of $x'$ is bounded from above by that of $x$. Therefore, the divisorial index function is upper semi-continuous.

Next we investigate the closed substack ${\mc U}_{div}$. Using the notations introduced in the previous paragraph, we see that for $x \in |{\mc U}_{div}|$ covered by the orbifold chart $({\mb R}^k \times V) // \Gamma$, the preimage of $|{\mc U}_{div}|$ is locally given by
\beq
{\mb R}^k \times (\chi_1 \oplus \cdots \oplus \chi_{s-d}),
\eeq
which is preserved by the $\Gamma$-action. Arguing as in the proof of Lemma \ref{lem:abelianize-center}, we see that we can use these orbifold charts to endow ${\mc U}_{div} \to {\mc U}$ with a suborbifold structure because any subgroup of an abelian group is normal. The existence of complex structure on its normal bundle is manifest, as the $\Gamma$-representation $\chi_{s-d+1} \oplus \cdots \oplus \chi_s$ is complex.
\end{proof}

\begin{proof}[Proof of Proposition \ref{prop:divisorial}]
Given a normally complex pair $({\mc U}, \mathbf{D})$, if the divisorial index is constantly zero, the construction terminates.

Otherwise, we consider the effective and compact normally complex orbifold $\mathrm{Bl}_{{\mc U}_{div}} {\mc U}$ with divisors $\hat{\mathbf{D}} = \mathbf{D} \cup {\mb P}(N_{\mc U} {\mc U}_{div})$, in which $\mathbf{D}$ is the strict transform and the last component denotes the exceptional divisor that is endowed with an order which is greater than all other components from $\mathbf{D}$. Note that $(\mathrm{Bl}_{{\mc U}_{div}} {\mc U}, \hat{\mathbf{D}})$ is still a normally complex pair, which can be checked by looking at the local picture as exhibited in Lemma \ref{lem:divisorial-center}.

We wish to show that the maximum of the divisorial index for $(\mathrm{Bl}_{{\mc U}_{div}} {\mc U}, \hat{\mathbf{D}})$ is strictly smaller than that of the pair $({\mc U}, \mathbf{D})$. If this is the case, due to the compactness of ${\mc U}$, we just need finitely many steps to obtain $({\mc U}^{div}, \mathbf{D}^{div})$. Note that this construction is automatically functorial in the sense of the statement of Proposition \ref{prop:divisorial}: the order between divisorial indices is preserved by restricting to open suborbifolds and the blow-up is functorial for inclusion of open suborbifolds.

To finish the proof, we just need to check that for any $\hat{x} \in |\mathrm{Bl}_{{\mc U}_{div}}{\mc U}|$ on the exceptional divisor $|{\mb P}(N_{\mc U} {\mc U}_{div})|$, the divisorial index of $\hat{x}$ is strictly smaller than the divisorial index of its image $x$ under the blow-down map. Let's use the notations in the proof of Lemma \ref{lem:divisorial-center} to construct an orbifold chart $({\mb R}^k \times V) // \Gamma$ covering $x$. Then near $\hat{x}$, we can find a finite quotient orbifold embedded as a suborbifold of $\mathrm{Bl}_{{\mc U}_{div}}{\mc U}$ of the form
\beq
\big( {\mb R}^k \times (\chi_1 \oplus \cdots \oplus \chi_{s-d}) \times {\mc O}_{{\mb P}(\chi_{s-d+1} \oplus \cdots \oplus \chi_s)}(-1) \big) // \Gamma
\eeq
as exhibited in the proof of Theorem \ref{thm:abel}. For $\hat{x}$, its divisorial index must be bounded from above by $d-1$ because the normal direction of the exceptional divisor is certainly contained in $\chi_{div}(\hat{x})$.
\end{proof}

\subsection{Annihilating ${\mc D}$-divisorial index}
The aim here is to establish a counterpart of \cite[Algorithm D]{bergh} in the smooth category.

\begin{prop}\label{prop:kill-D-div-index}
Let $({\mc U}, \mathbf{D}, \acute{\mathbf{D}})$ be a normally complex pair with distinguished structure in the sense of Definition \ref{defn:pair-distinguish}. Then there exists a sequence of normally complex pairs with distinguished structure
\beq
({\mc U}^{dis}, \mathbf{D}^{dis}, \acute{\mathbf{D}}^{dis}) := ({\mc U}_n, \mathbf{D}_n, \acute{\mathbf{D}}_n) \xrightarrow{\pi_n} \cdots \xrightarrow{\pi_1} ({\mc U}_0, \mathbf{D}_0, \acute{\mathbf{D}}_0) = ({\mc U}, \mathbf{D}, \acute{\mathbf{D}})
\eeq
such that:
\begin{enumerate}
    \item each $\pi_i: {\mc U}_i \to {\mc U}_{i-1}$ is a smooth map induced from the blow-up along a suborbifold with complex normal bundle in ${\mc U}_{i-1}$;
    \item in the construction of (1), the locus along which each blow-up takes place is a closed substack of $\acute{\mathbf{D}}_{i-1}$;
    \item for any point $x \in |{\mc U}_n|$ and ${\mc D} \in \acute{\mathbf{D}}_n$, the ${\mc D}$-divisorial index (cf. Definition \ref{defn:D-div-index}) at $x$ defined using the pair $({\mc U}_n, \mathbf{D}_n)$ is $0$.
\end{enumerate}
If $f: ({\mc U}', \mathbf{D}', \acute{\mathbf{D}}') \to ({\mc U}, \mathbf{D}, \acute{\mathbf{D}})$ defines an open embedding of orbifolds ${\mc U}' \subset {\mc U}$ such that $\mathbf{D}' = f^* \mathbf{D}$ and $\acute{\mathbf{D}}' = f^* \acute{\mathbf{D}}$, then there exists a smooth map $\tilde{f}: (({\mc U}')^{dis}, (\mathbf{D}')^{dis}, (\acute{\mathbf{D}}')^{dis}) \to ({\mc U}^{dis}, \mathbf{D}^{dis}, \acute{\mathbf{D}}^{dis})$ inducing an open embedding of orbifolds $\tilde{f}: \tilde{\mc U}' \to \tilde{\mc U}$ such that the commutative diagram
\beq
\begin{tikzcd}
{(({\mc U}')^{dis}, (\mathbf{D}')^{dis}, (\acute{\mathbf{D}}')^{dis}) } \arrow[d] \arrow[r, "\tilde{f}"] & {({\mc U}^{dis}, \mathbf{D}^{dis}, \acute{\mathbf{D}}^{dis})} \arrow[d] \\
{({\mc U}', \mathbf{D}', \acute{\mathbf{D}}')} \arrow[r, "f"]                                 & {({\mc U}, \mathbf{D}, \acute{\mathbf{D}})}                        
\end{tikzcd}
\eeq
is Cartesian.
\end{prop}

As the reader could guess, before providing a proof of Proposition \ref{prop:kill-D-div-index}, we need to characterize the loci along which the ${\mc D}$-divisorial index achieves its maximum.

\begin{lemma}\label{lem:distinguish-center}
Let $({\mc U}, \mathbf{D})$ be a normally complex pair, and suppose ${\mc D} \in \mathbf{D}$ is a divisor. Let $|{\mc U}_{dis}| \subset |{\mc U}|$ be the subset which consists of points at which the ${\mc D}$-divisorial index achieves its maximum. Then $|{\mc U}_{dis}|$ is a closed subset of $|{\mc U}|$, and the closed substack ${\mc U}_{dis} \subset {\mc U}$ which is defined to be the preimage of $|{\mc U}_{dis}|$ under the coarse space map ${\mc U} \to |{\mc U}|$ is a suborbifold with complex normal bundle.
\end{lemma}
\begin{proof}
If all the points in $|{\mc U}|$ have ${\mc D}$-divisorial index $0$, by definition we have ${\mc U}_{dis} = {\mc U}$. So let's assume that there exist one point in $|{\mc U}|$ whose ${\mc D}$-divisorial index is nonzero.

To prove that $|{\mc U}_{dis}|$ is closed, it suffices to show that the ${\mc D}$-divisorial function defined over $|{\mc U}|$ is upper semi-continuous. This property is local in nature. Without loss of generality, if the normal invariant of $x \in |{\mc U}|$ is realized by $(\Gamma, V)$, we can replace the orbifold ${\mc U}$ by the global quotient orbifold $({\mb R}^k \times V) // \Gamma$ where $\Gamma$ acts trivially on ${\mb R}^k$, and $x$ corresponds to the image of $0 \in {\mb R}^k \times V$ in the coarse space. Assuming $k=0$ does not affect the argument, so we choose to do so. Using the decomposition \eqref{eqn:decompose-char}, based on the assumption that $({\mc U}, \mathbf{D})$ is everywhere divisorial, we know that
\beq
\chi_{div}(x) = \chi(x) = \langle \chi_1, \dots, \chi_s \rangle.
\eeq
Using the same notation, we assume that ${\mc D}$ corresponds to the complex codimension $1$ subspace $\{ \chi_{i_j} = 0 \}$. It suffices to show that for any other point in $x' \in |V // \Gamma|$, its ${\mc D}$-divisorial index is bounded from above by the quantity \eqref{eqn:d-div-index}. We proceed by assuming that $x'$ is contained in the image of $|{\mc D}|$, otherwise the proof is already finished. Suppose that the stabilizer of $x'$ corresponds to the subgroup $\Gamma' \leq \Gamma$. Because the divisorial index of $x'$ is $0$, we know that $\chi(x') = \chi_{div}(x')$. By definition, $\chi(x') \subset \chi(\Gamma')$ is generated by the $\Gamma'$-characters from restricting the $\Gamma$-characters $\chi_{1}, \dots, \chi_{s}$ to $\Gamma'$, which we denote by $\chi_1', \dots, \chi_s'$. The character associated to the normal direction of ${\mc D}$ at $x'$ is then denoted by $\chi_{i_j}'$. The restriction of any divisorial character $\chi_{i_1}, \dots, \chi_{i_{r'}}$ of $x$ to the group $\Gamma'$ is trivial if the corresponding divisor does not pass through $x'$. As a result, we just need to focus on the contribution of $\chi_{i_{r' +1}}', \dots, \chi_s'$ to the formula \eqref{eqn:d-div-index}. The group homomorphism \eqref{eqn:char-proj-D} associated to $x'$ can be written as
\beq
\chi(x') = \chi_{div}(x') \to \chi_{div}(x') / \chi_{div}(x')^{\mc D} = \langle \chi_1', \dots, \chi_s' \rangle / \langle \chi_1', \dots, \hat{\chi_{i_j}'}, \dots, \chi_s' \rangle.
\eeq
It fits into a commutative diagram 
\beq
\begin{tikzcd}
\chi(x) \arrow[d] \arrow[rr] &  & {\langle \chi_1, \dots, \chi_s \rangle / \langle \chi_1, \dots, \hat{\chi_{i_j}}, \dots, \chi_s \rangle} \arrow[d] \\
\chi(x') \arrow[rr]          &  & {\langle \chi_1', \dots, \chi_s' \rangle / \langle \chi_1', \dots, \hat{\chi_{i_j}'}, \dots, \chi_s' \rangle}     
\end{tikzcd}
\eeq
in which the vertical arrows are induced by restricting a $\Gamma$-character to a $\Gamma'$-character so that they are surjective. For a character $\chi_j' \in \{ \chi_{i_{r' +1}}', \dots, \chi_s' \}$ coming from the restriction of the character $\chi_j$, because the right vertical arrow takes a generator to a generator and takes the image of $\chi_j$ to $\chi_j'$, we know that the contribution of $\chi_j'$ to the ${\mc D}$-divisorial index in exactly $c_j$ as in \eqref{eqn:d-div-index} if $\chi_j'$ is nontrivial. Therefore, the ${\mc D}$-divisorial index of $x'$ is bounded from above by \eqref{eqn:d-div-index}.

To construct the orbifold structure, consider the \'etale chart ${\mb R}^k \times V$, where the preimage of $|{\mc U}_{dis}|$ corresponds to ${\mb R}^k \times V_{dis}$, in which $V_{dis}$ is the direct sum of $\chi_{i_1}, \dots \chi_{i_{r'}}$ and $\chi_j \in \{ \chi_{i_{r' +1}}, \dots, \chi_s \}$ such that the image of $\chi_j$ under the map \eqref{eqn:char-proj-D} is $0$. Accordingly, the finite global quotient orbifold $({\mb R}^k \times V_{dis}) // \Gamma$ can be used to construct a suborbifold structure on ${\mc U}_{dis}$. The fiber of the normal bundle $N_{\mc U} {\mc U}_{dis}$ can be identified with $V / V_{dis}$, which has a natural complex structure.
\end{proof}

\begin{proof}[Proof of Proposition \ref{prop:kill-D-div-index}]
Given a normally complex pair with distinguished structure $({\mc U}, \mathbf{D}, \acute{\mathbf{D}})$, if the ${\mc D}$-divisorial index is constantly zero for any ${\mc D} \in \acute{\mathbf{D}}$, the construction terminates.

Otherwise, we choose ${\mc D} \in \acute{\mathbf{D}}$ with minimal order such that there exists at least a point in $|{\mc U}|$ at which the ${\mc D}$-divisorial index is nonzero. Using ${\mc U}_{dis}$ defined from Lemma \ref{lem:distinguish-center}, we consider the effective and compact normally complex orbifold $\mathrm{Bl}_{{\mc U}_{dis}} {\mc U}$. It is endowed with divisors $\hat{\mathbf{D}} = \mathbf{D} \cup {\mb P}(N_{\mc U} {\mc U}_{div})$, in which $\mathbf{D}$ is the strict transform and the last component denotes the exceptional divisor that is endowed with an order which is greater than all other components from $\mathbf{D}$. The distinguished structure $\acute{\mathbf{D}}$ can also be extended to $\acute{\hat{\mathbf{D}}}$ by making the exceptional divisor ${\mc F} = {\mb P}(N_{\mc U} {\mc U}_{div})$ distinguished. By construction, the ordering constraint in Definition \ref{defn:pair-distinguish} is satisfied.

We claim the following holds for the normally complex pair with distinguished structure $(\mathrm{Bl}_{{\mc U}_{dis}} {\mc U}, \hat{\mathbf{D}}, \acute{\hat{\mathbf{D}}})$:
\begin{enumerate}
    \item given any $\hat{x} \in |{\mc F}|$ mapped to $x \in |{\mc U}|$ under the blow-down map, its $\tilde{\mc D}$-divisorial index for any $\tilde{\mc D}$ coming from the strict transform of a divisor $\tilde{\mc D}$ (here we abuse the notation by using the same symbol for the divisor in the two spaces) in $\mathbf{D}$ is at most the $\tilde{\mc D}$-divisorial index of $x$;
    \item if ${\mc D} = \tilde{\mc D}$ in (1), the ${\mc D}$-divisorial index of $\hat{x}$ is strictly smaller than that of $x$;
    \item for any $\hat{x} \in |{\mc F}|$, its ${\mc F}$-divisorial index is strictly smaller than the ${\mc D}$-divisorial index of its image under the blow-down map $x$.
\end{enumerate}
If the claim holds, by ranging over all the divisors in $\acute{\mathbf{D}}$ according to the given ordering, we can strictly decrease the maximum of the ${\mc D}$-divisorial index. By the compactness of ${\mc U}$, the construction terminates after applying the above blow-up construction for finitely many times. The desired functoriality statement holds as the ordering of the ${\mc D}$-divisorial indices are preserved by restricting to suborbifolds.

It remains to prove the claim. As in the proof of Lemma \ref{lem:distinguish-center}, we can assume that the orbifold ${\mc U}$ is isomorphic to the finite global quotient orbifold $V // \Gamma$, and $V$ is decomposed as \eqref{eqn:decompose-char}. Using the notations introduced in the proof of Lemma \ref{lem:distinguish-center}, ${\mc D}$ corresponds to the subspace $\{ \chi_{i_j} = 0 \}$, and we can decompose $V$ as a direct sum $V_{dis} \oplus V_{dis}^{\perp}$ such that $V_{dis} \times \{0\}$ defines an \'etale chart of ${\mc U}_{dis}$. After applying the blow-up, we can locally identify $\mathrm{Bl}_{{\mc U}_{dis}} {\mc U}$ with the global quotient $({\mc O}_{{\mb P}(V_{dis})}(-1) \times V_{dis}^{\perp}) // \Gamma = ((V_{dis} \setminus \{0\}) \times V_{dis}^{\perp}) // \Gamma$. We can assume that $\hat{x}$ is contained in $|({\mb P}(V_{dis}) \times V_{dis}^{\perp}) // \Gamma|$ We prove the claims one after another.
\begin{enumerate}
    \item Suppose the coordinate along $\chi_{j}$ being $0$ specifies the divisor $\tilde{\mc D}$. Because we have presented the blow-up locally as $((V_{dis} \setminus \{0\}) \times V_{dis}^{\perp})$, the same argument in the proof of Lemma \ref{lem:distinguish-center} shows that the $\tilde{\mc D}$-divisorial index of $\hat{x}$ is bounded from above by $\tilde{\mc D}$-divisorial index of the image of $0 \in V$ in the coarse space.
    \item Keep using the notations in (1). In the special case of ${\mc D} = \tilde{\mc D}$, by assumption we know that $\chi_j \subset V_{dis}^{\perp}$, which mean that the image of $\chi_j \in \chi(x)$ under the map \eqref{eqn:char-proj-D} is nonzero. For the point $\hat{x}$ in the blow-up, the image of the character $\chi_j \in \chi(\hat{x})$ under the map \eqref{eqn:char-proj-D} is necessarily zero. Therefore, the quantity \eqref{eqn:d-div-index} strictly decreases.
    \item Now suppose that $\hat{x} \in |{\mc F}| = |({\mb P}(V_{dis}) \times V_{dis}^{\perp}) // \Gamma|$. If $\chi_{j'} \subset V_{dis}^{\perp}$, namely $\chi_{j'}  $ is a character which is mapped to a nonzero element under \eqref{eqn:char-proj-D}, then it contributes $c_{j'}$ to \eqref{eqn:d-div-index}. In other words, the image of $\chi_{j'}$ is $c_{j'}$ times the generator of $\chi_{\mc D}(x)$. For the group
    \beq
    \chi_{\mc F}(\hat{x}) = \chi_{div}(\hat{x}) / \chi_{div}(\hat{x})^{\mc F}
    \eeq
    as defined in \eqref{eqn:chi-D-x}, if we view $\hat{x} \in |((V_{dis} \setminus \{0\}) \times V_{dis}^{\perp}) // \Gamma|$, we see that the image of $\chi_{j'}$, which is now thought of as an element in $\chi(\Gamma_{\hat{x}})$ after restriction, is actually $c_{j'}-1$ times the generator of $\chi_{\mc F}(\hat{x})$. So claim (3) is also proven after ranging over all $\chi_{j'} \subset V_{dis}^{\perp}$.
\end{enumerate}
\end{proof}

\subsection{Proof of Theorem \ref{thm:deorbi}}
\label{sec:proof-theor-refthm:d}

Starting from a normally complex pair $({\mc U}, \mathbf{D})$, we first apply Proposition \ref{prop:divisorial} to obtain a divisorial normally complex pair $({\mc U}^{div}, \mathbf{D}^{div})$. Then the deorbification is realized by iteratively applying the following four steps.
\begin{enumerate}
    \item Define an invariant for points in $|{\mc U}^{div}|$ to be the lexicographical composition of the independence index (Definition \ref{defn:independence-ind}), the toroidal index (Definition \ref{defn:toroidal-ind}), and the divisorial type (Definition \ref{defn:divisor-type}). Denote the locus along which such an invariant achieves its maximum by $|{\mc U}^{div}_{lex}|$.
    \item The divisorial type is constant along $|{\mc U}^{div}_{lex}|$ by definition. For any $x \in |{\mc U}^{div}_{lex}|$, its divisorial type $(A(x), \mathbf{v}(x))$ defines an orbifold fan $\mathbf{\Sigma}$ by Remark \ref{rem:orbi-fan}. Let $\mathbf{\Sigma}_0 = \mathbf{\Sigma}$ and apply the construction in Proposition \ref{prop:algo-A} to obtain a sequence of orbifold fans
    \beq\label{eqn:sigma-0-n}
    \mathbf{\Sigma}_n \to \cdots \to \mathbf{\Sigma}_0.
    \eeq
    \item Because each ray in $\Sigma(1)$ of $\mathbf{\Sigma}$ is marked by certain index from $\mathbf{D}^{div}$, every star subdivision lifts to the corresponding (possibly) ineffective orbifold chart $V//\Gamma$ and can be globalized to the blow-up of ${\mc U}^{div}$ along the corresponding intersection of divisors from $\mathbf{D}^{div}$. Similarly, the root stack construction along toric divisors can be globalized to a root stack construction along the corresponding divisor in the ambient orbifold. As a result, the sequence \eqref{eqn:sigma-0-n} can be lifted to a sequence of blow-ups and root constructions of $({\mc U}^{div}, \mathbf{D}^{div})$, whose output is a normally complex pair with distinguished structure $(({\mc U}^{div})', (\mathbf{D}^{div})', (\acute{\mathbf{D}}^{div})')$.
    \item Applying Proposition \ref{prop:kill-D-div-index} to $(({\mc U}^{div})', (\mathbf{D}^{div})', (\acute{\mathbf{D}}^{div})')$, we obtain a normally complex pair with distinguished structure $({\mc U}'', \mathbf{D}'', \acute{\mathbf{D}}'')$ whose ${\mc D}$-divisorial index is $0$. Then we forget the distinguished structure to get a normally complex pair $({\mc U}'', \mathbf{D}'')$.
\end{enumerate}

Then we define $(\tilde{\mc U}, \tilde{\mathbf{D}})$ to be the final output of the above algorithm. The proof that the above lexicographical composition of invariants decreases strictly for each time of the iteration is exactly the same as the argument in \cite[Algorithm E]{bergh}, because everything can be checked on the toric local models described by orbifold fans. By compactness, we know that the above algorithm must terminate after finitely many steps.

By construction, we see that the independence index of each point in $|\tilde{\mc U}|$ is $0$, which implies that $|\tilde{\mc U}|$ is a smooth manifold. Because all the modifications are functorial under open embeddings, so is the final output. \qed

\section{From orbifolds to manifolds}
\label{sec:from-orbif-manif}

In this section, we assemble the proofs of all the results stated in the introduction, involving geometric or orbifold bordism.

\subsection{Complex orbifold bordism}\label{subsec:Omega-C}

If ${\mc U}$ is a orbifold with boundary and $(X,A)$ is a pair of topological spaces, we use $f: {\mc U} \to (X,A)$ to denote a continuous map $f: ({\mc U}, \partial {\mc U}) \to (X,A)$ and express such data as $({\mc U}, f)$ \emph{over} $(X,A)$. The following definition is a convenient way to bypass the need to formulate the notion of (derived) orbifold with corners:
\begin{defn}
  A pair $({\mc U}, f)$ and $({\mc U}', f')$ of orbifolds over $(X,A)$, are  \emph{bordant}  if there exist an orbifold with boundary $\hat{\mc U}$ and a continuous map $F: \hat{\mc U} \to X$ such that:
    \begin{enumerate}
        \item there exists a codimension $0$ embedding ${\mc U} \coprod {\mc U}' \to \partial \hat{\mc U}$ along which $F$ restricts to $f \coprod f'$;
        \item the map $F$ restricts a continuous map of pairs $(\hat{\mc U}, \partial \hat{\mc U}  \setminus {\mc U}^{\circ} \coprod ({\mc U}')^{\circ}) \to (X,A)$.
    \end{enumerate}
\end{defn}
The \emph{orbifold bordism group} of $(X,A)$ is then defined to be the free abelian group generated by isomorphism classes of $({\mc U}, f)$, with ${\mc U}$ compact, over $(X,A)$ modulo the equivalence relations generated by bordism. It is straightforward to see that this is a homology theory, with boundary homomorphism given by restricting $f$ to the boundary of $\mc U$.

One can impose tangential structures on ${\mc U}$ to define \emph{oriented} or \emph{framed} orbifold bordism groups. We focus our attention on the following case:
\begin{defn}
  A \emph{stable complex structure} on an orbifold ${\mc U}$ (with boundary) consists of a complex a complex structure on $T{\mc U} \oplus \underline{\mb R}^k$, for some integer $k$.
  \end{defn}
  The boundary of any stably complex orbifold admits a natural stable complex structure, given by identifying the choice of normal vector with an additional ${\mb R}$ stabilization factor. Then we say that any stably complex pair $({\mc U}, f)$ and $({\mc U}', f')$ over $(X,A)$ are stably complex bordant if there exists a bordism $(\hat{\mc U}, F)$ as above further so that
\begin{itemize}
    \item the stable complex structure on $\hat{\mc U}$ restricts to the stable complex structures of ${\mc U}$ and ${\mc U}'$ for some choice of normal vector field.
\end{itemize}

\begin{defn}
The \emph{complex orbifold bordism} group of a pair of topological spaces $(X,A)$
\beq
\Omega_*^{U}(X,A)
\eeq
is the abelian group (with additive structure coming from disjoint union) generated by compact orbifolds with boundary over $(X,A)$, equipped with a stable complex structure, modulo the  equivalence relation generated by stably complex bordisms. It is \emph{graded} by the (virtual) dimension.
\end{defn}

\subsection{Constructing manifolds from orbifolds}\label{subsec:splitting-mani}

The goal of this section is to prove Theorem  \ref{thm:split_manifold_to_orbifold}. Although this result is stated for stable complex orbifolds, our construction is in fact designed to derive a manifold from a normally complex orbifold. Since every step in the algorithm preserves the additional datum of the stable complex structure, working in the normally complex setting implies the desired result. In view of the applicability to concrete geometric problems from symplectic topology, we thus prefer to state our construction just using the normal complex structures. 

\begin{mdframed}
\textbf{ALGORITHM I: From orbifolds to manifolds}
\end{mdframed}
\vspace{0.2cm}

\noindent \underline{\it Input:} A compact normally complex orbifold ${\mc Z}$. 
\vspace{0.2cm}

\noindent \underline{\it Step 1}
Invoke Theorem \ref{thm:abel} to construct a smooth map from a compact normally complex orbifold 
\beq
\pi^{ab}: {\mc Z}^{ab} \to {\mc Z}
\eeq
such that ${\mc Z}^{ab}$ only has abelian stabilizers.

\vspace{0.2cm}

\noindent \underline{\it Step 2}
Conduct the construction from Theorem \ref{thm:deorbi} to the pair $({\mc Z}^{ab}, \emptyset)$ and forget the data of the divisors $\tilde{\mathbf{D}}$ to obtain an effective and compact normally complex compact orbifold $\tilde{\mc Z}^{ab}$ with a smooth map
\beq
\pi^{M}: \tilde{\mc Z}^{ab} \to {\mc Z}^{ab}
\eeq
such that $|\tilde{\mc Z}^{ab}|$ is a smooth manifold.

The upshot is that we obtain a compact orbifold $\tilde{\mc Z}^{ab}$ with a map $\tilde{\mc Z}^{ab} \to {\mc Z}$ such that $|\tilde{\mc Z}^{ab}|$ is a smooth manifold by applying the above two steps. \qed

We call the combination of the two steps above the \emph{destackification algorithm}.

\begin{rem}\label{rem:destack-algorithm}
\begin{enumerate}
    \item The input of the above algorithm has ambient orbifold possibly with boundary. As remarked before, although Theorem \ref{thm:abel} and Theorem \ref{thm:deorbi} are stated for (derived) orbifolds without boundary, the same constructions go through for (derived) orbifolds with boundary by observing that for local charts for the form $({\mb R}_{\geq 0} \times {\mb R}^{n-1}) // \Gamma$, the group $\Gamma$ acts trivially on the ${\mb R}_{\geq 0}$ direction, and all the arguments only concern the directions \emph{normal} to the fixed point loci.
    \item The only auxiliary choices made in the destackification algorithm is the choice of flattened Riemannian metric in the definition of blowups. As discussed in Section \ref{lem:flattening}, this is a contractible choice, and the result of the blowup procedure is thus well-defined up to bordism.
\end{enumerate}
\end{rem}

We now prove our main result concerning the relationship between manifold and orbifold bordism:

\begin{proof}[Proof of Theorem \ref{thm:split_manifold_to_orbifold}]
 To show that Algorithm I is compatible with bordism, let $\mc Z$ be a normally complex orbifold, which bounds a compact orbifold  $\hat{\mc Z} $. Because of the compatibility of normal complex structures and the functoriality of open embedding as stated in Theorem \ref{thm:abel} and Theorem \ref{thm:deorbi}, we can use the collar embedding of the boundary to make sure that 
        \beq
        \widetilde{\hat{\mc Z}}^{ab}
        \eeq
        defines a null bordism of $|\tilde{\mc Z}^{ab}|$. Assuming that $\mc Z$ is a normally complex orbifold with boundary, lying over a pair $(X,A)$, and that $ \hat{\mc Z}$ bounds over this pair, we obtain the desired natural transformation by composing the projection $\widetilde{\hat{\mc Z}}^{ab} \to  \hat{\mc Z}$ with the map to the target space.

        To show compatibility with the module action, it suffices to prove that, if $M$ s a manifold, then the resolution of singularities of $M \times \mc Z $ is the product of $M$ with the resolution of singularities of $\mc Z$. This follows by noting that every step in the Algorithm I, which involves taking blowups and root stacks, will have such a product decomposition.
        
\end{proof}

\subsection{Geometric bordisms}
\label{sec:equivariant-bordisms-no-d}
We begin with results concerning $\Gamma$-equivariant geometric bordism, for a finite group $\Gamma$, which assigns to a pair $(X,A)$ of topological spaces the graded group $\Omega^{U,\Gamma}_*(X,A)$ of bordism classes of maps of pairs $(M,\partial M) \to (X,A)$, where $M$ is a compact stably complex $\Gamma$-manifold.  
\begin{proof}[Proof of Theorem \ref{thm:inclusion_ordinary_geometric_splits}]
  Our goal is to construct a map
  \begin{equation} \label{eq:splitting_map}
        \Omega^{U,\Gamma}_*(X,A) \to \Omega^{U}_*(X,A),
      \end{equation}
      which assigns to a free $\Gamma$-manifold its quotient by $\Gamma$. Appealing to Theorem \ref{thm:split_manifold_to_orbifold}, we may define such a map as the composition of the inclusion of equivariant bordism in orbifold bordism which takes a $\Gamma$-manifold to the quotient orbifold, followed by the splitting we construct.
    \end{proof}

    Tracing through the construction shows that the map we produced  assigns to a general $\Gamma$-manifold the result of applying the resolution of singularities algorithm to the associated orbifold quotient, consisting of the two steps of abelianization of the fixed point locus and desingularization of the associated quotient space. The former replaces the manifold $M$, by blow-ups along fixed point strata, with a $\Gamma$-manifold with the property that the representations associated to all fixed points are diagonalizable, following the proof of Theorem \ref{thm:abel}, appearing at the end of Section \ref{thm:abel}, and the latter procedure is the desingularization procedure from Theorem \ref{thm:deorbi}. As both procedures are algorithmic, the construction is functorial, and compatible both with cobordisms and with boundary homomorphisms. 

    \begin{rem} \label{rem:other_choices_of_splitting}
      There is some flexibility in the properties satisfied by the spliting map. For example, we may alternately set it to vanish on all bordism classes represented by manifolds for which the isotropy at a generic point is non-trivial, and use resolution of singularities only on the component coming from bordism classes equipped with faithful actions. Alternatively, we may factor Equation \eqref{eq:splitting_map} through a map to a direct sum of copies of $\Omega^{U}_*(X,A) $, indexed by conjugacy classes of normal sugroups in $\Gamma$, and which record the generic isotropy.
    \end{rem}

    The next proof requires a small modification of the results of Section \ref{sec:abelianize} and \ref{sec:bergh}, by imposing an additional condition of equivariance:
    \begin{proof}[Proof of Theorem \ref{thm:splitting_compatible_surjections}]
      Given a surjection $\Gamma \to \Gamma'$ with kernel $K$, and pair $(Y,B)$ of $\Gamma'$ spaces, we construct a map
        \begin{equation}
        \Omega^{U,\Gamma}_*(Y,B) \to \Omega^{U,\Gamma'}_*(Y,B).
      \end{equation}
 Given a  $\Gamma$-manifold $M$, we consider the quotient orbifold $M // K$, which lies over $(Y,B)$ and carries a residual $\Gamma'$ action.  We then apply the constructions of of Section \ref{sec:abelianize} and \ref{sec:bergh}, $\Gamma'$-equivariantly, to obtain a manifold over the pair $(Y,B)$: equivariance can be achieved because the only choice is that of a flattening of the almost complex structure near the fixed point strata, which can be done $\Gamma$-equivariantly on $M$, hence descends to a $\Gamma'$-equivariant choice on $M // K$.
    \end{proof}

\section{Functorial embedded resolution of singularities and universal zero loci}\label{sec:ATW}

\subsection{Abramovich--Temkin--Wlodarczyk's functorial resolution} 
We recall the functorial embedded resolution of singularities in characteristic $0$ developed in \cite{abramovich2019functorial}. This approach is well-adapted to our local-to-global construction in Section \ref{sec:derive-manifolds}, though earlier results on functorial resolution of singularities, e.g., \cite{Bierstone-Milman, Wlo05, kollar07}, may also suffice for our purpose. Throughout our discussion, we work over ${\mb C}$, although the results from \cite{abramovich2019functorial} are stated for any field of characteristic $0$.


\begin{defn}
The \emph{category of pairs with smooth morphisms} over ${\mb C}$ is given by the following data:
\begin{itemize}
    \item the objects are \emph{pairs} $(X \subset Y)$ for which $Y$ is a smooth variety over ${\mb C}$, and $X$ is a reduced closed subscheme of $Y$;
    \item given two pairs $(X_1 \subset Y_1)$ and $(X_2 \subset Y_2)$, a morphism $(X_1 \subset Y_1) \to (X_2 \subset Y_2)$ is given by a smooth morphism $Y_1 \to Y_2$ so that $X_1 = X_2 \times_{Y_2} Y_1$ is the pullback of $X_2$.
\end{itemize}
\end{defn}

\begin{thm}[{\cite[Theorem 1.1.1]{abramovich2019functorial}}]\label{thm:ATW-resolution}
There is an endofunctor $F^{\circ \infty}_{er}$ of the category of pairs with smooth morphisms over ${\mb C}$ whose image consists of pairs $(X \subset Y)$ satisfying the property that $X$ is  a smooth closed subscheme.
\end{thm}

More precisely, the above statement is a corollary of \cite[Theorem 1.1.1]{abramovich2019functorial}, so it deserves some further comments.
\begin{enumerate}
    \item The actual output of $F^{\circ \infty}_{er}$ as stated in \cite[Theorem 1.1.1]{abramovich2019functorial} is a pair of smooth \emph{Deligne--Mumford stacks}. One needs to appeal to Bergh's destackification theorem \cite{bergh2019functorial}, \cite[Theorem 8.1.2]{abramovich2019functorial} to obtain the statement involving only schemes, see the discussion in \cite[Section 1.5]{abramovich2019functorial}. In other words, our notation $F^{\circ \infty}_{er}$ is in fact a composition of $F^{\circ \infty}_{er}$ as in \cite[Theorem 1.1.1]{abramovich2019functorial} and Bergh's destackification functor.
    \item Before applying the destackification functor, the functor $F^{\circ \infty}_{er}$ is obtained as the stabilization of an endofunctor $F_{er}$ defined on the category of pairs with smooth \emph{surjective} morphisms. The functor $F^{\circ \infty}_{er}$ is constructed via \emph{weighted blow-ups}, and strictly reduces an invariant $\mathrm{maxinv}(X)$ associated to a pair $X \subset Y$, which measures the ``worst" singularities of $X$ via intersecting $X$ with hypersurfaces in $Y$.
\end{enumerate}

The following statement is a straightforward consequence of the functoriality property.

\begin{cor}\label{cor:cube-diagram}
Let $(X_1 \subset Y_1) \to (X_2 \subset Y_2)$ be a smooth morphisms between pairs. For $i=1,2$, define $(\tilde{X}_i, \tilde{Y}_i) := F^{\circ \infty}_{er}(X_i \subset Y_i)$. Then the diagram 
\begin{equation}
    \begin{tikzcd}
\tilde{X_1} \arrow[d] \arrow[rrd] \arrow[r] & \tilde{Y_1} \arrow[d] \arrow[rrd] &                                 &                       \\
X_1 \arrow[r] \arrow[rrd]                   & Y_1 \arrow[rrd]                   & \tilde{X}_2 \arrow[d] \arrow[r] & \tilde{Y}_2 \arrow[d] \\
                                            &                                   & X_2 \arrow[r]                   & Y_2                  
\end{tikzcd}
\end{equation}
is commumative. Moreover, the top and bottom faces are smooth morphisms of pairs, and the vertical arrows are compositions of destackifications and weighted blow-ups. \qed
\end{cor}

Functorial resolutions of singularities interact perfectly with group actions. The next statement is well-known, see, e.g., \cite[Proposition 3.9.1]{kollar07}. We include a proof for completeness.

\begin{cor}\label{cor:equivariance}
Given a pair $(X \subset Y)$, write $(\tilde{X} \subset \tilde{Y}) \subset F^{\circ \infty}_{er}(X \subset Y)$. If  an algebraic group $G$ acts on $Y$, preserving $X$, then there is a unique $G$ action on $\tilde{Y}$  so that the projection $\tilde{Y} \to Y$ is equivariant. Such an action necessarily preserves $\tilde{X}$. 
\end{cor}
\begin{proof}
The action of $G$ on the pair $(X \subset Y)$ is given by a smooth morphism
\beq
m: G \times (X \subset Y) = (G \times X \subset G \times Y) \to (X \subset Y).
\eeq
By functoriality, the pair $(\widetilde{G \times X}, \widetilde{G \times Y}) := F^{\circ \infty}_{er}(G \times X \subset G \times Y)$ is given by the pull-back of $(\tilde{X} \subset \tilde{Y}) := F^{\circ \infty}_{er}(X \subset Y)$ via the action $m$.

On the other hand, the projection $\pi_2: G \times (X \subset Y) \to (X \subset Y)$ is smooth. Using the functoriality again, we see that $(\widetilde{G \times X}, \widetilde{G \times Y}) = (G \times \tilde{X}, G \times \tilde{Y})$. Accordingly, we get a commutative diagram
\beq
\begin{tikzcd}
(G \times \tilde{X} \subset G \times \tilde{Y}) \arrow[r, "\cong"] \arrow[d] & (\widetilde{G \times X} \subset \widetilde{G \times Y}) \arrow[r, "m'"] & (\tilde{X} \subset \tilde{Y}) \arrow[d] \\
(G \times X \subset G \times Y) \arrow[r, "="]                               & (G \times X \subset G \times Y) \arrow[r, "m"]                    & (X \subset Y)                          
\end{tikzcd}
\eeq
Here the morphism of pairs $m': (\widetilde{G \times X}, \widetilde{G \times Y}) \to (\tilde{X} \subset \tilde{Y})$ is induced by the functor $F^{\circ \infty}_{er}$. We define the morphism
\beq
\tilde{m}: G \times (\tilde{X} \subset \tilde{Y}) \to (\tilde{X} \subset \tilde{Y})
\eeq
to be the composition of the top arrow.

We claim that $\tilde{m}$ defines a $G$-action on the pair $(\tilde{X} \subset \tilde{Y})$. In other words, we need to verify the commutativity of the diagram
\beq
\begin{tikzcd}
G \times G \times (\tilde{X} \subset \tilde{Y}) \arrow[d, "m_{G} \times id"'] \arrow[r, "id_G \times \tilde{m}"] & G \times (\tilde{X} \subset \tilde{Y}) \arrow[d] \\
G \times (\tilde{X} \subset \tilde{Y}) \arrow[r, "\tilde{m}"]                                                    & (\tilde{X} \subset \tilde{Y})        .           
\end{tikzcd}
\eeq
The point is, as $G$ acts on the pair $(X \subset Y)$, and the resolution $(\tilde{X} \subset \tilde{Y}) \to (X \subset Y)$ is birational, we know that commutativity holds over a dense open set. The desired commutativity is a result of the separatedness and reducedness of all the schemes in our discussion.
\end{proof}

We will use actions of both finite groups and general reductive groups in our later discussions.

For the proof of an important technical statement (Lemma \ref{lemma:group-inverse}), we need a counterpart of Theorem \ref{thm:ATW-resolution} in the complex analytic setting.

\begin{defn}
The \emph{category of complex analytic pairs with smooth morphisms} is given by the following data:
\begin{itemize}
    \item the objects are \emph{pairs} $(X \subset Y)$ for which $Y$ is a smooth complex manifold, and $X$ is a reduced closed complex analytic space of $Y$;
    \item a morphism $(X_1 \subset Y_1) \to (X_2 \subset Y_2)$ is given by a holomorphic map $Y_1 \to Y_2$ which is a submersion (i.e., locally modeled on the projection map $Y_2 \times {\mb C}^r \to Y_2$) and so that $X_1 = X_2 \times_{Y_2} Y_1$ is the pullback of $X_2$.
\end{itemize}
\end{defn}

\begin{thm}[\cite{wlodarczyk}]\label{thm:ATW-resolution-analytic}
  There is an endofunctor $F^{\circ \infty}_{an}$ of category of complex analytic pairs with smooth morphisms whose image consists of pairs $(X \subset Y)$ satisfying the property that $X$ is  a complex submanifold. This endofunctor is compatible with the corresponding algebraic functor in the sense that, for an algebraic pair $(X \subset Y)$, the analytification of $F^{\circ \infty}_{er}(X \subset Y)$ is naturally isomorphic to the result of applying  $F^{\circ \infty}_{an}$ to the analytic pair underlying $(X \subset Y)$. 
\end{thm}
\begin{proof}[Sketch of proof]
We list the important technical steps in \cite{abramovich2019functorial} and comment on how to adapt the arguments to the complex analytic setting.
\begin{enumerate}
    \item As in the algebraic setting, the final resolution functor $F^{\circ \infty}_{an}$ is obtained as the composition of Bergh's destackification and the stabilization of the counterpart of the endofunctor $F_{er}$ in the complex analytic setting. Bergh's destackification procedure has been reproduced in Section \ref{sec:abelianize} and Section \ref{sec:bergh} in the almost complex setting, which can be modified to work in the complex analytic setting by combining with the original reference \cite{bergh,bergh2019functorial}.
    \item The pair  $F^{\circ \infty}_{an}(X \subset Y)$ is obtained, as in the algebraic case, by applying the weighted blow-up to $Y$ along the center specified by the locus where the function $\mathrm{inv}_p({\mc I}_X)$ (cf. \cite[Section 5.1]{abramovich2019functorial}) achieves its maximum. Here ${\mc I}_X \subsetneq {\mc O}_Y$ is the ideal associated with $X$. The function $\mathrm{inv}_p({\mc I}_X)$ is constructed inductively by looking at the hypersurface of maximal contact (cf. \cite[Section 4.2]{abramovich2019functorial}) and the restricted coefficient ideals (cf. \cite[Section 4.1, 4.3, 5.1]{abramovich2019functorial}). Since the uniqueness of maximal contact hypersurfaces up to complex analytic isomorphism holds, and the arguments concerning coefficient ideals work in the same way for the rings of complex analytic functions, the center for the desired weighted blow-up is uniquely specified (cf. \cite[Theorem 5.3.1, Theorem 6.1.1]{abramovich2019functorial}). It follows from the local constructions that the centers from the algebraic arguments and complex analytic arguments agree.
    \item The final piece of this discussion is about weighted blow-ups. This is done in \cite[Section 3]{abramovich2019functorial} via a Rees algebra construction, but the equivalent constructions via local defining equation or toric models (cf. \cite[Section 3.5, 3.6]{abramovich2019functorial}) show that the same can be done in the complex analytic setting, working with complex orbifolds. 
\end{enumerate}
Therefore, the functor $F^{\circ \infty}_{an}$ can be defined via the above procedures. The desired functoriality follows from the functoriality of $\mathrm{inv}_p$ under smooth submersions (the conterpart of smooth morphisms in the complex analytic setting) (cf. \cite[Theorem 5.1.1]{abramovich2019functorial}). The coincidence of $F^{\circ \infty}_{er}$ and $F^{\circ \infty}_{an}$ on algebraic pairs holds because each step of the algorithm matches.
\end{proof}

\subsection{Resolving universal zero loci}\label{subsec:resolve-Z}
We revisit the crucial algebraic varieties $Z_d^{\Gamma}(V,W)$ constructed from the universal zero locus of equivariant polynomial maps between complex $\Gamma$-representations $V$ and $W$ where $\Gamma$ is a finite group.  The canonical equivariant Whitney stratifications on $Z_d^{\Gamma}(V,W)$ was studied extensively in \cite[Section 4]{Bai_Xu_2022}, but we study them instead through the lens of resolution of singularities.

\begin{defn}\label{defn:universal-zero-locus}
Let $\Gamma$ be a finite group. Suppose that $V$ and $W$ are finite-dimensional complex $\Gamma$-representations. For a positive integer $d$, recall that $\mathrm{Poly}_d^{\Gamma}(V,W)$ is the affine space of $\Gamma$-equivariant polynomial maps of degree at most $d$. Define the \emph{universal zero locus} to be  the $0$-locus of the evaluation map from $V \times \mathrm{Poly}_d^{\Gamma}(V,W)$ to $W$:
\beq
\begin{tikzcd}
  Z_d^{\Gamma}(V,W) := \ev^{-1}(0)\ar[r] \ar[d] & V \times \mathrm{Poly}_d^{\Gamma}(V,W) \ar[d,"\ev"] \\
  \{ 0\} \ar[r] & W 
\end{tikzcd}
\eeq
\end{defn}

The product $V \times \mathrm{Poly}_d^{\Gamma}(V,W)$ admits a $\Gamma$-action which acts on $V$ by the natural action and acts trivally on the second factor $\mathrm{Poly}_d^{\Gamma}(V,W)$. This action preserves the universal zero locus $Z_d^{\Gamma}(V,W)$. The proof of the following fact can be found at \cite[Proposition 4.9]{Bai_Xu_2022}.

\begin{lemma}[{Cf. \cite[Lemma 5]{Fukaya_Ono_integer}}]\label{lem:FO}
If $V$ is a faithful $\Gamma$-representation, there exists an integer $d_0 \geq 1$ such that for any $d \geq d_0$, the subvariety
\beq
Z_d^{\Gamma}(V,W)^{\circ} := \{(v,P) \in Z_d^{\Gamma}(V,W) \ | \ \text{the stabilizer of } (v, P) \text{ under the } \Gamma \text{-action is trivial} \}
\eeq
lies in the regular locus of the evaluation map $\ev$. In particular, it is smooth of dimension $\dim_{\mb C}\mathrm{Poly}_d^{\Gamma}(V,W) + \dim_{\mb C}V - \dim_{\mb C}W$. \qed
\end{lemma}

\begin{rem}
The above transversality result is eventually reduced to the following basic fact about the representation theory of finite groups: under the same assumption, there exists $d_0$ such that $W$ is included as a subrepresentation of the symmetric product $\mathrm{Sym}^{d_0}(V)$ with the induced $\Gamma$-action. The faithfulness of $V$ is crucial here.
\end{rem}

We will use Lemma \ref{lem:FO} to make sure that the resolution of the zero locus of the section is an orbifold 
which has the expected dimension. 



Denote by $\ov{Z}_d^{\Gamma}(V,W)^{\circ}$ the Zariski closure of $Z_d^{\Gamma}(V,W)^{\circ}$ in $Z_d^{\Gamma}(V,W)$. By construction, $\ov{Z}_d^{\Gamma}(V,W)^{\circ}$ is preserved under the $\Gamma$-action on $Z_d^{\Gamma}(V,W)$. Accordingly,
\beq\label{eqn:Zd-pair}
(\ov{Z}_d^{\Gamma}(V,W)^{\circ} \subset V \times \mathrm{Poly}_d^{\Gamma}(V,W))
\eeq
is pair over ${\mb C}$ acted on by the finite group $\Gamma$.

\begin{defn}\label{defn:Z-d-resolve}
Define $\tilde{Z}_d^{\Gamma}(V,W)^{\circ}$ to be the resolution of $\ov{Z}_d^{\Gamma}(V,W)^{\circ}$ by applying the functor $F_{er}^{\circ \infty}$ from Theorem \ref{thm:ATW-resolution} to the $\Gamma$-equivariant pair $\ov{Z}_d^{\Gamma}(V,W)^{\circ} \subset V \times \mathrm{Poly}_d^{\Gamma}(V,W)$, which comes with a $\Gamma$-equivariant morphism
\beq\label{eqn:blowdown-inclusion}
\tilde{Z}_d^{\Gamma}(V,W)^{\circ} \to V \times \mathrm{Poly}_d^{\Gamma}(V,W).
\eeq
\end{defn}

Next, we prove that when varying the cut-off degree $d$, the resolutions $\tilde{Z}_d^{\Gamma}(V,W)^{\circ}$ are compatible with each other in the following sense.

\begin{prop}\label{prop:cartesian-degree}
Given a sufficiently large integer $d \geq 1$ and any integer $d' \geq d$, the morphisms of Equation \eqref{eqn:blowdown-inclusion} fit into a $\Gamma$-equivariant Cartesian commutative diagram
\beq\label{eqn:change-of-degree}
\begin{tikzcd}
{\tilde{Z}_d^{\Gamma}(V,W)^{\circ}} \arrow[r] \arrow[d] & {V \times \mathrm{Poly}_d^{\Gamma}(V,W)} \arrow[d] \\
{\tilde{Z}_{d'}^{\Gamma}(V,W)^{\circ}} \arrow[r]      & {V \times \mathrm{Poly}_{d'}^{\Gamma}(V,W)}     
\end{tikzcd}
\eeq
where the right vertical arrow is the natural inclusion.
\end{prop}

Note that the morphism $V \times \mathrm{Poly}_d^{\Gamma}(V,W) \to V \times \mathrm{Poly}_{d'}^{\Gamma}(V,W)$ is \emph{not} smooth if $d' > d$. Therefore, Proposition \ref{prop:cartesian-degree} does not follow directly from the functoriality part of Theorem \ref{thm:ATW-resolution}. To establish the commutativity of Diagram \eqref{eqn:change-of-degree}, we need to appeal to the following statement:

\begin{prop}[{Cf. \cite[Lemma 4.14]{Bai_Xu_2022}}]\label{prop:inverse-degree}
Denote by $\phi_d: V \times \mathrm{Poly}_d^{\Gamma}(V,W) \to V \times \mathrm{Poly}_{d+1}^{\Gamma}(V,W)$ the natural inclusion map. Then there exists an integer $d_0 \geq 1$, which depends on $\Gamma, V$ and $W$ such that for any $d \geq d_0$, there exists a morphism
\beq
\ud{\phi}_d' : V \times \mathrm{Poly}_{d+1}^{\Gamma}(V,W) \to \mathrm{Poly}_d^{\Gamma}(V,W)
\eeq
so that the induced morphism
\beq\label{eqn:inverse-degree}
\begin{aligned}
\phi_d' : V \times \mathrm{Poly}_{d+1}^{\Gamma}(V,W) &\to V \times \mathrm{Poly}_{d}^{\Gamma}(V,W) \\
(v, P) &\mapsto (v, \ud{\phi}_d'(v, P))
\end{aligned}
\eeq
satisfies the following properties.
\begin{enumerate}
    \item The composition $\phi_d' \circ \phi_d$ is equal to the identity map on $V \times \mathrm{Poly}_{d}^{\Gamma}(V,W)$.
    \item The equality of evaluation 
    \beq
    P(v) = \ud{\phi}_d'(v, P) (v)
    \eeq
    holds for any pair $(v,P) \in V \times \mathrm{Poly}_{d+1}^{\Gamma}(V,W)$.
\end{enumerate}
\end{prop}
\begin{proof}
We reproduce the proof here because it will be needed for the following discussions. 

Because the space of $\Gamma$-equivariant polynomial maps of all degrees $\mathrm{Poly}^{\Gamma}(V,W)$ is a finitely generated module over the ring of $\Gamma$-invariant functions ${\mb C}[V]^{\Gamma}$, we can choose generators $Q_1, \dots, Q_n$. Let $d_0$ be the maximum of the degrees of $Q_1, \dots, Q_n$. Then for any $d \geq d_0$, let $\{ P_1, \dots, P_k \}$ be a complex basis of homogeneous $\Gamma$-equivariant polynomial maps from $V$ to $W$ of degree $(d+1)$. Then we can find $h_{ij} \in {\mb C}[V]^{\Gamma}$ for $1 \leq i \leq k, 1 \leq j \leq n$ such that
\begin{equation}
    P_i = \sum_{j} h_{ij}Q_j.
\end{equation}
Using these choices, we can define 
\beq
\begin{aligned}
\phi_d' : V \times \mathrm{Poly}_{d+1}^{\Gamma}(V,W) &\to V \times \mathrm{Poly}_{d}^{\Gamma}(V,W) \\
(v, P) &\mapsto (v, P - P' + \sum_{i,j} a_i h_{ij}(v)Q_j),
\end{aligned}
\eeq
where $P'$ is the degree $d+1$ part of $P$ which can be written as $P' = \sum_{i} a_i P_i$ using the basis $\{ P_1, \dots, P_k \}$. This map satisfies the desired properties.
\end{proof}

\begin{cor}\label{cor:degree-submersion}
The morphism $\phi_d'$ from Equation \eqref{eqn:inverse-degree} is a smooth morphism. Moreover, there exists a $\Gamma$-equivariant Cartesian diagram
\beq\label{diag:degree-reversed}
\begin{tikzcd}
{\tilde{Z}_{d+1}^{\Gamma}(V,W)^{\circ}} \arrow[r] \arrow[d] & {V \times \mathrm{Poly}_{d+1}^{\Gamma}(V,W)} \arrow[d] \\
{\tilde{Z}_{d}^{\Gamma}(V,W)^{\circ}} \arrow[r]      & {V \times \mathrm{Poly}_{d}^{\Gamma}(V,W)}.    
\end{tikzcd}
\eeq
\end{cor}
\begin{proof}
By Proposition \ref{prop:inverse-degree}, the morphism $\phi_d'$ and its induced maps between K\"ahler differentials at each point admit a left inverse defined by modifying \eqref{eqn:inverse-degree} as follows. Given an element $(v, \tilde{P}) \in V \times \mathrm{Poly}_{d+1}^{\Gamma}(V,W)$, we denote the degree $(d+1)$ part of $\tilde{P}$ by $\tilde{P}'$, and write it as $\tilde{P}' = \sum_{i} \tilde{a}_i P_i$ using the basis from the proof of Proposition \ref{prop:inverse-degree}. Note that the image of $(v, \tilde{P})$ under $\phi_d'$ is $(v, \tilde{P} - \tilde{P}' + \sum_{i,j} \tilde{a}_i h_{ij}(v)Q_j)$. We produce a right inverse of $\phi_d'$ around this point by
\begin{equation}
   V \times \mathrm{Poly}_{d}^{\Gamma}(V,W) \ni (v, P) \mapsto (v, P + \tilde{P}' - \sum_{i,j} \tilde{a}_i h_{ij}(v)Q_j).
\end{equation}
Because the composition of this map with $\phi_d'$ on the right is the identity, we see that the induced map of $\phi_d'$ on K\"ahler differentials admits a left inverse at $(v, \tilde{P})$.

By the implicit function theorem \cite[Chapter 8.5, Proposition 10]{Bosch2013}, we conclude that $\phi_d'$ is smooth since the source and target of this map are both smooth. By item (2) of Proposition \ref{prop:inverse-degree}, we know that we have a Cartesian diagram
\beq
\begin{tikzcd}
{\ov{Z}_{d+1}^{\Gamma}(V,W)^{\circ}} \arrow[r] \arrow[d] & {V \times \mathrm{Poly}_{d+1}^{\Gamma}(V,W)} \arrow[d, "\phi_d'"] \\
{\ov{Z}_d^{\Gamma}(V,W)^{\circ}} \arrow[r]      & {V \times \mathrm{Poly}_{d}^{\Gamma}(V,W)}     
\end{tikzcd}
\eeq
such that all the arrows are $\Gamma$-equivariant, and it defines a smooth morphism of pairs
\beq
(\ov{Z}_{d+1}^{\Gamma}(V,W)^{\circ} \subset V \times \mathrm{Poly}_{d+1}^{\Gamma}(V,W)) \to (\ov{Z}_d^{\Gamma}(V,W)^{\circ} \subset V \times \mathrm{Poly}_{d}^{\Gamma}(V,W)).
\eeq
As a result, we can apply the resolution algorithm Theorem \ref{thm:ATW-resolution} and Corollary \ref{cor:cube-diagram} to obtain Diagram \eqref{diag:degree-reversed}. The $\Gamma$-equivariance follows from Corollary \ref{cor:equivariance}.
\end{proof}

\begin{proof}[Proof of Proposition \ref{prop:cartesian-degree}] 
We consider the ($\Gamma$-equivariant) scheme $Z$ defined as the fiber product
\beq
\begin{tikzcd}
{Z} \arrow[r] \arrow[d] & {V \times \mathrm{Poly}_d^{\Gamma}(V,W)} \arrow[d] \\
{\tilde{Z}_{d+1}^{\Gamma}(V,W)^{\circ}} \arrow[r]      & {V \times \mathrm{Poly}_{d+1}^{\Gamma}(V,W)}     
\end{tikzcd}
\eeq
where the bottom horizontal arrow is the composition of the blow-down map $\tilde{Z}_{d+1}^{\Gamma}(V,W)^{\circ} \to \ov{Z}_{d+1}^{\Gamma}(V,W)^{\circ}$ and the inclusion map $\ov{Z}_{d+1}^{\Gamma}(V,W)^{\circ} \to V \times \mathrm{Poly}_{d+1}^{\Gamma}(V,W)$. It can be concatenated with Diagram \eqref{diag:degree-reversed}, which gives us the commutative diagram
\beq
\begin{tikzcd}
{Z} \arrow[d] \arrow[r]                                & {\tilde{Z}_{d+1}^{\Gamma}(V,W)^{\circ}} \arrow[r] \arrow[d] & {\tilde{Z}_{d}^{\Gamma}(V,W)^{\circ}} \arrow[d] \\
{V \times \mathrm{Poly}_d^{\Gamma}(V,W)} \arrow[r] & {V \times \mathrm{Poly}_{d+1}^{\Gamma}(V,W)} \arrow[r]      & {V \times \mathrm{Poly}_d^{\Gamma}(V,W).}       
\end{tikzcd}
\eeq
Because the two ``inner" squares are Cartesian, we see that the ``outer" square
\beq
\begin{tikzcd}
{Z} \arrow[r] \arrow[d] & {\tilde{Z}_{d}^{\Gamma}(V,W)^{\circ}} \arrow[d] \\
{V \times \mathrm{Poly}_{d}^{\Gamma}(V,W)} \arrow[r]      & {V \times \mathrm{Poly}_{d}^{\Gamma}(V,W)}     
\end{tikzcd}
\eeq
is also Cartesian. Now we observe that the bottom arrow is the identity map by Proposition \ref{prop:inverse-degree}, which implies that the top arrow is an isomorphism. This in turn establishes Proposition \ref{prop:cartesian-degree} for the case $d'=1$. For general $d'$, we just need to concatenate diagrams of the form \eqref{eqn:change-of-degree}.
\end{proof}

We are also interested in restricted representations. Let $\Gamma_1$ be a finite group and let $V$ and $W$ be finite-dimensional $\Gamma_1$-representations. Suppose that $\Gamma_2 < \Gamma_1$ is a proper subgroup. By restricting the representations along $\Gamma_2$, we obtain a decomposition
\beq\label{eqn:decomp-v1-v2}
V = \mathring{V}^{\Gamma_2} \oplus \check{V}^{\Gamma_2},
\eeq
where $\mathring{V}^{\Gamma_2}$ is the direct sum of all $\Gamma_2$-trivial irreducible subrepresentations of $V$. Introduce the following notation
\beq\label{eqn: Gamma-2-open}
V_{+}^{\Gamma_2} := \{ v \in V \ | \text{ the stabilizer of } v \text{ is a subgroup of } \Gamma_2 \}.
\eeq

Consider $V$ and the product $ \mathring{V}^{\Gamma_2} \times W$  as the total spaces of $\Gamma_2$-equivariant vector bundles over the affine space $\mathring{V}^{\Gamma_2}$ equipped with the trivial $\Gamma_2$-action
\beq
\underline{\check{V}^{\Gamma_2}} = \mathring{V}^{\Gamma_2} \times \check{V}^{\Gamma_2}, \quad \quad \quad \underline{W} = \mathring{V}^{\Gamma_2} \times W.
\eeq
Any element $\check{P} \in \mathrm{Poly}_d^{\Gamma_2}(\mathring{V}^{\Gamma_2} \oplus \check{V}^{\Gamma_2}, W)$ can be viewed as a $\Gamma_2$-equivariant bundle morphism $\underline{\check{V}^{\Gamma_2}} \to \underline{W}$ which over for any $\mathring{v} \in \mathring{V}^{\Gamma_2}$ is given by $\check{P}(\mathring{v}, -) \in \mathrm{Poly}_d^{\Gamma_2}(\check{V}^{\Gamma_2}, W)$. Write
\beq\label{eqn:polynomial-bundle-map}
\mathrm{Poly}_d^{\Gamma_2}(\underline{\check{V}^{\Gamma_2}}, \underline{W})
\eeq
the bundle of $\Gamma_2$-equivariant fiberwise polynomial maps $\underline{\check{V}^{\Gamma_2}} \to \underline{W}$ of degree at most $d$.

Now given $P \in \mathrm{Poly}_d^{\Gamma_1}(V, W)$, the natural inclusion $\mathrm{Poly}_d^{\Gamma_1}(V, W) \to \mathrm{Poly}_d^{\Gamma_2}(V, W)$ by viewing a $\Gamma_1$-equivariant polynomial map as a $\Gamma_2$-equivariant polynomial map defines a morphism
\beq\label{eqn:partial-ev}
\begin{aligned}
V \times \mathrm{Poly}_d^{\Gamma_1}(V, W) &\to \underline{\check{V}^{\Gamma_2}} \times_{\mathring{V}^{\Gamma_2}} \mathrm{Poly}_d^{\Gamma_2}(\underline{\check{V}^{\Gamma_2}}, \underline{W}) \\
(v, P) &\mapsto \big( (\mathring{v}, \check{v}), P(\mathring{v}, -) \big),
\end{aligned}
\eeq
where for any $v \in V$, we use the decomposition in Equation \eqref{eqn:decomp-v1-v2} to write $v = (\mathring{v}, \check{v})$.

After this preliminary discussion, we can start investigating the universal zero loci. Continuing the above setting, define
\beq\label{eqn:fiberwise-closure}
\ov{Z}_d^{\Gamma_2}(\underline{\check{V}^{\Gamma_2}}, \underline{W})^\circ\subset \underline{\check{V}^{\Gamma_2}} \times_{\mathring{V}^{\Gamma_2}} \mathrm{Poly}_d^{\Gamma_2}(\underline{\check{V}^{\Gamma_2}}, \underline{W}),
\eeq
to be the Zariski closure of the subset of the fiberwise zero locus consisting of points with trivial stabilizer under the $\Gamma_2$-action, which should be thought of as a family version of \eqref{eqn:Zd-pair}. By construction, we have a Cartesian diagram
\beq
\begin{tikzcd}
{\ov{Z}_d^{\Gamma_1}(V, W)^\circ} \arrow[d] \arrow[rr]                                    &  & {V \times \mathrm{Poly}_d^{\Gamma_1}(V, W)} \arrow[d]                                                                               \\
{\ov{Z}_d^{\Gamma_2}(\underline{\check{V}^{\Gamma_2}}, \underline{W})^\circ} \arrow[rr] &  & \underline{\check{V}^{\Gamma_2}} \times_{\mathring{V}^{\Gamma_2}} \mathrm{Poly}_d^{\Gamma_2}(\underline{\check{V}^{\Gamma_2}}, \underline{W})
\end{tikzcd}
\eeq
\begin{rem}
We wish to show that by applying the functor $F^{\circ \infty}_{er}$ to the pairs defined by the horizontal arrows of the above diagram, we can lift the Cartesian diagram to the resolutions. Unfortunately, we do not know a proof within the algebraic category, and we prove the desired result in the complex analytic setting. Nevertheless, we expect that the Cartesian square \eqref{eqn:change-of-group} holds in the algebraic category.  
\end{rem}

Write $\tilde{Z}_d^{\Gamma_2}(\underline{\check{V}^{\Gamma_2}}, \underline{W})^\circ$ the resolution of $\ov{Z}_d^{\Gamma_2}(\underline{\check{V}^{\Gamma_2}}, \underline{W})^\circ$ by applying $F^{\circ \infty}_{an}$ from Theorem \ref{thm:ATW-resolution-analytic} to the pair $(\ov{Z}_d^{\Gamma_2}(\underline{\check{V}^{\Gamma_2}}, \underline{W})^\circ \subset \underline{\check{V}^{\Gamma_2}} \times_{\mathring{V}^{\Gamma_2}} \mathrm{Poly}_d^{\Gamma_2}(\underline{\check{V}^{\Gamma_2}}, \underline{W}) )$, and it admits a morphism
\begin{equation}\label{eqn:family-inclusion}
    \tilde{Z}_d^{\Gamma_2}(\underline{\check{V}^{\Gamma_2}}, \underline{W})^\circ \to \underline{\check{V}^{\Gamma_2}} \times_{\mathring{V}^{\Gamma_2}} \mathrm{Poly}_d^{\Gamma_2}(\underline{\check{V}^{\Gamma_2}}, \underline{W}).
\end{equation}
Here we abuse the notations because after passing to the underlying complex manifolds, the outputs of $F^{\circ \infty}_{an}$ and $F^{\circ \infty}_{er}$ agree in view of Theorem \ref{thm:ATW-resolution-analytic}.

\begin{prop}\label{prop:cartesian-group}
Denote by $\underline{\check{V}_{+}^{\Gamma_2}}$ the open subvariety of $\check{V}^{\Gamma_2}$ as in Equation \eqref{eqn: Gamma-2-open}. For $d$ sufficiently large, after restricting to $V_{+}^{\Gamma_2}$, the analogue of the morphism \eqref{eqn:blowdown-inclusion} from $F^{\circ \infty}_{an}$ and the morphism \eqref{eqn:family-inclusion} fit into a $\Gamma_2$-equivariant Cartesian commutative diagram
\beq\label{eqn:change-of-group}
\begin{tikzcd}
{\tilde{Z}_d^{\Gamma_1}(V_{+}^{\Gamma_2}, W)^\circ} \arrow[r] \arrow[d] & {V_{+}^{\Gamma_2} \times \mathrm{Poly}_d^{\Gamma_1}(V, W)} \arrow[d] \\
{\tilde{Z}_d^{\Gamma_2}(\underline{\check{V}_{+}^{\Gamma_2}}, \underline{W})^\circ} \arrow[r]      & {\underline{\check{V}_{+}^{\Gamma_2}} \times_{\mathring{V}^{\Gamma_2}} \mathrm{Poly}_d^{\Gamma_2}(\underline{\check{V}^{\Gamma_2}}, \underline{W}).}     
\end{tikzcd}
\eeq
where the right vertical arrow is Equation \eqref{eqn:partial-ev}.
\end{prop}

We need some preparation to prove the above Proposition.

\begin{lemma}\label{lemma:group-inverse}
    For $d$ sufficiently large, near any point $(v, P) \in \ov{Z}_d^{\Gamma_1}(V, W)^\circ \cap V_{+}^{\Gamma_2} \times \mathrm{Poly}_d^{\Gamma_1}(V, W)$, the natural inclusion map
    \beq
    V_{+}^{\Gamma_2} \times \mathrm{Poly}_d^{\Gamma_1}(V, W) \subset V_{+}^{\Gamma_2} \times \mathrm{Poly}_d^{\Gamma_2}(V, W)
    \eeq
    has a complex analytic left inverse $\theta_d$ which preserves the evaluation map $(v, P) \mapsto P(v)$.
\end{lemma}

\begin{proof}
The space $\mathrm{Poly}_{d}^{\Gamma_2}(V, W)$ is a finite-dimensional complex $\Gamma_1$-representation under the conjugation action whose trivial isotypic piece is identified with $\mathrm{Poly}_{d}^{\Gamma_1}(V, W)$.. Therefore, given any $P \in \mathrm{Poly}_{d}^{\Gamma_2}(V, W)$, it admits a canonical decomposition $P = P' + P''$ where $P'$ is $\Gamma_1$-equivariant.

By the Lagrange interpolation method, there exists an integer $d_0$ such that for any $v \in V_{+}^{\Gamma_2}$, there is a $\Gamma_2$-invariant polynomial $L_{v}: V \to {\mb C}$ of degree at most $d_0$ which satisfies $L_v(\gamma v) = 1$ for $\gamma \in \Gamma_2$ and $L_v(\gamma v) = 0$ for $\gamma \in \Gamma_1 \setminus \Gamma_2$. Locally, $L_v$ depends complex analytically on $v$. Indeed, fixing $v_0$, one may choose a $1$-dimensional subspace $l_{v_0} \subset V$ and a complex linear projection from $V$ to $l_{v_0}$ such that the $\Gamma_1$ orbits of $v_0$ have distinct images. This property holds for all $v$ sufficiently close to $v_0$. Then $L_v$ is obtained by composing the standard $1$-dimensional Lagrangian interpolation polynomial with this linear transformation followed by taking average over $\Gamma_2$. We define a map $V_{+}^{\Gamma_2} \times \mathrm{Poly}_d^{\Gamma_2}(V, W) \to V_{+}^{\Gamma_2} \times \mathrm{Poly}_{d+d_0}^{\Gamma_1}(V, W)$ by
\beq
(v, P) \mapsto (v, P' + \frac{1}{|\Gamma_2|} \sum_{\gamma \in \Gamma_1} \gamma^{-1} \circ (L_{v} \cdot P'') \circ \gamma).
\eeq

Note that this map may increase the degree by $d_0$, but we can compose it with the map $\phi_d'$ from Equation \eqref{eqn:inverse-degree} to bring the degree back in range. We denote such a map by $\theta_d$, and one can check from the definition that it satisfies the desired properties.
\end{proof}

\begin{lemma}
By applying $F^{\circ \infty}_{an}$ to the pairs $(\ov{Z}_d^{\Gamma_1}(V_{+}^{\Gamma_2}, W)^\circ \subset V_{+}^{\Gamma_2} \times \mathrm{Poly}_d^{\Gamma_1}(V, W))$ and $(\ov{Z}_d^{\Gamma_2}(V_{+}^{\Gamma_2}, W)^\circ \subset V_{+}^{\Gamma_2} \times \mathrm{Poly}_d^{\Gamma_2}(V, W))$, the morphisms from Equaion \eqref{eqn:blowdown-inclusion} fit into a $\Gamma_2$-equivariant Cartesian diagram
\beq\label{eqn:change-of-group-prelim}
\begin{tikzcd}
{\tilde{Z}_d^{\Gamma_1}(V_{+}^{\Gamma_2}, W)^\circ} \arrow[r] \arrow[d] & {V_{+}^{\Gamma_2} \times \mathrm{Poly}_d^{\Gamma_1}(V, W)} \arrow[d] \\
{\tilde{Z}_d^{\Gamma_2}(V_{+}^{\Gamma_2}, W)^\circ} \arrow[r]      & {V_{+}^{\Gamma_2} \times \mathrm{Poly}_d^{\Gamma_2}(V, W)}     
\end{tikzcd}
\eeq
\end{lemma}
\begin{proof}
    By Lemma \ref{lemma:group-inverse}, locally near any given $(v, P) \in V_{+}^{\Gamma_1} \times \mathrm{Poly}_d^{\Gamma_1}(V, W) \subset V_{+}^{\Gamma_2} \times \mathrm{Poly}_d^{\Gamma_2}(V, W)$, there exists a complex analytic submersion $V_{+}^{\Gamma_2} \times \mathrm{Poly}_d^{\Gamma_2}(V, W) \to V_{+}^{\Gamma_1} \times \mathrm{Poly}_d^{\Gamma_1}(V, W)$ which preserves the evaluation map. Then the Cartesian-ness of Diagram \eqref{eqn:change-of-group-prelim} holds locally as in the proof of Proposition \ref{prop:cartesian-degree}. Using the functoriality again, we see that these local diagrams can be patched together to give the desired result. 
\end{proof}

\begin{proof}[Proof of Proposition \ref{prop:cartesian-group}]
Consider the following map
\begin{equation}
\begin{aligned}
    V_{+}^{\Gamma_2} \times \mathrm{Poly}_d^{\Gamma_2}(V, W) &\to \underline{\check{V}_{+}^{\Gamma_2}} \times_{\mathring{V}^{\Gamma_2}} \mathrm{Poly}_d^{\Gamma_2}(\underline{\check{V}^{\Gamma_2}}, \underline{W}) \\
    (v, P) &\mapsto \big( (\mathring{v}, \check{v}), P(\mathring{v}, -) \big)
\end{aligned}
\end{equation}
as in Equation \eqref{eqn:partial-ev}. By inspecting the definition, this is a complex analytic submersion (in fact, an algebraic one) such that the evaluation map is preserved. Therefore, we obtain a smooth morphism of complex analytic pairs
\begin{equation}
    (\ov{Z}_d^{\Gamma_2}(V_{+}^{\Gamma_2}, W)^\circ  \subset V_{+}^{\Gamma_2} \times \mathrm{Poly}_d^{\Gamma_2}(V, W)) \to (\ov{Z}_d^{\Gamma_2}(\underline{\check{V}_{+}^{\Gamma_2}}, \underline{W})^\circ \subset \underline{\check{V}_{+}^{\Gamma_2}} \times_{\mathring{V}^{\Gamma_2}} \mathrm{Poly}_d^{\Gamma_2}(\underline{\check{V}^{\Gamma_2}}, \underline{W}) ).
\end{equation}
Applying the resolution functor $F^{\circ \infty}_{an}$, we obtain a Cartesian diagram
\beq
\begin{tikzcd}
{\tilde{Z}_d^{\Gamma_2}(V_{+}^{\Gamma_2}, W)^\circ} \arrow[r] \arrow[d] & {V_{+}^{\Gamma_2} \times \mathrm{Poly}_d^{\Gamma_2}(V, W)} \arrow[d] \\
{\tilde{Z}_d^{\Gamma_2}(\underline{\check{V}_{+}^{\Gamma_2}}, \underline{W})^\circ} \arrow[r]      & {\underline{\check{V}_{+}^{\Gamma_2}} \times_{\mathring{V}^{\Gamma_2}} \mathrm{Poly}_d^{\Gamma_2}(\underline{\check{V}^{\Gamma_2}}, \underline{W}).}     
\end{tikzcd}
\eeq
By concatenating this diagram with \eqref{eqn:change-of-group-prelim}, we see that \eqref{eqn:change-of-group} is Cartesian. The $\Gamma_2$-equivariance follows from funtoriality.
\end{proof}

The above Proposition provides the key property which eventually permits the possibility of the local-to-global approach to the resolution of singularities. This replaces the role of the canonical Whitney stratification as developed in \cite[Section 4]{Bai_Xu_2022}. 

We need a result concerning the behavior of zero-loci under stabilizations.  Let $\Gamma$ be a finite group, and suppose that $V$, $W$, and $V'$ are finite-dimensional complex $\Gamma$-representations such that $V$ is a faithful $\Gamma$-representation. Then we can define a $\Gamma$-equivariant morphism
\beq\label{eqn:poly-stabilization}
\begin{aligned}
V \times \mathrm{Poly}_d^{\Gamma}(V, W) &\to (V \oplus V') \times \mathrm{Poly}_d^{\Gamma}(V \oplus V', W \oplus V') \\
(v, P) &\mapsto \big( (v, 0), \mathrm{diag}(P, \mathrm{Id}) \big),
\end{aligned}
\eeq
where $\mathrm{Id}: V' \to V'$ is the identity matrix, thought of as a $\Gamma$-equivariant polynomial map of degree $1$. This induces a morphism of pairs
\beq
\begin{aligned}
\big( \ov{Z}_d^{\Gamma}(V,W)^\circ &\subset V \times \mathrm{Poly}_d^{\Gamma}(V, W)  \big) \to \\
\big( \ov{Z}_d^{\Gamma}(V \oplus V',W \oplus V')^\circ &\subset (V \oplus V') \times \mathrm{Poly}_d^{\Gamma}(V \oplus V', W \oplus V')  \big).
\end{aligned}
\eeq

The following statement asserts the  compatibility of resolutions under such stabilization.
\begin{prop}\label{prop:stabilization-compatible}

We have a $\Gamma$-equivariant Cartesian commutative diagram
\beq\label{eqn:stabilization-Cartesian}
\begin{tikzcd}
{\tilde{Z}_d^{\Gamma}(V,W)^\circ} \arrow[rr] \arrow[d]           &  & {V \times \mathrm{Poly}_d^{\Gamma}(V, W)} \arrow[d]                       \\
{\tilde{Z}_d^{\Gamma}(V \oplus V',W \oplus V')^\circ} \arrow[rr] &  & {(V \oplus V') \times \mathrm{Poly}_d^{\Gamma}(V \oplus V', W \oplus V')},
\end{tikzcd}
\eeq
in which the horizontal arrows are the outcomes of Theorem \ref{thm:ATW-resolution} composed with the blow-down maps.
\end{prop}
\begin{proof}
We shall use the following description of the zero-locus as a fiber product:
\beq
Z_d^{\Gamma}(V \oplus V',W \oplus V') = Z_d^{\Gamma}(V \oplus V',W) \times_{V \oplus V'} Z_d^{\Gamma}(V \oplus V',V').
\eeq
For an element $(v,v') \in V \oplus V'$, note that if either $v$ or $v'$ has trivial stabilizer, then $(v,v')$ has trivial stabilizer under the diagonal $\Gamma$-action. On the other hand, note that the variety $Z_d^{\Gamma}(V \oplus V',V')$ is smooth near the point $((v, 0), \mathrm{Id}_{V'})$ for any $v \in V$, where the identity matrix $\mathrm{Id}_{V'}$ is treated as a $\Gamma$-equivariant polynomial map of degree $1$ from $V \oplus V'$ to $V'$ which is independent of $v$. As a result, for the subvariety $\ov{Z}_d^{\Gamma}(V \oplus V',W)^\circ \times_{V \oplus V'} Z_d^{\Gamma}(V \oplus V',V')$, near $V \times \{0\} \times \mathrm{Poly}_d^{\Gamma}(V \oplus V',W) \times \{ \mathrm{Id}_{V'} \}$, the resolution can be identified with 
\beq
\tilde{Z}_d^{\Gamma}(V \oplus V',W)^\circ \times_{V \oplus V'} Z_d^{\Gamma}(V \oplus V',V').
\eeq

Furthermore, we have an inclusion $Z_d^{\Gamma}(V, W) \to Z_d^{\Gamma}(V \oplus V',W)$ defined by $(v, P) \mapsto ((v,0), P)$, where $P \in \mathrm{Poly}_d^{\Gamma}(V ,W)$ is viewed as an element of $\mathrm{Poly}_d^{\Gamma}(V \oplus V' ,W)$ which is independent of $V'$. We claim that we have a Cartesian diagram
\beq\label{eqn:cartesian-stab-1}
\begin{tikzcd}
{\tilde{Z}_d^{\Gamma}(V,W)^\circ} \arrow[rr] \arrow[d]           &  & {V \times \mathrm{Poly}_d^{\Gamma}(V, W)} \arrow[d]                       \\
{\tilde{Z}_d^{\Gamma}(V \oplus V',W)^\circ} \arrow[rr] &  & {(V \oplus V') \times \mathrm{Poly}_d^{\Gamma}(V \oplus V', W)}.
\end{tikzcd}
\eeq
Indeed, we can construct a left inverse of the right vertical arrow by the formula
\beq
((v,v'), P) \mapsto (v, P(-, v'))
\eeq
which preserves the evaluation map. It defines a surjective smooth morphism of pairs
\beq
(\tilde{Z}_d^{\Gamma}(V \oplus V',W)^\circ \subset (V \oplus V') \times \mathrm{Poly}_d^{\Gamma}(V \oplus V', W)) \to (\tilde{Z}_d^{\Gamma}(V,W)^\circ \subset V \times \mathrm{Poly}_d^{\Gamma}(V, W)),
\eeq
so we can again argue as in the proof of Proposition \ref{prop:cartesian-degree} to show that we can construct Diagram \eqref{eqn:cartesian-stab-1}. Then we  concatenate Diagram \eqref{eqn:cartesian-stab-1} with the following diagram
\beq
\begin{tikzcd}
{\tilde{Z}_d^{\Gamma}(V \oplus V',W)^\circ} \arrow[d]    \arrow[rr]       &  & {(V \oplus V') \times \mathrm{Poly}_d^{\Gamma}(V \oplus V', W)} \arrow[d]                       \\
{\tilde{Z}_d^{\Gamma}(V \oplus V',W)^\circ \times_{V \oplus V'} Z_d^{\Gamma}(V \oplus V',V')} \arrow[rr] &  & {(V \oplus V') \times \mathrm{Poly}_d^{\Gamma}(V \oplus V', W \oplus V')}.
\end{tikzcd}
\eeq
By the smoothness of $Z_d^{\Gamma}(V \oplus V',V')$ near $((v, 0), \mathrm{Id}_{V'})$, we know the bottom horizontal arrow is identified with
\beq
\tilde{Z}_d^{\Gamma}(V \oplus V',W \oplus V')^\circ \to (V \oplus V') \times \mathrm{Poly}_d^{\Gamma}(V \oplus V', W \oplus V')
\eeq
near $((v, 0), \mathrm{Id}_{V'})$. Therefore, such a concatenation finishes the proof.
\end{proof}

Finally, we discuss the following situation which will be used to show that the splitting maps in Theorem \ref{thm:inclusion_geometric_homotopical_splits} and Theorem \ref{thm:main} are module map. Slightly differently from the above setting, let $\Gamma_1$ and $\Gamma_2$ be finite groups, and suppose that $V_1$ and $W_1$ are finite-dimensional complex $\Gamma_1$-representations such that $V_1$ is faithful, while $V_2$ is a finite-dimensional complex $\Gamma_2$-representation. Then there exists a $\Gamma_1 \times \Gamma_2$-equivariant morphism
\beq\label{eqn:poly-product}
\begin{aligned}
(V_1 \oplus V_2) \times \mathrm{Poly}_d^{\Gamma_1}(V_1, W_1) &\to (V_1 \oplus V_2) \times \mathrm{Poly}_d^{\Gamma_1 \times \Gamma_2}(V_1 \oplus V_2, W_1) \\
\big( (v_1, v_2), P_1) \big) &\mapsto \big( (v_1, v_2), P_1 \big),
\end{aligned}
\eeq
where we declare that the $\Gamma_2$-action on $W_1$ is trivial, and we view $P_1 \in \mathrm{Poly}_d^{\Gamma_1 \times \Gamma_2}(V_1 \oplus V_2, W_1)$ which does not depend on the $V_2$-factor. It induces a morphism of pairs
\beq\label{eqn:pair-product}
\begin{aligned}
\big( \ov{Z}_d^{\Gamma_1}(V_1,W_1)^\circ \times V_2 &\subset (V_1 \oplus V_2) \times \mathrm{Poly}_d^{\Gamma_1}(V_1, W_1)  \big) \to \\
\big( \ov{Z}_d^{\Gamma_1 \times \Gamma_2}(V_1 \oplus V_2,W_1)^\circ &\subset (V_1 \oplus V_2) \times \mathrm{Poly}_d^{\Gamma_1 \times \Gamma_2}(V_1 \oplus V_2, W_1)  \big).
\end{aligned}
\eeq

\begin{prop}\label{prop:product-compatible}
    The morphism \eqref{eqn:pair-product} lifts to a $\Gamma_1 \times \Gamma_2$-equivariant Cartesian commutative diagram
    \beq\label{eqn:product-Cartesian}
\begin{tikzcd}
{\tilde{Z}_d^{\Gamma_1}(V_1,W_1)^\circ \times V_2} \arrow[rr] \arrow[d]           &  & {(V_1 \oplus V_2 ) \times \mathrm{Poly}_d^{\Gamma_1}(V_1, W_1)} \arrow[d]                       \\
{\tilde{Z}_d^{\Gamma_1 \times \Gamma_2}(V_1 \oplus V_2,W_1)^\circ} \arrow[rr] &  & {(V_1 \oplus V_2) \times \mathrm{Poly}_d^{\Gamma_1 \times \Gamma_2}(V_1 \oplus V_2, W_1),}
\end{tikzcd}
\eeq
in which the horizontal arrows are the outcomes of Theorem \ref{thm:ATW-resolution} composed with the blow-down maps.
\end{prop}
\begin{proof}
    Consider the following $\Gamma_1 \times \Gamma_2$-equivariant morphism:
    \beq
    \begin{aligned}
     (V_1 \oplus V_2) \times \mathrm{Poly}_d^{\Gamma_1 \times \Gamma_2}(V_1 \oplus V_2, W_1) &\to ( V_1 \oplus V_2 ) \times \mathrm{Poly}_d^{\Gamma_1}(V_1, W_1) \\
    \big( (v_1,v_2), P_2 \big) &\mapsto \big( (v_1, v_2), P_2(-, v_2) \big).
    \end{aligned}
    \eeq
    This is a smooth morphism such that the universal zero loci are preserved. Then we have the commutative diagram
    \beq
    \begin{tikzcd}
{\ov{Z}_d^{\Gamma_1}(V_1,W_1)^\circ \times V_2} \arrow[rr] \arrow[d]                          &  & {(V_1 \oplus V_2) \times \mathrm{Poly}_d^{\Gamma_1}(V_1, W_1)} \arrow[d]                                      \\
{\ov{Z}_d^{\Gamma_1 \times \Gamma_2}(V_1 \oplus V_2,W_1)^\circ} \arrow[rr] \arrow[d] &  & {(V_1 \oplus V_2) \times \mathrm{Poly}_d^{\Gamma_1 \times \Gamma_2}(V_1 \oplus V_2, W_1)} \arrow[d] \\
{\ov{Z}_d^{\Gamma_1}(V_1,W_1)^\circ \times V_2} \arrow[rr]                                    &  & {(V_1 \oplus V_2) \times \mathrm{Poly}_d^{\Gamma_1}(V_1, W_1)}                                               
\end{tikzcd}
    \eeq
    where the upper square comes from \eqref{eqn:pair-product} and the lower square is induced from the aforementioned smooth morphism. By construction, the outer square has vertical arrows given by the identity map. Then applying the functoriality aspect of Theorem \ref{thm:ATW-resolution} to the lower square and the outer square, as in the proof of Proposition \ref{prop:cartesian-degree}, we conclude the desired result on the compatibility of the resolutions of singularities.
\end{proof}

\subsection{Smooth bundles of algebraic varieties}\label{subsec:variety-bundle}
We need to introduce a notion of \emph{smooth} families of \emph{algebraic varieties}. This is of course a standard notion when the varieties are smooth, but we need a generalization to the singular case, which is less commonly considered. . 

\begin{defn}
Let $X$ be a smooth algebraic variety over ${\mb C}$. Given a smooth manifold $M$, an \emph{$X$-bundle} over $M$ is a smooth fiber bundle with fibre $X$, together with a reduction of the structure group from $\mathrm{Diff}(X)$ to the  \emph{algebraic automorphism group} $ \mathrm{Aut}(X) $ of the variety $X$.
\end{defn}
As usual, such a structure may be presented by an open cover $\{ U_{\alpha} \}$, and transition functions
    \beq
    g_{\alpha \beta}: U_{\alpha} \cap U_{\beta} \to \mathrm{Aut}(X)
    \eeq
    for every pairs of elements of this cover, which satisfy the cocycle cocycle condition $g_{\alpha \beta} \cdot g_{\beta \gamma} = g_{\alpha \gamma}$ over each triple intersection $U_{\alpha} \cap U_{\beta} \cap U_{\gamma}$, such that $g_{\alpha \alpha}$ is the identity.

For example, a complex vector bundle $E \to M$ of rank $k$ defines a ${\mb C}^k$-bundle over $M$ in the above sense. Our primary example is the following:

\begin{lemma}\label{lemma:Z-d-bundle}
Suppose $\Gamma$ is a finite group. Let $M$ be a connected smooth manifold which is viewed as a $\Gamma$-manifold with the trivial $\Gamma$-action. Let $V$ and $W$ be $\Gamma$-representations, with $\Gamma$ assumed to be faithful, and let $\mathbf{V}$ and $ \mathbf{W}$ be  $\Gamma$-equivariant vector bundles over $M$ with fibers respectively isomorphic to $V$ and $W$. For a sufficiently large $d \geq 1$, let $\tilde{Z}_d^{\Gamma}(V,W)^{\circ}$ be the smooth $\Gamma$-variety from Definition \ref{defn:Z-d-resolve}. Then there exists a $\tilde{Z}_d^{\Gamma}(V,W)^{\circ}$-bundle over $M$
\beq
\tilde{Z}_d^{\Gamma}(\mathbf{V},\mathbf{W})^{\circ} \to M
\eeq
admitting a $\Gamma$-equivariant bundle map to
\beq
\mathbf{V} \times_{M} \mathrm{Poly}_{d}^{\Gamma}(\mathbf{V}, \mathbf{W}),
\eeq
where the later denotes the $\mathrm{Poly}_{d}^{\Gamma}(V, W)$-bundle over $M$ induced from the vector bundle structure.
\end{lemma}
\begin{proof}
Suppose $\{ U_{\alpha} \}$ is an open cover of $M$ such that over any $U_{\alpha}$, the fiber product $\mathbf{V} \times_{M} \mathrm{Poly}_{d}^{\Gamma}(\mathbf{V}, \mathbf{W})$ is isomorphic to the product $U_{\alpha} \times V \times \mathrm{Poly}_d^{\Gamma}(V,W)$. Furthermore, we observe that $\mathbf{V} \times_{M} \mathrm{Poly}_{d}^{\Gamma}(\mathbf{V}, \mathbf{W})$ can be reconstructed from these local pieces from smooth transition maps
\beq
g_{\alpha \beta}: U_{\alpha} \cap U_{\beta} \to GL(V)^{\Gamma} \times GL(W)^{\Gamma}.
\eeq
Indeed, given $(v, P) \in GL(V)^{\Gamma} \times GL(W)^{\Gamma}$ and $(g_1, g_2) \in GL(V)^{\Gamma} \times GL(W)^{\Gamma}$, the latter takes the former to $(g_1 v, g_2 P(g_1^{-1} v))$.
Observe that the pair
\beq
\ov{Z}_d^{\Gamma}(V,W)^{\circ} \subset V \times \mathrm{Poly}_d^{\Gamma}(V,W)
\eeq
from \eqref{eqn:Zd-pair} is preserved under the $GL(V)^{\Gamma} \times GL(W)^{\Gamma}$-action. By Corollary \ref{cor:equivariance}, the algebraic group $GL(V)^{\Gamma} \times GL(W)^{\Gamma}$ acts on the resolution $\tilde{Z}_d^{\Gamma}(V,W)^{\circ}$ as automorphisms, and the morphism $\tilde{Z}_d^{\Gamma}(V,W)^{\circ} \to V \times \mathrm{Poly}_d^{\Gamma}(V,W)$ is both $GL(V)^{\Gamma} \times GL(W)^{\Gamma}$-equivariant and $\Gamma$-equivariant. Accordingly, the transition maps $g_{\alpha \beta}$ define
\beq
g_{\alpha}: U_{\alpha} \cap U_{\beta} \to GL(V)^{\Gamma} \times GL(W)^{\Gamma} \to \mathrm{Aut}(\tilde{Z}_d^{\Gamma}(\mathbf{V},\mathbf{W})^{\circ})
\eeq
and they satisfy the cocycle condition by construction. The $\tilde{Z}_d^{\Gamma}(V,W)^{\circ}$-bundle $\tilde{Z}_d^{\Gamma}(\mathbf{V},\mathbf{W})^{\circ} \to M$ is constructed by globalizing, and the $\Gamma$-equivariant bundle map to $\mathbf{V} \times_{M} \mathrm{Poly}_{d}^{\Gamma}(\mathbf{V}, \mathbf{W})$ is just the fiberwise blow-down map composed with the inclusion.
\end{proof}

The total space $\tilde{Z}_d^{\Gamma}(\mathbf{V},\mathbf{W})^{\circ}$ is acted on by $\Gamma$ and it preserves the fibration structure over $M$. The $\Gamma$-equivariant complex structure on $\tilde{Z}_d^{\Gamma}(V,W)^{\circ}$ defines a normal complex structure on the $\Gamma$-manifold $\tilde{Z}_d^{\Gamma}(\mathbf{V},\mathbf{W})^{\circ}$ in the sense of Definition \ref{defn:normal-complex-manifold}. We will use bundles of the form $\tilde{Z}_d^{\Gamma}(\mathbf{V},\mathbf{W})^{\circ}$ to define a suitable version of equivariant transversality in the next section and resolve the singularities of certain orbispaces inside normally complex orbifolds.

\subsection{Zero loci with prescribed stabilizer}
\label{sec:zero-loci-with}

In the above discussions, we studied the resolution of singularities of $\ov{Z}_d^{\Gamma}(V,W)^{\circ}$ the compactification of the isotropy-free part of the universal zero locus. Following \cite{Bai_Xu_2022}, we study pieces of $Z_d^{\Gamma}(V,W)$ with prescribed isotropy groups. 

Let $\Gamma$, $V$, and $W$ be the same as in Definition \ref{defn:universal-zero-locus}. Let $\Gamma' \leq \Gamma$ be a subgroup. Define 
\beq
Z_d^{\Gamma}(V,W)^{\circ}_{\Gamma'} := \{(v,P) \in Z_d^{\Gamma}(V,W) \ | \ \text{the stabilizer of } (v, P) \text{ under the } \Gamma \text{-action is } \Gamma' \}.
\eeq

\begin{lemma}[{Cf. \cite[Proposition 4.9]{Bai_Xu_2022}}]\label{lem:FO-1}
If $V$ is a faithful $\Gamma$-representation, then for $d$ sufficiently large, $Z_d^{\Gamma}(V,W)^{\circ}_{\Gamma'}$ is a smooth complex algebraic variety of dimension $\dim_{\mb C}\mathrm{Poly}_d^{\Gamma}(V,W) + \dim_{\mb C} \mathring{V}_{\Gamma'} - \dim_{\mb C} \mathring{W}_{\Gamma'}$, where $\mathring{V}_{\Gamma'}$ and $\mathring{W}_{\Gamma'}$ are respectively the trivial isotypic piece of $V$ and $W$ under the restricted $\Gamma'$-action. \qed
\end{lemma}

Denote by $\gamma'$ the set of subgroups of $\Gamma$ which are conjugate to $\Gamma'$. Then the variety 
\beq
Z_d^{\Gamma}(V,W)^{\circ}_{\gamma'} := \bigcup_{\Gamma' \in \gamma'} Z_d^{\Gamma}(V,W)^{\circ}_{\Gamma'} = \coprod_{\Gamma' \in \gamma'} Z_d^{\Gamma}(V,W)^{\circ}_{\Gamma'}
\eeq
is invariant under the $\Gamma$-action, and it is smooth provided that $d$ is large enough. We always take such a $d$ for subsequent discussions. Following Definition \ref{defn:Z-d-resolve}, write $\tilde{Z}_d^{\Gamma}(V,W)^{\circ}_{\gamma'}$ the resolution of the Zariski closure $\ov{Z}_d^{\Gamma}(V,W)^{\circ}_{\gamma'}$ by applying the functor $F_{er}^{\circ \infty}$ from Theorem \ref{thm:ATW-resolution} to the $\Gamma$-equivariant pair $\ov{Z}_d^{\Gamma}(V,W)^{\circ}_{\gamma'} \subset V \times \mathrm{Poly}_d^{\Gamma}(V,W)$. Then we have a $\Gamma$-equivariant morphism
\beq\label{eqn:blowdown-inclusion-gamma}
\tilde{Z}_d^{\Gamma}(V,W)^{\circ}_{\gamma'} \to V \times \mathrm{Poly}_d^{\Gamma}(V,W).
\eeq
Then we have the following generalization of Proposition \ref{prop:cartesian-degree}, which follows from the same proof.
\begin{prop}\label{prop:cartesian-degree-gamma}
Given a sufficiently large integer $d \geq 1$ and any integer $d' \geq d$, the morphisms of Equation \eqref{eqn:blowdown-inclusion-gamma} fit into a $\Gamma$-equivariant Cartesian commutative diagram
\beq\label{eqn:change-of-degree-gamma}
\begin{tikzcd}
{\tilde{Z}_d^{\Gamma}(V,W)^{\circ}_{\gamma'}} \arrow[r] \arrow[d] & {V \times \mathrm{Poly}_d^{\Gamma}(V,W)} \arrow[d] \\
{\tilde{Z}_{d'}^{\Gamma}(V,W)^{\circ}_{\gamma'}} \arrow[r]      & {V \times \mathrm{Poly}_{d'}^{\Gamma}(V,W)}     
\end{tikzcd}
\eeq
where the right vertical arrow is the natural inclusion. \qed
\end{prop}

Now consider a pair of groups $\Gamma_2 \leq \Gamma_1$ with a pair of finite-dimensional complex $\Gamma_1$-representations $V$ and $W$ such that $V$ is faithful. For a subgroup $\Gamma' \leq \Gamma_2$, write $\gamma_1'$ and $\gamma_2'$ the set of subgroups of $\Gamma_1$ and $\Gamma_2$ which are conjugate to $\Gamma'$ respectively. Using the same notation as before, write
\beq
\ov{Z}_d^{\Gamma_2}(\underline{\check{V}^{\Gamma_2}}, \underline{W})^\circ_{\gamma'_2} \subset \underline{\check{V}^{\Gamma_2}} \times_{\mathring{V}^{\Gamma_2}} \mathrm{Poly}_d^{\Gamma_2}(\underline{\check{V}^{\Gamma_2}}, \underline{W}), 
\eeq
the analogue of \eqref{eqn:fiberwise-closure}, which is the Zariski closure of $Z_d^{\Gamma_2}(\underline{\check{V}^{\Gamma_2}}, \underline{W})^\circ_{\gamma'_2}$, consisting of points whose stabilizer is conjugate to $\Gamma'$. Then its resolution under $F^{\circ \infty}_{an}$ is denoted by $\tilde{Z}_d^{\Gamma_2}(\underline{\check{V}^{\Gamma_2}}, \underline{W})^\circ_{\gamma'_2}$ and it comes with a $\Gamma_2$-equivariant morphism to $\underline{\check{V}^{\Gamma_2}} \times_{\mathring{V}^{\Gamma_2}} \mathrm{Poly}_d^{\Gamma_2}(\underline{\check{V}^{\Gamma_2}}, \underline{W})$. Abusing the notations, we have the following analogue of Proposition \ref{prop:cartesian-group}, which follows from the same proof.
\begin{prop}\label{prop:cartesian-group-gamma}
For $d$ sufficiently large, after restricting to $V_{+}^{\Gamma_2}$, using the resolution functor $F^{\circ \infty}_{an}$, we have a $\Gamma_2$-equivariant Cartesian commutative diagram
\beq\label{eqn:change-of-group-gamma}
\begin{tikzcd}
{\tilde{Z}_d^{\Gamma_1}(V_{+}^{\Gamma_2}, W)^\circ_{\gamma_1'}} \arrow[r] \arrow[d] & {V_{+}^{\Gamma_2} \times \mathrm{Poly}_d^{\Gamma_1}(V, W)} \arrow[d] \\
{\tilde{Z}_d^{\Gamma_2}(\underline{\check{V}_{+}^{\Gamma_2}}, \underline{W})^\circ_{\gamma_2'}} \arrow[r]      & {\underline{\check{V}_{+}^{\Gamma_2}} \times_{\mathring{V}^{\Gamma_2}} \mathrm{Poly}_d^{\Gamma_2}(\underline{\check{V}^{\Gamma_2}}, \underline{W}),}     
\end{tikzcd}
\eeq
where the right vertical arrow is Equation \eqref{eqn:partial-ev}. \qed
\end{prop}

Under Equation \eqref{eqn:poly-stabilization}, the analogue of Proposition \ref{prop:stabilization-compatible} is the following:

\begin{prop}\label{prop:stabilization-compatible-gamma}
If $V'$ is a $\Gamma$-representation, we have a $\Gamma$-equivariant Cartesian commutative diagram
\beq\label{eqn:stabilization-Cartesian-gamma}
\begin{tikzcd}
{\tilde{Z}_d^{\Gamma}(V,W)^\circ_{\gamma'}} \arrow[rr] \arrow[d]           &  & {V \times \mathrm{Poly}_d^{\Gamma}(V, W)} \arrow[d]                       \\
{\tilde{Z}_d^{\Gamma}(V \oplus V',W \oplus V')^\circ_{\gamma'}} \arrow[rr] &  & {(V \oplus V') \times \mathrm{Poly}_d^{\Gamma}(V \oplus V', W \oplus V')}. 
\end{tikzcd}
\eeq  \qed
\end{prop}

Finally, the following is the generalization of Proposition \ref{prop:product-compatible}, for which we use the $\gamma$-version of Equation \eqref{eqn:pair-product}. Note that the conjugacy classes of subgroups of the product group $\Gamma_1 \times \Gamma_2$ are given by pairs of conjugacy classes of $\Gamma_1$ and $\Gamma_2$.

\begin{prop}\label{prop:product-compatible-gamma}
    The morphism \eqref{eqn:pair-product} can be lifted to a $\Gamma_1 \times \Gamma_2$-equivariant Cartesian commutative diagram
    \beq\label{eqn:product-Cartesian-gamma}
\begin{tikzcd}
{\tilde{Z}_d^{\Gamma_1}(V_1,W_1)^\circ_{\gamma_1'} \times V_2} \arrow[rr] \arrow[d]           &  & { (V_1 \oplus V_2) \times \mathrm{Poly}_d^{\Gamma_1}(V_1, W_1)} \arrow[d]                       \\
{\tilde{Z}_d^{\Gamma_1 \times \Gamma_2}(V_1 \oplus V_2,W_1)^\circ_{\gamma_1'}} \arrow[rr] &  & {(V_1 \oplus V_2) \times \mathrm{Poly}_d^{\Gamma_1 \times \Gamma_2}(V_1 \oplus V_2, W_1),}
\end{tikzcd}
\eeq
where for $\tilde{Z}_d^{\Gamma_1 \times \Gamma_2}(V_1 \oplus V_2,W_1)^\circ_{\gamma_1'}$, we only prescribe the conjugacy class along the $\Gamma_1$-factor. \qed
\end{prop}

Similar to Section \ref{subsec:variety-bundle}, we can construct bundles
\beq\label{eqn:Z-bundle-gamma}
\tilde{Z}_d^{\Gamma}(\mathbf{V},\mathbf{W})^{\circ}_{\gamma'}
\eeq
whose fibres are given by this construction. We will use the discussions in this subsection in the study of structures of equivariant complex bordisms.

\section{Regular FOP sections and resolution of singularities}\label{sec:derive-manifolds}

In this section, we demonstrate how to construct a smooth normally complex orbifold from the zero locus of an FOP section satisfying certain regularity conditions using the machinery developed in Section \ref{sec:ATW}. Then we present our construction of manifolds from  normally complex derived orbifolds,  
leaving the proof of invariance to the next section after introducing various versions of orbifold bordisms. 

\subsection{Regular FOP sections}\label{subsec:regular-FOP}
We describe a form of equivariant transversality of FOP sections. To distinguish it from the notion of \emph{strong transversality} of FOP sections introduced in \cite[Definition 6.1]{Bai_Xu_2022}, we refer to our version as \emph{regular} FOP sections.

\begin{hyp}
    Unless otherwise stated, $M$ is a smooth manifold with a \emph{faithful} $\Gamma$-action.
\end{hyp}

\begin{rem}\label{rem:faithful}
    The attentive reader may worry that this somewhat restrictive assumption might preclude applications to $\Gamma$-manifolds with nontrivial generic isotropy groups, which inevitably show up in the definition of $\Gamma$-equivariant bordism. However, this does not affect the construction in this paper. Indeed, as in \cite{Bai_Xu_2022}, we rely on the result of Pardon \cite{pardon19} showing that after a stabilization, any derived orbifold can be assumed to have an effective ambient space (cf. Section \ref{subsec:d-Omega-C}). Using the geometric model of homotopical equivariant bordisms (cf. Section \ref{sec:equivariant-bordisms}), the same trick can be applied so that when discussing perturbations, so that any class in homotopical bordism can be represented by a derived manifold whose ambient space admits a faithful action. The faithfulness assumption is closely related to Lemma \ref{lem:FO} and Lemma \ref{lem:FO-1} which guarantee the generic smoothness of (strata of) the universal zero loci, following \cite{Bai_Xu_2022}. It is possible to develop a more general framework directly incorporating non-faithful actions, but one would have to modify these statements and the relevant discussion.
\end{rem}

For the next definition, we assume, as in Definition \ref{defn:FOP-lift-equivariant}, that $M$ is equipped with with a normal complex structure and a straightened connection, and $E \to M$ is a $\Gamma$-equivariant vector bundle with a normal complex structure and straightened connection. We choose in addition an integer $d$ which is sufficiently large so that the construction of  Lemma \ref{lemma:Z-d-bundle} results in a bundle
\beq
\tilde{Z}_d^{\Gamma'}(NM^{\Gamma'},\check{E}^{\Gamma'})^{\circ}
\eeq
over $M^{\Gamma'}$ with fibre  $\tilde{Z}_d^{\Gamma'}(-,-)^{\circ}$. More generally, applying the construction in Equation \eqref{eqn:Z-bundle-gamma}, for any conjugacy class $\gamma''$ of subgroups of $\Gamma'$, we can construct the bundle
\beq
\tilde{Z}_d^{\Gamma'}(NM^{\Gamma'},\check{E}^{\Gamma'})^{\circ}_{\gamma''}
\eeq
over $M^{\Gamma'}$ coming the part of of the universal zero locus whose isotropy group under the $\Gamma'$-action lies in the class $\gamma''$.

\begin{defn}\label{defn:equivariant-regular}
A smooth $\Gamma$-equivariant section $S: M \to E$ is said to be \emph{regular} at $p \in M$ if the following holds. Suppose the isotropy group of $p$ under the $\Gamma$-action is $\Gamma' \leq \Gamma$. Writing
\beq
S = (\mathring{S}_{\Gamma'}, \check{S}_{\Gamma'}): D_r (NM^{\Gamma'}) \to \mathring{E}^{\Gamma'} \oplus \check{E}^{\Gamma'},
\eeq
 we   then require  that the following condition hold near $p$:
\begin{enumerate}
\item $\check{S}_{\Gamma'}$ admits a normally complex lift (cf. Definition \ref{defn:FOP-lift-equivariant}) of degree at most $d$ 
\beq
\check{\mathfrak{s}}_{\Gamma'}: D_r(NM^{\Gamma'}) \to \mathrm{Poly}_d^{\Gamma'}(NM^{\Gamma'}, \check{E}^{\Gamma'}).
\eeq
\item The resulting $\Gamma'$-equivariant bundle map
\beq\label{eqn:gamma-prime-lift}
\begin{aligned}
D_r (NM^{\Gamma'}) &\to \mathring{E}^{\Gamma'} \oplus (NM^{\Gamma'} \times_{M^{\Gamma'}} \mathrm{Poly}_d^{\Gamma'}(NM^{\Gamma'},\check{E}^{\Gamma'})) \\
y &\mapsto (\mathring{S}_{\Gamma'}(y), (y, \check{\mathfrak{s}}_{\Gamma'}(y)))
\end{aligned}
\eeq
is transverse to the smooth map
\beq\label{eqn:inclusion-Z}
 \{0 \} \times \tilde{Z}_d^{\Gamma'}(NM^{\Gamma'},\check{E}^{\Gamma'})^{\circ}_{\gamma''} \to \mathring{E}^{\Gamma'} \oplus (NM^{\Gamma'} \times_{M^{\Gamma'}} \mathrm{Poly}_d^{\Gamma'}(NM^{\Gamma'},\check{E}^{\Gamma'}))
\eeq
\end{enumerate}
for all conjugacy classes $\gamma''$. If such a condition holds for all points $p \in M$, we say that $S$ is \emph{regular}.
\end{defn}

Note that the regularity condition particularly implies that $S$ is transverse to $0$ over the isotropy free locus of $M$. It may not be a priori clear whether the regularity condition for different subgroups are compatible, and in particular whether, given a pair of subgroups $\Gamma_2 < \Gamma_1 \leq \Gamma$, a choice of such regular section near $M^{\Gamma_1}$ yields a section which satifies the regularity condition for $\Gamma_2$ on a neighbourhood of $M^{\Gamma_1}$. This is the content of Proposition \ref{prop:open-ness} below, which requires some preliminary discussion.

Let $E \to M$ be as above. 
Using the restriction of the $\Gamma_1$ action on the normal bundle $NM^{\Gamma_1}$ to $\Gamma_2$, we obtain a decomposition
\beq\label{eqn:N-decomposition}
NM^{\Gamma_1} = \mathring{({NM^{\Gamma_1}})}^{\Gamma_2} \oplus \check{({NM^{\Gamma_1}})}^{\Gamma_2}
\eeq
where the first component denotes the $\Gamma_2$-trivial isotypic piece. Using the exponential map to identify the disc bundle $D_r(NM^{\Gamma_1})$ with an open neighborhood of $M^{\Gamma_1}$, we see that $M^{\Gamma_2}$ can be identified with the total space of $\mathring{({NM^{\Gamma_1}})}^{\Gamma_2}$ locally. As for the vector bundle $E|_{D_r(NM^{\Gamma_1})}$, we have the decomposition
\beq\label{eqn:E-decomposition}
\check{E}^{\Gamma_1} = (\mathring{E}^{\Gamma_2} \cap \check{E}^{\Gamma_1}) \oplus \check{E}^{\Gamma_2}.
\eeq
Let us look at a section $S = (\mathring{S}_{\Gamma_1}, \check{S}_{\Gamma_1}): D_r (NM^{\Gamma_1}) \to \mathring{E}^{\Gamma_1} \oplus \check{E}^{\Gamma_1}$ with a lift $\check{\mathfrak{s}}_{\Gamma_1}$ of degree at most $d$. Locally near $M^{\Gamma_1}$, for $v = (\mathring{v}, \check{v}) \in NM^{\Gamma_1}$ with respect to the decomposition in Equation \eqref{eqn:N-decomposition} where $NM^{\Gamma_1}$ is viewed as an open subset of $M$, if we write
\beq
\check{\mathfrak{s}}_{\Gamma_1} = (\mathring{\mathfrak{s}}_{\Gamma_1}, \check{\mathfrak{s}}_{\Gamma_2})
\eeq
using the decomposition of Equation \eqref{eqn:E-decomposition}, we obtain a bundle map
\beq\label{eqn:nc-lift-change-of-group}
\begin{aligned}
D_r(NM^{\Gamma_2}) &\to \mathrm{Poly}_d^{\Gamma_2}(NM^{\Gamma_2}, \check{E}^{\Gamma_2}) \\
(\mathring{v}, \check{v}) &\mapsto \check{\mathfrak{s}}_{\Gamma_2}(\mathring{v}, \check{v})(\mathring{v}, -)
\end{aligned}
\eeq
covering the identity map $\mathring{v} \mapsto \mathring{v}$ on $M^{\Gamma_2}$ which is locally identified with $\mathring{({NM^{\Gamma_1}})}^{\Gamma_2}$. A word on the notations: $\check{\mathfrak{s}}_{\Gamma_2}(\mathring{v}, \check{v})(-,-)$ stands for the polynomial determined by $\check{\mathfrak{s}}_{\Gamma_2}$ after evaluating at $(\mathring{v}, \check{v})$ in which $\check{v}$ is the bundle coordinate, and the symbol $(\mathring{v}, -)$ is the partial evaluation map so that we can get a polynomial map with the correct domain. This is just a global version of the construction in Equation \eqref{eqn:partial-ev}, with the exception that now we have a family of polynomials parametrized by $D_r(NM^{\Gamma_1})$. As a result, if we write $S = (\mathring{S}_{\Gamma_2}, \check{S}_{\Gamma_2}): D_r(NM^{\Gamma_2}) \to \mathring{E}^{\Gamma_2} \oplus \check{E}^{\Gamma_2}$, we have the expressions
\beq
\begin{aligned}
\mathring{S}_{\Gamma_2}(\mathring{v}, \check{v}) &= (\mathring{S}_{\Gamma_1}(\mathring{v}, \check{v}), \mathring{\mathfrak{s}}_{\Gamma_1}(\mathring{v}, \check{v})(\mathring{v}, \check{v})) \in \mathring{E}^{\Gamma_2} = \mathring{E}^{\Gamma_1} \oplus(\mathring{E}^{\Gamma_2} \cap \check{E}^{\Gamma_1}), \\
\check{S}_{\Gamma_2} &= \check{\mathfrak{s}}_{\Gamma_2}(\mathring{v}, \check{v})(\mathring{v}, \check{v}) \in \check{E}^{\Gamma_2},
\end{aligned}
\eeq
so we know that Equation \eqref{eqn:nc-lift-change-of-group} defines a normally complex lift of $\check{S}_{\Gamma_2}$.

\begin{prop}\label{prop:open-ness}
Using the notations introduced above, if $S$ is regular near $p \in M^{\Gamma_1}$, then near $p$, the $\Gamma_2$-equivariant smooth map 
\beq\label{eqn:gamma-2-lift}
\begin{aligned}
D_r (NM^{\Gamma_2}) &\to \mathring{E}^{\Gamma_2} \oplus (NM^{\Gamma_2} \times_{M^{\Gamma_2}} \mathrm{Poly}_d^{\Gamma_2}(NM^{\Gamma_2},\check{E}^{\Gamma_2})) \\
(\mathring{v}, \check{v}) &\mapsto (\mathring{S}_{\Gamma_2}(\mathring{v}, \check{v}), ((\mathring{v}, \check{v}), \check{\mathfrak{s}}_{\Gamma_2}(\mathring{v}, \check{v})(\mathring{v}, -)))
\end{aligned}
\eeq
is transverse to the smooth map
\beq
 \{ 0 \} \times \tilde{Z}_d^{\Gamma_2}(NM^{\Gamma_2},\check{E}^{\Gamma_2})^{\circ}_{\gamma_2'} \to \mathring{E}^{\Gamma_2} \oplus (NM^{\Gamma_2} \times_{M^{\Gamma_2}} \mathrm{Poly}_d^{\Gamma_2}(NM^{\Gamma_2},\check{E}^{\Gamma_2}),
\eeq
for any conjugacy class $\gamma_2'$ of subgroups of $\Gamma_2$.
\end{prop}

Before we proceed to the proof, we need a digression concerning transversality. The following elementary lemma is useful for showing the consistency of the regularity condition.

\begin{lemma}\label{lemma:transverse-Cartesian}
    Suppose the commutative diagram
    \begin{equation}\label{eqn:pullback-square}
\begin{tikzcd}
N_1 \arrow[d, "h"'] \arrow[rr, "f_1"] &  & M_1 \arrow[d, "g"] \\
N_2 \arrow[rr, "f_2"']                &  & M_2               
\end{tikzcd}
    \end{equation}
    is a pullback square of smooth manifolds such that $g$ and $f_2$ are transverse to each other. Suppose $F: Q \to M_1$ is a smooth map which is transverse to $f_1$. Then the composition $g \circ F: Q \to M$ is transverse to $f_2$.
\end{lemma}
\begin{proof}
    Take $n_2 \in N_2$ and $q \in Q$ such that $f_2(n_2) = (g \circ Q) (m_2)$, we need to show that
    \beq
    (f_2)_* (T_{n_2} N_2) + g_*  F_* (T_q Q) = T_{m_2} M_2. 
    \eeq
    Because \eqref{eqn:pullback-square} is a pullback square, there exists unique $n_1 \in N_1$ such that $f_1(n_1) = F(q)$ and $h(n_1) = n_2$ becasue $F(q)$ and $n_2$ are respectively mapped by $g$ and $f_2$ to the same point $m_2$. Using the transversality of $f_1$ and $F$, we know that 
    \beq\label{eqn:transverse-sum}
    F_* (T_q Q) + (f_1)_* (T_{n_1} N_1) = T_{F(q)} M_1.
    \eeq
    Using the transversality of $g$ and $f_2$, we have
    \beq\label{eqn:sum-1}
    g_* (T_{F(q)} M_1) + (f_2)_* (T_{n_2} N_2) = T_{m_2} M_2.
    \eeq
    Applying the pushforward $g_*$ to Equation \eqref{eqn:transverse-sum}, we get
    \beq
    g_* F_* (T_q Q) + g_* (f_1)_* (T_{n_1} N_1) = g_* (T_{F(q)} M_1).
    \eeq
    Note that $g_* (f_1)_* = (f_2)_* h_*$, so the above equation implies
    \beq\label{eqn:sum-2}
    g_* F_* (T_q Q) + (f_2)_* h_* (T_{n_1} N_1) = g_* (T_{F(q)} M_1).
    \eeq
    Combining \eqref{eqn:sum-1} and \eqref{eqn:sum-2}, we obtain the desired result.
\end{proof}

\begin{proof}[Proof of Proposition \ref{prop:open-ness}]
To simply the notations in the proof, we first work out the details in the case when $M^{\Gamma_1}$ is a single point $p$ and $\mathring{E}^{\Gamma_1}$ is trivial, and the general case can be proven using exactly the same ideas. Also, we omit the subscript by $\gamma_2'$ in the bundle $\tilde{Z}_d^{\Gamma_2}(NM^{\Gamma_2},\check{E}^{\Gamma_2})^{\circ}_{\gamma_2'}$.

Recall the notations from Section \ref{subsec:resolve-Z}, and write $NM^{\Gamma_1}$ as a finite-dimensional complex and faithful $\Gamma_1$-representation $V$ and $\check{E}^{\Gamma_1}$ a finite-dimensional complex $\Gamma_1$-representation $W$. We identify $M^{\Gamma_2}$ with $\mathring{V}^{\Gamma_2}$, over which we have the $\Gamma_2$-equivariant vector bundles $\underline{\check{V}^{\Gamma_2}}$ and $\underline{W} = \underline{\mathring{W}^{\Gamma_2}} \oplus \underline{\check{W}^{\Gamma_2}}$. The data of a normally complex lift of a smooth $\Gamma$-equivariant section near $M^{\Gamma_1}$ is a smooth $\Gamma_1$-equivariant map
\begin{equation}
    f: V \to \mathrm{Poly}_d^{\Gamma_1}(V, W)
\end{equation}
near $0 \in V$ with associated section
\begin{equation}
\begin{aligned}
    S_f: V &\to W \\
    v &\mapsto f(v)(v).
\end{aligned}
\end{equation}
Under such a lift, $S_f$ is regular if and only if the graph map of $f$ is transverse to the map $\tilde{Z}_d^{\Gamma_1}(V,W)^{\circ} \to V \times \mathrm{Poly}_d^{\Gamma_1}(V, W)$. Now for a vector $v' \in \mathring{V}^{\Gamma_2}$ with isotropy group $\Gamma_2$ near $0$, write the induced normal complex lift of $S_f$ near $v'$ as
\begin{equation}
\begin{aligned}
    (\mathring{S}_f, \check{\mathfrak{s}}_{f}): \underline{\check{V}^{\Gamma_2}} &\to \underline{W} = \underline{\mathring{W}^{\Gamma_2}} \oplus (\check{V}^{\Gamma_2}_1 \times_{\mathring{V}^{\Gamma_2}} \mathrm{Poly}_d^{\Gamma_2}(\underline{\check{V}^{\Gamma_2}},\underline{\check{W}^{\Gamma_2}})) \\
    (\mathring{v}, \check{v}) &\mapsto (\mathring{f}(\mathring{v}, \check{v})(\mathring{v}, \check{v}), \check{f}(\mathring{v}, \check{v})(\mathring{v}, -)),
\end{aligned}
\end{equation}
where we have written $f = (\mathring{f}, \check{f})$ under the decomposition $W = \mathring{W}^{\Gamma_2} \oplus \check{W}^{\Gamma_2}$. We need to verify that the above map is transverse to  $\{ 0 \} \times \tilde{Z}_d^{\Gamma_2}(\underline{\check{V}^{\Gamma_2}},\underline{\check{W}^{\Gamma_2}})^{\circ} \to \underline{W} = \underline{\mathring{W}^{\Gamma_2}} \oplus (\check{V}^{\Gamma_2}_1 \times_{\mathring{V}^{\Gamma_2}} \mathrm{Poly}_d^{\Gamma_2}(\underline{\check{V}^{\Gamma_2}},\underline{\check{W}^{\Gamma_2}}))$.

Note that we have a $\Gamma_2$-equivariant morphism
\beq
\begin{aligned}
\check{V}^{\Gamma_2} \times_{\mathring{V}^{\Gamma_2}} \mathrm{Poly}_d^{\Gamma_2}(\underline{\check{V}^{\Gamma_2}},\underline{W}) &\to \underline{\mathring{W}^{\Gamma_2}} \oplus \check{V}^{\Gamma_2} \times_{\mathring{V}^{\Gamma_2}} \mathrm{Poly}_d^{\Gamma_2}(\underline{\check{V}^{\Gamma_2}},\underline{\check{W}^{\Gamma_2}}) \\
((\mathring{v}, \check{v}), \mathring{P}_{\mathring{v}}, \check{P}_{\mathring{v}}) &\mapsto ((\mathring{v}, \mathring{P}_{\mathring{v}}(\mathring{v})), \check{P}_{\mathring{v}})
\end{aligned}
\eeq
where we use the decomposition $\underline{W} = \underline{\mathring{W}^{\Gamma_2}} \oplus \underline{\check{W}^{\Gamma_2}}$ and $\mathring{P}_{\mathring{v}}, \check{P}_{\mathring{v}}$ are the fiberwise polynomials over $\mathring{v}$. This is a smooth morphism of algebraic varieties, thus by using the functorial resolution Theorem \ref{thm:ATW-resolution}, we obtain a $\Gamma_2$-equivariant Cartesian diagram
\begin{equation}
    \begin{tikzcd}
{\tilde{Z}_d^{\Gamma_2}(\underline{\check{V}^{\Gamma_2}},\underline{W})^{\circ}} \arrow[rr] \arrow[d]                           &  & {\check{V}^{\Gamma_2}_1 \times_{\mathring{V}^{\Gamma_2}} \mathrm{Poly}_d^{\Gamma_2}(\underline{\check{V}^{\Gamma_2}},\underline{W})} \arrow[d]                                                           \\
{\{ 0 \} \times \tilde{Z}_d^{\Gamma_2}(\underline{\check{V}^{\Gamma_2}},\underline{\check{W}^{\Gamma_2}})^{\circ}} \arrow[rr] &  & {\underline{\mathring{W}^{\Gamma_2}} \oplus \check{V}^{\Gamma_2} \times_{\mathring{V}^{\Gamma_2}} \mathrm{Poly}_d^{\Gamma_2}(\underline{\check{V}^{\Gamma_2}},\underline{\check{W}^{\Gamma_2}})}
\end{tikzcd}
\end{equation}
of compatible resolutions of the universal zero loci. By concatenating the above diagram with Diagram \ref{eqn:change-of-group}, we obtain a Cartesian diagram of smooth varieties
\begin{equation}\label{eqn:diagram-useful}
    \begin{tikzcd}
{\tilde{Z}_d^{\Gamma_1}(V_{+}^{\Gamma_2}, W)^\circ} \arrow[rr] \arrow[d]                           &  & {V_{+}^{\Gamma_2} \times \mathrm{Poly}_d^{\Gamma_1}(V, W)} \arrow[d]                                                           \\
{\{ 0 \} \times \tilde{Z}_d^{\Gamma_2}(\underline{\check{V}^{\Gamma_2}},\underline{\check{W}^{\Gamma_2}})^{\circ}} \arrow[rr] &  & {\underline{\mathring{W}^{\Gamma_2}} \oplus \check{V}^{\Gamma_2} \times_{\mathring{V}^{\Gamma_2}_{+}} \mathrm{Poly}_d^{\Gamma_2}(\underline{\check{V}^{\Gamma_2}},\underline{\check{W}^{\Gamma_2}})}.
\end{tikzcd}
\end{equation}
Now we can apply Lemma \ref{lemma:transverse-Cartesian} to the graph map of $f$, which establishes the desired transversality result for $(\mathring{S}_f, \check{\mathfrak{s}}_{f})$ for near $v'$ because $v'$ lies in the open subset $V^{\Gamma_2}_{+}$.

In general, using the bundles constructed in Section \ref{subsec:variety-bundle}, we can assemble the fiberwise diagram \eqref{eqn:diagram-useful}, we get
\beq\label{eqn:change-group-res-compatible}
\begin{tikzcd}
{\{0\} \times \tilde{Z}_d^{\Gamma_1}(NM^{\Gamma_1},\check{E}^{\Gamma_1})^{\circ}} \arrow[r] \arrow[d] & {\mathring{E}^{\Gamma_1} \oplus NM^{\Gamma_1} \times_{M^{\Gamma_1}} \mathrm{Poly}_d^{\Gamma_1}(NM^{\Gamma_1},\check{E}^{\Gamma_1})} \arrow[d] \\
{\{0\} \times \tilde{Z}_d^{\Gamma_2}(NM^{\Gamma_2},\check{E}^{\Gamma_1})^{\circ}} \arrow[r]      & {\mathring{E}^{\Gamma_2} \oplus (NM^{\Gamma_2} \times_{M^{\Gamma_2}} \mathrm{Poly}_d^{\Gamma_2}(NM^{\Gamma_2},\check{E}^{\Gamma_1}))}.  
\end{tikzcd}
\eeq
Then the same argument shows that regularity holds near points with isotropy group $\Gamma_2$.
\end{proof}

Proposition \ref{prop:open-ness} can be summarized as the \emph{openness} of FOP regularity. This ensures the compatibility of resolutions of singularities which will be used later. 

Because the regularity is a notion of equivariant transversality, it might be possible that it can never be achieved. However, in our situation, it exactly meets the sufficient condition to achieve equivariant transversality as observed by Wasserman \cite{wasserman} and discussed in detail in \cite[Section 5.4]{pardon20}.

\begin{lemma}\label{lem:local-FOP-regular}
Assume that we are in the situation of Definition \ref{defn:equivariant-regular}. Furthermore, assume that $M$ is compact (or has compact closure as a subspace in a topological space). Then a generic choice of FOP section is regular. 
\end{lemma}
\begin{proof}
Note that there is no difficulty to achieve the transversality of $\mathring{S}_{\Gamma'}$, because this is a non-equivariant situation. Therefore, after choosing $\mathring{S}_{\Gamma'}$ generically, we are reduced to the toy case when $\mathring{E}^{\Gamma'} = \{0\}$. In other words, we want to show for the graph of a generic $\Gamma'$-equivariant smooth map
\beq
\check{\mathfrak{s}}_{\Gamma'}: D_r(NM^{\Gamma'}) \to \mathrm{Poly}_d^{\Gamma'}(NM^{\Gamma'}, \check{E}^{\Gamma'})
\eeq
is transverse to
\beq
\tilde{Z}_d^{\Gamma'}(NM^{\Gamma'},\check{E}^{\Gamma'})^{\circ}_{\gamma''} \to NM^{\Gamma'} \times_{M^{\Gamma'}} \mathrm{Poly}_d^{\Gamma'}(NM^{\Gamma'},\check{E}^{\Gamma'}).
\eeq
Because we are taking the graph of $\check{\mathfrak{s}}_{\Gamma'}$, for any $y \in D_r(NM^{\Gamma'})$, the pushforward of the tangent space of $D_r(NM^{\Gamma'})$ at $y$ by the differential $d(\mathrm{graph}(\check{\mathfrak{s}}_{\Gamma'}))$ must be surjective onto $T(NM^{\Gamma'}) \subset T(NM^{\Gamma'} \times_{M^{\Gamma'}} \mathrm{Poly}_d^{\Gamma'}(NM^{\Gamma'},\check{E}^{\Gamma'}))$. Therefore, in order to guarantee that $\mathrm{graph}(\check{\mathfrak{s}}_{\Gamma'})$ meets transversely with $\tilde{Z}_d^{\Gamma'}(NM^{\Gamma'},\check{E}^{\Gamma'})^{\circ}_{\gamma''}$ at $(v, P) = (\mathring{v}, \check{v}, P)$, we just need to make sure that $d(\mathrm{graph}(\check{\mathfrak{s}}_{\Gamma'}))$ is surjective onto the complement of the image of $T\tilde{Z}_d^{\Gamma'}(NM^{\Gamma'},\check{E}^{\Gamma'})^{\circ}_{\gamma''}$ in $T_{P}\mathrm{Poly}_d^{\Gamma'}(N_{\mathring{v}}M^{\Gamma'},\check{E}_{\mathring{v}}^{\Gamma'})$. Note that the space in which $\check{\mathfrak{s}}_{\Gamma'}$ lives is the Fr\'echet space of smooth $\Gamma'$-equivariant bundle maps
\beq
C^{\infty}_{\Gamma'}(D_r(NM^{\Gamma'}), \mathrm{Poly}_d^{\Gamma'}(NM^{\Gamma'}, \check{E}^{\Gamma'})).
\eeq
In the special case when $M^{\Gamma'}$ is a point, if the fibers of $NM^{\Gamma'}$ and $\check{E}^{\Gamma'}$ correspond to complex $\Gamma'$-representations $V'$ and $W'$ respectively, the above space is
\beq
C^{\infty}_{\Gamma'}(V', \mathrm{Poly}_d^{\Gamma'}(V', W')).
\eeq
Because constant maps from $V'$ to $\mathrm{Poly}_d^{\Gamma'}(V', W')$ are $\Gamma'$-equivariant, an application of the Sard--Smale theorem (one can either use Taubes' trick or Floer's $C_{\epsilon}^{\infty}$ space to set up the Banach spaces) shows that a generic $\check{\mathfrak{s}}_{\Gamma'}$ satisfies the desired transversality condition. In general, using the compactness assumption of $M^{\Gamma'}$, a covering argument then implies the result.
\end{proof}

Note that the notion of regularity \emph{a priori} depends on the choice of cut-off degree $d$. The next statement shows that such a dependence can be removed.

\begin{lemma}\label{lem:cut-degree}
We working is the same setting as in Definition \ref{defn:equivariant-regular}: for any pair of integers $d \leq d'$, if $S$ is a regular FOP section of degree at most $d$, it is also regular by viewing the lift as a normally complex lift of degree at most $d'$.
\end{lemma}
\begin{proof}
By definition, the map in Equation \eqref{eqn:gamma-prime-lift} is transverse to the map in Equation \eqref{eqn:inclusion-Z}. Using Corollary \ref{prop:cartesian-degree} and the construction in Lemma \ref{lemma:Z-d-bundle}, we have a Cartesian diagram
\beq
\begin{tikzcd}
{\{0\} \oplus \tilde{Z}_d^{\Gamma'}(NM^{\Gamma'},\check{E}^{\Gamma'})^{\circ}_{\gamma''}} \arrow[d] \arrow[rr] &  & {\mathring{E}^{\Gamma'} \oplus (NM^{\Gamma'} \times_{M^{\Gamma'}} \mathrm{Poly}_d^{\Gamma'}(NM^{\Gamma'},\check{E}^{\Gamma'}))} \arrow[d] \\
{\{0\} \oplus \tilde{Z}_{d'}^{\Gamma'}(NM^{\Gamma'},\check{E}^{\Gamma'})^{\circ}_{\gamma''}} \arrow[rr]      &  & {\mathring{E}^{\Gamma'} \oplus (NM^{\Gamma'} \times_{M^{\Gamma'}} \mathrm{Poly}_{d'}^{\Gamma'}(NM^{\Gamma'},\check{E}^{\Gamma'})).}     
\end{tikzcd}
\eeq
Just as in the case of Proposition \ref{prop:open-ness}, the fact that the Cartesian nature the above diagram and Lemma \ref{lemma:transverse-Cartesian} implies that \eqref{eqn:gamma-prime-lift} is transverse to the $d'$-version of \eqref{eqn:inclusion-Z}.
\end{proof}

After the discussions in the equivariant setting, we can go back to orbifolds.

\begin{defn}\label{defn:orbi-regular}
Let ${\mc E} \to {\mc U}$ be a vector bundle over an \emph{effective} orbifold, which are both endowed with normal complex structures and straightenings as in Definition \ref{defn:FOP-section}. Then an FOP section ${\mc S}: {\mc U} \to {\mc E}$ with a normally complex lift of degree at most $d$ is called \emph{regular} if its pullback as a section of $E_{\alpha} \to U_{\alpha}$ is regular in the sense of Definition \ref{defn:equivariant-regular}. 
\end{defn}

Note that to make sure that varieties of the form $\tilde{Z}_d^{\Gamma}(V,W)^{\circ}$ have the correct dimension, we need to make $d$ sufficiently large. In our applications, such a uniform choice can be made after introducing certain compactness assumptions.

\begin{prop}\label{prop:FOP-existence}
Let ${\mc E} \to {\mc U}$ be a vector bundle over an \emph{effective} orbifold, endowed with normal complex structures and straightenings as in Definition \ref{defn:FOP-section}. Suppose that ${\mc S}: {\mc U} \to {\mc E}$ is a smooth section such that the zero locus ${\mc S}^{-1}(0)$ is contained in an open suborbifold ${\mc U}' \subset {\mc U}$ with the property that $|{\mc U}'|$ is precompact. Then there exists an integer $d_0 \geq 1$ such that for any $d \geq d_0$, and any $\epsilon > 0$, there exists a smooth section ${\mc S}_{\epsilon}: {\mc U} \to {\mc E}$ satisfying the following properties:
\begin{enumerate}
    \item in an open neighborhood of ${\mc U}'$ whose image in $|{\mc U}|$ contains $|\ov{{\mc U}'}|$, the section ${\mc S}_{\epsilon}$ is a regular FOP section of degree at most $d$;
    \item $\|{\mc S}_{\epsilon} - {\mc S}\|_{C^0} < \epsilon$;
    \item for $\epsilon$ sufficiently small, the set $|{\mc S}_{\epsilon}^{-1}(0)|$ is compact.
\end{enumerate}
\end{prop}
\begin{proof}
The proof is exactly the same as the proof of \cite[Proposition 6.4]{Bai_Xu_2022}, with the following changes: (i) the existence over an \'etale chart as addressed by \cite[Lemma 6.5]{Bai_Xu_2022} is replaced by Lemma \ref{lem:local-FOP-regular}, and (ii) to apply the two-layer induction proof of \cite[Proposition 6.4]{Bai_Xu_2022}, the openness of strong transversality is replaced by the openness of regularity as proven in Proposition \ref{prop:open-ness}.
\end{proof}

For the next topic of this subsection, we discuss the behavior of such sections under stabilization. Namely, assume that we are in the situation of Definition \ref{defn:orbi-regular}. Let $\pi_{{\mc E}'} :{\mc E}' \to {\mc U}$ be another normally complex vector bundle and write $\tau_{{\mc E}'}: {\mc E}' \to \pi_{{\mc E}'}^* {\mc E}'$ for the tautological section. Assume that ${\mc E}'$ is endowed with a straightening, so we have a notion of FOP sections for $\pi_{{\mc E}'}^* {\mc E} \oplus \pi_{{\mc E}'}^* {\mc E}' \to {\mc E}'$.

\begin{lemma}\label{lem:stable-regular}
    If ${\mc S}: {\mc U} \to {\mc E}$ is regular, then so is $\pi_{{\mc E}'}^* {\mc S} \oplus \tau_{{\mc E}'}: {\mc E}' \to \pi_{{\mc E}'}^* {\mc E} \oplus \pi_{{\mc E}'}^* {\mc E}'$.
\end{lemma}
\begin{proof}
We just need to check the statement over \'etale charts, so let us assume that we are in the situation of Definition \ref{defn:equivariant-regular} and $\pi_{E'}: E' \to M$ is a $\Gamma$-equivariant vector bundle with a normal complex structure, and we have the tautological section $\tau_{E'}: E' \to \pi_{E'}^* E'$. Observe that for any $\Gamma' \leq \Gamma$, we have an inclusion $(E')^{\Gamma'} \subset \pi_{E'}^{-1}(M^{\Gamma'})$. As a result, the $\Gamma$-equivariant section $\pi_{E'}^* S \oplus \pi_{E'}^* \tau_{E'}$ admits a normally complex lift of degree at most $d$. 

For $\Gamma' \leq \Gamma$, the normal bundle $N(E')^{\Gamma'}$ admits a projection map to the normal bundle $NM^{\Gamma'}$. It induces a splitting
\beq
N(E')^{\Gamma'} = N(E')^{\Gamma'}_{\parallel} \oplus N(E')^{\Gamma'}_{\perp},
\eeq
where $N(E')^{\Gamma'}_{\perp}$ is the kernel of the projection map and $N(E')^{\Gamma'}_{\parallel}$ is the orthogonal complement under the induced metric. This is indeed a splitting of vector bundles when restricted to connected components. On the other hand, over $N(E')^{\Gamma'}$, we have splittings
\beq
\pi_{E'}^* E = \pi_{E'}^* \mathring{E}^{\Gamma'} \oplus \pi_{E'}^* \check{E}^{\Gamma'}, \textrm{and } \pi_{E'}^* E = \pi_{E'}^* \mathring{E'}^{\Gamma'} \oplus \pi_{E'}^* \check{E'}^{\Gamma'}.
\eeq
Then the section
\beq
\begin{aligned}
S^{\tau} := \pi_{E'}^* S \oplus \pi_{E'}^* \tau_{E'} &= (\mathring{S}^{\tau}_{\Gamma'}, \check{S}^{\tau}_{\Gamma'}): \\
D_r (N(E')^{\Gamma'}) &\to (\pi_{E'}^* \mathring{E}^{\Gamma'} \oplus \pi_{E'}^* \mathring{E'}^{\Gamma'}) \oplus (\pi_{E'}^* \check{E}^{\Gamma'} \oplus \pi_{E'}^* \check{E'}^{\Gamma'})
\end{aligned}
\eeq
has normally complex lift
\beq\label{eqn:stablized-lift}
\begin{aligned}
\check{\mathfrak{s}}_{\Gamma'}^{\tau}: D_{r/2} (N(E')^{\Gamma'}) &\subset D_r (N(E')^{\Gamma'}_{\parallel}) \oplus D_r (N(E')^{\Gamma'}_{\perp}) \\
&\to \mathrm{Poly}_d^{\Gamma'}(N(E')^{\Gamma'}_{\parallel}, \pi_{E'}^* \check{E}^{\Gamma'} ) \oplus \mathrm{Poly}_1^{\Gamma'}(N(E')^{\Gamma'}_{\perp}, \pi_{E'}^* \check{E'}^{\Gamma'} ) \\
&\subset \mathrm{Poly}_d^{\Gamma'}(N(E')^{\Gamma'}, \pi_{E'}^* \check{E}^{\Gamma'} ) \oplus \mathrm{Poly}_d^{\Gamma'}(N(E')^{\Gamma'}, \pi_{E'}^* \check{E'}^{\Gamma'} )
\end{aligned}
\eeq
written as the diagonal matrix $\mathrm{diag}(\pi_{E'}^* \check{\mathfrak{s}}_{\Gamma'}, \pi_{E'}^* \tau_{E'})$, in which we have identified $N(E')^{\Gamma'}_{\perp}$ with $\pi_{E'}^* \check{E'}^{\Gamma'}$. Then we see that $\check{\mathfrak{s}}_{\Gamma'}^{\tau}$ is obtained from $\check{\mathfrak{s}}_{\Gamma'}$ via a fiberwise/family version of the stabilization map in Equation \eqref{eqn:poly-stabilization}. Then the regularity of $S^{\tau}$ simply follows from the Cartesian property of Diagram \eqref{eqn:stabilization-Cartesian-gamma} and the fact that the tautogical section is transverse to $0$ along the $\Gamma'$-trivial part $\pi_{E'}^* \mathring{E'}^{\Gamma'}$.
\end{proof}

The final statements of this subsection are the following. 
\begin{lemma}\label{lem:regular-product}
    Let ${\mc S}_1: {\mc U}_1 \to {\mc E}_1$ be an FOP section which is regular. Suppose ${\mc U}_2$ is another effective normally complex orbifold. Then under the projection map
    \beqn
    \pi_1: {\mc U}_1 \times {\mc U_2} \to {\mc U}_1,
    \eeqn
    the pulled-back section $\pi_1^* {\mc S}_1$ is also FOP regular.
\end{lemma}
\begin{proof}
    Again, we can check the assertion in the equivariant setting for $\Gamma_1$-equivariant triples $(M_1, E_1, S_1)$ and the $\Gamma_2$-manifold $M_2$. Then the relevant Cartesian commutative diagram is
    \beq
    \begin{tikzcd}
{\{0\} \oplus (NM_2^{\Gamma_2'} \times \tilde{Z}_{d}^{\Gamma_1'}(NM_1^{\Gamma_1'} ,\check{E}^{\Gamma_1'})^{\circ})_{\gamma_1''}} \arrow[d] \arrow[rr] &  & {\begin{aligned} \pi_2^* \mathring{E}^{\Gamma_1'} |_{M_1^{\Gamma_1'} \times M_2^{\Gamma_2'}} \oplus ((NM_1^{\Gamma_1'} \times NM_2^{\Gamma_2'}) \times_{M_1^{\Gamma_1'} \times M_2^{\Gamma_2'}} \\ \mathrm{Poly}_d^{\Gamma_1'}(NM_1^{\Gamma_1'}, \pi_2^*\check{E}^{\Gamma_1'}|_{M_1^{\Gamma_1'} \times M_2^{\Gamma_2'}})) \end{aligned}} \arrow[d] \\
{\{0\} \oplus \tilde{Z}_d^{\Gamma_1' \times \Gamma_2'}(NM_1^{\Gamma_1'} \times NM_2^{\Gamma_2'},\pi_2^*\check{E}^{\Gamma_1'}|_{M_1^{\Gamma_1'} \times M_2^{\Gamma_2'}})^{\circ}_{\gamma_1'' \times \{1\} }} \arrow[rr]      &  & {\begin{aligned}\pi_2^* \mathring{E}^{\Gamma_1'}|_{M_1^{\Gamma_1'} \times M_2^{\Gamma_2'}} \oplus ((NM_1^{\Gamma_1'} \times NM_2^{\Gamma_2'}) \times_{M_1^{\Gamma_1'} \times M_2^{\Gamma_2'}} \\ \mathrm{Poly}_d^{\Gamma_1' \times \Gamma_2'}(NM_1^{\Gamma_1'} \times NM_2^{\Gamma_2'},\pi_2^*\check{E}^{\Gamma_1'}|_{M_1^{\Gamma_1'} \times M_2^{\Gamma_2'}})) \end{aligned}}     
\end{tikzcd}
\eeq
over the $\Gamma_1' \times \Gamma_2' \leq \Gamma_1 \times \Gamma_2$-invariant locus of $M_1 \times M_2$ under the product action, which is the fiberwise version of the diagram \eqref{eqn:product-Cartesian}. The right vertical map is constructed using the bundle vesion of \eqref{eqn:pair-product}. Then as in the case of Proposition \ref{prop:open-ness}, the Cartesian nature the above diagram and Lemma \ref{lemma:transverse-Cartesian} imply the desired claim.
\end{proof}

\begin{cor}\label{cor:product-regular}
Let ${\mc S}_1: {\mc U}_1 \to {\mc E}_1$ be an FOP section which is regular. Suppose ${\mc U}_2$ is a (not necessarily effective) normally complex orbifold. Let ${\mc F}_2 \xrightarrow{\pi_{{\mc F}_2}} {\mc U}_2$ be a complex vector bundle such that the total space ${\mc F}_2$ is an effective orbifold. Then the section
\beqn
\pi_1^* {\mc S}_1 \times \tau_{{\mc F}_2}: {\mc U}_1 \times {\mc F}_2 \to \pi_1^* {\mc E}_1 \oplus \pi_{{\mc F}_2}^* {\mc F}_2,
\eeqn
where $\pi_1 : {\mc U}_1 \times {\mc F}_2 \to {\mc U}_1$ denotes the projection, is FOP regular.
\end{cor}
\begin{proof}
    This follows from the Cartesian diagram in Lemma \ref{lem:regular-product} and Lemma \ref{lem:stable-regular} because the above section is a stabilization of the section $\pi_1^* {\mc S}$ (here we abuse the notation to use $\pi_1$ to denote the projection ${\mc U}_1 \times {\mc U_2} \to {\mc U}_1$) by the bundle ${\mc F}_2$. 
\end{proof}

\begin{rem}
    Following Remark \ref{rem:faithful}, we include the above corollary since we formulated FOP regularity only for effective orbifolds.
\end{rem}

\subsection{Resolution of zero loci: isotropy-free part}\label{subsec:resolve-free}
Now we describe how to construct a resolution of the closure of the isotropy-free part of the zero locus of a regular FOP section defined over an effective normally complex orbifold via pulling back the resolution produced from Theorem \ref{thm:ATW-resolution}. We explain the construction in the equivariant setting first, which we then globalize to the orbifold setting.

Let $E \to M$ be as in Definition \ref{defn:equivariant-regular} and suppose that $S: M \to E$ is a regular $\Gamma$-equivariant smooth section which admits a normally complex lift of degree at most $d$. By definition, $S$ comes with lifts defined over the disc bundles over the fixed point loci. If we range over all proper subgroups $\Gamma' < \Gamma$, we see that the collection of open subsets
\beq\label{eqn:disc-bundles}
\{ U(M, \Gamma', r):= D_r (NM^{\Gamma'}) \setminus \big( \bigcup_{\ \Gamma' < \Gamma''} D_{\leq r/2} (NM^{\Gamma''}) \big) \}_{\{1\} \neq \Gamma' \leq \Gamma}
\eeq
together with the open subset
forms an open cover of $M$. Then we can construct a new open cover as follows: observe that the set of subgroups of $\Gamma$ forms a poset under inclusion. We can form a new poset $2^{\Gamma}$ from it, whose elements are given by a chain of subgroups 
\beq
\Gamma_{\vec{k}} := (\Gamma_1 < \cdots < \Gamma_k), \text{ where } \Gamma_i \leq \Gamma \text{ for } 1 \leq i \leq k,
\eeq
and the partial order is induced by the inverse inclusion of chains. Define
\beq
U(M, \Gamma_{\vec{k}}, r) := \bigcap_{1 \leq i \leq k} U(M, \Gamma_i, r),
\eeq
we see that $U(M, \Gamma_{\vec{k}}, r) \subset U(M, \Gamma_{\vec{k}'}, r)$ if and only if $\Gamma_{\vec{k}} \leq \Gamma_{\vec{k}'}$. Moreover, if $\Gamma_k$ and $\Gamma_{k'}$ are respectively the maximal elements in $\Gamma_{\vec{k}}$ and $\Gamma_{\vec{k}'}$, the inclusion $U(M, \Gamma_{\vec{k}}, r) \subset U(M, \Gamma_{\vec{k}'}, r)$ is equivariant under the group embedding $\Gamma_{k'} \to \Gamma_k$. Therefore, 
\beq
\{ U(M, \Gamma_{\vec{k}}, r) \}_{\Gamma_{\vec{k}} \in 2^{\Gamma}}
\eeq
defines an open cover of $M$ and for any two members $U(M, \Gamma_{\vec{k}}, r)$ and we can present $M$ as 
\beq\label{eqn:colim-M}
\colim_{\Gamma_{\vec{k}} \in 2^{\Gamma}} U(M, \Gamma_{\vec{k}}, r) 
\eeq
in which the arrows are inclusion of open subsets, which are compatible with the partial order on $2^{\Gamma}$. For $\Gamma_{\vec{k}} = (\Gamma_1 < \cdots < \Gamma_k)$, because any $U(M, \Gamma', r)$ is invariant under the $\Gamma'$-action, we see that $U(M, \Gamma_{\vec{k}}, r)$ is invariant under the $\Gamma_1$-action. By restricting the $\Gamma_1$-version of the equivariant map \eqref{eqn:gamma-prime-lift} to $U(M, \Gamma_{\vec{k}}, r) \subset U(M, \Gamma_1, r)$, we obtain a smooth map
\beq\label{eqn:gamma-prime-lift-2}
\begin{aligned}
U(M, \Gamma_{\vec{k}}, r) &\to \mathring{E}^{\Gamma_1} \oplus (NM^{\Gamma_1} \times_{M^{\Gamma_1}} \mathrm{Poly}_d^{\Gamma_1}(NM^{\Gamma_1},\check{E}^{\Gamma_1})) \\
y &\mapsto (\mathring{S}_{\Gamma_1}(y), (y, \check{\mathfrak{s}}_{\Gamma_1}(y))),
\end{aligned}
\eeq
we define $\tilde{S}^{-1}(0)_{(M, \Gamma_{\vec{k}}, r)}$ to be the fiber product
\beq\label{eqn:fiber-prod-free}
\begin{tikzcd}
{\tilde{S}^{-1}(0)_{(M, \Gamma_{\vec{k}}, r)}} \arrow[rr] \arrow[d]                  &  & {\{0\} \oplus \tilde{Z}_d^{\Gamma_1}(NM^{\Gamma_1},\check{E}^{\Gamma_1})^{\circ}} \arrow[d, "\eqref{eqn:inclusion-Z}"]                      \\
{U(M, \Gamma_{\vec{k}}, r)} \arrow[rr, "\eqref{eqn:gamma-prime-lift-2}"'] &  & {\mathring{E}^{\Gamma_1} \oplus (NM^{\Gamma_1} \times_{M^{\Gamma_1}} \mathrm{Poly}_d^{\Gamma_1}(NM^{\Gamma_1},\check{E}^{\Gamma_1}))}.
\end{tikzcd}
\eeq
Because $S$ has been assumed to be regular, $\tilde{U}(M, \Gamma_{\vec{k}}, r)$ is a smooth $\Gamma_1$-manifold with an induced normal complex structure from the ones on $M$ and $E$.

\begin{lemma}\label{lem:local-S-resolve}
The equivariant inclusion map $U(M, \Gamma_{\vec{k}}, r) \subset U(M, \Gamma_{\vec{k}'}, r)$ for $\Gamma_{\vec{k}} \leq \Gamma_{\vec{k}'}$ induces an equivariant open embedding $\tilde{S}^{-1}(0)_{(M, \Gamma_{\vec{k}}, r)} \to \tilde{S}^{-1}(0)_{(M, \Gamma_{\vec{k}'}, r)}$ which respects the normal complex structures.
\end{lemma}
\begin{proof}
This is a straightforward consequence of the Cartesian nature of the diagram \eqref{eqn:change-group-res-compatible} after decomposing the local lifts of $S$ in the form of \eqref{eqn:gamma-2-lift}.
\end{proof}

Using the colimit presentation, we see that
\beq
S^{-1}(0) = \colim_{\Gamma_{\vec{k}} \in 2^{\Gamma}} \big( S^{-1}(0) \cap U(M, \Gamma_{\vec{k}}, r) \big).
\eeq
Based on Lemma \ref{lem:local-S-resolve}, we can define a resolution of the closure of the isotropy-free locus of $S^{-1}(0)$ to be
\beq\label{eqn:resolution-equivarint}
\tilde{S}^{-1}(0):= \colim_{\Gamma_{\vec{k}} \in 2^{\Gamma}} \tilde{S}^{-1}(0)_{(M, \Gamma_{\vec{k}}, r)}.
\eeq

\begin{prop}\label{prop:equivariant-resolve}
    The smooth manifold $\tilde{S}^{-1}(0)$ is a $\Gamma$-manifold with a normal complex structure which admits a $\Gamma$-equivariant map to $S^{-1}(0)$. This map defines a diffeomorphism over the isotropy-free part of $S^{-1}(0)$. Furthermore, $\tilde{S}^{-1}(0)$ is independent of the choice of sufficiently small $r > 0$. 
    
    Finally, for any open embedding $M' \to M$ which is equivariant under an injective group homomorphism $\Gamma' \to \Gamma$ such that $E' = E|_{M'}$ and $S' = S|_{M'}$ with compatible straightenings and normal complex structures, there is an equivariant open embedding $(\tilde{S}')^{-1}(0) \to \tilde{S}^{-1}(0)$ covering $(S')^{-1}(0) \to S^{-1}(0)$ which respects the normal complex structures.
\end{prop}
\begin{proof}
Note that for any $\gamma \in \Gamma$, if we define 
\beq
\gamma \Gamma_{\vec{k}} \gamma^{-1} := (\gamma \Gamma_1 \gamma^{-1} < \cdots < \gamma \Gamma_k \gamma^{-1}),
\eeq
we see that $U(M, \Gamma_{\vec{k}}, r) $ is mapped isomorphically onto $U(M, \gamma \Gamma_{\vec{k}} \gamma^{-1}, r)$ by the $\gamma$-translation. Its compatibility with the arrows in \eqref{eqn:colim-M} recovers the $\Gamma$-action on $M$. Such an isomorphism can also be lifted to an isomorphism between the vector bundles over them, and it intertwines with the normal complex structures and straightenings. As a result, the corresponding (fiberwise) universal zero loci are identified, and so are their resolutions. Then we see that the colimit presentation of $\tilde{S}^{-1}(0)$ inherits a $\Gamma$-action, which defines a smooth $\Gamma$-action on $\tilde{S}^{-1}(0)$. The existence of normal complex structure can be checked locally. Indeed, in the diagram \eqref{eqn:fiber-prod-free}, the vertical arrow \eqref{eqn:inclusion-Z} comes from a complex algebraic map, therefore respects the complex, in particular, normal complex structure. For the horizontal arrow \eqref{eqn:gamma-prime-lift-2}, it suffices to check that its linearization respects the complex structures along the $\Gamma_1$-fixed point locus. This follows from the fact that the fibers of $NM^{\Gamma_1}$ does not contain any $\Gamma_1$-trivial summand because this implies that the differential of $\check{\mathfrak{s}}_{\Gamma_1}$ is necessarily trivial along the zero section $M^{\Gamma_1}$: we are only left with the identity map along the fibers of $NM^{\Gamma_1}$, which is evidently complex linear. By definition, $\tilde{S}^{-1}(0)$ maps diffeomorphically onto $S^{-1}(0)$ over the isotropy free locus, over which the resolution does nothing. The assertion that $\tilde{S}^{-1}(0)$ is independent of $r$ follows by inspection. The functoriality of the resolution under open embeddings is a straightforward consequence of the construction.
\end{proof}

With these preparations in the equivariant setting, we can discuss the situation for orbifolds. Let ${\mc U}, {\mc E}, {\mc S}$ be as in Definition \ref{defn:orbi-regular}. Denote by $S_{\alpha}: U_{\alpha} \to E_{\alpha}$ be the pullback of ${\mc S}$ to the \'etale chart $U_{\alpha} \to {\mc U}$. Writing $U = \coprod_{\alpha} U_{\alpha}$, from the discussions in Section \ref{sec:prelim}, we can present ${\mc U}$ using the topological groupoid $[U \times_{\mc U} U \rightrightarrows U]$. Similarly, if we write $E = \coprod_{\alpha} E_{\alpha}$, there is the associated groupoid $[E \times_{\mc E} E \rightrightarrows E]$ of ${\mc E}$. Then we can consider
\beq\label{eqn:orbi-resolve}
\coprod_{\alpha} \tilde{S}_{\alpha}^{-1}(0)
\eeq
the disjoint union of the equivariant resolutions of $S_{\alpha}^{-1}(0)$. By the definition of straightenings and normal complex structures, the functoriality part of Proposition \ref{prop:equivariant-resolve} implies that the topological groupoid $[U \times_{\mc U} U \rightrightarrows U]$ restricts to $\tilde{S}^{-1}(0) \subset U$ and it produces a normally complex orbifold, which we denote by $\tilde{\mc S}^{-1}(0)$. The following is the counterpart of Proposition \ref{prop:equivariant-resolve} in the orbifold setting. We state it for completeness and its proof should be evident.

\begin{prop}\label{prop:construct-orbifold}
The smooth normally complex orbifold $\tilde{\mc S}^{-1}(0)$  depends only on the straightenings, normal complex structures on ${\mc U}$ and ${\mc E}$, and the FOP section ${\mc S}$. It admits a map onto ${\mc S}^{-1}(0)$, which is an isomorphism of orbifolds over the isotropy-free part of ${\mc S}^{-1}(0)$. 

For any open embedding ${\mc U}' \to {\mc U}$ such that ${\mc E}' = {\mc E}|_{{\mc U}'}$ and ${\mc S}' = {\mc S}|_{{\mc U}'}$ with compatible straightenings and normal complex structures, there is an equivariant open embedding $(\tilde{\mc S}')^{-1}(0) \to \tilde{\mc S}^{-1}(0)$ covering $({\mc S}')^{-1}(0) \to {\mc S}^{-1}(0)$ which respects the normal complex structures. \qed
\end{prop}

Next, we need a statement concerning the resolution of the zero locus of a stabilization. 

\begin{prop}\label{prop:stabilization}
Suppose that we are in the situation of Lemma \ref{lem:stable-regular}. Then the resolution of the isotropy-free part of $({\mc S}^{\tau})^{-1}(0) := (\pi_{{\mc E}'}^* {\mc S} \oplus \tau_{{\mc E}'})^{-1}(0)$ agrees with $\tilde{\mc S}^{-1}(0)$.
\end{prop}
\begin{proof}
Lemma \ref{lem:stable-regular} proves that ${\mc S}^{\tau}$ is regular. The zero loci $({\mc S}^{\tau})^{-1}(0)$ and ${\mc S}^{-1}(0)$ agree with each other. Then the compatibility of resolutions as exhibited in Proposition \ref{prop:stabilization-compatible} implies that the resolutions also coincide, by following the information provided by the lift \eqref{eqn:stablized-lift}.
\end{proof}

Finally, we formulate a proposition about products.
\begin{prop}\label{prop:product}
    Suppose that we are in the situation of Lemma \ref{lem:regular-product}. Then the resolution $\tilde{(\pi_2^* {\mc S})}^{-1}(0)$ agrees with $\tilde{{\mc S}}^{-1}(0) \times {\mc U}_2$ if ${\mc U}_2$ is effective. 
\end{prop}
\begin{proof}
    By Lemma \ref{lem:regular-product}, we see that the pullback $\pi_2^* {\mc S}$ is regular. The Cartesian nature of the diagram in the proof implies the compatibility of resolutions. When ${\mc U}_2$ is effective, the isotropy-free part of $(\pi_2^* {\mc S})^{-1}(0)$ agrees with the isotropy-free part of ${\mc S}^{-1}(0)$ with the isotropy-free locus of ${\mc U}_2$ so we get the whole ${\mc U}_2$-factor.
\end{proof}

\subsection{Resolution with prescribed isotropy}
Following the discussion in Section \ref{sec:zero-loci-with}, we explain how to extend the construction in the previous subsection to obtain a resolution of the (closure of) the zero locus of a regular FOP section with a prescribed isomorphism class of stabilizer groups. We will not use the results in this subsection in the final discussion, but we include it here to get the reader prepared with the construction in the next subsection which may seem confusing for the first reading.

Following the equivariant set up in Section \ref{subsec:resolve-free}, the additional choice to make here is a finite group $\tilde{\Gamma}$ which is isomorphic to a subgroup of $\Gamma$. Then we can consider the open subsets $U(M, \Gamma', r)$ from Equation \eqref{eqn:disc-bundles} with the requirement that $\tilde{\Gamma}$ is isomorphic to a subgroup of $\Gamma'$. Introducing a new notation to imply the dependence on $\tilde{\Gamma}$, we will denote such an open subset by
\beq
U(M, \Gamma', r)_{\tilde{\Gamma}}.
\eeq
The union of all such $U(M, \Gamma', r)_{\tilde{\Gamma}}$ defines an open subset of $M$ which contains the fixed point locus under any subgroup of $\Gamma$ which is isomorphic to $\tilde{\Gamma}$. We can similarly define 
\beq
U(M, \Gamma_{\vec{k}}, r)_{\tilde{\Gamma}} := \bigcap_{1 \leq i \leq k} U(M, \Gamma_i, r)_{\tilde{\Gamma}}
\eeq
for chains of subgroups $\Gamma_{\vec{k}} = (\Gamma_1 < \cdots < \Gamma_k)$ such that $\Gamma_1$ admits a subgroup isomorphic to $\tilde{\Gamma}$. Using the map \eqref{eqn:gamma-prime-lift-2}, instead of taking the fiber product as in Diagram \eqref{eqn:fiber-prod-free}, we can define $\tilde{S}^{-1}(0)_{(M, \Gamma_{\vec{k}}, r), \tilde{\gamma} }$ to be the fiber product
\beq\label{eqn:fiber-prod-gamma}
\begin{tikzcd}
{\tilde{S}^{-1}(0)_{(M, \Gamma_{\vec{k}}, r), \tilde{\gamma}}} \arrow[rr] \arrow[d]                  &  & {\{0\} \oplus \tilde{Z}_d^{\Gamma_1}(NM^{\Gamma_1},\check{E}^{\Gamma_1})^{\circ}_{\gamma'}} \arrow[d, "\eqref{eqn:inclusion-Z}"]                      \\
{U(M, \Gamma_{\vec{k}}, r)} \arrow[rr, "\eqref{eqn:gamma-prime-lift-2}"'] &  & {\mathring{E}^{\Gamma_1} \oplus (NM^{\Gamma_1} \times_{M^{\Gamma_1}} \mathrm{Poly}_d^{\Gamma_1}(NM^{\Gamma_1},\check{E}^{\Gamma_1}))},
\end{tikzcd}
\eeq
where $\gamma'$ stands for the conjugacy class of subgroups of $\Gamma_1$ conjugate to the chosen subgroup isomorphic to $\tilde{\Gamma}$, and $\tilde{Z}_d^{\Gamma_1}(NM^{\Gamma_1},\check{E}^{\Gamma_1})^{\circ}_{\gamma'}$ is the bundle from Equation \eqref{eqn:Z-bundle-gamma}. Similar to Lemma \ref{lem:local-S-resolve}, the compatibility of the resolutions (now we view $\gamma'$ representing a conjugacy class of subgroups in this chain of finite groups) ensures that the colimit
\beq\label{eqn:resolution-equivarint-gamma}
\tilde{S}^{-1}(0)_{\tilde{\gamma}}:= \colim_{\Gamma_{\vec{k}}} \tilde{S}^{-1}(0)_{(M, \Gamma_{\vec{k}}, r), \tilde{\gamma} }
\eeq
is well-defined and it admits a smooth map onto the closure of the subset of $S^{-1}(0)$ whose isotropy group under the $\Gamma$-action is isomorphic to $\tilde{\Gamma}$. Similar to Proposition \ref{prop:equivariant-resolve}, we conclude the following.

\begin{prop}\label{prop:equivariant-resolve-gamma}
    The smooth manifold $\tilde{S}^{-1}(0)_{\tilde{\gamma}}$ is a $\Gamma$-manifold with a normal complex structure which admits a $\Gamma$-equivariant map to $S^{-1}(0)$. This map defines a diffeomorphism over the part of $S^{-1}(0)$ whose isotropy group under the $\Gamma$-action is isomorphic to $\tilde{\Gamma}$. Furthermore, $\tilde{S}^{-1}(0)_{\tilde{\gamma}}$ is independent of the choice of sufficiently small $r > 0$. 
    
    For any open embedding $M' \to M$ which is equivariant under an injective group homomorphism $\Gamma' \to \Gamma$ such that $E' = E|_{M'}$ and $S' = S|_{M'}$ with compatible straightenings and normal complex structures, there is an equivariant open embedding $(\tilde{S}')^{-1}(0)_{\tilde{\gamma}} \to \tilde{S}^{-1}(0)_{\tilde{\gamma}}$ which respects the normal complex structures, provided that $\tilde{\Gamma}$ is isomorphic to a subgroup of $\Gamma'$. \qed
\end{prop}

As before, we can globalize the discussion to the orbifold setting: given the finite group $\tilde{\Gamma}$, we can consider the analog of \eqref{eqn:orbi-resolve}
\beq\label{eqn:orbi-resolve-gamma}
\coprod_{\alpha} \tilde{S}_{\alpha}^{-1}(0)_{\tilde{\gamma}}
\eeq
which can be patched together to define an orbifold, denoted by $\tilde{\mc S}^{-1}(0)_{\tilde{\gamma}}$. Then the following statement holds for the same reasons as before.

\begin{prop}\label{prop:resolve-gamma}
    The smooth normally complex orbifold $\tilde{\mc S}^{-1}(0)_{\tilde{\gamma}}$  depends only on the straightenings, normal complex structures on ${\mc U}$ and ${\mc E}$, and the regular FOP section ${\mc S}$.

    For any open embedding ${\mc U}' \to {\mc U}$ such that ${\mc E}' = {\mc E}|_{{\mc U}'}$ and ${\mc S}' = {\mc S}|_{{\mc U}'}$ with compatible straightenings and normal complex structures, there is an equivariant open embedding $(\tilde{\mc S}')^{-1}(0)_{\tilde{\gamma}} \to \tilde{\mc S}^{-1}(0)_{\tilde{\gamma}}$ which respects the normal complex structures. 

    If we are in the situation of Lemma \ref{lem:stable-regular}, then the resolution $({\mc S}^{\tau})^{-1}(0)_{\tilde{\gamma}} := (\pi_{{\mc E}'}^* {\mc S} \oplus \tau_{{\mc E}'})^{-1}(0)_{\tilde{\gamma}}$ agrees with $\tilde{\mc S}^{-1}(0)_{\tilde{\gamma}}$. \qed
\end{prop}

\subsection{Resolution of prescribed isotropy type}
In the previous subsection, we discussed how to resolve the zero loci of FOP regular sections with prescribed stabilizers. We need a slight refinement, which allows us to further decompose such resolutions into pieces indexed by certain representation-theoretic data. This will be used to ensure the module map property of the splitting map to be described in the next section.

\begin{defn}
    An \emph{isotropy type} is a triple
    \beq
    \tilde{\boldsymbol{\gamma}} = (\tilde{\Gamma}, \tilde{V}, \tilde{W})
    \eeq
    where $\tilde{\Gamma}$ is a finite group, while $\tilde{V}$ and $\tilde{W}$ are 
    finite-dimensional complex $\tilde{\Gamma}$-representations which do not contain any trivial subrepresentation (of non-trivial rank).  
    An isomorphism of isotropy types means an isomorphism of groups which intertwine the corresponding representations.
\end{defn}

Given a finite group $\Gamma$, we can consider the set of isotropy types where the group $\tilde{\Gamma}$ is a subgroup of $\Gamma$. We say this is the set of isotropy types \emph{subordinate to} $\Gamma$. This set has three features.
\begin{enumerate}
    \item It admits a conjugation action by $\Gamma$: for any $\gamma \in \Gamma$ and an isotropy type $(\tilde{\Gamma}, \tilde{V}, \tilde{W})$ with $\tilde{\Gamma} \leq \Gamma$, we can form a new isotropy type with group given by $\gamma \tilde{\Gamma} \gamma^{-1}$ and the representations are induced by the conjugation action. Therefore, in this setting, it makes sense to talk about \emph{conjugacy classes} of isotropy types.
    \item This set has a natural partial order. Namely, we declare that 
    \beq
    (\tilde{\Gamma}_2, \tilde{V}_2, \tilde{W}_2) \leq (\tilde{\Gamma}_1, \tilde{V}_1, \tilde{W}_1)
    \eeq
    if $\tilde{\Gamma}_2 \leq \tilde{\Gamma}_1$, and if we restrict the representations $\tilde{V}_1$ and $\tilde{W}_1$ along the subgroup $\tilde{\Gamma}_2$ and write the decomposition we had before as
    \beq
    \tilde{V}_1 = \mathring{\tilde{V}}_1^{\tilde{\Gamma}_2} \oplus \check{\tilde{V}}_1^{\tilde{\Gamma}_2}, \quad \quad \quad \tilde{W}_1 = \mathring{\tilde{W}}_1^{\tilde{\Gamma}_2} \oplus \check{\tilde{W}}_1^{\tilde{\Gamma}_2},
    \eeq
    then we have $\tilde{V}_2 \cong \check{\tilde{V}}_1^{\tilde{\Gamma}_2}$ and $\tilde{W}_2 \cong \check{\tilde{W}}_1^{\tilde{\Gamma}_2}$.
    \item There is a \emph{stabilization} operation. Namely, given $(\tilde{\Gamma}, \tilde{V}, \tilde{W})$, if we have another finite-dimensional complex $\tilde{\Gamma}$-representation $\tilde{V}'$ without trivial summands, 
    \beq
    (\tilde{\Gamma}, \tilde{V} \oplus \tilde{V}', \tilde{W} \oplus \tilde{V}')
    \eeq
    is another well-defined isotropy type. 
\end{enumerate}

In the more abstract setting, where the group $\tilde{\Gamma}$ does not necessarily arise from a subgroup of a given group, we can consider the collection of all isomorphism classes of isotropy types. This set inherits the last two features listed above.
\begin{enumerate}
    \item We have that $\tilde{\boldsymbol{\gamma}}_2 \leq \tilde{\boldsymbol{\gamma}}_1 $, for isomorphism classes $\tilde{\boldsymbol{\gamma}}_2$ and $\tilde{\boldsymbol{\gamma}}_1$, if and only if the group in $\tilde{\boldsymbol{\gamma}}_2$ admits a embedding as a subgroup of the group in $\tilde{\boldsymbol{\gamma}}_1$, and the representations pull back as described above.
    \item We can stabilize an isomorphism class of isotropy types $\tilde{\boldsymbol{\gamma}} = [(\tilde{\Gamma}, \tilde{V}, \tilde{W})]$ by finite-dimensional complex $\tilde{\Gamma}$-representation $\tilde{V}'$ without trivial summands, which gives arise to a well-defined isomorphism classes.
\end{enumerate}

In either case, we call a class from the set of isotropy types modulo the equivalence relation generated by stabilizations a \emph{stable isotropy type}.

Geometrically, such consideration naturally shows up as follows. For the $\Gamma$-equivariant version, continuing the setup in Section \ref{subsec:regular-FOP}, we consider a $\Gamma$-equivariant $E \to M$. Given a subgroup $\Gamma' \leq \Gamma$ and any component of $M^{\Gamma'}$, the fibers of $NM^{\Gamma'}$ and $\check{E}^{\Gamma'}$ are naturally complex $\Gamma'$-representations without trivial summands. Therefore, we can further decompose the submanifold of $M$ whose isotropy group is $\Gamma'$ into a union indexed by isotropy types whose underlying group is $\Gamma'$. Accordingly, by shrinking $r$ if necessary, we can refine the open subsets from \eqref{eqn:disc-bundles} into 
\beq
U(M, \boldsymbol{\gamma}', r),
\eeq
arising from (punctured) disc bundles over points with prescribed isotropy type $\boldsymbol{\gamma}'$. Then given an isotropy type $\tilde{\boldsymbol{\gamma}}$ isomorphic to one which is subordinate to $\Gamma$, we can consider the open subsets $U(M, \boldsymbol{\gamma}', r)$ such that $\tilde{\boldsymbol{\gamma}} \leq \boldsymbol{\gamma}'$. We can similarly define 
\beq
U(M, \boldsymbol{\gamma}_{\vec{k}}, r)_{\tilde{\boldsymbol{\gamma}}} := \bigcap_{1 \leq i \leq k} U(M, \boldsymbol{\gamma}_i, r)
\eeq
for chains of isotropy types $\boldsymbol{\gamma}_{\vec{k}} = (\boldsymbol{\gamma}_1 < \cdots < \boldsymbol{\gamma}_k)$ such that $\tilde{\boldsymbol{\gamma}} \leq \boldsymbol{\gamma}_1$. Then we obtain the colimit
\beq\label{eqn:resolution-equivarint-isotropy}
\tilde{S}^{-1}(0)_{\tilde{\boldsymbol{\gamma}}}:= \colim_{\boldsymbol{\gamma}_{\vec{k}}} \tilde{S}^{-1}(0)_{(M, \boldsymbol{\gamma}_{\vec{k}}, r), \tilde{\boldsymbol{\gamma}} }
\eeq
where $\tilde{S}^{-1}(0)_{(M, \boldsymbol{\gamma}_{\vec{k}}, r), \tilde{\boldsymbol{\gamma}} }$ is obtained from the pullback diagram \eqref{eqn:fiber-prod-gamma} replacing the bottom left entry by the open subset $U(M, \boldsymbol{\gamma}_{\vec{k}}, r)_{\tilde{\boldsymbol{\gamma}}}$.
It is well-defined and it admits a smooth map onto the closure of the subset of $S^{-1}(0)$ whose isotropy type under the $\Gamma$-action is isomorphic to $\tilde{\boldsymbol{\gamma}}$.

\begin{prop}\label{prop:equivariant-resolve-bold-gamma}
    The smooth manifold $\tilde{S}^{-1}(0)_{\tilde{\boldsymbol{\gamma}}}$ is a $\Gamma$-manifold with a normal complex structure which admits a $\Gamma$-equivariant map to $S^{-1}(0)$. This map defines a diffeomorphism over the part of $S^{-1}(0)$ whose isotropy type under the $\Gamma$-action is isomorphic to $\tilde{\boldsymbol{\gamma}}$. Furthermore, $\tilde{S}^{-1}(0)_{\tilde{\gamma}}$ is independent of the choice of sufficiently small $r > 0$. 
    
    For any open embedding $M' \to M$ which is equivariant under an injective group homomorphism $\Gamma' \to \Gamma$ such that $E' = E|_{M'}$ and $S' = S|_{M'}$ with compatible straightenings and normal complex structures, there is an equivariant open embedding $(\tilde{S}')^{-1}(0)_{\tilde{\boldsymbol{\gamma}}} \to \tilde{S}^{-1}(0)_{\tilde{\boldsymbol{\gamma}}}$ which respects the normal complex structures, provided that $\tilde{\boldsymbol{\gamma}}$ is isomorphic to an isotropy type subordinate to $\Gamma'$. \qed
\end{prop}

Globalizing to the orbifold setting, the generalization of \eqref{eqn:orbi-resolve-gamma} is now given by
\beq\label{eqn:orbi-resolve-bold-gamma}
\coprod_{\alpha} \tilde{S}_{\alpha}^{-1}(0)_{\tilde{\boldsymbol{\gamma}}},
\eeq
and the resulting orbifold is written as $\tilde{\mc S}^{-1}(0)_{\tilde{\boldsymbol{\gamma}}}$, coming from all isotropy types lying in the isomorphism class $\tilde{\boldsymbol{\gamma}}$.

\begin{prop}\label{prop:resolve-bold-gamma}
    The smooth normally complex orbifold $\tilde{\mc S}^{-1}(0)_{\tilde{\boldsymbol{\gamma}}}$  depends only on the straightenings, normal complex structures on ${\mc U}$ and ${\mc E}$, and the regular FOP section ${\mc S}$.

    For any open embedding ${\mc U}' \to {\mc U}$ such that ${\mc E}' = {\mc E}|_{{\mc U}'}$ and ${\mc S}' = {\mc S}|_{{\mc U}'}$ with compatible straightenings and normal complex structures, there is an equivariant open embedding $(\tilde{\mc S}')^{-1}(0)_{\tilde{\boldsymbol{\gamma}}} \to \tilde{\mc S}^{-1}(0)_{\tilde{\boldsymbol{\gamma}}}$ which respects the normal complex structures. \qed
\end{prop}

To discuss stabilization, it is more convenient to work with stable isotropy types. Given a stable isotropy type $[\tilde{\boldsymbol{\gamma}}]$, the notations 
\beq
\tilde{S}^{-1}(0)_{[\tilde{\boldsymbol{\gamma}}]} \quad \quad \quad \tilde{\mc S}^{-1}(0)_{[\tilde{\boldsymbol{\gamma}}]}
\eeq
simply denote the disjoint union of all $\tilde{S}^{-1}(0)_{\tilde{\boldsymbol{\gamma}}}$ and $\tilde{\mc S}^{-1}(0)_{\tilde{\boldsymbol{\gamma}}}$ such that $\tilde{\boldsymbol{\gamma}}$ lies in the class $[\tilde{\boldsymbol{\gamma}}]$. The following proposition follows from the reasoning behind Proposition \ref{prop:resolve-gamma}.

\begin{prop}\label{prop:bold-stable}
    In the of situation of Lemma \ref{lem:stable-regular}, then the resolution $({\mc S}^{\tau})^{-1}(0)_{[\tilde{\boldsymbol{\gamma}}]} := (\pi_{{\mc E}'}^* {\mc S} \oplus \tau_{{\mc E}'})^{-1}(0)_{[\tilde{\boldsymbol{\gamma}}]}$ agrees with $\tilde{\mc S}^{-1}(0)_{[\tilde{\boldsymbol{\gamma}}]}$. \qed
\end{prop}

Finally, note that isotropy admits a product structure, sending a pair $\boldsymbol{\gamma}_1 = (\tilde{\Gamma}_1, \tilde{V}_1, \tilde{W}_1)$ and $\boldsymbol{\gamma}_2 = (\tilde{\Gamma}_2, \tilde{V}_2, \tilde{W}_2)$ to
\beq
\boldsymbol{\gamma}_1 \times \boldsymbol{\gamma}_2 := (\tilde{\Gamma}_1 \times \tilde{\Gamma}_2, \tilde{V}_1 \times \tilde{V}_2, \tilde{W}_1 \times \tilde{W}_2),
\eeq
and it is well-defined after passing to isomorphism classes and stabilization. We have the following generalization of Proposition \ref{prop:product}.
\begin{prop}\label{prop:module-gamma}
    Let ${\mc S}_1: {\mc U}_1 \to {\mc E}_1$ be a regular FOP section. Suppose ${\mc U}_2$ is a (not necessarily effective) normally complex orbifold whose generic stabilizer is isomorphic to the finite group $\Gamma_2$. Let ${\mc F}_2 \xrightarrow{\pi_{{\mc F}_2}} {\mc U}_2$ be a complex vector bundle such that the total space ${\mc F}_2$ is an effective orbifold. Then the resolution (cf. Corollary \ref{cor:product-regular})
    \beqn
    (\pi_1^* {\mc S}_1 \times \tau_{{\mc F}_2})^{-1}_{[\boldsymbol{\gamma}_1] \times [\boldsymbol{\gamma}_2]}(0)
    \eeqn
    agrees with $\tilde{\mc S}_1^{-1}(0)_{[\boldsymbol{\gamma}_1]}$, where $[\boldsymbol{\gamma}_2]$ is the stable isotropy class represented by $(\Gamma_2, \{0\}, \{0\})$.
\end{prop}
\begin{proof}
    By Corollary \ref{cor:product-regular}, we know that $\pi_1^* {\mc S}_1 \times \tau_{{\mc F}_2}$ is FOP regular, and the generic part of its zero locus indexed by $[\boldsymbol{\gamma}_1] \times [\boldsymbol{\gamma}_2]$ is isomorphic to the product of ${\mc U}_2$ with the the part of ${\mc S}_1^{-1}(0)$ indexed by $[\boldsymbol{\gamma}_1]$. The agreement of resolutions is a consequence of Lemma \ref{lem:regular-product} and Lemma \ref{lem:stable-regular}.
\end{proof}

\section{From derived orbifolds to orbifolds}\label{sec:split}

The goal of this section is to prove the results from the introduction which describe the relationship between derived and underived orbifold bordism.

\subsection{Complex (derived) orbifold bordism}\label{subsec:d-Omega-C} 
Elaborating on the discussion from Section \ref{subsec:Omega-C}, we recall the definition of complex derived orbifold bordism following \cite[Section 5]{pardon20} and \cite[Section 7]{Bai_Xu_2022}. For the discussion of this section, let $(X,A)$ be a pair of orbispaces.

Firt, we introduce the notion of derived orbifold with boundary ${\mc D}$ over $(X,A)$: this consists of lifting the underlying orbifold ${\mc U}$ to an orbifold \emph{over} $(X,A)$.
\begin{defn}
  A pair $({\mc D}, f)$ and $({\mc D}', f')$ of derived orbifolds over $(X,A)$, are  \emph{bordant}  if there exist a derived orbifold with boundary $(\hat{\mc U}, \hat{\mc E}, \hat{\mc S})$ such that:
    \begin{enumerate}
        \item $  \hat{\mc U}$ is a bordism from $({\mc U}, f)$ to $({\mc U}', f')$, and
        \item using the embedding from (1), $(\hat{\mc E}, \hat{\mc S})$ restricts to $({\mc E}, {\mc S})$ and $({\mc E}', {\mc S}')$ respectively.
    \end{enumerate}
\end{defn}

The \emph{derived orbifold bordism group} of $(X,A)$ is then defined to be the free abelian group generated by $({\mc D}, f)$, with ${\mc D}$ compact, over $(X,A)$ modulo the equivalence relations generated by stabilization and bordism (the latter includes the notion of restriction to a neighbourhood of the zero-locus). The proof that this is a homology theory (i.e. that some version of the Mayer-Viteoris condition holds) is non-trivial, and is due to Pardon  \cite{pardon20}.

As in the underived case, one can impose tangential structures on ${\mc D}$ to define structured derived orbifold bordism groups.  In particular, one can define stable complex structures on a derived orbifold as stable complex structures on the underlying orbifold together with a stable complex structure on the defining vector bundle. We shall however use the following more general notion, which ends up being equivalent after stabilization:
\begin{defn}\label{defn:stable-complex-d-chart}
A \emph{stable complex structure} on a derived orbifold ${\mc D} = ({\mc U}, {\mc E}, {\mc S})$ (with boundary) is an isomorphism class of:
\begin{enumerate}
    \item a vector bundle ${\mc E}' \to {\mc U}$, together with
    \item a complex structure on $T{\mc U} \oplus {\mc E}' \oplus  \underline{\mb R}^k$, for some integer $k$, and a complex structure on ${\mc E} \oplus {\mc E}'$.
\end{enumerate}
The stable complex structures specified by $({\mc E}_1', k_1)$ and $({\mc E}_2', k_2)$ are isomorphic if there exist \emph{complex} vector bundles ${\mc E}_1'', {\mc E}_2'' \to {\mc U}$ such that we have isomorphisms of complex vector bundles
\begin{align}
  (T{\mc U} \oplus {\mc E}'_1 \oplus  \underline{\mb R}^{k_1}) \oplus {\mc E}_1'' & \cong (T{\mc U} \oplus {\mc E}'_2\oplus  \underline{\mb R}^{k_2}) \oplus {\mc E}_2''  \\
  ({\mc E} \oplus {\mc E}'_1) \oplus {\mc E}_1'' & \cong ({\mc E} \oplus {\mc E}'_2) \oplus {\mc E}_2''. 
\end{align}
\end{defn}

Returning to the discussion of Section \ref{sec:derived-orbifolds}, the restriction of any stably complex derived orbifold admits a natural stable complex structure,  so does any stabilization (this is a minor change from \cite{pardon20,Bai_Xu_2022} who only allow stabilisation by complex vector bundles, which is immaterial because any vector bundle can be further stabilised to a complex bundle). The bordism relation can be refined to incorporate stable complex structures. Namely, a choice of normal vector to the inclusion of a  derived orbifold ${\mc D} = ({\mc U}, {\mc E}, {\mc S})$ in the boundary of a stable complex derived orbifold $(\hat{\mc D}, F) = (\hat{\mc U}, \hat{\mc E}, \hat{\mc S}, F)$ induces a stable complex structure on ${\mc D}$, with respect to the restriction of the bundle $ \hat{\mc E}' $, with the normal direction identified with an additional $\underline{\mb R} $ stabilisation factor.

Then we say that any stably complex pair $({\mc D}, f)$ and $({\mc D}', f')$ over $(X,A)$ are stably complex bordant if there exists a bordism $(\hat{\mc D}, F) = (\hat{\mc U}, \hat{\mc E}, \hat{\mc S}, F)$ as above further satisfying:
\begin{itemize}
    \item the stable complex structure of the pair $(\hat{\mc U}, \hat{\mc E})$ restricts to the stable complex structures of $({\mc U}, {\mc E})$ and $({\mc U}', {\mc E}')$ for some choice of normal vector field.
\end{itemize}

\begin{defn}
The \emph{complex derived orbifold bordism} group of a pair of orbispaces $(X,A)$
\beq
d\Omega_*^{U}(X,A)
\eeq
is the abelian group (with additive structure coming from disjoint union) generated by compact derived orbifolds with boundary over $(X,A)$, equipped with a stable complex structure, modulo the the equivalence relation generated by stabilization and stably complex bordisms. It is \emph{graded} by the virtual dimension.
\end{defn}

We remark that the stable complex bordism relation is actually an equivalence relation, as shown in \cite[Proposition 5.1]{pardon20}. From \cite[Proposition 7.23]{Bai_Xu_2022}, we know that $d\Omega_*^{U}$ actually defines a generalized homology theory.

\subsection{Constructing orbifolds from derived orbifolds}\label{subsec:splitting}
We now prove Theorem \ref{thm:main-derived-to-underived}, which, then completes the proof of Theorem \ref{thm:main}. As in Section \ref{sec:from-orbif-manif}, we formulate all constructions in terms of normally complex orbifolds. The main difference from the constructions of that section is that many more choices are required for the algorithm we use:

\vspace{0.2cm}

\begin{mdframed}
\textbf{ALGORITHM II: From derived orbifolds to orbifolds}
\end{mdframed}

\vspace{0.2cm}

\noindent \underline{\it Input:} Suppose we are given an effective  normally complex derived orbifold ${\mc D}$: i.e. an effective normally complex orbifold ${\mc U}$ possibly with boundary, a normally complex vector bundle ${\mc E} \to {\mc U}$, and a continuous section ${\mc S}: {\mc U} \to {\mc E}$ such that $|{\mc S}^{-1}(0)|$ is compact. 


\vspace{0.2cm}

\noindent \underline{\it Step 1}
Apply \cite[Lemma 3.15, Lemma 3.20]{Bai_Xu_2022} to equip ${\mc E} \to {\mc U}$ with a straightening in the sense of Definition \ref{defn:straight-metric} and Definition \ref{defn:vector bundle-straighting}. 

\vspace{0.2cm}

\noindent \underline{\it Step 2}
For a sufficiently small $\epsilon > 0$, apply Proposition \ref{prop:FOP-prop} and Proposition \ref{prop:FOP-existence} to obtain a smooth section such that:
\begin{itemize}
    \item ${\mc S}_{\epsilon}$ is a regular FOP section over ${\mc U}'$ of degree at most $d \geq d_0$, where ${\mc U}' \subset {\mc U}$ is an open suborbifold with ${\mc S}^{-1}(0) \subset {\mc U}'$ and $|{\mc U}'|$ compact,
    \item $\|{\mc S}_{\epsilon} - {\mc S}\|_{C^0} < \epsilon$, and
    \item $|{\mc S}_{\epsilon}^{-1}(0)|$ is compact.
\end{itemize}

\vspace{0.2cm}

\noindent \underline{\it Step 3}
For each isomorphism class of stable isotropy types $[\tilde{\boldsymbol{\gamma}}]$, apply Proposition \ref{prop:resolve-bold-gamma}, to obtain a compact normally complex orbifold ${\mc Z}^{orb}_{[\tilde{\boldsymbol{\gamma}}]}$ whose generic stabilizer point is indexed by $[\tilde{\boldsymbol{\gamma}}]$, which is the resolution of singularities 
\beq
{\mc Z}^{orb}_{[\tilde{\boldsymbol{\gamma}}]} := \tilde{\mc S}^{-1}(0)_{[\tilde{\boldsymbol{\gamma}}]}.
\eeq
Moreover, ${\mc Z}^{orb}_{[\tilde{\boldsymbol{\gamma}}]}$ comes with a map to ${\mc U}$.
\qed

\vspace{0.2cm}

We call the combination of the three steps above the \emph{Perturbation-Resolution algorithm}.

\begin{rem}\label{rem:algorithm}
\begin{enumerate}
    \item As in Section \ref{sec:from-orbif-manif}, we can easily extend the results stated for derived manifolds to the case of derived manifolds with boundary, in particular, Proposition  \ref{prop:resolve-bold-gamma}.
    \item The following auxiliary choices are made in the perturbation-resolution algorithm:
    \begin{itemize}
        \item restrictions and stabilizations which make the ambient orbifold effective and induce genuine (normal) complex structures from the stable version;
        \item straightenings, in the form of a Riemannian metric on ${\mc U}$ and a connection on ${\mc E}$;
        \item a cut-off degree $d$ in the construction of the FOP sections;
        \item FOP sections which achieve regularity.
    \end{itemize} 
    \item After making ${\mc U}$ effective, if $T{\mc U}$ and ${\mc E}$ are both endowed with a genuine complex structure, it follows from the construction that the tangent bundle of ${\mc Z}^{orb}_{[\tilde{\boldsymbol{\gamma}}]}$ is equipped with a stable complex structure. 
\end{enumerate}
\end{rem}

\begin{proof}[Proof of Theorem \ref{thm:main-derived-to-underived}] 
  We first prove the following statement: given a compact derived orbifold (with boundary) $({\mc D}, f) = ({\mc U}, {\mc E}, {\mc S}, f)$ over a pair of orbispaces $(X,A)$ which represents a class in $d\Omega_*^{U}(X,A)$,  for each isomorphism class of groups $\gamma$, the output of applying Algorithm II with input datum $({\mc D}, f)$ is independent up to cobordism of the representative and various auxiliary data as listed in Remark \ref{rem:algorithm} (2). This relies on the following:

\vspace{0.2cm}

\noindent {\it Claim}
If $({\mc D}, f)$ and $({\mc D}', f') = ({\mc U}', {\mc E}', {\mc S}', f')$ are bordant via $(\hat{\mc D}, F) = (\hat{\mc U}, \hat{\mc E}, \hat{\mc S}, F)$ such that
\begin{itemize}
    \item the bordism $\hat{\mc U}$ is effective,
    \item $\hat{\mc U}$ is normally complex and the restriction of the normal complex structure along the collars of ${\mc U}$ and ${\mc U}'$ agrees with the normal complex structure on $({\mc U} \coprod {\mc U}') \times [0,1)$, and
    \item $\hat{\mc E}$ is normally complex and the restriction of the normal complex structure along the collars of ${\mc U}$ and ${\mc U}'$ agrees with the normal complex structure on the pullback of ${\mc E} \coprod {\mc E}'$ on $({\mc U} \coprod {\mc U}') \times [0,1)$;
\end{itemize}
then for any auxiliary choices, the outputs of Algorithm II
\beq
{\mc Z}^{orb}_{[\tilde{\boldsymbol{\gamma}}]}, \quad \quad \quad {\mc Z}^{\prime orb}_{[\tilde{\boldsymbol{\gamma}}]}
\eeq
are bordant over $(X,A)$ via the composition of the blow-down map and $f$ (or $f'$).

The proof of this claim goes as follows.
\begin{enumerate}
    \item[(a)] Given the straightenings on ${\mc E} \to {\mc U}$ and ${\mc E}' \to {\mc U}'$, using the collar $({\mc D} \coprod {\mc D}') \times [0,1) \hookrightarrow \hat{\mc D}$ and isomorphism between $\hat{\mc E}|_{({\mc D} \coprod {\mc D}') \times [0,1)}$ and the pullback of ${\mc E} \coprod {\mc E'}$ via the projection $({\mc D} \coprod {\mc D}') \times [0,1) \to {\mc D} \coprod {\mc D}'$, we can apply the constructions in \cite[Section 3.3, 3.4]{Bai_Xu_2022} to equip $\hat{E} \to \hat{\mc D}$ with a straightening, whose underlying Riemannian metric coincides with the product Riemannian metric on $({\mc D} \coprod {\mc D}') \times [0,1)$ (the interval $[0,1)$ is endowed with the flat $dt^2$ metric), and the connection coincides with the pullback connection.
    \item[(b)] If the FOP perturbations ${\mc S}_{\epsilon}$ on ${\mc D}$ and ${\mc S}_{\epsilon'}'$ on ${\mc D}'$ are chosen using cut-off degrees $d$ and $d'$ respectively, by shrinking $\epsilon$ and $\epsilon'$ if necessary, from a relative form of Proposition \ref{prop:FOP-existence} as stated in \cite[Proposition 6.7]{Bai_Xu_2022}, we can find $\hat{d} \geq \max\{d,d'\}$, $\hat{\epsilon} > 0$, and a regular FOP section ${\mc S}_{\hat{\epsilon}}: \hat{\mc U} \to \hat{\mc E}$ of cut-off degree $\hat{d}$ with respect to the straightenings chosen in (a) such that it is equal to the pullback of ${\mc S}_{\epsilon} \coprod {\mc S}_{\epsilon'}'$ on the collar $({\mc D} \coprod {\mc D}') \times [0,1)$. Note here we need to apply Lemma \ref{lem:cut-degree} to make sure that enlarging the degree  preserves the regularity of the FOP section.
    \item[(c)] After making the choices in (a) and (b), we observe that the functoriality as explored in Proposition \ref{prop:resolve-bold-gamma} (see Proposition \ref{prop:construct-orbifold} for the isotropy-free case) makes sure that ${\mc Z}^{orb}_{[\tilde{\boldsymbol{\gamma}}]}$ and ${\mc Z}^{\prime orb}_{[\tilde{\boldsymbol{\gamma}}]}$ are bordant through the resolution $\hat{\mc Z}^{orb}_{[\tilde{\boldsymbol{\gamma}}]}$, and we actually have a collar $({\mc Z}^{orb}_{[\tilde{\boldsymbol{\gamma}}]} \coprod {\mc Z}^{\prime orb}_{[\tilde{\boldsymbol{\gamma}}]} ) \times [0,1) \hookrightarrow \hat{\mc Z}^{orb}_{[\tilde{\boldsymbol{\gamma}}]}$.
 \end{enumerate}

Given the claim and item (3) in Remark \ref{rem:algorithm}, the only ambiguity that remains comes from restrictions and stabilizations. Indeed, just to remind the reader, different choices of straightenings and FOP sections on the same derived orbifold chart can be compared over the trivial bordism.

Note that taking restriction does not affect the output of Algorithm II because all the constructions only depend on the information over any arbitrarily small open suborbifold of ${\mc U}$ containing the zero locus ${\mc S}^{-1}(0)$. As for the invariance under stabilization, Lemma \ref{lem:stable-regular} shows that the regularity of FOP sections is preserved under stabilization, then so does the orbifold resolution of the zero loci by Proposition \ref{prop:stabilization} and Proposition \ref{prop:bold-stable}. Therefore, bordism invariance under stabilization is reduced to the case proven in the claim.

In order to prove the compatibility of this splitting with the module structure, we observe that the geometric orbifold bordism, as a homology theory, splits as the direct sum
\beq\label{eqn:orbi-bord-decompose}
\Omega^U = \bigoplus_{\gamma} \Omega^U_{\gamma},
\eeq
where $\gamma$ ranges over isomorphism classes of finite groups and $\Omega^U_{\gamma}$ consists of bordism classes of orbifolds whose ambient orbifold has generic stabilizer group lies in the isomorphism class $\gamma$. This decomposition, the index set respects the multiplicative structure in the sense that the Cartesian product induces a map
\beq
\Omega^U_{\gamma} \times \Omega^U_{\gamma'} \to \Omega^U_{\gamma \times \gamma'}.
\eeq

 With this in mind, the natural transformation in Theorem \ref{thm:main-derived-to-underived}, is defined in the following way: given a derived orbifold $({\mc U}, {\mc E}, {\mc S})$ (we omit the map to the pair $(X,A)$ for notational simplicity), we consider all the stable isotropy types of the form
 \beq\label{eqn:geoemtric-isotropy-type}
 [(\tilde{\Gamma}, \{0\}, \{0\})],
 \eeq
 namely, isomorphism classes of $(\tilde{\Gamma}, \tilde{V}, \tilde{W})$ where $\tilde{V}$ is isomorphic to $\tilde{W}$. Then for such a $\tilde{\boldsymbol{\gamma}} = (\tilde{\Gamma}, \{0\}, \{0\})$, the bordism class ${\mc Z}^{orb}_{[\tilde{\boldsymbol{\gamma}}]}$ lands in $\Omega^U_{\tilde{\gamma}}$, where $\tilde{\gamma}$ stands for the isomorphism class of the group $\tilde{\Gamma}$. Then we construct the transformation ${\mc Z}^{orb}$ to be the evident map defined by assembling such a construction over all such stable isotropy classes:
\beq
{\mc Z} := \bigoplus {\mc Z}^{orb}_{[\boldsymbol{\gamma}]}: d \Omega^{U} \to \bigoplus_{\gamma} \Omega^U_{\gamma}.
\eeq

It remains to see that ${\mc Z}$ does define a natural transformation of generalized homology theories which respect the module structure and splits the inclusion map $\Omega^U \to d \Omega^U$.

The functoriality follows from the fact that ${\mc Z}$ is defined by pushing forward the bordism class produced from the above algorithm. Therefore, it suffices to show the following diagram commutes:
\beq\label{eqn:naturality}
\vcenter{ 
\xymatrix{
{d\Omega_{*+1}^{U}(X,A)} \ar[d]^{\partial} \ar[rr]^{{\mc Z}} & & {\Omega^{U}_{*+1}(X,A)} \ar[d]^{\partial} \\
{d\Omega_*^{U}(A)} \ar[rr]^{{\mc Z}}                            & & \Omega_*^{U}(A).                            
}
}
\eeq
Consider a stably complex derived orbifold chart with boundary ${\mc D} = ({\mc U}, {\mc E}, {\mc S})$ with a map $f: ({\mc U}, \partial {\mc U}) \to (X, A)$. Assume that there is an embedding $\partial {\mc U} \times [0,1) \rightarrow {\mc U}$. After choosing a straightening on $(\partial {\mc U}, {\mc E}|_{\partial {\mc U}})$, we can extend it to a straightening on $({\mc U}, {\mc E})$ which coincides with the product structure on the collar $\partial {\mc U} \times [0,1)$. After choosing a regular FOP section ${\mc S}' |_{\partial {\mc U}}$ of ${\mc S}|_{\partial {\mc U}}$, we can extend it to an FOP perturbation $\tilde {\mc S}'$ of ${\mc S}$. The rest of the argument goes exactly like the discussion of bordisms.

For the module map property, given a bordism class in the summand $\Omega^U_{\gamma}$ and a class in $d\Omega^U$, we can appeal to Proposition \ref{prop:module-gamma} to see that applying ${\mc Z}^{orb}_{[\boldsymbol{\gamma}]}$ commutes with taking Cartesian product with bordism classes in $\Omega^U_{\gamma}$. For the splitting property, this follows from the observation that the tautological section of a stabilization of an orbifold by a faithful vector bundle, which represents the image of a geometric bordism class inside derived bordism, is FOP regular and the resolution does nothing to its zero locus.
\end{proof}

\subsection{Homotopical bordism}
\label{sec:equivariant-bordisms}

We now turn attention to our results comparing geometric and homotopical bordism.  Following tom Dieck \cite{tomdieck1970}, the standard formulation of the stably complex homotopical bordism of a $\Gamma$-pair $(X,A)$ is in terms of the stably complex bordism of the product with $(X,A)$ with the pair $(W,W \setminus{0})$ associated to a complex $\Gamma$-representation $W$. An inclusion $W \to W'$ of such representations induces, by taking the product with the disc bundle of the orthogonal complement (and rounding corners), a map
\begin{equation}
  \Omega^{U,\Gamma}_{*}((X,A) \times (W, W \setminus{0}))   \to  \Omega^{U,\Gamma}_{* + \dim W'/W }((X,A) \times (W', W' \setminus{0})).
\end{equation}
\begin{defn}
  The homotopical bordism group is the direct limit 
  \begin{equation}
h \Omega^{U,\Gamma}(X,A) \equiv \lim_{W}    \Omega^{U,\Gamma}_{* - \dim W}((X,A) \times (W, W \setminus{0}))
\end{equation}
taken over all inclusions of complex finite-dimensional sub-representations of a universal $\Gamma$-representation.
\end{defn}

An alternate definition is that of derived $\Gamma$-manifolds. This  consists of a $\Gamma$-manifold $M$, a $\Gamma$-vector bundle $E$, and an equivariant section $s$ of $E$. When $s^{-1}(0)$ is compact, we say that a derived $\Gamma$-manifold is proper. We can specialize the discussion of Section \ref{subsec:d-Omega-C} to obtain a notion of stably complex derived $\Gamma$-manifold. The derived stably complex bordism group
\begin{equation}
    d \Omega^{U,\Gamma}(X,A) 
\end{equation}
is then generated by proper derived $\Gamma$-manifolds equipped with an equivariant map $(M,\partial M) \to (X,A)$. The following is a standard consequence of the fact that vector bundles can be stabilized to representation (compatibly with complex structures).
\begin{prop}
  Homotopical and derived bordism are naturally isomorphic homology theories of $\Gamma$-spaces.
\end{prop}
\begin{proof}[Sketch of proof]
  It is convenient for the proof to remove the requirement that the representative of a bordism class of the pair $ (X,A) \times (W, W \setminus{0})$ be compact with boundary. There is a reformulation of this bordism group as a bordism group of stably almost complex manifolds
  \begin{equation}
    (M,\partial M) \to \left( X \times W , A \times W \right)   
  \end{equation}
whose projection to $W$ is proper in a neighbourhood of the origin, and the comparison maps are given in one direction by passing to the interior of the manifold, and in the other by replacing $M$ by the inverse image of a small ball in $W$.

Proceeding with this model for homotopical bordism, we obtain a tautological map from homotopical to derived bordism by considering $W$ as a vector bundle over $M$, and the projection to $W$ as a section. The properness conditions on the two sides match up, as do the equivalence relations of bordism and stabilization. In the other direction, we stabilise a derived orbifold so that its total space $\mc U$ and the vector bundle $\mc E$ are both stably complex, then stabilize further (by a complex complement for $\mc E)$, so that $\mc E$ is obtained from a $\Gamma$-representation.
\end{proof}

\begin{proof}[Proof of Theorem \ref{thm:inclusion_geometric_homotopical_splits}:]
  For this statement, our goal is to construct a map
  \begin{align} \label{eq:resolution_derived_equivariant}
    d\Omega^{U,\Gamma}_*(X,A) & \to \Omega^{U,\Gamma}_*(X,A) \\ \notag
    (M,E,S) & \mapsto \tilde{Z}
      \end{align}
      which splits the inclusion of $\Gamma$-manifolds in derived $\Gamma$-manifolds. The right hand side decomposes as a direct sum indexed by the conjugacy classes of subgroups of $\Gamma$, corresponding to the generic stabilizer group.

      The proof is exactly the same as that of Theorem \ref{thm:main-derived-to-underived}, appealing directly to Proposition \ref{prop:equivariant-resolve}  in the isotropy-free part (instead of Proposition \ref{prop:construct-orbifold}), and to Proposition \ref{prop:equivariant-resolve-bold-gamma} in general, and splitting according to conjugacy classes of subgroups (and virtual representations) rather than abstract isomorphism classes: given a $\Gamma$-equivariant section $S$ of a $\Gamma$-vector bundle $E$ over an effective $\Gamma$-manifold $M$, the result of our construction is the union, over all subgroups of $\Gamma$ of the resolution algorithm applied to the intersection of a nearby FOP-regular section with the subset of the fixed point locus for groups in this conjugacy class, with the property that the normal representation and the obstruction representation are isomorphic (cf. the discussion after Equation \eqref{eqn:geoemtric-isotropy-type}).      The statement that this construction is compatible with the module structures follows from the compatibility of Algorithm II with taking products with manifolds.
\end{proof}

We complete this section with the proof of the last remaining result from Section \ref{sec:appl-equiv-stable}:
\begin{proof}[Proof of Theorem  \ref{thm:split-change-group-compatible}]
  For the statement arising from an injection $\Gamma \to \Gamma'$ of groups, we have to show that the induction functor which takes a $\Gamma$-manifold $M$ to $M \times_{\Gamma} \Gamma'$, a $\Gamma$-vector bundle $E$ to the induced bundle $E \times_{\Gamma} \Gamma'$, and acts on sections in the evident way is compatible with the perturbation-resolution algorithm. To establish this, we observe that the choices of $\Gamma$-equivariant straightenings of $M$ and $E$ that are required for the construction induce $\Gamma'$-equivariant straightenings: indeed, we have a $\Gamma$-equivariant inclusion map $M \to M \times_{\Gamma} \Gamma'$, and every non-empty fixed point stratum $\left( M \times \Gamma' \right)^{\Sigma'}$ for a subgroup $\Sigma'$ of $\Gamma$ is the image of a fixed point stratum $M^{\Sigma}$ for a subgroup $\Sigma$ of $\Gamma$, under an element of $\Gamma'$ which is unique up to left multiplication by appropriate elements of $\Gamma$, so that the choice made near $M^{\Sigma}$ freely induces such a choice for $ \left( M \times \Gamma' \right)^{\Sigma'}$ (in the sense that no additional conditions are imposed). Tracing through the construction then shows that, using at every step the induced choices yields a diffeomorphism of $\Gamma'$-manifold between $\tilde{Z} \times_{\Gamma} \Gamma$ (with $\tilde{Z}$ as in Equation \eqref{eq:resolution_derived_equivariant}) and the result of applying the algorithm to the induced group.

  In the case of a surjection $\Gamma \to \Gamma'$ of groups, the argument is even more straightforward: a $\Gamma'$-manifold gives rise to a $\Gamma$-manifold via the group homomorphism, and so are all the additional structures (of a vector, of a section, and of straightenings) in an evident way. The perturbation-resolution algorithm for $\Gamma'$ derived manifolds imposes more constraints (arising from conjugacy in $\Gamma'$), but these nonetheless restrict to data of the required type for the corresponding $\Gamma$-manifold, yielding the compatibility of resolutions.
\end{proof}

\section{Applications}\label{sec:applications}

\subsection{Complex bordism valued Gromov--Witten invariants}
Let $(X, \omega)$ be a closed symplectic manifold and denote by ${\mc J}(X, \omega)$ the space of $\omega$-tame almost complex structures on $TX$. For integers $g, k \geq 0$, a homology class $A \in H_2(X;{\mb Z})$, and $J \in {\mc J}(X, \omega)$, write
\beq
\ov{\scrM}_{g,k}(X,J,A)
\eeq
for the moduli space of stable genus $g$ $J$-holomorphic maps, representing the homology class $A$, equipped with $k$ marked points. It is a compact Hausdorff space with respect to the Gromov topology. The stabilization map to the Deligne--Mumford space $\ov{\scrM}_{g,k}$ (we just use the coarse space here) and the evaluation maps at the marked points define a continuous map
\beq\label{eqn:st-times-ev}
\mathrm{st} \times \mathrm{ev}: \ov{\scrM}_{g,k}(X,J,A) \to \ov{\scrM}_{g,k} \times X^k.
\eeq
We explain how to extract a complex bordism class 
\beq
\ov{\mathbf{M}}_{g.k}(X,A)
\eeq
over $\ov{\scrM}_{g,k} \times X^k$ from the moduli space $\ov{\scrM}_{g,k}(X,J,A)$ which is independent of $J \in {\mc J}(X, \omega)$, which thus gives rise to a \emph{complex bordism valued Gromov--Witten invariant}.

\subsubsection{Global Kuranishi charts}
A \emph{global Kuranishi chart} for a compact Hausdorff topological space $Z$ is a tuple
\beq
(G, M, E, \mathfrak{s}, \Psi)
\eeq
such that $G$ is a compact Lie group, $M$ is a topological $G$-manifold on which $G$ acts almost freely, $E \to M$ is a $G$-equivariant vector bundle, $\mathfrak{s}: M \to E$ is a $G$-equivariant section, and  $\Psi: \mathfrak{s}^{-1}(0)/G \xrightarrow{\sim} Z$ is a homeomorphism. We usually omit $\Psi$ for the notation. We can introduce the following operations on global Kuranishi charts.
\begin{enumerate}
    \item \emph{Shrinking}: if $M' \subset M$ is a $G$-invariant open subset containing $\mathfrak{s}^{-1}(0)$, then
    \beq
    (G, M', E|_{M'}, \mathfrak{s}|_{M'})
    \eeq
    also defines a global Kuranishi chart for $Z$.
    \item \emph{Stabilization}: if $\pi_{E'}: E' \to M$ is a $G$-equivariant vector bundle and denote by $\tau_{E'}: E' \to \pi_{E'}^* E$ the tautological section, then
    \beq
    (G, E', \pi_{E'}^* E \oplus \pi_{E'}^* E',\pi_{E'}^* \mathfrak{s} \oplus \tau_{E'})
    \eeq
    is a global Kuranishi chart for $Z$.
    \item \emph{Group enlargement}: if $G'$ is another compact Lie group and $\pi_{P}: P \to M$ is a principal $G'$-bundle, then
    \beq
    (G \times G', P, P \times_{M} E, \pi_{P}^* \mathfrak{s})
    \eeq
    defines a global Kuranishi chart for $Z$.
\end{enumerate}
Note that if $M$ comes with a continuous $G$-invariant map $f: M \to Y$ to a topological space $Y$, it induces $G$-invariant continuous maps from $M'$, $E'$, and $P$ from above by composing with the inclusion map, $\pi_{E'}$, and $\pi_{P}$ respectively.

A pair of global Kuranishi charts $(G, M, E, \mathfrak{s})$ and $(G, M', E', \mathfrak{s}')$ equipped with $G$-invariant continuous maps $f: M \to Y$, $f: M' \to Y$, are \emph{bordant} if there exists a tuple $(G, \hat{M}, \hat{E}, \hat{\mathfrak{s}}, \hat{F})$ such that $\hat{M}$ is a topological $G$-manifold with boundary with $\partial \hat{M} = M \coprod M'$, along which the $G$-equivariant vector bundle $\hat{E} \to \hat{M}$ and the $G$-equivariant section $\hat{\mathfrak{s}}: \hat{M} \to \hat{E}$ restricts to $E \coprod E'$ and $\mathfrak{s} \coprod \mathfrak{s}'$; furthermore, the restriction of the $G$-invariant $F$ along $\partial \hat{M}$ is required to agree with $f \coprod f'$.
 
Using this language, the main theorems of \cite{abouzaid2021complex, AbouzaidMcLeanSmith2023} and \cite{hirschi2022global} can be summarized as follows. 

\begin{thm}\label{thm:global-charts}
For any $g, k \in {\mb Z}_{\geq 0}$, $A \in H_2(M;{\mb Z})$, and $J \in {\mc J}(X,\omega)$, the moduli space $\ov{\scrM}_{g,k}(X,J,A)$ admits a global Kuranishi chart $(G, M, E, \mathfrak{s})$, that depends on auxiliary data, and which is equipped with a continuous map $f: M \to \ov{\scrM}_{g,k} \times X^k$ which restricts to \eqref{eqn:st-times-ev} over $\mathfrak{s}^{-1}(0)$. 

For different choices of auxiliary data, the global Kuranishi charts are related by shrinkings, stabilizations, and group enlargements. For any pair of choices of $J, J' \in {\mc J}(X,\omega)$, there exist auxiliary data such that the induced global Kuranishi charts $(G, M, E, \mathfrak{s})$ for $\ov{\scrM}_{g,k}(X,J,A)$ and $(G, M', E', \mathfrak{s}')$ for $\ov{\scrM}_{g,k}(X,J',A)$ together with the maps $f,f'$ are bordant. \qed
\end{thm}

\subsubsection{Smoothing and tangential structures}\label{sec:hand-waving}
In fact, a global Kuranishi chart $(G, M, E, \mathfrak{s})$ for $\ov{\scrM}_{g,k}(X,J,A)$ from Theorem \ref{thm:global-charts} is endowed with extra structures. The ambient $G$-manifold $M$ comes with a continuous map $\pi$ to a smooth $G$-manifold $F$ such that $\pi: M \to F$ is a $G$-equivariant fiberwise submersion and a $C^1_{loc}$ G-bundle \cite[Section 4.4, 4.5]{abouzaid2021complex}. 
As a result, $\pi: M \to F$ has a well-defined vertical tangent bundle $T^v M$. Additionally, the $C^1_{loc}$ $G$-bundle structure guarantees that the $G$-equivariant vector bundle
\beq
T^v M \oplus \pi^* TF
\eeq
defines a $G$-equivariant \emph{vector bundle lift} of the tangent microbundle of $M$. Accordingly, Lashof's stable $G$-smoothing theory \cite{Lashof_1979} can be applied to $M$ to equip a stabilization of it with a smooth structure such that $G$ acts smoothly. For such a smooth structure, the corresponding stabilization of the vector bundle $E \to M$ is isomorphic to a smooth $G$-equivariant vector bundle, which can be seen through approximating the classifying map.

Geometrically, the $G$-manifold $M$ is constructed from the moduli space of certain perturbed $J$-holomorphic curves mapping into $X$, the base $F$ is a $G$-equivariant smooth quasi-projective variety parametrizing suitable classes of stable maps of the standard complex projective spaces. An index theoretic analysis shows that the real $G$-equivariant virtual vector bundle
\beq
(T^v M \oplus \mathfrak{g} \oplus \pi^* TF) \ominus E \cong TM \oplus \mathfrak{g}  \ominus E 
\eeq
admits a  $G$-equivariant \emph{stable complex} structure \cite[Section 6.8]{abouzaid2021complex}.

\subsubsection{The bordism-valued invariants} 
Suppose that we have a global Kuranishi chart $(G, M, E, \mathfrak{s})$ satisfying the following assumptions:
\begin{itemize}
    \item $M$ is a smooth manifold on which $G$ acts smoothly;
    \item $E \to M$ is a smooth $G$-equivariant vector bundle and $\mathfrak{s}: M \to E$ is a smooth section;
    \item the virtual $G$-equivariant vector bundle $(TM \oplus \mathfrak{g}) \ominus E$ is stably  complex. 
\end{itemize}
Then we see that the $G$-quotient
\beq\label{eqn:quotient-d-chart}
(M // G, E // G, \mathfrak{s}^G)
\eeq
defines a stably complex (cf. Definition \ref{defn:stable-complex-d-chart}) and compact derived orbifold. If $f: M \to Y$ is a $G$-invariant continuous map to a topological space $Y$, then \eqref{eqn:quotient-d-chart} and the induced map $f: M // G \to Y$ give rise to an element of $d\Omega_*^{U}(Y)$. Combined with the facts recalled in Section \ref{sec:hand-waving}, Theorem \ref{thm:global-charts} admits the following enhancement:

\begin{thm}\label{eqn:d-orb-GW}
After taking stabilization, any global Kuranishi chart $(G, M, E, \mathfrak{s})$ for $\ov{\scrM}_{g,k}(X,J,A)$ can be promoted to one for which $M$ is a smooth $G$-manifold and $E \to M$ is a smooth $G$-equivariant vector bundle. 

Approximating $\mathfrak{s}$ by a smooth $C^0$-close $G$-equivariant section $\mathfrak{s}'$, the class in $d\Omega_*^{U}(\ov{\scrM}_{g,k} \times X^k)$ defined by the derived orbifold chart
\beq
(M // G, E // G, (\mathfrak{s}')^G)
\eeq
coupled with the continuous map $M // G \to \ov{\scrM}_{g,k} \times X^k$ induced from $f: M \to \ov{\scrM}_{g,k} \times X^k$ is independent of the choice of auxiliary data in the construction of global charts and smoothing, the almost complex structure $J$, and the approximation $\mathfrak{s}'$. \qed
\end{thm}

The cobordism invariance of choosing differrent smooth approximations is manifest. The rest is a compact form of the results from \cite[Corollary 4.60]{AbouzaidMcLeanSmith2023} and \cite[Theorem 2.16]{hirschi2022global}. For a more detailed exposition, the reader could consult \cite[Section 8]{Bai_Xu_2022}.

\begin{thm}\label{thm:GW-bordism}
Given a closed symplectic manifold $(X,\omega)$, for $g,k \in {\mb Z}_{\geq 0}$ and $A \in H_2(X;{\mb Z})$, there exists a complex bordism valued Gromov--Witten invariant
\beq\label{eqn:MU-GW}
(\ov{\mathbf{M}}_{g.k}(X,A), \mathrm{st} \times \mathrm{ev}) \in \Omega_*^{U}(\ov{\scrM}_{g,k} \times X^k)
\eeq
whose underlying complex manifold $\ov{\mathbf{M}}_{g.k}(X,A)$ parametrizes a `resolution' of the moduli space of stable pseudo-holomorphic maps of genus $g$ with $k$ marked points representing the class $A$.
\end{thm}
\begin{proof}
We simply apply the splitting map ${\mc Z}: d\Omega_*^{U}(\ov{\scrM}_{g,k} \times X^k) \to \Omega_*^{U}(\ov{\scrM}_{g,k} \times X^k)$ in \eqref{eqn:splitting-map} to the class constructed in Theorem \ref{eqn:d-orb-GW}. The bordism map $\mathrm{st} \times \mathrm{ev}$ is just defined to be the composition of the blow-down map with the restriction of the bordism map on the ambient orbifold.
\end{proof}

\begin{rem}
\begin{enumerate}
    \item The class \eqref{eqn:MU-GW} should be viewed as a virtual fundamental class of the moduli space $\ov{\scrM}_{g,k}(X,J,A)$ with value in $\Omega_*^U$. One can either take the intersection of it with other classes from $\Omega_*^{U}(\ov{\scrM}_{g,k} \times X^k)$ or look at, e.g., Chern numbers of the complex manifold $\ov{\mathbf{M}}_{g.k}(X,A)$ to define numerical enumerative invariants.
    \item It is an open problem that to what extent the Kontsevich--Manin axioms of ordinary Gromov--Witten invariants are met by the $\Omega_{*}^{U}$-valued Gromov--Witten invariants. We expect it to be a subtle problem, because the resolution algorithm in \cite{abramovich2019functorial}, which is adapted in Section \ref{sec:derive-manifolds}, fails to respect the Cartesian product. This is in contrast with the multiplicativity property of the strong transversality of FOP sections established in \cite[Section 2]{bai2022arnold}.
\end{enumerate}
\end{rem}

\subsection{Hamiltonian fibrations over $S^2$}\label{subsec:split}
This subsection is devoted to the proof of Theorem \ref{thm:splitting-B}. This shall require a little bit of mental gymnastics, as the construction of spectra from geometric bordism is not quite standardized in the literature, and we shall use the Thom spectrum $MU$ for the rest of this section.

\subsubsection{Reformulation in terms of Hamiltonian fibrations}
\label{sec:reform-terms-hamilt}

Via the family version of clutching construction over $S^2$, one associates to the sweepout map the (universal) Hamiltonian fibration over $S^2$ with fiber $X$ over $0 \in S^2$, which is a fiber bundle
\begin{equation} \label{eq:universal_fibre_bundle}
  P \to   \Omega \mathrm{Ham}(M, \omega) \times S^2.
\end{equation}

\begin{lemma}\label{lem:homotopy-classes}
  There is a bijection between the set of homotopy classes of null homotopies for \eqref{eqn:B-sweepout}, and maps of spectra
  \begin{equation} \label{eq:map_from_total_space_to_X}
 P_+ \wedge  {\mb S} \to X_+ \wedge MU
  \end{equation}
  so that the composite 
  \begin{equation}
    \left( \Omega \mathrm{Ham}(M, \omega) \times X \right)_+ \wedge  {\mb S} \to P_+ \wedge  {\mb S} \to X_+ \wedge MU
  \end{equation}
 agrees with the projection map.
\end{lemma}
\begin{proof}
  The most efficient way to perform the comparison is to use parametrized spectra: using the trivialization of the fibers over $0$ as a section, the fibration $P$ determines a parametrized spectrum $P \barwedge  {\mb S} $ over $ \Omega \mathrm{Ham}(M, \omega) \wedge S^2$. We claim that both data correspond to a map from $P \barwedge  {\mb S} $  to the trivial parametrized spectrum with fiber $X_+ \wedge MU$, extending the identity map over the basepoint.

  The fact that Equation \eqref{eq:map_from_total_space_to_X} is equivalent to such a map of parametrized spectra is a standard adjunction statement \cite[Theorem 11.4.1]{MaySigurdsson2006}: the spectrum $P_+ \wedge {\mb S}$ is equivalent to the ``total spectrum'' of the parametrized spectrum $P \barwedge  {\mb S}  $, and this procedure is (left) adjoint to the construction which assigns to the spectrum $X_+ \wedge MU$ the associated trivial spectrum over $ \Omega \mathrm{Ham}(M, \omega) \wedge S^2 $. Our condition about the basepoint corresponds to the requirement that the restriction to $\Omega \mathrm{Ham}(M, \omega) \times X  $ agree with projection.

  For the other equivalence, we use the fact that a parametrized spectrum over a base $B$ is equivalent to a module spectrum over the suspension spectrum of the based loop space $\Omega B$; the standard reference for this fact has become \cite{AndoBlumbergGepner2018}, though the idea goes back at least to Waldhausen \cite[Section 2.1]{Waldhausen1983}. By the James construction, the based loop space of $  \Omega \mathrm{Ham}(M, \omega) \wedge S^2$ is a free algebra on $ \Omega \mathrm{Ham}(M, \omega) \wedge S^1$, so that such a module spectrum is determined by the map in Equation \eqref{eqn:B-sweepout}. A null-homotopy of this map exactly corresponds to a module spectrum over $\Omega \mathrm{Ham}(M, \omega) \wedge S^1 \wedge [0,1]_+$, which is trivial over $1$, and agrees with $P \barwedge  {\mb S} $ over $0$. Because the projection from $[0,1]$ to either endpoint is a homotopy equivalence, we conclude that such a null-homotopy gives rise to the desired equivalence of parametrized spectra. 
\end{proof}

\subsubsection{Some Pontryagin--Thom theory}
\label{sec:revi-geom-bord}

Given a topological space $X$,   the Pontryagin-Thom construction implies that the cohomology theory associated to the spectrum $  X_+ \wedge MU$ is given, on smooth manifolds $Y$, by the geometric bordism theory of manifolds
\begin{equation}
 N \to Y \times X 
\end{equation}
which are proper over $Y$, and are stably complex relative $TY$, in the sense that they are equipped with a complex vector bundle $I$, and an isoorphism
\begin{equation}
  TN \oplus {\mb R}^{k} \cong I \oplus TY  
\end{equation}
for some integer $k$. This determines the cohomology theory on spaces having the homotopy type of a finite CW complex. We can extend this to general CW complex by a standard strategy that bypasses the fact that we cannot define the cohomology of a space as an inverse limit of the cohomology of subcomplexes because of the possible presence of $\lim^1$ terms (c.f. \cite{Milnor1995}): each such space may be represented as an increasing union $Y = \colim_{i} Y_i$ with $Y_i$ smooth manifolds equipped with smooth embeddings $Y_{i} \to Y_{i+1}$. In this case, the cohomology theory associated to $  X_+ \wedge MU$, evaluation on $Y$, is given by bordism classes of the following data for each $i$:
\begin{enumerate}
\item A manifold $N_i \to Y_i \times X$ which is proper over $Y_i$, equipped with a stable complex structure relative $T Y_i$, and which is transverse to $Y_{i-1}$.
\item A cobordism $N_{i,i+1}$ over $Y_{i} \times X$ between $N_{i}$ and the fibre product $N_{i+1} \times_{Y_{i}} Y_{i+1}$, which is also proper over $Y_{i}$, and equipped with a stable complex structure relative $TY_{i}$ extending the stable complex structure on the two boundary components.
\end{enumerate}
In the bordism relation between such data, we impose as usual properness, and include the data of stable complex structures. A completely analogous description applies to other bordism theories, including framed bordism, as we shall presently use:


\begin{lemma} \label{lem:correspondence_induces_map}
  Let $P$ be a smooth manifold. 
  A manifold $M \to P \times X$ whose projection to $P$ is proper  and which is equipped with a stable almost complex structure relative $TP$
  induces a morphism of spectra
  \beq\label{eqn:M-morphism}
    P_+ \wedge    {\mb S} \to X_+ \wedge MU.   
  \eeq
\end{lemma}
\begin{proof}
  Since maps of spectra are determined by maps of the corresponding cohomology theories, let $Y$ be a topological space with the homotopy type of a CW complex which we present, as discussed above, as an increasing union of manifolds $Y_i$. The cohomology theory associated to $P_+ \wedge    {\mb S}$, when evaluated on $Y$, is given by bordism classes of collections $\{(N_i, N_{i-1,i})\}$ of manifolds mapping to $P \times Y_i$, which are proper over $Y_i$, and stably framed relative $TY_{i}$. Up to bordism, we map require that all such manifolds be transverse to $M$, under the evaluation map to $P$. This implies that the fibre products  $\{(N_i \times_{P} M , N_{i-1,i} \times_{P} M )\}$ are manifolds, which are again proper over $Y_i$ (because $M$ is proper over $P$), and are equipped with stable complex structures arising from the stable complex structure on $M$ relative $T P$. The fibre products rae equipped with evaluation maps to $X$, and hence represents cobordism spaces for the cohomology theory associated to $X_+ \wedge MU$. It is easy to see that this map depends only on the bordism class, as any bordism over $P$ may be assumed to be transverse to $M$.

  To prove that this construction defines a natural transformation, it is convenient to restrict attention to filtrations with the property that the $k$-skeleton of $Y$ factors through $Y_{k_i}$ for some integer $k_i$. Given any a map $Y' \to Y$ we may assume, possibly after passing to a subsequence, and reindexing, that $Y'_i$ maps by a smooth map to $Y_{i}$, commuting with the inclusions on both towers (this can be done inductively on $i$). Choosing the pairs $\{(N_i, N_{i-1,i})\}$ representing a cobordism on $Y$ to be transverse to these maps, the fibre product defines a cobordism class on $Y'$. 
\end{proof}
Using the exact same method to a filtration $P = \colim P_i$ of a topological space by smooth manifolds, we conclude:
\begin{cor}  \label{cor:family_correspondence_map}
  The data of a collection $\{(M_i, M_{i, i+1})\}$ of smooth manifolds over $P_i \times X$, which are proper over $P_i$, are equipped with stable complex structures relative $TP_i$, and such that $M_{i,i+1}$ is a cobordism from $M_{i}$ to $M_{i+1} \times_{P_i} P_{i+1}$ determines a morphism of spectra
  \begin{equation}
    P_+ \wedge    {\mb S} \to X_+ \wedge MU. 
  \end{equation} \qed
\end{cor}

Let us assume further that $X$ is a smooth manifold, and that we are given a map
\begin{equation}
  X \times B \to P  
\end{equation}
for some smooth manifold $B$. 

\begin{lemma} \label{lem:composite_given_by_fibre_product} 
  Let $M$ be manifold over $  P \times X$ as in Lemma \ref{lem:correspondence_induces_map}, so that the projection map $M \to P$ is transverse to $X \times B$. If the fiber product $M \times_{P}  (X \times B)$ agrees with the product of $B$ with diagonal in $X$, then the composite
  \begin{equation}
     (X \times B)_+ \wedge  {\mb S} \to   P_+ \wedge    {\mb S} \to X_+ \wedge MU    
   \end{equation}
   is given by the projection $X \times B \to X$ and the inclusion $ {\mb S} \to MU  $.
\end{lemma}
\begin{proof}
  We only discuss the map on homotopy groups, as the lift to the map of cohomology theories follows by modifying the argument below in exactly the same way as the proof of Lemma   \ref{lem:correspondence_induces_map}. The map $ (X \times B)_+ \wedge  {\mb S}  \to  P_+ \wedge    {\mb S}$ corresponds, on homotopy groups, to assigning to a stably framed manifold  over $X \times B $ its image under the inclusion in $ P$. 
  The image of a stably framed manifold  over $X \times B $ after composition with the map $  P_+ \wedge    {\mb S} \to X_+ \wedge MU    $ is given by fibre product as in Lemma \ref{lem:correspondence_induces_map}. Under the assumption that  the fiber product $M \times_{P}  (X \times B)$ agrees with the product of $B$ with diagonal in $X$, the composite is simply the image  under the projection map to $X$, considered as a stably complex manifold. This exactly corresponds to the map of spectra $ (X \times B)_+ \wedge  {\mb S}  \to   X_+ \wedge MU $.
\end{proof}
Assuming now that $B$ is a topological space, filtered by manifolds $B_i$, together with inclusions $X \times B_i \to P_i$ commuting with the maps $P_i \to P_{i+1}$, we may extend the above result as follows:
\begin{cor} \label{cor:composite_filtered_fibre_product}
   Let $\{(M_i, M_{i,i+1})\}$ be manifolds over $  P \times X$ as in Corollary \ref{cor:family_correspondence_map}, so that the projections maps from $M_i$ and $M_{i,i+1} $ to $P_i$ are transverse to $X \times B_i$. If the fiber product $M_i \times_{P_i}  (X \times B_i)$ agrees with the product of $B_i$ with diagonal in $X_i$, and the  fiber product $M_{i,i+1} \times_{P_i}  (X \times B_i)$ is the trivial cobordism, then the composite
  \begin{equation}
     (X \times B)_+ \wedge  {\mb S} \to   P_+ \wedge    {\mb S} \to X_+ \wedge MU    
   \end{equation}
   is given by the projection $X \times B \to X$ and the inclusion $ {\mb S} \to MU  $. \qed
\end{cor}

\subsubsection{The topological part of the proof}

Consider a $k$-dimensional stably almost complex compact manifold $B$ with boundary, equipped with a smooth map to $\Omega  \mathrm{Ham}(X, \omega)$. We write $P_B$ for the pullback of $P$ to $B$ under this map. As explained in, e.g., \cite[Section 2]{savelyev}, the following holds:
\begin{enumerate}
    \item If $b \in B$ corresponds to a Hamiltonian loop $\phi_b = \phi_b(t) \in \Omega\mathrm{Ham}(X, \omega)$, the fiber of $P_B$ over $b$ is given by a symplectic manifold $(P_{\phi_b}, \omega_b)$, which is the clutching construction over $S^2$ using $\phi_b$. For any point in $S^2$, the symplectic connection of the Hamiltonian fibration $P_{\phi_b} \to S^2$ identifies the horizontal symplectic vector space with the tangent space of $S^2$ equipped with the symplectic form induced from the standard area form on $S^2$.
    \item There exists a globally defined closed $2$-form $\Omega_B$ on $P_B$ such that for any $b \in B$, the restriction $\Omega_B |_{P_{\phi_b}}$ agrees with $\omega_b$.
    \item The structure group of $P_B \to B$ is given by the group of bundle maps over $S^2$ which act on the fibers by Hamiltonian diffeomorphisms and cover the identity map on $S^2$. Furthermore, such bundle maps can be assumed to be the identity near the fibers over $0, \infty \in S^2$.
\end{enumerate}
The space of choices of the $2$-form $\Omega_B$ is weakly contractible. For each fiber $P_{\phi_b}$, there is also an affine space parametrizing the space of symplectic forms $\omega_b$ satisfying the prescribed condition. The choices here do not affect our discussion, except for an intermediate step in the construction of global Kuranishi charts, which will be implemented in Section \ref{sec:glob-kuran-charts}.



Let $S$ be the one-point blow-up of $\mb{CP}^1 \times \mb{CP}^1$. It admits a ruling $S \to \mb{CP}^1$ with a unique reducible fiber $\mb{CP}^1 \vee \mb{CP}^1$ over $0$ and in particular a smooth fiber over $\infty$. Taking the product with $B$, we have a smooth map
\beq\label{eqn:proj-S}
\pi_B: \tilde{S} = S \times B \to \mb{CP}^1 \times B
\eeq
with reducible fibers $\mb{CP}^1 \vee \mb{CP}^1$ over $\{0\} \times B$, and we have a family of smooth fibers over $\{\infty\} \times B$.

\begin{lemma}
There is a fibration 
\beq\label{eqn:B-family-fibration}
\pi_{\tilde{S}}: \tilde{P}_B \to \tilde{S}
\eeq
where $\tilde{P}_B$ admits a closed $2$-form $\tilde{\Omega}_B$ satisfying the following properties:
\begin{enumerate}
    \item  the restriction of $\tilde{\Omega}_B$ to each fiber is symplectomorphic to  $(X, \omega)$;
    \item there exists an open neighborhood $W_{\infty}$ of $\infty \in \mb{CP}^1$ such that the fibration $\tilde{P}_B |_{\pi_B^{-1}(W_{\infty} \times B)}$ with the $2$-form $\tilde{\Omega}_B$ is isomorphic to the trivial fibration
    \beq\label{eqn:trivial-region}
    (W_\infty \times B \times S^2 \times X, \omega_{S^2}|_{W_{\infty}} \oplus 0 \oplus \omega_{S^2} \oplus \omega)
    \eeq
    with projection $W_\infty \times B \times S^2 \times X \to W_\infty \times B \times S^2$, where $\omega_{S^2}$ is the standard area form on $S^2$.
    \item for any $b \in B$, the space $\tilde{P}_B |_{\pi_B^{-1}(\{0\} \times \{b\})} \to \mb{CP}^1 \vee \mb{CP}^1$ is isomorphic to $P_{\phi_b} \vee P_{\phi_b^{-1}} \to \mb{CP}^1 \vee \mb{CP}^1$, where each component of the reducible fiber carries the canonical defomation class of symplectic structures.
\end{enumerate}
\end{lemma}
\begin{proof}
This is a family version of \cite[Lemma 3.3]{abouzaid2021complex}: if we compose \eqref{eqn:B-family-fibration} with the projection $\tilde{S} \to B$, each fiber of this map recovers in construction in \emph{loc. cit.}. The existence of the closed $2$-form $\tilde{\Omega}_B$ follows easily from a patching argument as exhibited in \cite[Section 2.2]{savelyev}.
\end{proof}

Let $(S^2 \times B)_{hor}$ be a section of \eqref{eqn:B-family-fibration} such that for each $b \in B$, the image of $(0,b)$ lies in the projection of $P_{\phi_b} \setminus P_{\phi_b^{-1}}$. Define $(S^2 \times B \times X)_{hor}$ to be the preimage of $(S^2 \times B)_{hor}$ under $\pi_B$. By suitably choosing the section, we can assume that $(S^2 \times B \times X)_{hor}$ together with the restriction of $\tilde{\Omega}_B$ is isomorphic to the product $(S^2 \times B \times X, \omega_{S^2} \oplus 0 \oplus \omega)$. For $t \in S^2 \times B$, denote by $X_t$ the fiber of $(S^2 \times B \times X)_{hor} \to B \times X$, and define $P_t$ to be the preimage of $t$ under $\pi_B \circ \pi_{\tilde{S}}$. For $t \notin \{0\} \times B$, the space $(P_t, \tilde{\Omega}_b |_{P_t})$ with the projection to $S^2$ is isotopic as a Hamiltonian fibration to the product $(S^2 \times X, \omega_{S^2} \oplus \omega)$, which is further canonical when $t \in W_\infty \times B$.

Consider the projection
\beq\label{eqn:project-B}
p_B: \tilde{P}_B \xrightarrow{\pi_{\tilde{S}}} \tilde{S} \xrightarrow{\pi_B} \mb{CP}^1 \times B \to B,
\eeq
which is a smooth submersion such that the restriction of $\tilde{\Omega}_B$ to each fiber defines a symplectic form. Based on the discussion in \cite[Section 2.5.1]{savelyev}, there exists a smooth family of almost complex structures $\{J_b\}_{b \in B}$, varying smoothly in $b \in B$, such that:
\begin{enumerate}
    \item for each $b \in B$, the almost complex structure $J_b$ is compatible with $\tilde{\Omega}_B |_{p_B^{-1}(b)}$;
    \item endowing $S$ with the integrable complex structure, the projection map $\tilde{P}_B |_{p_B^{-1}(b)} \to S$ induced from $\pi_B$ is pseudo-holomorphic;
    \item over the region \eqref{eqn:trivial-region}, for any $b \in B$, the endomorphism $J_b$ is a product.
\end{enumerate}
We use such a family of almost complex structures to construct suitable moduli spaces of $J$-holomorphic curves. For any $b \in B$, let $A$ be the homology class represented by the sphere $\{\infty\} \times \{b\} \times S^2 \times \{pt\}$ in $W_{\infty} \times \{b\} \times S^2 \times X$ from \eqref{eqn:trivial-region}. Consider the moduli space
\beq\label{eqn:moduli-2-pointed}
\ov{\scrM}_{0,2}(\tilde{P}_B, \{J_b\}, A) := \{ (b, u) \ | \ b \in B, u \in \ov{\scrM}_{0,2}(p_B^{-1}(b), J_b, A) \}
\eeq
which comes with the evaluation map
\beq
\mathrm{ev}: \ov{\scrM}_{0,2}(\tilde{P}_B, \{J_b\}, A) \to \tilde{P}_B \times \tilde{P}_B.
\eeq

The virtual dimension of $\ov{\scrM}_{0,2}(\tilde{P}_B, \{J_b\}, A)$ is $k + 2n + 6$. Define the moduli space
\beq
\ov{\scrM}_h^B := \mathrm{ev}^{-1}(\tilde{P}_B \times (S^2 \times B \times X)_{hor})
\eeq
which has virtual dimension $k + 2n + 4$. Over the open subset \eqref{eqn:trivial-region}, we see that 
\beq
(S^2 \times B \times X)_{hor} \cap (W_\infty \times B \times S^2 \times X) = W_\infty \times B \times X.
\eeq
Then define the moduli space
\beq
\ov{\scrM}_h^B \supset \ov{\scrM}^B_{W_{\infty}} := \mathrm{ev}^{-1}((W_\infty \times B \times S^2 \times X) \times (W_\infty \times B \times X)).
\eeq
\begin{lemma}\label{lem:moduli-near-infty}
$\ov{\scrM}^B_{W_{\infty}}$ is regular, and the evaluation map defines a diffeomorphism onto the graph of the projection map
\beq
(W_\infty \times B \times S^2 \times X) \to (W_\infty \times B \times X).
\eeq
\end{lemma}
\begin{proof}
Recall that for any $b \in B$, we have asked $J_b$ to be of product form over \eqref{eqn:trivial-region}. Therefore, the regularity follows directly from the automatic transversality of the Cauchy--Riemann operator associated with the degree $1$ holomorphic map from $S^2$ to itself. As there is exactly one such curve after imposing the evaluation constraint, we see that $\mathrm{ev}$ defines a diffeomorphism onto to the manifold described above.
\end{proof}

\begin{rem}\label{rem:lazy}
  According to our assumption, $B$ has a stable complex structure. Therefore,  $\ov{\scrM}^B_{W_{\infty}}$ has a stable complex structure 
  because the symplectic structures on $S^2$ and $X$ induces a canonical class of stable complex structures on their tangent bundles. On the other hand, index theory equips $\ov{\scrM}^B_{W_{\infty}}$ with a stable complex structure because the linearized Cauchy--Riemann operator associated to each element in the moduli space is homotopic to a complex linear Fredholm operator. The diffeomorphism $\mathrm{ev}$ furthermore respects the stable complex structures.
\end{rem}

We similarly introduce
\beq
\ov{\scrM}_{\infty}^B := \mathrm{ev}^{-1}(\pi_B^{-1}(\{\infty\} \times B) \times (S^2 \times B \times X)_{hor}).
\eeq
Now we are ready to state the key geometric input from the theory of global charts \cite{abouzaid2021complex} of moduli spaces, whose proof is discussed in the next subsection.
\begin{thm}\label{thm:family-global-chart}
There exists a stably complex derived orbifold
\beq
{\mc D}_B = ({\mc U}_B, {\mc E}_B, {\mc S}_B),
\eeq
 whose section section is   ${\mc S}_B$ continuous \footnote{We require the section to be smooth in the definition of derived orbifold charts, so hopefully emphasizing the continuity here would not cause any confusion.}, 
together with a map
\beq
\widetilde{\mathrm{ev}}: {\mc U}_B \to \tilde{P}_B \times_{B} \tilde{P}_B
\eeq
satisfying the following properties.
\begin{enumerate}
    \item ${\mc U}_B$ is an effective orbifold, and there exists a homeomorphism ${\mc S}_B^{-1}(0) \cong \ov{\scrM}_{0,2}(\tilde{P}_B, \{J_b\}, A)$.
    \item The restriction of $\widetilde{\mathrm{ev}}$ along ${\mc S}_B^{-1}(0)$ agrees with $\mathrm{ev}: \ov{\scrM}_{0,2}(\tilde{P}_B, \{J_b\}, A) \to \tilde{P}_B \times_{B} \tilde{P}_B$.
      \item The projection map to either factor $ \tilde{P}_B$ is proper. 
\end{enumerate}
\end{thm}

\begin{cor}
The derived orbifold chart ${\mc D}_h = ({\mc U}_h, {\mc E}_h, {\mc S}_h)$ induced from the Cartesian square
\beq
\begin{tikzcd}
{\mc U}_{h} \arrow[d] \arrow[rr]                  &  & {\mc U}_B \arrow[d]            \\
\tilde{P}_B \times_{B} (S^2 \times B \times X)_{hor} \arrow[rr] &  & \tilde{P}_B \times_{B} \tilde{P}_B
\end{tikzcd}
\eeq
inherits a stable complex structure 
and ${\mc S}_{h}^{-1}(0)$ is homeomorphic to $\ov{\scrM}_{h}^B$.


Similarly, the derived orbifold chart ${\mc D}_{\infty} = ({\mc U}_{\infty}, {\mc E}_{\infty}, {\mc S}_{\infty})$ specified by the fiber product diagram
\beq
\begin{tikzcd}
{\mc U}_{\infty} \arrow[d] \arrow[rr]                                           &  & {\mc U}_B \arrow[d]            \\
\pi_B^{-1}(\{\infty\} \times B) \times_{B} (S^2 \times B \times X)_{hor} \arrow[rr] &  & \tilde{P}_B \times_{B} \tilde{P}_B
\end{tikzcd}
\eeq
is stably complex, and ${\mc S}_{\infty}^{-1}(0)$ is homeomorphic to $\ov{\scrM}_{\infty}^B$.
\end{cor}
\begin{proof}
This is a direct consequence of Lemma \ref{lem:moduli-near-infty} and Theorem \ref{thm:family-global-chart}. 
\end{proof}

By applying the perturbation-resolution algorithm to the derived orbifold chart ${\mc D}_{h}$, we obtain a stably  almost complex manifold
\beq
\ov{\mathbf{M}}^{B}_{h}.
\eeq
The composition of $\widetilde{\mathrm{ev}}$ with the natural map $\ov{\mathbf{M}}^B_{h} \to {\mc U}_{h}$ produces a pair of maps
\beq
\begin{tikzcd}
    & \ov{\mathbf{M}}^B_{h} \arrow[ld, "\mathrm{ev}_1"'] \arrow[rd, "\mathrm{ev}_2"] &                               \\
\tilde{P}_B &                                                                                  & (S^2 \times B \times X)_{hor}.
\end{tikzcd}
\eeq
We compose with the projection from $ S^2 \times B \times X$ to $X$, to obtain a manifold which is proper over $\tilde{P}_B $.
Consider the inclusion $X \times B \to \tilde{P}_B$ associated to the point $(\infty,\infty)$ in $\mathbb{C} \mathbb{P}^1 \times \mathbb{C} \mathbb{P}^1 $. According to Lemma \ref{lem:moduli-near-infty}, the evaluation map from $\ov{\mathbf{M}}^B_{h}$ to $\tilde{P}_B  $  is transverse to this submanifold, and the associated submanifold of $X \times B \times X$ is therefore the product of $B$ with the diagonal.

Moreover, given an inclusion $B \to B'$ of bases, the perturbation-resolution algorithm yield a cobordism of almost complex manifold
\beq
\ov{\mathbf{M}}^{B,B'}_{h}.
\eeq
between $ \mathbf{M}^B_{h}$ and the fibre product of $\mathbf{M}^{B'}_{h}$ with the inclusion of $\tilde{P}_{B'} $ in $ \tilde{P}_{B'}$. Indeed, the space of data for defining the derived orbifold $ {\mc D}_B $ is contractible, and any choice of data for defining $ {\mc D}_{B'}$ thus pulls back to define a derived orbifold which is cobordant to $ {\mc D}_B $ after stabilization. The claim then follows form the compatibility of the perturbation-resolution algorithm with cobordisms and stabilizations.  This cobordism will again have the property that its fibre product with $X \times B$ gives the product of $B \times I$ with the diagonal of $X$.

We can now complete the proof of our application, modulo the geometric construction of parametric families of derived orbifolds: 
   \begin{proof}[Proof of Proposition \ref{thm:splitting-B}]
     Choose a sequence $B_1 \to B_{2} \to \cdots $ of embeddings of compact stably almost manifolds with boundary, lying over $ \Omega \mathrm{Ham}(M, \omega) \times S^2$, with the property that the map from $\colim B_i $ is a homotopy equivalence. Given the preceding discussion, the result now follows from Corollary \ref{cor:composite_filtered_fibre_product}. 
\end{proof}

\subsubsection{Global Kuranishi charts for families} \label{sec:glob-kuran-charts}
We move on to explain the constructions in Theorem \ref{thm:family-global-chart}. It is almost the same as the original construction described in \cite[Section 6]{abouzaid2021complex}, so we will be brief and refer the reader to \emph{loc. cit.} for more detailed arguments. As noted in \cite[Section 6.9]{abouzaid2021complex}, the global Kuranishi chart for a moduli space of $J$-holomorphic curves with marked points can be obtained by pulling back such a chart for a moduli space of $J$-holomorphic curves without marked points via the forgetful map. Therefore, we will focus on constructing global Kuranishi charts for $\ov{\scrM}_{0,0}(\tilde{P}_B, \{J_b\}, A)$, which is the unmarked version of \eqref{eqn:moduli-2-pointed}.

\begin{lemma}\label{lem:integral-2-form}
There exists a closed $2$-form $\tilde{\Omega}_B'$ defined on $\tilde{P}_B$ such that:
\begin{enumerate}
    \item the restriction of $\tilde{\Omega}_B'$ to each fiber of $p_B$ from \eqref{eqn:project-B} defines a symplectic form which is tamed by $J_b$;
    \item the cohomology class $[\tilde{\Omega}_B' |_{p_B^{-1}(b)}]$ lies in the image of the natural inclusion map of cohomology groups $H^2(p_B^{-1}(b);{\mb Z}) \to H^2(p_B^{-1}(b);{\mb R})$.
\end{enumerate}
\end{lemma}
\begin{proof}
For point $b \in B$, using the description from \cite[Section 2.2]{savelyev}, we see that the cohomology class of $[\tilde{\Omega}_B |_{p_B^{-1}(b)}]$ is uniquely determined by the cohomology class of the symplectic form on the complex surface $S$ and the class $[\omega] \in H^2(X;{\mb R})$. By choosing a sufficiently fine rational approximation and applying rescaling, we can find the desired $\tilde{\Omega}_B' \in H^2(\tilde{P}_B; {\mb R})$.
\end{proof}

From now on, let us fix:
\begin{enumerate}
    \item a closed $2$-form $\tilde{\Omega}_B'$ as in Lemma \ref{lem:integral-2-form};
    \item a positive integer $d := \langle \tilde{\Omega}_B', A \rangle$.
\end{enumerate}

Denote by $\scrF \subset \ov{\scrM}_{0,0}(\mb{CP}^d, d)$ the Zariski open subset of the moduli space of genus $0$ degree $d$ holomorphic maps into $\mb{CP}^d$ which consists of maps whose image is not contained in any linear subspace, and let $\scrC \to \scrF$ be the universal curve. By \cite[Lemma 6.4]{abouzaid2021complex}, both $\scrF$ and $\scrC$ are smooth complex quasi-projective varieties.

Using the integral $2$-form $\tilde{\Omega}_B'$ defined over $\tilde{P}_B$, for any pair $(b,u) \in \ov{\scrM}_{0,0}(\tilde{P}_B, \{J_b\}, A)$ such that $u: \Sigma \to p^{-1}_B(b)$ is a $J_b$-holomorphic map defined over a nodal genus $0$ curve $\Sigma$, by \cite[Lemma 6.8]{abouzaid2021complex}, there exists a holomorphic Hermitian line bundle $L_{u^* \tilde{\Omega}_B'} \to \Sigma$ whose Chern connection has curvature $-2\pi i u^* \tilde{\Omega}_B'$, and it is uniquely defined up to isomorphism. Moreover, due to the positivity of $u^* \tilde{\Omega}_B'$, as implied by the fact that $\tilde{\Omega}_B'|_{p^{-1}_B(b)}$ is tamed by $J_b$, we know that $H^1(\Sigma, L_{u^* \tilde{\Omega}_B'}) = 0$ and $\dim_{\mb C} H^0(\Sigma, L_{u^* \tilde{\Omega}_B'}) = d+1$. Upon choosing a basis $\{s_0, \dots, s_d \}$ of $H^0(\Sigma, L_{u^* \tilde{\Omega}_B'})$, the projectivization $[s_0 : \cdots: s_d]$ defines a holomorphic map $\Sigma \to \mb{CP}^d$, which actually lies in $\scrF \subset \ov{\scrM}_{0,0}(\mb{CP}^d, d)$. This motivates the following definition.

\begin{defn}
A \emph{framed genus $0$ curve in $p_B: \tilde{P}_B \to B$} is a quadruple $(b,u,\Sigma, F)$ where
\begin{enumerate}
    \item $\Sigma$ is a genus $0$ nodal curve;
    \item $b \in B$, and $u \to p_B^{-1}(b)$ is a smooth map in the class $A$ such that $2$-form $u^* \tilde{\Omega}_B'$ is positive over any unstable component of $\Sigma$;
    \item $F = (s_0, \dots, s_d)$ forms a basis of $H^0(\Sigma, L_{u^* \tilde{\Omega}_B'})$ such that the Hermitian matrix
    \beq
    \scrH(b,u,\Sigma, F) = (\int_{\Sigma} \langle s_i, s_j \rangle u^* \tilde{\Omega}_B' )_{i,j}
    \eeq
    has positive eigenvalues.
\end{enumerate}
The holomorphic map $\iota_{F} := [s_0 :  \cdots: s_d]: \Sigma \to \mb{CP}^d$ is called the \emph{domain map}.
\end{defn}
We can define the isomorphism relation between a pair of framed curves in the standard way. 

Next, we choose
\begin{enumerate}
    \item a Hermitian line bundle $\scrL \to \scrC$ whose restriction to each fiber is ample and the $PU(d+1)$-action on $\scrC$ induced from the $PGL(d+1)$-action on $\mb{CP}^d$ is lifted to a $PU(d+1)$-action on $\scrL$ which preserves the Hermitian structure;
    \item a sufficiently large integer $k \gg 1$.
\end{enumerate}

\begin{defn}
The \emph{thickened moduli space} $\scrT$ (which is endowed with the Gromov topology after performing a graph construction) is the moduli space of tuples $(b,u,\Sigma, F, \eta)$ for which:
\begin{enumerate}
    \item $(b,u,\Sigma, F)$ is a framed genus $0$ curve;
    \item $\eta \in H^0(\ov{\mathrm{Hom}}(\iota_F^* T\scrC, u^* Tp_B^{-1}(b)) \otimes \iota_F^* \scrL^k) \otimes_{\mb C} \ov{H^0(\iota_F^* \scrL^k)}$;
    \item The equation
    \beq\label{eqn:perturbed-J-hol}
    \ov{\partial}_{J_b} u + \langle \eta \rangle \circ d\iota_F = 0 
    \eeq
    holds where $\langle \eta \rangle$ is the image of the Hermitian pairing
    \beq
    H^0(\ov{\mathrm{Hom}}(\iota_F^* T\scrC, u^* Tp_B^{-1}(b)) \otimes \iota_F^* \scrL^k) \otimes_{\mb C} \ov{H^0(\iota_F^* \scrL^k)} \to C^\infty(\ov{\mathrm{Hom}}(\iota_F^* T\scrC, u^* Tp_B^{-1}(b))).
    \eeq
\end{enumerate}

The \emph{obstruction bundle} $\scrE$ is defined to be the bundle over $\scrT$ whose fiber over $(b,u,\Sigma, F, \eta)$ is
\beq
H^0(\ov{\mathrm{Hom}}(\iota_F^* T\scrC, u^* Tp_B^{-1}(b)) \otimes \iota_F^* \scrL^k) \otimes_{\mb C} \ov{H^0(\iota_F^* \scrL^k)} \oplus \scrH_{d+1},
\eeq
where $\scrH_{d+1}$ is the space of $(d+1) \times (d+1)$ Hermitian matrices.

The \emph{Kuranishi map} is the section $\mathfrak{s}: \scrT \to \scrE$ defined by
\beq
\mathfrak{s}: (b,u,\Sigma, F, \eta) \mapsto (\eta, \exp^{-1}(\scrH(b,u,\Sigma, F))),
\eeq
where $\exp^{-1}$ is the inverse of the exponential map, which is well-defined over the space of positive definite Hermitian matrices.
\end{defn}

If $k \gg 1$, the bundle $\scrE$ is indeed a vector bundle by positivity and compactness of $B$. The thickened moduli space $\scrT$ admits a continuous map
\beq
\begin{aligned}
pr: \scrT &\to \scrF \times B \\
(b,u,\Sigma, F, \eta) &\mapsto (\iota_F, b).
\end{aligned}
\eeq
Introduce the notation $G = PU(d+1)$. Then $\scrT$ admits a $G$-action and $\scrE$ is a $G$-equivariant vector bundle. We summarize some salient features of $(G, \scrT, \scrE, \mathfrak{s})$ as follows.
\begin{prop}\label{prop:global-chart-property}
\begin{enumerate}
    \item The induced topology on $\ov{\scrM}_{0,0}(\tilde{P}_B, \{J_b\}, A)$ as the zero locus $\mathfrak{s}^{-1}(0)$ agrees with the Gromov topology.
    \item The $G$-action on $\scrT$ has finite stabilizers, and the orders of stabilizers are uniformly bounded.
    \item For $k$ sufficiently large, the linearization of \eqref{eqn:perturbed-J-hol} is surjective over a (precompact) open neighborhood of $\mathfrak{s}^{-1}(0)$. Therefore, up to a further shrinking, $(G, \scrT, \scrE, \mathfrak{s})$ defines a global Kuranishi chart for $\ov{\scrM}_{0,0}(\tilde{P}_B, \{J_b\}, A)$.
    \item The projection $pr: \scrT \to \scrF \times B$ is a topological submersion which defines a $C^1_{loc}$ $G$-bundle structure in the sense of \cite[Section 4.5]{abouzaid2021complex}.
    \item Using the structure described in (4), the virtual $G$-equivariant vector bundle 
    \beq
    T^v \scrT \oplus \mathfrak{g} \oplus pr^* T(\scrF \times B) - \scrE
    \eeq
    defines a class in the \emph{complex} G-equivariant $K$-theory.
\end{enumerate}
\end{prop}
\begin{proof}
These are all proven in \cite[Section 6]{abouzaid2021complex}. The compactness of $B$ makes sure that the H\"ormander techniques used there can also work in the family setting due to the openness of transversality. Moreover, because $TB$ is stably complex, the complex orientation can be constructed. As for the $C^1_{loc}$ $G$-bundle property, the topological gluing construction in \cite[Appendix B]{pardon-VFC} can be done without difficulties in the family setting, whose counterpart in the setting of Hamiltonian Floer theory can be found in \cite[Appendix C]{pardon-VFC}. Accordingly, the arguments of \cite[Theorem 6.1]{abouzaid2021complex} carry over here.
\end{proof}

\begin{proof}[Proof of Theorem \ref{thm:family-global-chart}]
This follows from the proof of \cite[Proposition 5.35]{abouzaid2021complex}, which shows that the Property (4) ensures that one can apply Lashof's $G$-equivariant stable smoothing theory to equip a stabilization of the global Kuranishi chart $(G, \scrT, \scrE, \mathfrak{s})$ with a smooth structure, except that the stabilized Kuranishi section may only be continuous. Pulling back such a global Kuranishi chart via the forgetful map
\beq
\ov{\scrM}_{0,2}(\tilde{P}_B, \{J_b\}, A) \to \ov{\scrM}_{0,0}(\tilde{P}_B, \{J_b\}, A),
\eeq
\cite[Lemma 4.5]{abouzaid2021complex} shows that a further stabilization and shrinking can make sure that the evaluation map on the thickened moduli space becomes submersive. By taking the corresponding derived orbifold chart as in \eqref{eqn:quotient-d-chart}, we obtain ${\mc D}_B = ({\mc U}_B, {\mc E}_B, {\mc S}_B)$.
\end{proof}

\bibliography{ref}

\bibliographystyle{amsalpha}

\end{document}